\documentclass{article}

\usepackage[T2A]{fontenc}
\usepackage[utf8]{inputenc}
\usepackage[russian]{babel}
\usepackage{amsfonts}
\usepackage{amsthm}
\usepackage{amsmath}
\usepackage{amssymb}
\usepackage{cancel}
\usepackage{tabu}
\usepackage{array}
\usepackage{tikz}
\usepackage{hyperref}
\usepackage{makecell}
\usepackage{xcolor}
\usepackage{graphicx}
\usepackage[final]{pdfpages}

\graphicspath{ {./images/} }

\title{Эргодичность границы Мартина графа Юнга--Фибоначчи. II}

\author{В. Ю. Евтушевский}

\begin{document}

\maketitle

\tableofcontents

\newpage

\newtheorem{Lemma}{Лемма}

\newtheorem{Alg}{Алгоритм}

\newtheorem{Col}{Следствие}

\newtheorem{theorem}{Теорема}

\newtheorem{Def}{Определение}

\newtheorem{Prop}{Утверждение}

\newtheorem{Problem}{Задача}

\newtheorem{Zam}{Замечание}

\newtheorem{Oboz}{Обозначение}

\newtheorem{Ex}{Пример}

\newtheorem{Nab}{Наблюдение}

\section{Введение}

Эта работа является продолжением работы \cite{Evtuh3}. Здесь мы завершаем доказательство гипотезы Керова -- Гудмана об эргодичности всех мер из границы Мартина графа Юнга -- Фибоначчи. 

\renewcommand{\labelenumi}{\arabic{enumi}$)$}
\renewcommand{\labelenumii}{\arabic{enumi}.\arabic{enumii}$^\circ$}
\renewcommand{\labelenumiii}{\arabic{enumi}.\arabic{enumii}.\arabic{enumiii}$^\circ$}

Напомним, что вершинами ранга $n$ графа Юнга -- Фибоначчи служат слова над алфавитом $\{1,2\}$ с данной суммой цифр $n$. Рёбра ``вверх'' из данного
слова $x$ ведут в слова, получаемые из $x$ одной из двух операций:
\begin{enumerate}
    \item  заменить самую левую единицу на двойку;

    \item вставить единицу левее чем самая левая единица.
\end{enumerate}

\begin{center}
\includegraphics[width=12cm, height=10cm]{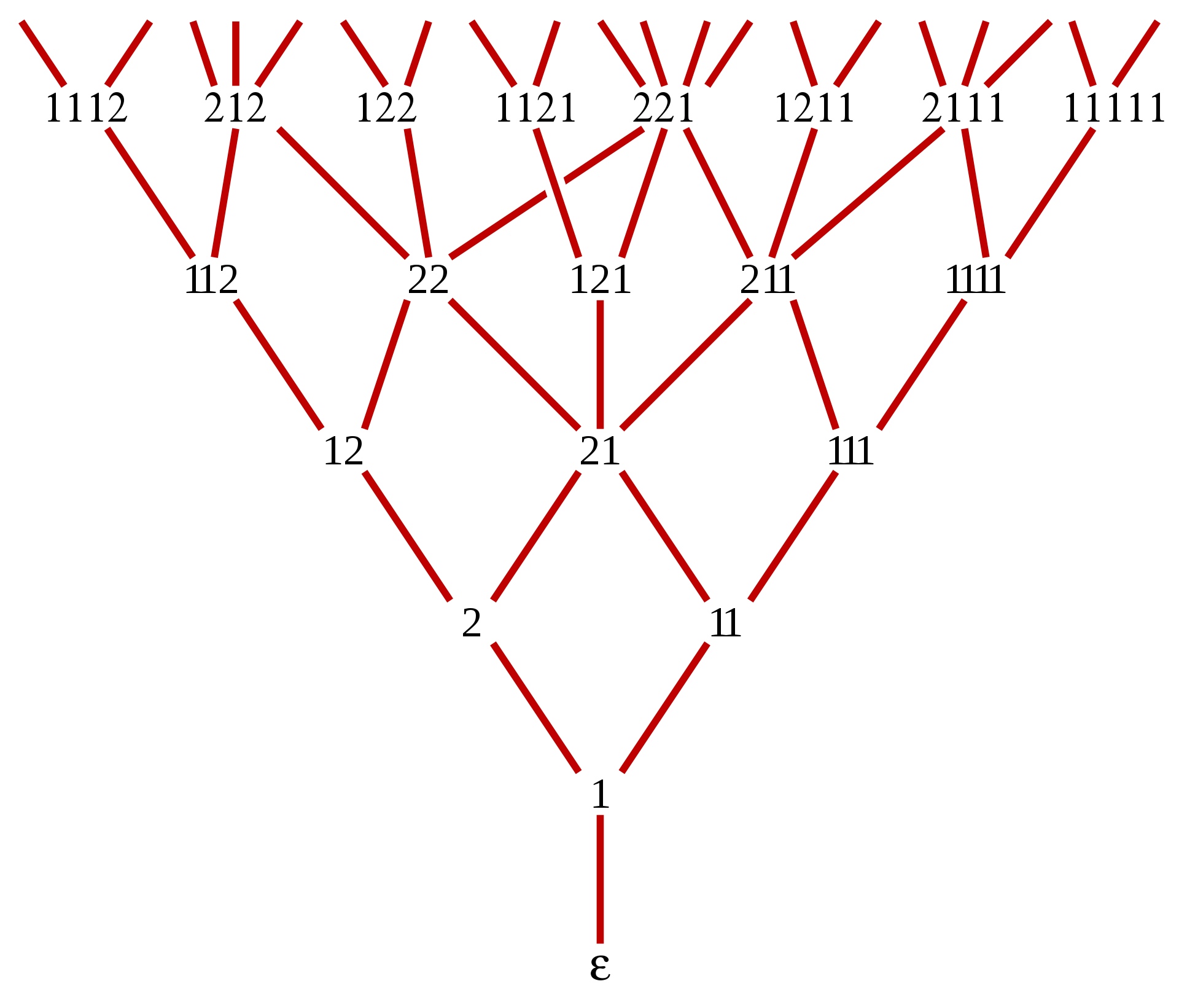}
\end{center}

Керов и Гудман доказали, что список интересующих нас центральных мер исчерпывается следующими мерами:
\begin{enumerate}
    \item Мера Планшереля: мера множества путей, проходящих через данную вершину $v$, равна $\frac{d(\varepsilon,v)^2}{n!}$, где $d(u,v)$ -- количество путей ``вниз'' из $v$ в $u$.
    \item Меры $\mu_{\{w_i'\}}$, параметризующиеся некоторой бесконечной последовательностью вершин графа Юнга--Фибоначчи. Нам удобнее другое эквивалентное определение в терминах некоторого бесконечного слова $w$ (содержащего ``достаточно мало'' двоек) и числа $\beta\in(0,1]$. См. подробнее Лемму \ref{yavsyo}.
\end{enumerate}

Доказательство эргодичности меры Планшереля было получено Керовым и Гнединым \cite{KerGned}. Оно основано на следующей Лемме: мера Планшереля сосредоточена на путях, вершины которых содержат ``достаточно много'' двоек. Мы доказываем аналогичное утверждение для остальных мер $\mu_{w,\beta}$, откуда стандартным рассуждением получается эргодичность.

Ключевыми утверждениями являются Теорема \ref{t2} и Теорема \ref{t3}, при этом используется результат работы Бочкова -- Евтушевского \cite{Evtuh3}.

\newpage

\renewcommand{\labelenumi}{\arabic{enumi}$^\circ$}
\renewcommand{\labelenumii}{\arabic{enumi}.\arabic{enumii}$^\circ$}
\renewcommand{\labelenumiii}{\arabic{enumi}.\arabic{enumii}.\arabic{enumiii}$^\circ$}

\section{Подготовка к доказательству гипотезы}

\renewcommand{\labelenumi}{\arabic{enumi}$^\circ$}
\renewcommand{\labelenumii}{\arabic{enumi}.\arabic{enumii}$^\circ$}
\renewcommand{\labelenumiii}{\arabic{enumi}.\arabic{enumii}.\arabic{enumiii}$^\circ$}

\begin{Oboz}
Пусть $\mathbb{YF}$ -- это граф Юнга -- Фибоначчи.
\end{Oboz}

\begin{Def}
Пусть $x \in \mathbb{YF}$. Тогда номером $x$ будем называть слово из единиц и двоек, соответствующее вершине $x$.
\end{Def}

\begin{Def}
Пусть $\{\alpha_i\}_{i=1}^\infty\in\{1,2\}^\infty$ -- бесконечная последовательность из единиц и двоек. Этой последовательности сопоставим ``бесконечно удалённую вершину'' графа Юнга -- Фибоначчи с номером $x=\ldots\alpha_2\alpha_1$. 
\end{Def}

\begin{Oboz}
Пусть $\mathbb{YF}_\infty$ -- это множество ``бесконечно удалённых вершин'' графа Юнга--Фибоначчи. 
\end{Oboz}

\begin{Oboz}

$\;$

\begin{itemize}
    \item Пусть $x\in \mathbb{YF},$ $n\in\mathbb{N}_0$, $\{\alpha_i\}_{i=1}^n\in\{1,2\}^n:$ номер $x$ -- это $\alpha_n\ldots\alpha_2\alpha_1$. Тогда будем писать, что $x=\alpha_n\ldots\alpha_2\alpha_1$.
    \item Пусть $x\in \mathbb{YF}_\infty$, $\{\alpha_i\}_{i=1}^\infty\in\{1,2\}^\infty:$ номер $x$ -- это $\ldots\alpha_2\alpha_1$. Тогда будем писать, что $x=\ldots\alpha_2\alpha_1$.
\end{itemize}
\end{Oboz}

\begin{Oboz}

$\;$

\begin{itemize}
    \item Пусть $x \in \mathbb{YF}$. Тогда сумму цифр в номере $x$ обозначим за $|x|$.
    \item Пусть $x \in \mathbb{YF}_\infty$. Тогда скажем, что $|x|=\infty$.
\end{itemize}
\end{Oboz}

\begin{Oboz}
Пусть $n\in\mathbb{N}_0$. Тогда
$$\mathbb{YF}_n:=\left\{v\in\mathbb{YF}:\;|v|=n\right\}.$$
\end{Oboz}

\begin{Zam}

$\;$

\begin{itemize}
    \item Пусть $x\in\mathbb{YF}$. Тогда $|x|$ -- это ранг вершины $x$ в графе Юнга -- Фибоначчи.
    \item Пусть $n\in\mathbb{N}_0$. Тогда $\mathbb{YF}_n$ -- это множество вершин графа Юнга -- Фибоначчи ранга $n$.
\end{itemize}
\end{Zam}

\begin{Oboz} 
Пусть $n,m\in\mathbb{N}_0:$ $m\le n$. Тогда
\begin{itemize}
    \item $$\overline{n}:=\{0,1,\ldots,n\};$$
    \item $$\overline{m,n}:=\{m,m+1\ldots,n\}.$$
\end{itemize}
\end{Oboz}

\begin{Def}\label{vniz}
Пусть $x,y,\{y_i\}_{i=0}^n\in \mathbb{YF}:$ $n\in\mathbb{N}_0$, 
$$y=y_0y_1y_2...y_n=x$$
-- это такой путь в графе Юнга--Фибоначчи, что $\forall i \in \overline{n}$ $$|y_i|=|y|-i.$$
Тогда путь
$$y=y_0y_1y_2...y_n=x$$
назовём $yx$-путём ``вниз'' в $\mathbb{YF}$.
\end{Def}

\begin{Zam}
В Определении \ref{vniz} $n=|y|-|x|$.
\end{Zam}

\begin{Oboz}
Пусть $x,y \in \mathbb{YF}$. Тогда количество $yx$-путей ``вниз'' в $\mathbb{YF}$ будем обозначать как $d(x,y)$.
\end{Oboz}

\begin{Zam}
Пусть $x,y \in \mathbb{YF}:$ $|y|<|x|$. Тогда
$$d(x,y)=0.$$
\end{Zam}

\begin{Oboz}
Пусть $x,y\in \mathbb{YF}$. Тогда множество всех $yx$-путей ``вниз'' обозначим за $T(x,y)$.
\end{Oboz}

\begin{Oboz}
$$T(\mathbb{YF}):=\bigcup_{\left\{(x,y)\in \mathbb{YF}^2\right\}}{T(x,y)}.$$
\end{Oboz}

\begin{Oboz}
Пусть $x,y\in \mathbb{YF},$ $t\in T(\mathbb{YF}):$ $t\in T(x,y)$. Тогда будем обозначать вершины этого пути как 
$$y=t(|y|), \; t(|y|-1),\ldots,\; t(|x|+1),\; t(|x|)=x,$$
а также считать, что если $z\in\left(\mathbb{N}_0\textbackslash\overline{|x|,|y|}\right)$, то $t(z)$ не определено.
\end{Oboz}

\begin{Zam}
Ясно, что $t(z)$ -- это вершина, через которую путь $t$ проходит на уровне $z\in\left(\mathbb{N}_0\textbackslash\overline{|x|,|y|}\right)$.
\end{Zam}

\begin{Oboz}

$\;$

\begin{itemize}
    \item Пусть $x \in \mathbb{YF}$. Тогда количество цифр в номере $x$ обозначим за $\#x$.
    \item Пусть $x \in \mathbb{YF}_\infty$. Тогда скажем, что $\#x=\infty$.
\end{itemize}
\end{Oboz}

\begin{Oboz}
Пусть $x \in \left(\mathbb{YF}\cup \mathbb{YF}_\infty\right)$. Тогда:
\begin{itemize} 
    \item Количество единиц в номере $x$ обозначим за $e(x)$;
    \item Количество двоек в номере $x$ обозначим за $d(x)$.
\end{itemize}
\end{Oboz}

\begin{Zam}
Пусть $x \in \mathbb{YF}_\infty$. Тогда:
\begin{itemize} 
    \item $d(x)$ может быть равно бесконечности;
    \item $e(x)$ может быть равно бесконечности.
\end{itemize}
\end{Zam}

\begin{Zam}
Пусть $x \in \left(\mathbb{YF}\cup \mathbb{YF}_\infty\right)$. Тогда:
\begin{itemize} 
    \item $e(x)+d(x)=\#x;$
    \item $e(x)+2d(x)=|x|;$
    \item $\#x+d(x)=|x|.$
\end{itemize}
\end{Zam}

\begin{Oboz}
$$f(x,y,z): \left\{(x,y,z)\subseteq\mathbb{YF}\times \mathbb{N}_0\times\mathbb{N}_0:\;y\in\overline{|x|},\;z\in\overline{\#x}\right\}\to\mathbb{R}$$
-- это функция, определённая следующим образом:

При $z=0$:
\begin{itemize}
\item Если $x\in \mathbb{YF}$ представляется в виде $x=\alpha_1...\alpha_m\alpha_{m+1}...\alpha_n$, где $|\alpha_{m+1}...\alpha_n|=y,$ $ \alpha_i \in \{1,2 \}$, то
$$f(x,y,0):=\frac{1}{(\alpha_{m+1})(\alpha_{m+1}+\alpha_{m+2})...(\alpha_{m+1}+...+\alpha_{n})}\cdot(-1)^{n-m}\cdot$$
$$\cdot\frac{1}{(\alpha_m)(\alpha_m+\alpha_{m-1})(\alpha_m+\alpha_{m-1}+\alpha_{m-2})...(\alpha_m+...+\alpha_1)}=$$
$$=\frac{1}{(-\alpha_{m+1})(-\alpha_{m+1}-\alpha_{m+2})...(-\alpha_{m+1}-...-\alpha_{n})}\cdot$$
$$\cdot\frac{1}{(\alpha_m)(\alpha_m+\alpha_{m-1})(\alpha_m+\alpha_{m-1}+\alpha_{m-2})...(\alpha_m+...+\alpha_1)};$$
\item Если $x\in \mathbb{YF}$ не представляется в виде $x=\alpha_1...\alpha_m\alpha_{m+1}...\alpha_n$, где  $|\alpha_{m+1}...\alpha_n|=y,$ $\alpha_i \in \{1,2 \}$, то
$$f(x,y,0)=0.$$
\end{itemize}

При $z>0$ (рекурсивное определение):
\begin{itemize}
    \item Если $y=0$, то
$$f(x1,0,z)=f(x1,0,0);$$
    \item Если $y>0$, то
$$f(x1,y,z)=f(x1,y,0)+f(x,y-1,z-1);$$
    \item 
\begin{equation*}
f(x2,y,z)= 
 \begin{cases}
   $$\frac{f(x11,y,z+1)}{1-y}$$ &\text {если $y \ne 1$}\\
   0 &\text{если $y=1$.}
 \end{cases}
\end{equation*}
\end{itemize}
\end{Oboz}

\begin{Ex} \label{exf1}
Значения $f(x,y,z)$ при всех возможных тройках $(x,y,z)\subseteq\mathbb{YF}\times \mathbb{N}_0\times\mathbb{N}_0:$ $x=21221,$ $y\in\overline{|x|},$ $z=0:$

\begin{itemize}
\item $$f(21221,0,0)=\frac{1}{1\cdot3\cdot5\cdot6\cdot8}=\frac{1}{720};$$
\item $$f(21221,1,0)=\frac{1}{(-1)\cdot2\cdot4\cdot5\cdot7}=-\frac{1}{280};$$
\item $$f(21221,2,0)=0;$$
\item $$f(21221,3,0)=\frac{1}{(-2)\cdot(-3)\cdot2\cdot3\cdot5}=\frac{1}{180};$$
\item $$f(21221,4,0)=0;$$
\item $$f(21221,5,0)=\frac{1}{(-2)\cdot(-4)\cdot(-5)\cdot1\cdot3}=-\frac{1}{120};$$
\item $$f(21221,6,0)=\frac{1}{(-1)\cdot(-3)\cdot(-5)\cdot(-6)\cdot2}=\frac{1}{180};$$
\item $$f(21221,7,0)=0;$$
\item $$f(21221,8,0)=\frac{1}{(-2)\cdot(-3)\cdot(-5)\cdot(-7)\cdot(-8)}=-\frac{1}{1680}.$$
\end{itemize}
\end{Ex}

\begin{Ex} \label{exf2}
Значения $f(x,y,z)$ при всех возможных тройках $(x,y,z)\subseteq\mathbb{YF}\times \mathbb{N}_0\times\mathbb{N}_0:$ $|x|\in\overline{4},$ $y\in\overline{|x|},$ $z\in\overline{\#x}:$

\begin{itemize}
    \item $x=\varepsilon$
\begin{center}
\begin{tabular}{ | m{0.5cm} || m{0.5cm} | } 
  \hline
 & $y=0$ \\
  \hline \hline
$z=0$ & $1$ \\ 
  \hline
\end{tabular}
\end{center}
    \item $x=1$
\begin{center}
\begin{tabular}{ | m{0.5cm} || m{0.5cm} | m{0.5cm} | } 
  \hline
 & $y=0$ & $y=1$ \\
  \hline \hline
$z=0$ & $1$ & $-1$ \\
  \hline
$z=1$ & $1$ & $0$ \\ 
  \hline
\end{tabular}
\end{center}
    \item $x=2$
\begin{center}
\begin{tabular}{ | m{0.5cm} || m{0.5cm} | m{0.5cm} | m{0.5cm} | } 
  \hline
 & $y=0$ & $y=1$ & $y=2$ \\
  \hline \hline
$z=0$ & $\frac{1}{2}$ & $0$ & $-\frac{1}{2}$\\
  \hline
$z=1$ & $\frac{1}{2}$ & $0$ &$-\frac{1}{2}$ \\ 
  \hline
\end{tabular}
\end{center}
    \item $x=11$
\begin{center}
\begin{tabular}{ | m{0.5cm} || m{0.5cm} | m{0.5cm} | m{0.5cm} | } 
  \hline
 & $y=0$ & $y=1$ & $y=2$ \\
  \hline \hline
$z=0$ & $\frac{1}{2}$ & $-1$ & $\frac{1}{2}$\\
  \hline
$z=1$ & $\frac{1}{2}$ & $0$ &$-\frac{1}{2}$ \\ 
  \hline
$z=2$ & $\frac{1}{2}$ & $0$ & $\frac{1}{2}$\\ 
  \hline
\end{tabular}
\end{center}
    \item $x=12$
\begin{center}
\begin{tabular}{ | m{0.5cm} || m{0.5cm} | m{0.5cm} | m{0.5cm} | m{0.5cm} | } 
  \hline
 & $y=0$ & $y=1$ & $y=2$ & $y=3$\\
  \hline \hline
$z=0$ & $\frac{1}{6}$ & $0$ & $-\frac{1}{2}$ &$\frac{1}{3}$\\
  \hline
$z=1$ & $\frac{1}{6}$ & $0$ &$-\frac{1}{2}$ &$\frac{1}{3}$\\ 
  \hline
$z=2$ & $\frac{1}{6}$ & $0$ & $-\frac{1}{2}$ &$-\frac{1}{6}$\\ 
  \hline
\end{tabular}
\end{center}
    \item $x=21$
\begin{center}
\begin{tabular}{ | m{0.5cm} || m{0.5cm} | m{0.5cm} | m{0.5cm} | m{0.5cm} | } 
  \hline
 & $y=0$ & $y=1$ & $y=2$ & $y=3$\\
  \hline \hline
$z=0$ & $\frac{1}{3}$ & $-\frac{1}{2}$ & $0$ &$\frac{1}{6}$\\
  \hline
$z=1$ & $\frac{1}{3}$ & $0$ &$0$ &$-\frac{1}{3}$\\ 
  \hline
$z=2$ & $\frac{1}{3}$ & $0$ & $0$ &$-\frac{1}{3}$\\ 
  \hline
\end{tabular}
\end{center}
    \item $x=111$
\begin{center}
\begin{tabular}{ | m{0.5cm} || m{0.5cm} | m{0.5cm} | m{0.5cm} | m{0.5cm} | } 
  \hline
 & $y=0$ & $y=1$ & $y=2$ & $y=3$\\
  \hline \hline
$z=0$ & $\frac{1}{6}$ & $-\frac{1}{2}$ & $\frac{1}{2}$ &$-\frac{1}{6}$\\
  \hline
$z=1$ & $\frac{1}{6}$ & $0$ &$-\frac{1}{2}$ &$\frac{1}{3}$\\ 
  \hline
$z=2$ & $\frac{1}{6}$ & $0$ & $\frac{1}{2}$ &$-\frac{2}{3}$\\ 
  \hline
$z=3$ & $\frac{1}{6}$ & $0$ & $\frac{1}{2}$ &$\frac{1}{3}$\\ 
  \hline
\end{tabular}
\end{center}
    \item $x=112$
\begin{center}
\begin{tabular}{ | m{0.5cm} || m{0.5cm} | m{0.5cm} | m{0.5cm} | m{0.5cm} | m{0.5cm} | } 
  \hline
 & $y=0$ & $y=1$ & $y=2$ & $y=3$ & $y=4$\\
  \hline \hline
$z=0$ & $\frac{1}{24}$ & $0$ & $-\frac{1}{4}$ &$\frac{1}{3}$ &$-\frac{1}{8}$\\
  \hline
$z=1$ & $\frac{1}{24}$ & $0$ &$-\frac{1}{4}$ &$\frac{1}{3}$ &$-\frac{1}{8}$\\ 
  \hline
$z=2$ & $\frac{1}{24}$ & $0$ & $-\frac{1}{4}$ &$-\frac{1}{6}$ &$\frac{5}{24}$\\ 
  \hline
$z=3$ & $\frac{1}{24}$ & $0$ & $-\frac{1}{4}$ &$-\frac{1}{6}$ &$-\frac{1}{8}$\\ 
  \hline
\end{tabular}
\end{center}
    \item $x=22$
\begin{center}
\begin{tabular}{ | m{0.5cm} || m{0.5cm} | m{0.5cm} | m{0.5cm} | m{0.5cm} | m{0.5cm} | } 
  \hline
 & $y=0$ & $y=1$ & $y=2$ & $y=3$ & $y=4$\\
  \hline \hline
$z=0$ & $\frac{1}{8}$ & $0$ & $-\frac{1}{4}$ & $0$ &$\frac{1}{8}$\\
  \hline
$z=1$ & $\frac{1}{8}$ & $0$ & $-\frac{1}{4}$ & $0$ &$\frac{1}{8}$\\
  \hline
$z=2$ & $\frac{1}{8}$ & $0$ & $-\frac{1}{4}$ & $0$ &$\frac{1}{8}$\\
  \hline
\end{tabular}
\end{center}
\item $x=121$
\begin{center}
\begin{tabular}{ | m{0.5cm} || m{0.5cm} | m{0.5cm} | m{0.5cm} | m{0.5cm} | m{0.5cm} | } 
  \hline
 & $y=0$ & $y=1$ & $y=2$ & $y=3$ & $y=4$\\
  \hline \hline
$z=0$ & $\frac{1}{12}$ & $-\frac{1}{6}$ & $0$ &$\frac{1}{6}$ &$-\frac{1}{12}$\\
  \hline
$z=1$ & $\frac{1}{12}$ & $0$ &$0$ &$-\frac{1}{3}$ &$\frac{1}{4}$\\ 
  \hline
$z=2$ & $\frac{1}{12}$ & $0$ & $0$ &$-\frac{1}{3}$ &$\frac{1}{4}$\\
  \hline
$z=3$ & $\frac{1}{12}$ & $0$ & $0$ &$-\frac{1}{3}$ &$-\frac{1}{4}$\\
  \hline
\end{tabular}
\end{center}
\item $x=211$
\begin{center}
\begin{tabular}{ | m{0.5cm} || m{0.5cm} | m{0.5cm} | m{0.5cm} | m{0.5cm} | m{0.5cm} | } 
  \hline
 & $y=0$ & $y=1$ & $y=2$ & $y=3$ & $y=4$\\
  \hline \hline
$z=0$ & $\frac{1}{8}$ & $-\frac{1}{3}$ & $\frac{1}{4}$ & $0$ &$-\frac{1}{24}$\\
  \hline
$z=1$ & $\frac{1}{8}$ & $0$ &$-\frac{1}{4}$ &$0$ &$\frac{1}{8}$\\ 
  \hline
$z=2$ & $\frac{1}{8}$ & $0$ & $\frac{1}{4}$ &$0$ &$-\frac{3}{8}$\\ 
  \hline
$z=3$ & $\frac{1}{8}$ & $0$ & $\frac{1}{4}$ &$0$ &$-\frac{3}{8}$\\ 
  \hline
\end{tabular}
\end{center}
\item $x=1111$
\begin{center}
\begin{tabular}{ | m{0.5cm} || m{0.5cm} | m{0.5cm} | m{0.5cm} | m{0.5cm} | m{0.5cm} | } 
  \hline
 & $y=0$ & $y=1$ & $y=2$ & $y=3$ & $y=4$\\
  \hline \hline
$z=0$ & $\frac{1}{24}$ & $-\frac{1}{6}$ & $\frac{1}{4}$ &$-\frac{1}{6}$ &$\frac{1}{24}$\\
  \hline
$z=1$ & $\frac{1}{24}$ & $0$ &$-\frac{1}{4}$ &$\frac{1}{3}$ &$-\frac{1}{8}$\\ 
  \hline
$z=2$ & $\frac{1}{24}$ & $0$ & $\frac{1}{4}$ &$-\frac{2}{3}$ &$\frac{3}{8}$\\ 
  \hline
$z=3$ & $\frac{1}{24}$ & $0$ & $\frac{1}{4}$ &$\frac{1}{3}$ &$-\frac{5}{8}$\\ 
  \hline
$z=4$ & $\frac{1}{24}$ & $0$ & $\frac{1}{4}$ &$\frac{1}{3}$ &$\frac{3}{8}$\\ 
  \hline
\end{tabular}
\end{center}
\end{itemize}
\end{Ex}

\begin{Oboz}
$$g(x,y):\left\{(x,y)\in\left(\mathbb{YF}\cup\mathbb{YF}_\infty\right)\times\mathbb{N}:\;y\le d(x)\right\} \to \mathbb{N}$$
-- это функция, определённая следующим образом:

Рассмотрим представление $x\in\mathbb{YF}$ в виде $$x=\ldots 2\underbrace{1\ldots1}_{\beta_m}2\ldots2\underbrace{1\ldots1}_{\beta_1}2\underbrace{1\ldots1}_{\beta_0}$$
и определим:
\begin{itemize}
    \item $g(x,1)=\beta_0+1;$
    \item $g(x,2)=\beta_0+\beta_1+3;$
    \item $\ldots$
    \item $g(x,m)=\beta_0+\ldots+\beta_{m-1}+2m-1;$
    \item $\ldots$.
\end{itemize}

\end{Oboz}

\begin{Oboz}
Пусть $x\in\left(\mathbb{YF}\cup\mathbb{YF}_\infty\right),$ $y \in \mathbb{YF}$. Тогда вершину графа Юнга -- Фибоначчи, номер которой -- это конкатенация номеров $x$ и $y$, обозначим за $xy$. 
\end{Oboz}

\begin{Oboz}
$\;$
\begin{itemize}
    \item Пусть $x,y \in \left(\mathbb{YF}\cup\mathbb{YF}_\infty\right)$, и при этом не выполняется то, что $x=y\in\mathbb{YF}_\infty$. Тогда максимальное $z\in\mathbb{N}_0:$ $\exists x',y'\in \left(\mathbb{YF}\cup\mathbb{YF}_\infty\right),$ $z'\in\mathbb{YF}:$ $ x=x'z'$, $y=y'z'$, $|z'|=z$, обозначим за $h'(x,y)$.
    \item Пусть $x=y\in \mathbb{YF}_\infty$. Тогда будем считать, что $h'(x,y)=\infty$.
\end{itemize}
\end{Oboz}

\begin{Oboz}
$\;$
\begin{itemize}
    \item Пусть $x,y \in \left(\mathbb{YF}\cup\mathbb{YF}_\infty\right)$, и при этом не выполняется то, что $x=y\in\mathbb{YF}_\infty$. Тогда максимальное $z\in\mathbb{N}_0:$ $\exists x',y'\in \left(\mathbb{YF}\cup\mathbb{YF}_\infty\right),$ $z'\in\mathbb{YF}:$ $x=x'z'$, $y=y'z'$, $\#z'=z$, обозначим за $h(x,y)$.
    \item Пусть $x=y\in \mathbb{YF}_\infty$. Тогда будем считать, что $h(x,y)=\infty$.
\end{itemize}
\end{Oboz}

\begin{Zam}

Пусть $x,y \in \left(\mathbb{YF}\cup\mathbb{YF}_\infty\right)$. Тогда 

\begin{itemize}
    \item $h'(x,y)$ -- это сумма цифр в самом длинном общем суффиксе номеров $x$ и $y$;
    \item $h(x,y)$ -- это количество цифр в самом длинном общем суффиксе номеров $x$ и $y$;
    \item $h'(x,y)=h'(y,x)$;
    \item $h(x,y)=h(y,x)$.
\end{itemize}
\end{Zam}

\begin{theorem}[Теорема 1\cite{Evtuh1}, Теорема 1\cite{Evtuh2}] \label{evtuh}
Пусть $x,y \in \mathbb{YF}:$ $|y|\ge|x|$. Тогда
$$d(x,y)=\sum_{i=0}^{|x|}\left( {f\left(x,i,h(x,y)\right)}\prod_{j=1}^{d(y)}\left(g\left(y,j\right)-i\right)\right).$$ 
\end{theorem}
\begin{Col}[Следствие 1\cite{Evtuh3}]
Пусть $y \in \mathbb{YF}$. Тогда
$$d(\varepsilon,y)=\prod_{j=1}^{d(y)}g\left(y,j\right).$$ 
\end{Col}

\begin{Oboz}
Пусть $w\in \mathbb{YF}_{\infty}$, $m\in\mathbb{N}_0$. Тогда за $w_m\in\mathbb{YF}$ обозначим такую вершину графа Юнга--Фибоначчи, что $\#w_m=m$ и $\exists w'\in\mathbb{YF}_\infty:$ $w=w'w_m$.
\end{Oboz}

\begin{Zam} 

\;

\begin{itemize}
    \item Очевидно, что для любых $w\in \mathbb{YF}_{\infty}$ и $m\in\mathbb{N}_0$ такая вершина существует и однозначно определена.
    \item Ясно, что это просто вершина, номер которой -- это последние (то есть самые правые) $m$ символов номера $w$.
\end{itemize}
\end{Zam}

\begin{Oboz}
Пусть $w\in\mathbb{YF}_\infty$, $v\in\mathbb{YF}$, $m\in\mathbb{N}_0$. Тогда
    $$\mu_w(v,m):= \frac{d(\varepsilon,v)d(v,w_m)}{d(\varepsilon,w_m)}. $$
\end{Oboz}



\begin{Prop}[Утверждение 1\cite{Evtuh3}]
Пусть $w\in\mathbb{YF}_\infty,$ $v\in\mathbb{YF}$. Тогда существует предел
$$\lim_{m \to \infty}\mu_w(v,m)=\lim_{m \to \infty} \frac{d(\varepsilon,v)d(v,w_m)}{d(\varepsilon,w_m)}.$$
\end{Prop}

\begin{Oboz}
Пусть $w\in\mathbb{YF}_\infty$, $v\in\mathbb{YF}$. Тогда
    $$\mu_w(v):=\lim_{m \to \infty}\mu_w(v,m)=\lim_{m \to \infty} \frac{d(\varepsilon,v)d(v,w_m)}{d(\varepsilon,w_m)}. $$
\end{Oboz}

\begin{Oboz}
Пусть $x\in \mathbb{YF}$, $y\in\left(\mathbb{YF}\cup \mathbb{YF}_{\infty}\right),$  $\beta\in(0,1]:$ $|y| \ge |x|$. Тогда
$$d'_{\beta}(x,y):=\sum_{i=0}^{|x|}\left(\beta^i {f\left(x,i,h(x,y)\right)}\prod_{j=1}^{d(y)}\frac{\left(g\left(y,j\right)-i\right)}{g\left(y,j\right)}\right).$$
\end{Oboz}

\begin{Zam}
Из определения функции $g$ ясно, что данное выражение определено.
\end{Zam}

\begin{Oboz} \label{bitvaextasensov}
Пусть $w\in\mathbb{YF}_\infty$, $\beta\in(0,1]$, $v\in\mathbb{YF}$. Тогда
$$\mu_{w,\beta}(v):=d(\varepsilon,v)\cdot d_\beta'(v,w). $$
\end{Oboz}

\begin{Oboz}
Пусть $\{w_i'\}_{i=1}^\infty\in\left(\mathbb{YF}\right)^\infty$, $v\in\mathbb{YF}$, $m\in\mathbb{N}_0$. Тогда
    $$\mu_{\{w_i'\}}(v,m):= \frac{d(\varepsilon,v)d(v,w_m')}{d(\varepsilon,w_m')}. $$
\end{Oboz}

\begin{Zam}
В данном обозначении $\{w_i'\}_{i=1}^\infty\in\left(\mathbb{YF}\right)^\infty$ -- это произвольная бесконечная последовательность вершин графа Юнга -- Фибоначчи.
\end{Zam}

\begin{Def}
Пусть $\{w_i'\}_{i=1}^\infty\in\left(\mathbb{YF}\right)^\infty$, $w\in\mathbb{YF}_\infty:$ $\forall h\in\mathbb{N}_0$ $\exists M\in\mathbb{N}_0:$ $\forall m\ge M$
$$h(w_m',w)> h.$$
Тогда скажем, что бесконечная последовательность $\{w_i'\}_{i=1}^\infty\in\left(\mathbb{YF}\right)^\infty$  вершин графа Юнга -- Фибоначчи сходится к ``бесконечно удалённой вершине'' $w$ графа Юнга -- Фибоначчи.
\end{Def}

\begin{Oboz}
Пусть $\{w_i'\}_{i=1}^\infty\in\left(\mathbb{YF}\right)^\infty$, $w\in\mathbb{YF}_\infty$ такие, что бесконечная последовательность $\{w_i'\}_{i=1}^\infty\in\left(\mathbb{YF}\right)^\infty$  вершин графа Юнга -- Фибоначчи сходится к ``бесконечно удалённой вершине'' $w$ графа Юнга -- Фибоначчи. Тогда будем писать, что
$$\{w_i'\}\xrightarrow{i\to\infty}w.$$
\end{Oboz}

\begin{Oboz}
Пусть $w\in \left(\mathbb{YF}\cup\mathbb{YF}_\infty\right)$. Тогда 
$$\pi(w):=\prod_{i:g(w,i)> 1  } \frac{g(w,i)-1}{g(w,i)}.$$
\end{Oboz}
\begin{Oboz}
Пусть $w\in \left(\mathbb{YF}\cup\mathbb{YF}_\infty\right),$ $k\in\mathbb{N}_0:$ $k\ge 2$. Тогда 
$$\pi_k(w):=\prod_{i:g(w,i)> k  } \frac{g(w,i)-k}{g(w,i)}.$$
\end{Oboz}

\begin{Zam} \label{promezhutok}
Ясно, что 
\begin{itemize}
    \item Если $w\in\mathbb{YF}$, то 
$$\pi(w)\in(0,1];$$
    \item Если $w\in\mathbb{YF}_\infty$, то 
$$\pi(w)\in[0,1];$$
    \item Если $w\in\mathbb{YF}$, $k\in\mathbb{N}_0:$ $k\ge 2$,  то 
$$\pi_k(w)\in(0,1];$$
    \item Если $w\in\mathbb{YF}_\infty$, $k\in\mathbb{N}_0:$ $k\ge 2$, то 
$$\pi_k(w)\in[0,1].$$
\end{itemize}
\end{Zam}

\begin{Oboz}
$$\mathbb{YF}_\infty^+:=\{w\in\mathbb{YF}_\infty:\;\pi(w)>0\}.$$
\end{Oboz}
\begin{Zam}
Ясно, что
$$\mathbb{YF}_\infty^+=\{w\in\mathbb{YF}_\infty:\;\pi(w)\in(0,1)\}.$$
\end{Zam}

\begin{Prop}[Proposition 8.6\cite{KerGood}]\label{KerovGoodman}
Пусть $\{w_i'\}_{i=1}^\infty\in\left(\mathbb{YF}\right)^\infty$, $w\in\mathbb{YF}_\infty^+$, $\beta\in(0,1]$, $v\in\mathbb{YF},$ $k\in\mathbb{N}_0:$ $k\ge 2,$ $\{w_i'\}\xrightarrow{i\to\infty}w$ и при этом существует предел
$$\lim_{m\to\infty}\frac{\pi(w'_m)}{\pi(w)}=\beta.$$

Тогда
$$\lim_{m\to\infty}\frac{\pi_k(w'_m)}{\pi_k(w)}=\beta^k.$$
\end{Prop}

\begin{Lemma}\label{yavsyo}
Пусть $\{w_i'\}_{i=1}^\infty\in\left(\mathbb{YF}\right)^\infty$, $w\in\mathbb{YF}_\infty^+,$ $\beta\in(0,1]$, $v\in\mathbb{YF}:$  $\{w_i'\}\xrightarrow{i\to\infty}w$ и при этом существует предел
$$\lim_{m\to\infty}\frac{\pi(w'_m)}{\pi(w)}=\beta.$$
Тогда предел
$$\lim_{m\to\infty}\mu_{\{w_i'\}}(v,m)=\lim_{m\to\infty} \frac{d(\varepsilon,v)d(v,w_m')}{d(\varepsilon,w_m')}$$
существует и равен
$$\mu_{w,\beta}(v).$$
\end{Lemma}
\begin{proof}
По определению если $\{w_i'\}\xrightarrow{i\to\infty}w$, то $\exists M\in\mathbb{N}_0:$ $\forall m\in\mathbb{N}_0:$ $m\ge M$
$$h(w_m',w)\ge |v|\Longrightarrow |w_m'|\ge\#w_m'\ge  h(w_m',w)\ge |v|,$$ 
а значит, по Теореме \ref{evtuh} при $v,w_m'\in\mathbb{YF}$
$$\lim_{m\to\infty}\mu_{\{w_i'\}}(v)=\lim_{m\to\infty}\frac{d(\varepsilon,v)d(v,w'_m)}{d(\varepsilon,w'_m)}
=\lim_{m\to\infty}\frac{\displaystyle d(\varepsilon,v) \sum_{i=0}^{|v|}\left( {f\left(v,i,h(v,w'_m)\right)}\prod_{j=1}^{d\left(w'_m\right)}{\left(g\left(w'_m,j\right)-i\right)}\right)}{\displaystyle\prod_{j=1}^{d\left(w'_m\right)}{g\left(w'_m,j\right)}}=$$
$$=d(\varepsilon,v)\lim_{m\to\infty}\left(\sum_{i=0}^{|v|}\left( {f\left(v,i,h(v,w'_m)\right)}\prod_{j=1}^{d\left(w'_m\right)}\frac{\left(g\left(w'_m,j\right)-i\right)}{g\left(w'_m,j\right)}\right)\right).$$

По определению если $\{w_i'\}\xrightarrow{i\to\infty}w$, то $\exists M\in\mathbb{N}_0:$ $\forall m\in\mathbb{N}_0:$ $m\ge M$
$$h(w_m',w)\ge |v|\Longrightarrow h(v,w_m')=h(v,w),$$
а это значит, что наше выражение равняется следующему:
$$d(\varepsilon,v)\lim_{m\to\infty}\left(\sum_{i=0}^{|v|}\left( {f\left(v,i,h(v,w)\right)}\prod_{j=1}^{d\left(w'_m\right)}\frac{\left(g\left(w'_m,j\right)-i\right)}{g\left(w'_m,j\right)}\right)\right)=$$
$$=d(\varepsilon,v)\sum_{i=0}^{|v|}\left( {f\left(v,i,h(v,w)\right)}\cdot\lim_{m\to\infty}\prod_{j=1}^{d\left(w'_m\right)}\frac{\left(g\left(w'_m,j\right)-i\right)}{g\left(w'_m,j\right)}\right).$$

Далее рассмотрим три случая:
\begin{enumerate}
    \item $v\in\mathbb{YF}:$ $|v|=0$.
    
    Ясно, что в данном случае $v=\varepsilon$. А значит наше выражение равняется следующему:
    $$d(\varepsilon,v)\sum_{i=0}^{|\varepsilon|}\left( {f\left(v,i,h(v,w)\right)}\cdot\lim_{m\to\infty}\prod_{j=1}^{d\left(w'_m\right)}\frac{\left(g\left(w'_m,j\right)-i\right)}{g\left(w'_m,j\right)}\right)=$$
    $$=d(\varepsilon,v)\sum_{i=0}^{0}\left( {f\left(v,i,h(v,w)\right)}\cdot\lim_{m\to\infty}\prod_{j=1}^{d\left(w'_m\right)}\frac{\left(g\left(w'_m,j\right)-i\right)}{g\left(w'_m,j\right)}\right)=$$
    $$=d(\varepsilon,v)\cdot {f\left(v,0,h(v,w)\right)}\cdot\lim_{m\to\infty}\prod_{j=1}^{d\left(w'_m\right)}\frac{\left(g\left(w'_m,j\right)-0\right)}{g\left(w'_m,j\right)}=$$
    $$=d(\varepsilon,v)\cdot {f\left(v,0,h(v,w)\right)}\cdot\lim_{m\to\infty}1=$$
    $$=d(\varepsilon,v)\cdot\beta^0 {f\left(v,0,h(v,w)\right)}\prod_{j=1}^{d(w)}\frac{\left(g\left(w,j\right)-0\right)}{g\left(w,j\right)}=$$
    $$=d(\varepsilon,v)\sum_{i=0}^{0}\left(\beta^i {f\left(v,i,h(v,w)\right)}\prod_{j=1}^{d(w)}\frac{\left(g\left(w,j\right)-i\right)}{g\left(w,j\right)}\right)=$$
    $$=d(\varepsilon,v)\sum_{i=0}^{|\varepsilon|}\left(\beta^i {f\left(v,i,h(v,w)\right)}\prod_{j=1}^{d(w)}\frac{\left(g\left(w,j\right)-i\right)}{g\left(w,j\right)}\right)=$$
    $$=d(\varepsilon,v)\sum_{i=0}^{|v|}\left(\beta^i {f\left(v,i,h(v,w)\right)}\prod_{j=1}^{d(w)}\frac{\left(g\left(w,j\right)-i\right)}{g\left(w,j\right)}\right)=d(\varepsilon,v) \cdot d'_{\beta}(v,w)=\mu_{w,\beta}(v),$$
что и требовалось.

В данном случае Лемма доказана.

    \item $v\in\mathbb{YF}:$ $|v|=1$.
    
    Ясно, что в данном случае $v=1$. А значит наше выражение равняется следующему:
    $$d(\varepsilon,v)\sum_{i=0}^{|1|}\left( {f\left(v,i,h(v,w)\right)}\cdot\lim_{m\to\infty}\prod_{j=1}^{d\left(w'_m\right)}\frac{\left(g\left(w'_m,j\right)-i\right)}{g\left(w'_m,j\right)}\right)=$$
    $$=d(\varepsilon,v)\sum_{i=0}^{0}\left( {f\left(v,i,h(v,w)\right)}\cdot\lim_{m\to\infty}\prod_{j=1}^{d\left(w'_m\right)}\frac{\left(g\left(w'_m,j\right)-i\right)}{g\left(w'_m,j\right)}\right)+$$
    $$+d(\varepsilon,v)\sum_{i=1}^{1}\left( {f\left(v,i,h(v,w)\right)}\cdot\lim_{m\to\infty}\prod_{j=1}^{d\left(w'_m\right)}\frac{\left(g\left(w'_m,j\right)-i\right)}{g\left(w'_m,j\right)}\right)=$$
    $$=d(\varepsilon,v)\cdot {f\left(v,0,h(v,w)\right)}\cdot\lim_{m\to\infty}\prod_{j=1}^{d\left(w'_m\right)}\frac{\left(g\left(w'_m,j\right)-0\right)}{g\left(w'_m,j\right)}+$$
    $$+d(\varepsilon,v)\cdot {f\left(v,1,h(v,w)\right)}\cdot\lim_{m\to\infty}\prod_{j=1}^{d\left(w'_m\right)}\frac{\left(g\left(w'_m,j\right)-1\right)}{g\left(w'_m,j\right)}=$$
    $$=d(\varepsilon,v)\cdot {f\left(v,0,h(v,w)\right)}\cdot\lim_{m\to\infty}1+$$
    $$+d(\varepsilon,v)\cdot {f\left(v,1,h(v,w)\right)}\cdot\lim_{m\to\infty}\left(\left(\prod_{j:\;g\left(w'_m,j\right)=1}\frac{\left(g\left(w'_m,j\right)-1\right)}{g\left(w'_m,j\right)}\right)\left(\prod_{j:\;g\left(w'_m,j\right)>1}\frac{\left(g\left(w'_m,j\right)-1\right)}{g\left(w'_m,j\right)}\right)\right).$$

По определению если $\{w_i'\}\xrightarrow{i\to\infty}w$, то $\exists M\in\mathbb{N}_0:$ $\forall m\in\mathbb{N}_0:m\ge M$
$$h(w_m',w)\ge 1.$$

А значит, из определения функции $g$ ясно, что $\exists M\in\mathbb{N}_0:$ $\forall m,j\in\mathbb{N}_0:$ $m\ge M,$ 
$$g(w,j)=1\Longleftrightarrow g(w_m',j)=1.$$

Применим это наблюдение к нашему выражению и поймём, что оно равняется следующему:
$$d(\varepsilon,v)\cdot {f\left(v,0,h(v,w)\right)}\lim_{m\to\infty}1+$$
$$+d(\varepsilon,v)\cdot {f\left(v,1,h(v,w)\right)}\lim_{m\to\infty}\left(\left(\prod_{j:\;g\left(w,j\right)=1}\frac{\left(g\left(w,j\right)-1\right)}{g\left(w,j\right)}\right)\left(\prod_{j:\;g\left(w'_m,j\right)>1}\frac{\left(g\left(w'_m,j\right)-1\right)}{g\left(w'_m,j\right)}\right)\right)=$$
$$=d(\varepsilon,v)\cdot {f\left(v,0,h(v,w)\right)}\cdot\lim_{m\to\infty}1+$$
$$+d(\varepsilon,v)\cdot {f\left(v,1,h(v,w)\right)}\left(\prod_{j:\;g(w
,j)=1}\frac{\left(g\left(w,j\right)-1\right)}{g\left(w,j\right)}\right)\left(\lim_{m\to\infty}\prod_{j:\;g(w'_m,j)>1}\frac{\left(g\left(w'_m,j\right)-1\right)}{g\left(w'_m,j\right)}\right)=$$
$$=(\text{По определению функции $\pi$})=$$
$$=d(\varepsilon,v)\cdot {f\left(v,0,h(v,w)\right)}\cdot\lim_{m\to\infty}1+$$
$$+d(\varepsilon,v)\cdot {f\left(v,1,h(v,w)\right)}\left(\prod_{j:\;g(w
,j)=1}\frac{\left(g\left(w,j\right)-1\right)}{g\left(w,j\right)}\right)\lim_{m\to\infty}\pi(w_m').$$

В нашем случае
$$\lim_{m\to\infty}\frac{\pi(w'_m)}{\pi(w)}=\beta\Longleftrightarrow \lim_{m\to\infty}{\pi(w_m')}=\beta {\pi(w)}.$$

Таким образом, наше выражение равняется следующему:
$$d(\varepsilon,v)\cdot {f\left(v,0,h(v,w)\right)}\cdot\lim_{m\to\infty}1+$$
$$+d(\varepsilon,v)\cdot {f\left(v,1,h(v,w)\right)}\left(\prod_{j:\;g(w
,j)=1}\frac{\left(g\left(w,j\right)-1\right)}{g\left(w,j\right)}\right)\beta\pi(w)=$$
$$=d(\varepsilon,v)\cdot\beta^0 {f\left(v,0,h(v,w)\right)}\prod_{j=1}^{d(w)}\frac{\left(g\left(w,j\right)-0\right)}{g\left(w,j\right)}+$$
$$+d(\varepsilon,v)\cdot\beta^1 {f\left(v,1,h(v,w)\right)}\left(\prod_{j:\;g(w
,j)=1}\frac{\left(g\left(w,j\right)-1\right)}{g\left(w,j\right)}\right)\pi(w)=$$
$$=(\text{По определению функции $\pi$})=$$
$$=d(\varepsilon,v)\cdot\beta^0 {f\left(v,0,h(v,w)\right)}\prod_{j=1}^{d(w)}\frac{\left(g\left(w,j\right)-0\right)}{g\left(w,j\right)}+$$
$$+d(\varepsilon,v)\cdot\beta^1 {f\left(v,1,h(v,w)\right)}\left(\prod_{j:\;g(w
,j)=1}\frac{\left(g\left(w,j\right)-1\right)}{g\left(w,j\right)}\right)\left(\prod_{j:\;g(w,j)>1}\frac{\left(g\left(w,j\right)-1\right)}{g\left(w,j\right)}\right)=$$
$$=d(\varepsilon,v)\cdot\beta^0 {f\left(v,0,h(v,w)\right)}\prod_{j=1}^{d(w)}\frac{\left(g\left(w,j\right)-0\right)}{g\left(w,j\right)}+$$
$$+d(\varepsilon,v)\cdot\beta^1 {f\left(v,1,h(v,w)\right)}\prod_{j=1}^{d(w)}\frac{\left(g\left(w,j\right)-1\right)}{g\left(w,j\right)}=$$
$$=d(\varepsilon,v)\sum_{i=0}^{0}\left(\beta^i {f\left(v,i,h(v,w)\right)}\prod_{j=1}^{d(w)}\frac{\left(g\left(w,j\right)-i\right)}{g\left(w,j\right)}\right)+$$
$$+d(\varepsilon,v)\sum_{i=1}^{1}\left(\beta^i {f\left(v,i,h(v,w)\right)}\prod_{j=1}^{d(w)}\frac{\left(g\left(w,j\right)-i\right)}{g\left(w,j\right)}\right)=$$
$$=d(\varepsilon,v)\sum_{i=0}^{1}\left(\beta^i {f\left(v,i,h(v,w)\right)}\prod_{j=1}^{d(w)}\frac{\left(g\left(w,j\right)-i\right)}{g\left(w,j\right)}\right)=$$
$$=d(\varepsilon,v)\sum_{i=0}^{|1|}\left(\beta^i {f\left(v,i,h(v,w)\right)}\prod_{j=1}^{d(w)}\frac{\left(g\left(w,j\right)-i\right)}{g\left(w,j\right)}\right)=$$
$$=d(\varepsilon,v)\sum_{i=0}^{|v|}\left(\beta^i {f\left(v,i,h(v,w)\right)}\prod_{j=1}^{d(w)}\frac{\left(g\left(w,j\right)-i\right)}{g\left(w,j\right)}\right)=d(\varepsilon,v)\cdot d'_{\beta}(v,w)=\mu_{w,\beta}(v),$$
что и требовалось.

В данном случае Лемма доказана.

\item $v\in\mathbb{YF}:$ $|v|\ge 2$.
    
Ясно, что в данном случае наше выражение равняется следующему:
$$d(\varepsilon,v)\sum_{i=0}^{0}\left( {f\left(v,i,h(v,w)\right)}\cdot\lim_{m\to\infty}\prod_{j=1}^{d\left(w'_m\right)}\frac{\left(g\left(w'_m,j\right)-i\right)}{g\left(w'_m,j\right)}\right)+$$
$$+d(\varepsilon,v)\sum_{i=1}^{1}\left( {f\left(v,i,h(v,w)\right)}\cdot\lim_{m\to\infty}\prod_{j=1}^{d\left(w'_m\right)}\frac{\left(g\left(w'_m,j\right)-i\right)}{g\left(w'_m,j\right)}\right)+$$
$$+d(\varepsilon,v)\sum_{i=2}^{|v|}\left( {f\left(v,i,h(v,w)\right)}\cdot\lim_{m\to\infty}\prod_{j=1}^{d\left(w'_m\right)}\frac{\left(g\left(w'_m,j\right)-i\right)}{g\left(w'_m,j\right)}\right)=$$
    $$=d(\varepsilon,v)\cdot {f\left(v,0,h(v,w)\right)}\cdot\lim_{m\to\infty}\prod_{j=1}^{d\left(w'_m\right)}\frac{\left(g\left(w'_m,j\right)-0\right)}{g\left(w'_m,j\right)}+$$
    $$+d(\varepsilon,v)\cdot {f\left(v,1,h(v,w)\right)}\cdot\lim_{m\to\infty}\prod_{j=1}^{d\left(w'_m\right)}\frac{\left(g\left(w'_m,j\right)-1\right)}{g\left(w'_m,j\right)}+$$
$$+d(\varepsilon,v)\sum_{i=2}^{|v|}\left( {f\left(v,i,h(v,w)\right)}\cdot\lim_{m\to\infty}\prod_{j=1}^{d\left(w'_m\right)}\frac{\left(g\left(w'_m,j\right)-i\right)}{g\left(w'_m,j\right)}\right)=$$
$$=d(\varepsilon,v)\cdot {f\left(v,0,h(v,w)\right)}\cdot\lim_{m\to\infty}1+$$
$$+d(\varepsilon,v)\cdot {f\left(v,1,h(v,w)\right)}\cdot\lim_{m\to\infty}\left(\left(\prod_{j:\;g(w'_m
,j)=1}\frac{\left(g\left(w'_m,j\right)-1\right)}{g\left(w'_m,j\right)}\right)\left(\prod_{j:\;g(w'_m,j)>1     
}\frac{\left(g\left(w'_m,j\right)-1\right)}{g\left(w'_m,j\right)}\right)\right)+$$
$$+d(\varepsilon,v)\sum_{i=2}^{|v|}\left( {f\left(v,i,h(v,w)\right)}\cdot\lim_{m\to\infty}\left(\left(\prod_{j:\;g(w'_m
,j)\le i}\frac{\left(g\left(w'_m,j\right)-i\right)}{g\left(w'_m,j\right)}\right)\left(\prod_{j:\;g(w'_m,j)>i}\frac{\left(g\left(w'_m,j\right)-i\right)}{g\left(w'_m,j\right)}\right)\right)\right).$$

По определению если $\{w_i'\}\xrightarrow{i\to\infty}w$, то $\exists M\in\mathbb{N}_0:$ $\forall m\in\mathbb{N}_0:m\ge M$
$$h(w_m',w)\ge |v|.$$

А значит, из определению функции $g$ ясно, что $\exists M\in\mathbb{N}_0:$ $\forall m,i,j\in\mathbb{N}_0:$ $i\in\overline{|v|},$ $m\ge M$
$$g(w_m',j)\le i\Longleftrightarrow g(w,j)\le i\Longrightarrow g(w,j)=g(w_m',j).$$

Применим это наблюдение к нашему выражению и поймём, что оно равняется следующему:
$$d(\varepsilon,v)\cdot {f\left(v,0,h(v,w)\right)}\cdot\lim_{m\to\infty}1+$$
$$+d(\varepsilon,v)\cdot {f\left(v,1,h(v,w)\right)}\cdot\lim_{m\to\infty}\left(\left(\prod_{j:\;g(w
,j)=1}\frac{\left(g\left(w,j\right)-1\right)}{g\left(w,j\right)}\right)\left(\prod_{j:\;g(w'_m,j)>1     
}\frac{\left(g\left(w'_m,j\right)-1\right)}{g\left(w'_m,j\right)}\right)\right)+$$
$$+d(\varepsilon,v)\sum_{i=2}^{|v|}\left( {f\left(v,i,h(v,w)\right)}\cdot\lim_{m\to\infty}\left(\left(\prod_{j:\;g(w
,j)\le i}\frac{\left(g\left(w,j\right)-i\right)}{g\left(w,j\right)}\right)\left(\prod_{j:\;g(w'_m,j)>i}\frac{\left(g\left(w'_m,j\right)-i\right)}{g\left(w'_m,j\right)}\right)\right)\right)=$$
$$=d(\varepsilon,v)\cdot {f\left(v,0,h(v,w)\right)}\cdot\lim_{m\to\infty}1+$$
$$+d(\varepsilon,v)\cdot {f\left(v,1,h(v,w)\right)}\left(\prod_{j:\;g(w
,j)=1}\frac{\left(g\left(w,j\right)-1\right)}{g\left(w,j\right)}\right)\left(\lim_{m\to\infty}\prod_{j:\;g(w'_m,j)>1}\frac{\left(g\left(w'_m,j\right)-1\right)}{g\left(w'_m,j\right)}\right)+$$
$$+d(\varepsilon,v)\sum_{i=2}^{|v|}\left( {f\left(v,i,h(v,w)\right)}\left(\prod_{j:\;g(w
,j)\le i}\frac{\left(g\left(w,j\right)-i\right)}{g\left(w,j\right)}\right)\left(\lim_{m\to\infty}\prod_{j:\;g(w'_m,j)>i}\frac{\left(g\left(w'_m,j\right)-i\right)}{g\left(w'_m,j\right)}\right)\right)=$$
$$=(\text{По определению функций $\pi$ и $\pi_i$})=$$
$$=d(\varepsilon,v)\cdot {f\left(v,0,h(v,w)\right)}\cdot\lim_{m\to\infty}1+$$
$$+d(\varepsilon,v)\cdot {f\left(v,1,h(v,w)\right)}\left(\prod_{j:\;g(w
,j)=1}\frac{\left(g\left(w,j\right)-1\right)}{g\left(w,j\right)}\right)\lim_{m\to\infty}\pi(w_m')+$$
$$+d(\varepsilon,v)\sum_{i=2}^{|v|}\left( {f\left(v,i,h(v,w)\right)}\left(\prod_{j:\;g(w
,j)\le i}\frac{\left(g\left(w,j\right)-i\right)}{g\left(w,j\right)}\right)\lim_{m\to\infty}\pi_i(w_m')\right).$$

По Утверждению \ref{KerovGoodman} при наших $\{w_i'\}_{i=1}^\infty\in\left(\mathbb{YF}\right)^\infty$, $w\in\mathbb{YF}_\infty,$ $\beta\in(0,1]$, $v\in\mathbb{YF}$ и произвольном   $k\in\overline{2,|v|}$
$$\lim_{m\to\infty}\frac{\pi_k(w_m')}{\pi_k(w)}=\beta^k\Longleftrightarrow \lim_{m\to\infty}{\pi_k(w_m')}=\beta^k {\pi_k(w)}.$$

Кроме того, в нашем случае
$$\lim_{m\to\infty}\frac{\pi(w'_m)}{\pi(w)}=\beta\Longleftrightarrow \lim_{m\to\infty}{\pi(w_m')}=\beta {\pi(w)}.$$

Таким образом, можно понять, что наше выражение равняется следующему:
$$d(\varepsilon,v)\cdot {f\left(v,0,h(v,w)\right)}\cdot\lim_{m\to\infty}1+$$
$$+d(\varepsilon,v)\cdot {f\left(v,1,h(v,w)\right)}\left(\prod_{j:\;g(w
,j)=1}\frac{\left(g\left(w,j\right)-1\right)}{g\left(w,j\right)}\right)\beta\pi(w)+$$
$$+d(\varepsilon,v)\sum_{i=2}^{|v|}\left( {f\left(v,i,h(v,w)\right)}\left(\prod_{j:\;g(w
,j)\le i}\frac{\left(g\left(w,j\right)-i\right)}{g\left(w,j\right)}\right)\beta^i\pi_i(w)\right)=$$
$$=d(\varepsilon,v)\cdot\beta^0 {f\left(v,0,h(v,w)\right)}\prod_{j=1}^{d(w)}\frac{\left(g\left(w,j\right)-0\right)}{g\left(w,j\right)}+$$
$$+d(\varepsilon,v)\cdot\beta^1 {f\left(v,1,h(v,w)\right)}\left(\prod_{j:\;g(w
,j)=1}\frac{\left(g\left(w,j\right)-1\right)}{g\left(w,j\right)}\right)\pi(w)+$$
$$+d(\varepsilon,v)\sum_{i=2}^{|v|}\left(\beta^i {f\left(v,i,h(v,w)\right)}\left(\prod_{j:\;g(w
,j)\le i}\frac{\left(g\left(w,j\right)-i\right)}{g\left(w,j\right)}\right)\pi_i(w)\right)=$$
$$=(\text{По определению функций $\pi$ и $\pi_i$})=$$
$$=d(\varepsilon,v)\cdot\beta^0 {f\left(v,0,h(v,w)\right)}\prod_{j=1}^{d(w)}\frac{\left(g\left(w,j\right)-0\right)}{g\left(w,j\right)}+$$
$$+d(\varepsilon,v)\cdot\beta^1 {f\left(v,1,h(v,w)\right)}\left(\prod_{j:\;g(w
,j)=1}\frac{\left(g\left(w,j\right)-1\right)}{g\left(w,j\right)}\right)\left(\prod_{j:\;g(w,j)>1}\frac{\left(g\left(w,j\right)-1\right)}{g\left(w,j\right)}\right)+$$
$$+d(\varepsilon,v)\sum_{i=2}^{|v|}\left(\beta^i {f\left(v,i,h(v,w)\right)}\left(\prod_{j:\;g(w
,j)\le i}\frac{\left(g\left(w,j\right)-i\right)}{g\left(w,j\right)}\right)\left(\prod_{j:\;g(w,j)>i}\frac{\left(g\left(w,j\right)-i\right)}{g\left(w,j\right)}\right)\right)=$$
$$=d(\varepsilon,v)\cdot\beta^0 {f\left(v,0,h(v,w)\right)}\prod_{j=1}^{d(w)}\frac{\left(g\left(w,j\right)-0\right)}{g\left(w,j\right)}+$$
$$+d(\varepsilon,v)\cdot\beta^1 {f\left(v,1,h(v,w)\right)}\prod_{j=1}^{d(w)}\frac{\left(g\left(w,j\right)-1\right)}{g\left(w,j\right)}+$$
$$+d(\varepsilon,v)\cdot\sum_{i=2}^{|v|}\left(\beta^i {f\left(v,i,h(v,w)\right)}\prod_{j=1}^{d(w)}\frac{\left(g\left(w,j\right)-i\right)}{g\left(w,j\right)}\right)=$$
$$=d(\varepsilon,v)\sum_{i=0}^{0}\left(\beta^i {f\left(v,i,h(v,w)\right)}\prod_{j=1}^{d(w)}\frac{\left(g\left(w,j\right)-i\right)}{g\left(w,j\right)}\right)+$$
$$+d(\varepsilon,v)\sum_{i=1}^{1}\left(\beta^i {f\left(v,i,h(v,w)\right)}\prod_{j=1}^{d(w)}\frac{\left(g\left(w,j\right)-i\right)}{g\left(w,j\right)}\right)+$$
$$+d(\varepsilon,v)\sum_{i=2}^{|v|}\left(\beta^i {f\left(v,i,h(v,w)\right)}\prod_{j=1}^{d(w)}\frac{\left(g\left(w,j\right)-i\right)}{g\left(w,j\right)}\right)=$$
$$=d(\varepsilon,v)\sum_{i=0}^{|v|}\left(\beta^i {f\left(v,i,h(v,w)\right)}\prod_{j=1}^{d(w)}\frac{\left(g\left(w,j\right)-i\right)}{g\left(w,j\right)}\right)=d(\varepsilon,v)\cdot d'_{\beta}(v,w)=\mu_{w,\beta}(v),$$
что и требовалось.

В данном случае Лемма доказана.
\end{enumerate}

Ясно, что все случаи разобраны, и в каждом из них Лемма доказана.

Таким образом, Лемма доказана.
\end{proof}

\begin{Oboz}
Пусть $\{w_i'\}_{i=1}^\infty\in\left(\mathbb{YF}\right)^\infty$, $w\in\mathbb{YF}_\infty^+,$ $\beta\in(0,1],$ $v\in\mathbb{YF}:$  $\{w_i'\}\xrightarrow{i\to\infty}w$ и при этом существует предел
$$\lim_{m\to\infty}\frac{\pi(w'_m)}{\pi(w)}=\beta.$$
Тогда
$$\mu_{\{w_i'\}}(v):=\lim_{m\to\infty}\mu_{\{w_i'\}}(v,m)=\lim_{m\to\infty} \frac{d(\varepsilon,v)d(v,w_m')}{d(\varepsilon,w_m')}=\mu_{w,\beta}(v).$$
\end{Oboz}

\begin{Col}\label{neotr}
Пусть $w\in\mathbb{YF}_\infty^+,$ $\beta\in(0,1],$ $v\in\mathbb{YF}.$ Тогда
$$\mu_{w,\beta}(v)\ge 0.$$

\end{Col}

\begin{Zam}
Пусть $\{w_i'\}_{i=1}^\infty\in\left(\mathbb{YF}\right)^\infty$, $v\in\mathbb{YF}$, $m,n\in\mathbb{N}_0:$ $|w_m'|\ge |v|=n$. Тогда
\begin{itemize}
    \item $$ {d(\varepsilon,v)d(v,w_m')}={|\left\{t\in T(\varepsilon,w_m'):\;t(n)=v\right\}|} ;$$

    \item $$\mu_{\{w_i'\}}(v,m)= \frac{d(\varepsilon,v)d(v,w_m')}{d(\varepsilon,w_m')}=\frac{|\left\{t\in T(\varepsilon,w_m'):\;t(n)=v\right\}|}{|\left\{T(\varepsilon,w_m')\right\}|} .$$
    
\end{itemize}
\end{Zam}

\begin{Prop}\label{mera2}
Пусть $\{w_i'\}_{i=1}^\infty\in\left(\mathbb{YF}\right)^\infty$, $m,n\in\mathbb{N}_0:$ $|w_m'|\ge |v|=n$. Тогда
$$\sum_{v\in\mathbb{YF}_n}\mu_{\{w_i'\}}(v,m)= \sum_{v\in\mathbb{YF}_n}\frac{d(\varepsilon,v)d(v,w_m')}{d(\varepsilon,w_m')}=1. $$
\end{Prop}
\begin{proof}
$$\sum_{v\in\mathbb{YF}_n}\mu_{\{w_i'\}}(v,m)= \sum_{v\in\mathbb{YF}_n}\frac{d(\varepsilon,v)d(v,w_m')}{d(\varepsilon,w_m')}=(\text{так как $|w_m'|\ge n$})=\sum_{v\in\mathbb{YF}_n}\frac{|\left\{t\in T(\varepsilon,w_m'):\;t(n)=v\right\}|}{|\left\{T(\varepsilon,w_m')\right\}|}=$$
$$=\frac{\displaystyle\sum_{v\in\mathbb{YF}_n}\left|\left\{t\in T(\varepsilon,w_m'):\;t(n)=v\right\}\right|}{\displaystyle|\left\{T(\varepsilon,w_m')\right\}|} =\frac{\displaystyle|\left\{T(\varepsilon,w_m')\right\}|}{\displaystyle|\left\{T(\varepsilon,w_m')\right\}|}=1. $$
\end{proof}

\begin{Col} \label{mera1}
Пусть $\{w_i'\}_{i=1}^\infty\in\left(\mathbb{YF}\right)^\infty$, $w\in\mathbb{YF}_\infty^+,$ $\beta\in(0,1],$ $n\in\mathbb{N}_0:$
$\{w_i'\}\xrightarrow{i\to\infty}w$ и при этом существует предел
$$\lim_{m\to\infty}\frac{\pi(w'_m)}{\pi(w)}=\beta.$$

Тогда
$$\sum_{v\in\mathbb{YF}_n}\mu_{\{w_i'\}}(v)=\sum_{v\in\mathbb{YF}_n}\mu_{w,\beta}(v)=1.$$
\end{Col}
\begin{proof}
$$\sum_{v\in\mathbb{YF}_n}\mu_{w,\beta}(v)=\left(\text{По Лемме \ref{yavsyo}, просуммированной по $v\in\mathbb{YF}_n$}\right)=$$
$$=\sum_{v\in\mathbb{YF}_n}\mu_{\{w_i'\}}(v)=\sum_{v\in\mathbb{YF}_n}\left(\lim_{m \to \infty} \frac{d(\varepsilon,v)d(v,w_m')}{d(\varepsilon,w_m')}\right)=\lim_{m \to \infty}\left(\sum_{v\in\mathbb{YF}_n}\left( \frac{d(\varepsilon,v)d(v,w_m')}{d(\varepsilon,w_m')}\right)\right).$$

По определению если $\{w_i'\}\xrightarrow{i\to\infty}w$, то $\exists M\in\mathbb{N}_0:$ $\forall m\in\mathbb{N}_0:$ $m\ge M$
$$h(w_m',w)\ge |v|\Longrightarrow |w_m'|\ge\#w_m'\ge  h(w_m',w)\ge |v|,$$
а значит наше выражение равняется 
$$\lim_{m\to\infty}1=1.$$
\end{proof}


\renewcommand{\labelenumi}{\arabic{enumi}$)$}
\renewcommand{\labelenumii}{\arabic{enumi}.\arabic{enumii}$^\circ$}
\renewcommand{\labelenumiii}{\arabic{enumi}.\arabic{enumii}.\arabic{enumiii}$^\circ$}

\begin{Oboz}
$$n(x,a):\left\{(x,a)\subseteq\mathbb{YF}\times\mathbb{N}_0:\; a\in\overline{\#x}\right\}\to \mathbb{YF}$$
-- это функция, определённая следующим образом:

Если $\exists x',x''\in \mathbb{YF}:$ $x=x'x''$ и $\#x''=a$, то $n(x,a)=x'$.

\end{Oboz}

\begin{Oboz}
$$k(x,a):\left\{(x,a)\subseteq\mathbb{YF}\times\mathbb{N}_0:\; a\in\overline{\#x}\right\}\to \mathbb{YF}$$
-- это функция, определённая следующим образом:

Если $\exists x',x''\in \mathbb{YF}:$ $x=x'x''$ и $\#x''=a$, то $k(x,a)=x''$.

\end{Oboz}

\begin{Zam}

\;

\begin{itemize}
    \item Ясно, что $\forall x\in\mathbb{YF}$ и $a\in\overline{\#x}$ значение функции $n(x,a)$ -- это вершина графа Юнга -- Фибоначчи, номер которой -- это первые (то есть самые левые) $(\#x-a)$ цифр номера $x$.
    \item Ясно, что $\forall x\in\mathbb{YF}$ и $a\in\overline{\#x}$ значение функции $k(x,a)$ -- это вершина графа Юнга -- Фибоначчи, номер которой -- это последние (то есть самые правые) $a$ цифр номера $x$.
    \item Очевидно, что $\forall x\in\mathbb{YF}$ и $a\in\overline{\#x}$ значение функции $n(x,a)$ всегда существует и однозначно определено.  
    \item Очевидно, что $\forall x\in\mathbb{YF}$ и $a\in\overline{\#x}$ значение функции $k(x,a)$ всегда существует и однозначно определено.
\end{itemize}
\end{Zam}

\begin{Zam} \label{prunk}
     Пусть $x\in\mathbb{YF}$, $a\in\mathbb{N}_0:$ $a\in\overline{ \#x}$. Тогда
     \begin{itemize}
        \item  $$x=n(x,a)k(x,a);$$
        \item  $$|x|=|n(x,a)|+|k(x,a)|.$$
     \end{itemize}
\end{Zam}
\begin{Zam} \label{mega}
     Пусть $x\in\mathbb{YF}$. Тогда
     \begin{itemize}
        \item $$n(x,0)=k(x,\#x)=x;$$
        \item $$n(x,\#x)=k(x,0)=\varepsilon.$$
     \end{itemize}
\end{Zam}

\begin{Prop} \label{rofl}
Пусть $x\in\mathbb{YF}$, $a\in\mathbb{N}_0$, $\alpha_0\in\{1,2\}:$ $a\in\overline{ \#x}$. Тогда
\begin{itemize}
    \item $$n(\alpha_0x,a)=\alpha_0n(x,a);$$
    \item $$k(\alpha_0x,a)=k(x,a).$$
\end{itemize}
\end{Prop}
\begin{proof}
Воспользуемся определением:
$$x=n(x,a)k(x,a), \; \#(k(x,a))=a \Longleftrightarrow$$ $$\Longleftrightarrow\alpha_0x=\alpha_0n(x,a)k(x,a),\; \#(k(x,a))=a \Longleftrightarrow$$
$$\Longleftrightarrow n(\alpha_0x,a)=\alpha_0n(x,a),\;k(\alpha_0x,a)=k(x,a),$$
что и требовалось.

Утверждение доказано.
\end{proof}

\renewcommand{\labelenumi}{\arabic{enumi})}

\renewcommand{\labelenumi}{\arabic{enumi}$^\circ$}
\renewcommand{\labelenumii}{\arabic{enumi}.\arabic{enumii}$^\circ$}
\renewcommand{\labelenumiii}{\arabic{enumi}.\arabic{enumii}.\arabic{enumiii}$^\circ$}

\begin{Oboz}
Пусть $n,y\in\mathbb{N}_0:$ $y\le n $. Тогда 
\begin{itemize}
    \item $$K(n,y):=\{v\in\mathbb{YF}_n:\;\exists v',v''\in\mathbb{YF}:\; v=v'v'' \text{ и } |v''|=y\};$$
    \item $$\overline{K}(n,y):=\{v\in\mathbb{YF}_n:\;\nexists v',v''\in\mathbb{YF}:\; v=v'v'' \text{ и } |v''|=y\}.$$
\end{itemize}
\end{Oboz}

\begin{Zam} \label{pruzhinka}
Пусть $n,y\in\mathbb{N}_0:$ $y\le n $. Тогда 
\begin{itemize}
    \item $$K(n,y)=\{v'v'':\; v'\in\mathbb{YF}_{n-y},\;v''\in\mathbb{YF}_y\}=\mathbb{YF}_{n-y}\cdot\mathbb{YF}_{y};$$
    \item $$\overline{K}(n,y)=\mathbb{YF}_n\textbackslash K(n,y)=\mathbb{YF}_n\textbackslash\left(\mathbb{YF}_{n-y}\cdot\mathbb{YF}_{y}\right);$$
    \item $$K(n,y)\cup \overline{K}(n,y)=\mathbb{YF}_n;$$
    \item $$K(n,y)\cap \overline{K}(n,y)=\varnothing.$$
\end{itemize}
\end{Zam}

\begin{Oboz} 
Пусть $v\in\mathbb{YF}$, $y\in\mathbb{N}_0:$ $y\le |v|$. Тогда
\begin{itemize}
    \item 
    \begin{equation*}
v(y):= \begin{cases}
$$v'$$
    &\text {если $\exists v',v''\in\mathbb{YF}:$ $v=v'v'' \text{ и } |v''|=y$
}\\
   \text{не определено} &\text{иначе}
 \end{cases};
\end{equation*}
    \item 
    \begin{equation*}
v'(y):= \begin{cases}
$$v''$$
    &\text {если $\exists v',v''\in\mathbb{YF}:$ $v=v'v'' \text{ и } |v''|=y$}\\
   \text{не определено} &\text{иначе}
 \end{cases}.
\end{equation*}
\end{itemize}
\end{Oboz}

\begin{Zam}
Пусть $v\in\mathbb{YF},$ $n,y\in\mathbb{N}_0:$ $v\in\mathbb{YF}_n$, $y\le n$.  Тогда
\begin{itemize}
    \item $$v(y) \text{ определено} \Longleftrightarrow v'(y) \text{ определено} \Longleftrightarrow v\in K(n,y) ;$$ 
    \item $$v(y) \text{ не определено} \Longleftrightarrow v'(y) \text{ не определено} \Longleftrightarrow v\in \overline{K}(n,y) .$$ 

\end{itemize}
\end{Zam}

\begin{Zam} \label{kerambus}
Пусть $v\in\mathbb{YF},$ $n,y\in\mathbb{N}_0:$ $y\le n$, $v\in K(n,y)$.  Тогда
     \begin{itemize}
        \item  $$v=v(y)v'(y);$$
        \item  $$|v|=|v(y)|+|v'(y)|.$$
     \end{itemize}
\end{Zam}
\begin{Zam} 
Пусть $v\in\mathbb{YF}$. Тогда
     \begin{itemize}
        \item $$v(0)=v'(|v|)=v;$$
        \item $$v(|v|)=v'(0)=\varepsilon.$$
     \end{itemize}
\end{Zam}

\begin{Oboz} \label{str}
Пусть $x\in \mathbb{YF}$, $\beta\in(0,1]$. Тогда
$$d_{\beta}(x):=\sum_{i=0}^{|x|}\left(\beta^i {f\left(x,i,0\right)}\right).$$

\end{Oboz}

\begin{Prop} \label{beta1}
Пусть $x\in \mathbb{YF}$, $y\in\mathbb{YF}_{\infty}$. Тогда
$$d'_{1}(x,y)= \sum_{i=0}^{\#x} \left(1^{|k(x,i)|}d_{1}(n(x,i))\cdot d'_1(k(x,i),y)\right).$$
\end{Prop}
\begin{proof}

По обозначению
$$d_{1}(n(x,i))=\sum_{j=0}^{|n(x,i)|}\left(1^i {f\left(n(x,i),j,0\right)}\right)=\sum_{j=0}^{|n(x,i)|} {f\left(n(x,i),j,0\right)}.$$

\begin{Prop}[Утверждение 5\cite{Evtuh1}, Утверждение 1.7\cite{Evtuh2}] \label{evtuh5}
Пусть $x \in \mathbb{YF}:$ $x\ne\varepsilon$. Тогда
$$\sum_{i=0}^{|x|} {f\left(x,i,0\right)}=0.$$
\end{Prop}

Подставим в Утверждение \ref{evtuh5} $n(x,i)$ на место $x$ и получим, что если $n(x,i)\ne\varepsilon$, то 
$$d_{1}(n(x,i))=\sum_{j=0}^{|n(x,i)|} {f\left(n(x,i),j,0\right)}=0.$$

Несложно заметить, что если $i\in\overline{\#x-1}$, то $n(x,i)\ne\varepsilon$, а если $i=\#x$, то $n(x,i)=\varepsilon$. А значит
$$\sum_{i=0}^{\#x} \left(1^{|k(x,i)|}d_{1}(n(x,i))\cdot d'_1(k(x,i),y)\right)=\sum_{i=0}^{\#x} \left(d_{1}(n(x,i))\cdot d'_1(k(x,i),y)\right)=$$
$$=\left(\sum_{i=0}^{\#x-1} \left(d_{1}(n(x,i))\cdot d'_1(k(x,i),y)\right)\right)+\left(\sum_{i=\#x}^{\#x} \left(d_{1}(n(x,i))\cdot d'_1(k(x,i),y)\right)\right)=$$
$$=\left(\sum_{i=0}^{\#x-1} \left(0 \cdot d'_1(k(x,i),y)\right)\right)+d_{1}(n(x,\#x))\cdot d'_1(k(x,\#x),y)= d_{1}(n(x,\#x))\cdot d'_1(k(x,\#x),y)=$$
$$=(\text{По Замечанию \ref{mega} при $x\in\mathbb{YF}$})=d_{1}(\varepsilon)\cdot d'_1(x,y)=$$
$$=\left(\sum_{i=0}^{|\varepsilon|}\left(1^if(\varepsilon,i,0)\right)\right)d'_1(x,y)=\left(\sum_{i=0}^{0}\left(1^if(\varepsilon,i,0)\right)\right)d'_1(x,y)=1^0f(\varepsilon,0,0)\cdot d'_1(x,y)=d'_1(x,y),$$
что и требовалось.

Утверждение доказано.
\end{proof}

\begin{Prop} \label{kusok}
Пусть $x\in \mathbb{YF}$, $y\in\mathbb{YF}_\infty,$ $\beta\in(0,1]$. Тогда
$$d'_{\beta}(x,y)= \sum_{i=0}^{\#x} \left(\beta^{|k(x,i)|}d_{\beta}(n(x,i))\cdot d'_1(k(x,i),y)\right).$$
\end{Prop}
\begin{proof}

Зафиксируем данный $y\in\mathbb{YF}_\infty$ и будем решать задачу по индукции по $\#x$.

\underline{\textbf{База}}. $x\in\mathbb{YF}:\;\#x=0\Longleftrightarrow x=\varepsilon$:
$$\sum_{i=0}^{\#x}\left( \beta^{|k(x,i)|}d_{\beta}(n(x,i))\cdot d'_1(k(x,i),y)\right)=\sum_{i=0}^{\#\varepsilon}\left( \beta^{|k(\varepsilon,i)|}d_{\beta}(n(\varepsilon,i))\cdot d'_1(k(\varepsilon,i),y)\right)=$$
$$=\sum_{i=0}^{0}\left( \beta^{|k(\varepsilon,i)|}d_{\beta}(n(\varepsilon,i))\cdot d'_1(k(\varepsilon,i),y)\right)=\beta^{|k(\varepsilon,0)|}d_{\beta}(n(\varepsilon,0))\cdot d'_1(k(\varepsilon,0),y)=$$
$$=\beta^{|\varepsilon|}d_{\beta}(\varepsilon)\cdot d'_1(\varepsilon,y)=\beta^0\left(\sum_{i=0}^{|\varepsilon|}\left(\beta^i {f\left(\varepsilon,i,0\right)}\right)\right)\sum_{i=0}^{|\varepsilon|}\left(1^i {f\left(\varepsilon,i,h(\varepsilon,y)\right)}\prod_{j=1}^{d(y)}\frac{\left(g\left(y,j\right)-i\right)}{g\left(y,j\right)}\right)=$$
$$=\beta^0\left(\sum_{i=0}^{0}\left(\beta^i {f\left(\varepsilon,i,0\right)}\right)\right)\sum_{i=0}^{0}\left(1^i {f\left(\varepsilon,i,h(\varepsilon,y)\right)}\prod_{j=1}^{d(y)}\frac{\left(g\left(y,j\right)-i\right)}{g\left(y,j\right)}\right)=$$
$$=\beta^0\cdot\beta^0 {f\left(\varepsilon,0,0\right)}\left(1^0 {f\left(\varepsilon,0,0\right)}\prod_{j=1}^{d(y)}\frac{\left(g\left(y,j\right)-0\right)}{g\left(y,j\right)}\right)=\prod_{j=1}^{d(y)}\frac{g\left(y,j\right)}{g(y,j)}=1.$$
$$d'_{\beta}(x,y)=d'_{\beta}(\varepsilon,y)=\sum_{i=0}^{|\varepsilon|}\left(\beta^i {f\left(\varepsilon,i,h(\varepsilon,y)\right)}\prod_{j=1}^{d(y)}\frac{\left(g\left(y,j\right)-i\right)}{g\left(y,j\right)}\right)=$$
$$=\sum_{i=0}^{0}\left(\beta^i {f\left(\varepsilon,i,h(\varepsilon,y)\right)}\prod_{j=1}^{d(y)}\frac{\left(g\left(y,j\right)-i\right)}{g\left(y,j\right)}\right)=\beta^0 {f\left(\varepsilon,0,0\right)}\prod_{j=1}^{d(y)}\frac{\left(g\left(y,j\right)-0\right)}{g\left(y,j\right)}=\prod_{j=1}^{d(y)}\frac{g\left(y,j\right)}{g(y,j)}=1.$$

Таким образом, мы доказали, что 
$$\sum_{i=0}^{\#\varepsilon}\left( \beta^{|k(\varepsilon,i)|}d_{\beta}(n(\varepsilon,i))\cdot d'_1(k(\varepsilon,i),y)\right)=1=d'_{\beta}(\varepsilon,y),$$
что и требовалось.

\underline{\textbf{База}} доказана.

\underline{\textbf{Переход}} к $x\in\mathbb{YF}:\;\#x\ge 1$:

Ясно, что $\#x\ge 1\Longrightarrow \exists \alpha_0\in\{1,2\},$ $ x'\in\mathbb{YF}:$ $x=\alpha_0x'$.

\begin{Prop}[Лемма 7\cite{Evtuh1}, Утверждение 1.8\cite{Evtuh2}]\label{evtuh7}
Пусть $x\in\mathbb{YF},$ $y,z\in\mathbb{N}_0,$ $\alpha_0 \in \{1,2\}:$ $y \in \overline{|x|},$ $z \in \overline{\#x}$. Тогда 
    $$f(x,y,z)=f(\alpha_0x,y,z)(|\alpha_0x|-y).$$
\end{Prop}

    Ясно, что в условии Утверждения $\ref{evtuh7}$ $|\alpha_0x|-y\ge |\alpha_0x|-|x|=|\alpha_0|+|x|-|x|=\alpha_0>0$, что значит, что при
    $x\in\mathbb{YF},$ $y,z\in\mathbb{N}_0,$ $\alpha_0 \in \{1,2\}:$ $y \in \overline{|x|},$ $z \in \overline{\#x}$
    $$f(\alpha_0x,y,z)=\frac{f(x,y,z)}{(|\alpha_0x|-y)}.$$
    
Посчитаем, воспользовавшись этим Утверждением. Рассмотрим два случая:
\begin{enumerate}
    \item $h(x,y)<\# x$. 
    
    Снова рассмотрим два случая:

\begin{enumerate}
    \item $\alpha_0=1$.
    
    Вначале посчитаем, чему равна левая часть нашего равенства:
$$d'_{\beta}(x,y)=\sum_{i=0}^{|x|}\left(\beta^i {f\left(x,i,h(x,y)\right)}\prod_{j=1}^{d(y)}\frac{\left(g\left(y,j\right)-i\right)}{g\left(y,j\right)}\right)=$$
$$=\sum_{i=0}^{|1x'|}\left(\beta^i {f\left(1x',i,h(1x',y)\right)}\prod_{j=1}^{d(y)}\frac{\left(g\left(y,j\right)-i\right)}{g\left(y,j\right)}\right)=$$
$$=\sum_{i=0}^{|x'|}\left(\beta^i {f\left(1x',i,h(1x',y)\right)}\prod_{j=1}^{d(y)}\frac{\left(g\left(y,j\right)-i\right)}{g\left(y,j\right)}\right)+$$
$$+\sum_{i=|1x'|}^{|1x'|}\left(\beta^i {f\left(1x',i,h(1x',y)\right)}\prod_{j=1}^{d(y)}\frac{\left(g\left(y,j\right)-i\right)}{g\left(y,j\right)}\right)=$$
$$=\sum_{i=0}^{|x'|}\left(\beta^i {f\left(1x',i,h(1x',y)\right)}\prod_{j=1}^{d(y)}\frac{\left(g\left(y,j\right)-i\right)}{g\left(y,j\right)}\right)+$$
$$+\beta^{|1x'|} {f\left(1x',|1x'|,h(1x',y)\right)}\prod_{j=1}^{d(y)}\frac{\left(g\left(y,j\right)-|1x'|\right)}{g\left(y,j\right)}.$$
Применим Утверждение \ref{evtuh7} при $x'\in\mathbb{YF},$ $i\in \overline{|x'|},$ $h(1x',y)\in\overline{\#x'} $ и $1\in\{1,2\}$ к каждому слагаемому суммы
(Действительно верно, что $h(1x',y)\in\overline{\#x'}$, так как в данном случае $h(1x',y)=h(x,y)<\#x\Longrightarrow h(1x',y)\le \#x -1 =\#(1x')-1=1+\#x'-1=\#x'$).

Таким образом, наше выражение равняется следующему:
$$\sum_{i=0}^{|x'|}\left(\beta^i \frac{f\left(x',i,h(1x',y)\right)}{|1x'|-i}\prod_{j=1}^{d(y)}\frac{\left(g\left(y,j\right)-i\right)}{g\left(y,j\right)}\right)+$$
$$+\beta^{|1x'|} f\left(1x',|1x'|,h(1x',y)\right)\prod_{j=1}^{d(y)}\frac{\left(g\left(y,j\right)-|1x'|\right)}{g\left(y,j\right)}.$$

Тут мы посчитали, чему равна левая сторона нашего равенства. Теперь будем считать, чему равна правая:
$$ \sum_{i=0}^{\#x} \left(\beta^{|k(x,i)|}d_{\beta}(n(x,i))\cdot d'_1(k(x,i),y)\right)=$$
$$=\sum_{i=0}^{\#x}\left(  \beta^{|k(x,i)|}\left(\sum_{j=0}^{|n(x,i)|}\left(\beta^j {f\left(n(x,i),j,0\right)}\right)\right)d'_1(k(x,i),y)\right)=$$
$$=\sum_{i=0}^{\#(1x')}\left(  \beta^{|k(1x',i)|}\left(\sum_{j=0}^{|n(1x',i)|}\left(\beta^j {f\left(n(1x',i),j,0\right)}\right)\right)d'_1(k(1x',i),y)\right)=$$
$$=\sum_{i=0}^{\#x'}\left(  \beta^{|k(1x',i)|}\left(\sum_{j=0}^{|n(1x',i)|}\left(\beta^j {f\left(n(1x',i),j,0\right)}\right)\right)d'_1(k(1x',i),y)\right)+$$
$$+\sum_{i=\#(1x')}^{\#(1x')}\left(  \beta^{|k(1x',i)|}\left(\sum_{j=0}^{|n(1x',i)|}\left(\beta^j {f\left(n(1x',i),j,0\right)}\right)\right)d'_1(k(1x',i),y)\right)=$$
$$=\sum_{i=0}^{\#x'}\left(  \beta^{|k(1x',i)|}\left(\sum_{j=0}^{|n(1x',i)|}\left(\beta^j {f\left(n(1x',i),j,0\right)}\right)\right)d'_1(k(1x',i),y)\right)+$$
$$ + \beta^{|k(1x',\#(1x'))|}\left(\sum_{j=0}^{|n(1x',\#(1x'))|}\left(\beta^j {f\left(n(1x',\#(1x')),j,0\right)}\right)\right)d'_1(k(1x',\#(1x')),y)=$$
$$=\text{(По Замечанию \ref{mega} при $(1x')\in\mathbb{YF}$)}=$$
$$=\sum_{i=0}^{\#x'}\left(  \beta^{|k(1x',i)|}\left(\sum_{j=0}^{|n(1x',i)|-1}\left(\beta^j {f\left(n(1x',i),j,0\right)}\right)\right)d'_1(k(1x',i),y)\right)+$$
$$+\sum_{i=0}^{\#x'}\left(  \beta^{|k(1x',i)|}\left(\sum_{j=|n(1x',i)|}^{|n(1x',i)|}\left(\beta^j {f\left(n(1x',i),j,0\right)}\right)\right)d'_1(k(1x',i),y)\right)+$$
$$ + \beta^{|1x'|}\left(\sum_{j=0}^{|\varepsilon|}\left(\beta^j {f\left(n(1x',\#(1x')),j,0\right)}\right)\right)d'_1(k(1x',\#(1x')),y)=$$
$$=\sum_{i=0}^{\#x'}\left(  \beta^{|k(1x',i)|}\left(\sum_{j=0}^{|n(1x',i)|-1}\left(\beta^j {f\left(n(1x',i),j,0\right)}\right)\right)d'_1(k(1x',i),y)\right)+$$
$$+\sum_{i=0}^{\#x'}\left(  \beta^{|k(1x',i)|}\left(\beta^{|n(1x',i)|} {f\left(n(1x',i),|n(1x',i)|,0\right)}\right)d'_1(k(1x',i),y)\right)+$$
$$ + \beta^{|1x'|}\left(\sum_{j=0}^{0}\left(\beta^j {f\left(n(1x',\#(1x')),j,0\right)}\right)\right)d'_1(k(1x',\#(1x')),y)=$$
$$=\text{(По Утверждению \ref{rofl} при $x'\in\mathbb{YF},$ $i\in\mathbb{N}_0,$ $1\in\{1,2\}$, ко всем слагаемым первой строчки)}=$$
$$=\sum_{i=0}^{\#x'}\left(  \beta^{|k(1x',i)|}\left(\sum_{j=0}^{|1n(x',i)|-1}\left(\beta^j {f\left(1n(x',i),j,0\right)}\right)\right)d'_1(k(1x',i),y)\right)+$$
$$+\sum_{i=0}^{\#x'}\left(  \beta^{|k(1x',i)|+|n(1x',i)|} {f\left(n(1x',i),|n(1x',i)|,0\right)}\cdot d'_1(k(1x',i),y)\right)+$$
$$ + \beta^{|1x'|}\cdot\beta^0 {f\left(n(1x',\#(1x')),0,0\right)}\cdot d'_1(k(1x',\#(1x')),y).$$

Применим Утверждение \ref{evtuh7} при $(n(x',i))\in\mathbb{YF},$ $j\in \overline{|1n(x',i)|-1}=\overline{|n(x',i)|},$ $0\in\overline{\#(n(x',i))} $ и $1\in\{1,2\}$ к каждому слагаемому первой строчки.

Кроме того, применим Замечание \ref{prunk} при $(1x')\in\mathbb{YF},$ $i\in\mathbb{N}_0$ к каждому слагаемому второй строчки  и получим, что наше выражение равняется следующему:
$$\sum_{i=0}^{\#x'}\left(  \beta^{|k(1x',i)|}\left(\sum_{j=0}^{|1n(x',i)|-1}\left(\beta^j \frac{f\left(n(x',i),j,0\right)}{|1n(x',i)|-j}\right)\right)d'_1(k(1x',i),y)\right)+$$
$$+\sum_{i=0}^{\#x'}\left(  \beta^{|1x'|} {f\left(n(1x',i),|n(1x',i)|,0\right)}\cdot d'_1(k(1x',i),y)\right)+$$
$$ + \beta^{|1x'|} {f\left(n(1x',\#(1x')),0,0\right)}\cdot d'_1(k(1x',\#(1x')),y)=$$
$$=\text{(По Утверждению \ref{rofl} при $x'\in\mathbb{YF},$ $i\in\mathbb{N}_0,$ $ 1\in\{1,2\}$ ко всем слагаемым первой строчки)}=$$
$$=\sum_{i=0}^{\#x'}\left(  \beta^{|k(1x',i)|}\left(\sum_{j=0}^{|n(1x',i)|-1}\left(\beta^j \frac{f\left(n(x',i),j,0\right)}{|n(1x',i)|-j}\right)\right)d'_1(k(1x',i),y)\right)+$$
$$+\beta^{|1x'|}\sum_{i=0}^{\#x'}\left(   {f\left(n(1x',i),|n(1x',i)|,0\right)}\cdot d'_1(k(1x',i),y)\right)+$$
$$ + \beta^{|1x'|} {f\left(n(1x',\#(1x')),0,0\right)}\cdot d'_1(k(1x',\#(1x')),y).$$

Очевидно, что $\forall i\in\overline{\#x'}$
$$|k(1x',i)|+(|n(1x',i)|-1)=(|k(1x',i)|+|n(1x',i)|)-1=$$
$$=(\text{По Замечанию \ref{prunk} при $(1x')\in\mathbb{YF},$ $i\in\mathbb{N}_0$})=$$
$$=|1x'|-1=1+|x'|-1=|x'|.$$
А это значит, что наше выражение равняется следующему:
$$\sum_{k=0}^{|x'|}\left(  \beta^{k}\sum_{\begin{smallmatrix}(l,m)\in\mathbb{N}_0\times\mathbb{N}_0: \\l+|k(1x',m)|=k              \end{smallmatrix}}\left( \frac{f\left(n(x',m),l,0\right)}{|n(1x',m)|-l}\cdot  d'_1(k(1x',m),y)\right)\right)+$$
$$+\beta^{|1x'|}\sum_{i=0}^{\#x'}\left(   {f\left(n(1x',i),|n(1x',i)|,0\right)}\cdot d'_1(k(1x',i),y)\right)+$$
$$ + \beta^{|1x'|} {f\left(n(1x',\#(1x')),0,0\right)}\cdot d'_1(k(1x',\#(1x')),y).$$

Очевидно, что если у нас есть пара $(l,m)\in\mathbb{N}_0\times\mathbb{N}_0:$ $l+|k(1x',m)|=k$, то по Замечанию \ref{prunk} при $(1x')\in\mathbb{YF},$ $m\in\mathbb{N}_0$ ясно, что $l+|1x'|-|n(1x',m)|=k,$ а это значит, что $|n(1x',m)|-l=|1x'|-k.$ Таким образом, наше выражение равно следующему:
$$\sum_{k=0}^{|x'|}\left(  \beta^{k}\sum_{\begin{smallmatrix}(l,m)\in\mathbb{N}_0\times\mathbb{N}_0: \\l+|k(1x',m)|=k              \end{smallmatrix}}\left( \frac{f\left(n(x',m),l,0\right)}{|1x'|-k}\cdot d'_1(k(1x',m),y)\right)\right)+$$
$$+\beta^{|1x'|}\sum_{i=0}^{\#x'}\left(   {f\left(n(1x',i),|n(1x',i)|,0\right)}\cdot d'_1(k(1x',i),y)\right)+$$
$$ + \beta^{|1x'|} {f\left(n(1x',\#(1x')),0,0\right)}\cdot d'_1(k(1x',\#(1x')),y).$$

Тут у нас есть пары $(l,m)\in\mathbb{N}_0\times\mathbb{N}_0:$ $l+|k(1x',m)|=k$ при $k\le |x'|$. Если $m=\#(1x')$, то 
$$k=l+|k(1x',m)|=l+|k(1x',\#(1x'))|=$$
$$=\text{(По Замечанию \ref{mega} при $(1x')\in\mathbb{YF}$)}=$$
$$=l+|1x'|=l+1+|x'|\ge 0+1+k>k.$$
Противоречие. А это значит, что $m<\#(1x')$, то есть $m\in\overline{\#x'}$. Таким образом, 
$$1x'=n(1x',m)k(1x',m),\; \#(k(1x',m))=m,\; m\in\overline{\#x'}\Longleftrightarrow$$
$$\Longleftrightarrow\text{(По Утверждению \ref{rofl} при $x'\in\mathbb{YF},\; m\in\mathbb{N}_0,\; 1\in\{1,2\}$)}\Longleftrightarrow$$
$$\Longleftrightarrow 1x'=1n(x',m)k(1x',m), \;\#(k(1x',m))=m,\; m\in\overline{\#x'}\Longleftrightarrow$$
$$\Longleftrightarrow x'=n(x',m)k(1x',m),\; \#(k(1x',m))=m,\; m\in\overline{\#x'}\Longleftrightarrow$$
$$\Longleftrightarrow k(x',m)=k(1x',m).$$
А значит наше выражение равно следующему:
$$\sum_{k=0}^{|x'|}\left(\beta^k  
\sum_{\begin{smallmatrix}(l,m)\in\mathbb{N}_0\times\mathbb{N}_0: \\l+|k(x',m)|=k              \end{smallmatrix}}\left( \frac{f\left(n(x',m),l,0\right)}{|1x'|-k}\cdot d'_1(k(x',m),y)\right)\right)+$$
$$+\beta^{|1x'|}\sum_{i=0}^{\#x'}\left(   {f\left(n(1x',i),|n(1x',i)|,0\right)}\cdot d'_1(k(1x',i),y)\right)+$$
$$ + \beta^{|1x'|} {f\left(n(1x',\#(1x')),0,0\right)}\cdot d'_1(k(1x',\#(1x')),y).$$

Заметим, что 
$$0=|\varepsilon|=\text{(По Замечанию \ref{mega} при $(1x')\in\mathbb{YF}$)}=|n(1x',\#(1x'))|.$$
Таким образом, наше выражение равно следующему:
$$\sum_{k=0}^{|x'|}\left(\beta^k  
\sum_{\begin{smallmatrix}(l,m)\in\mathbb{N}_0\times\mathbb{N}_0: \\l+|k(x',m)|=k              \end{smallmatrix}}\left( \frac{f\left(n(x',m),l,0\right)}{|1x'|-k}\cdot d'_1(k(x',m),y)\right)\right)+$$
$$+\beta^{|1x'|}\sum_{i=0}^{\#x'}\left(   {f\left(n(1x',i),|n(1x',i)|,0\right)}\cdot d'_1(k(1x',i),y)\right)+$$
$$ + \beta^{|1x'|} {f\left(n(1x',\#(1x')),|n(1x',\#(1x'))|,0\right)}\cdot d'_1(k(1x',\#(1x')),y)=$$
$$=\sum_{k=0}^{|x'|}\left(\beta^k  
\sum_{\begin{smallmatrix}(l,m)\in\mathbb{N}_0\times\mathbb{N}_0: \\l+|k(x',m)|=k              \end{smallmatrix}}\left( \frac{f\left(n(x',m),l,0\right)}{|1x'|-k}\cdot d'_1(k(x',m),y)\right)\right)+$$
$$+\beta^{|1x'|}\sum_{i=0}^{\#x'}\left(   {f\left(n(1x',i),|n(1x',i)|,0\right)}\cdot d'_1(k(1x',i),y)\right)+$$
$$ + \beta^{|1x'|} \sum_{i=\#(1x')}^{\#(1x')}\left(  {f\left(n(1x',i),|n(1x',i)|,0\right)}\cdot d'_1(k(1x',i),y)\right)=$$
$$=\sum_{k=0}^{|x'|}\left(\beta^k  
\sum_{\begin{smallmatrix}(l,m)\in\mathbb{N}_0\times\mathbb{N}_0: \\l+|k(x',m)|=k              \end{smallmatrix}}\left( \frac{f\left(n(x',m),l,0\right)}{|1x'|-k}\cdot d'_1(k(x',m),y)\right)\right)+$$
$$+\beta^{|1x'|}\sum_{i=0}^{\#(1x')}\left(   {f\left(n(1x',i),|n(1x',i)|,0\right)}\cdot d'_1(k(1x',i),y)\right).$$

Тут мы посчитали, чему равна правая сторона нашего равенства. Таким образом, мы поняли, что (вычтем правую сторону равенства из левого):
$$d'_{\beta}(x,y)- \sum_{i=0}^{\#x}\left( \beta^{|k(x,i)|}d_{\beta}(n(x,i))\cdot d'_1(k(x,i),y)\right)= $$
$$=\sum_{i=0}^{|x'|}\left(\beta^i \frac{f\left(x',i,h(1x',y)\right)}{|1x'|-i}\prod_{j=1}^{d(y)}\frac{\left(g\left(y,j\right)-i\right)}{g\left(y,j\right)}\right)+$$
$$+\beta^{|1x'|} f\left(1x',|1x'|,h(1x',y)\right)\prod_{j=1}^{d(y)}\frac{\left(g\left(y,j\right)-|1x'|\right)}{g\left(y,j\right)}-$$
$$-\sum_{k=0}^{|x'|}\left(\beta^k  
\sum_{\begin{smallmatrix}(l,m)\in\mathbb{N}_0\times\mathbb{N}_0: \\l+|k(x',m)|=k              \end{smallmatrix}}\left( \frac{f\left(n(x',m),l,0\right)}{|1x'|-k}\cdot d'_1(k(x',m),y)\right)\right)-$$
$$-\beta^{|1x'|}\sum_{i=0}^{\#(1x')}  \left( {f\left(n(1x',i),|n(1x',i)|,0\right)}\cdot d'_1(k(1x',i),y)\right).$$

Запомним это равенство.

Теперь воспользуемся предположением индукции при $x'\in\mathbb{YF}$:
$$d'_{\beta}(x',y)= \sum_{i=0}^{\#x'} \left(\beta^{|k(x',i)|}d_{\beta}(n(x',i))\cdot d'_1(k(x',i),y) \right)\Longleftrightarrow$$
$$\Longleftrightarrow \sum_{i=0}^{|x'|}\left(\beta^i {f\left(x',i,h(x',y)\right)}\prod_{j=1}^{d(y)}\frac{\left(g\left(y,j\right)-i\right)}{g\left(y,j\right)}\right)=\sum_{i=0}^{\#x'} \left(\beta^{|k(x',i)|}d_{\beta}(n(x',i))\cdot d'_1(k(x',i),y)\right)\Longleftrightarrow$$
$$\Longleftrightarrow(\text{По определению функции $h$, так как в данном случае $h(1x',y)<\#(1x')$})\Longleftrightarrow$$
$$\Longleftrightarrow \sum_{i=0}^{|x'|}\left(\beta^i {f\left(x',i,h(1x',y)\right)}\prod_{j=1}^{d(y)}\frac{\left(g\left(y,j\right)-i\right)}{g\left(y,j\right)}\right)=\sum_{i=0}^{\#x'}\left( \beta^{|k(x',i)|} d_{\beta}(n(x',i))\cdot d'_1(k(x',i),y)\right).$$

Заметим, что по обозначению
$$\sum_{i=0}^{\#x'} \left(\beta^{|k(x',i)|}d_{\beta}(n(x',i))\cdot d'_1(k(x',i),y)\right)=$$
$$=\sum_{i=0}^{\#x'}\left( \beta^{|k(x',i)|}\left(\sum_{j=0}^{|n(x',i)|}\left(\beta^j \left({f\left(n(x',i),j,0\right)}\right)\right)\right)d'_1(k(x',i),y)\right) =$$
$$=(\text{По Замечанию \ref{prunk} при $x'\in\mathbb{YF},$ $i\in\mathbb{N}_0$, если $i\in\overline{\#x'}$, то $|k(x',i)|+|n(x',i)|=|x'|$})=$$
$$=\sum_{k=0}^{|x'|}\left( \beta^{k}\sum_{\begin{smallmatrix}(l,m)\in\mathbb{N}_0\times\mathbb{N}_0: \\l+|k(x',m)|=k              \end{smallmatrix}} \left( {f\left(n(x',m),l,0\right)}\cdot d'_1(k(x',m),y)\right)\right).$$

А значит
$$ \sum_{i=0}^{|x'|}\left(\beta^i {f\left(x',i,h(1x',y)\right)}\prod_{j=1}^{d(y)}\frac{\left(g\left(y,j\right)-i\right)}{g\left(y,j\right)}\right)=$$
$$=\sum_{k=0}^{|x'|}\left( \beta^{k}\sum_{\begin{smallmatrix}(l,m)\in\mathbb{N}_0\times\mathbb{N}_0: \\l+|k(x',m)|=k              \end{smallmatrix}} \left( {f\left(n(x',m),l,0\right)}\cdot d'_1(k(x',m),y)\right)\right)\Longleftrightarrow$$
$$\Longleftrightarrow \sum_{i=0}^{|x'|}\left(\beta^i\left( {f\left(x',i,h(1x',y)\right)}\prod_{j=1}^{d(y)}\frac{\left(g\left(y,j\right)-i\right)}{g\left(y,j\right)}-\right.\right.$$
$$-\left.\left. \sum_{\begin{smallmatrix}(l,m)\in\mathbb{N}_0\times\mathbb{N}_0: \\l+|k(x',m)|=i              \end{smallmatrix}} \left( {f\left(n(x',m),l,0\right)}\cdot d'_1(k(x',m),y)\right)\right)\right)=0.$$

И это равенство верно для любого $\beta\in(0,1]$. А значит слева находится многочлен от $\beta$, тождественно равный нулю. А значит любой его коэффициент при $\beta^i$ при $i\in\overline{|x'|}$ равен нулю. То есть
$\forall i\in\overline{|x'|}$ 
$$ {f\left(x',i,h(1x',y)\right)}\prod_{j=1}^{d(y)}\frac{\left(g\left(y,j\right)-i\right)}{g\left(y,j\right)}-\sum_{\begin{smallmatrix}(l,m)\in\mathbb{N}_0\times\mathbb{N}_0: \\l+|k(x',m)|=i            \end{smallmatrix}} \left( {f\left(n(x',m),l,0\right)}\cdot d'_1(k(x',m),y)\right)=0\Longleftrightarrow$$
$$\Longleftrightarrow{f\left(x',i,h(1x',y)\right)}\prod_{j=1}^{d(y)}\frac{\left(g\left(y,j\right)-i\right)}{g\left(y,j\right)}=\sum_{\begin{smallmatrix}(l,m)\in\mathbb{N}_0\times\mathbb{N}_0: \\l+|k(x',m)|=i            \end{smallmatrix}} \left( {f\left(n(x',m),l,0\right)}\cdot d'_1(k(x',m),y)\right).$$

Теперь заметим, что если $i\in\overline{|x'|}$, то $|1x'|-i\ge |1x'|-|x'|=1+|x'|-|x'|=1>0$. А значит $\forall i\in\overline{|x'|}$
$$\beta^i \frac{f\left(x',i,h(1x',y)\right)}{|1x'|-i}\prod_{j=1}^{d(y)}\frac{\left(g\left(y,j\right)-i\right)}{g\left(y,j\right)}=$$
$$=\beta^i\frac{\displaystyle\sum_{\begin{smallmatrix}(l,m)\in\mathbb{N}_0\times\mathbb{N}_0: \\l+|k(x',m)|=i            \end{smallmatrix}} \left( {f\left(n(x',m),l,0\right)}\cdot d'_1(k(x',m),y)\right)}{\displaystyle|1x'|-i}=$$
$$=\beta^i  
\sum_{\begin{smallmatrix}(l,m)\in\mathbb{N}_0\times\mathbb{N}_0: \\l+|k(x',m)|=i              \end{smallmatrix}}\left( \frac{f\left(n(x',m),l,0\right)}{|1x'|-i}\cdot d'_1(k(x',m),y)\right).$$

Просуммируем данное равенство по $i\in\overline{|x'|}$:
$$\sum_{i=0}^{|x'|}\left(\beta^i \frac{f\left(x',i,h(1x',y)\right)}{|1x'|-i}\prod_{j=1}^{d(y)}\frac{\left(g\left(y,j\right)-i\right)}{g\left(y,j\right)}\right)=$$
$$=\sum_{i=0}^{|x'|}\left(\beta^i  
\sum_{\begin{smallmatrix}(l,m)\in\mathbb{N}_0\times\mathbb{N}_0: \\l+|k(x',m)|=i              \end{smallmatrix}}\left( \frac{f\left(n(x',m),l,0\right)}{|1x'|-i}\cdot d'_1(k(x',m),y)\right)\right)=$$
$$=\sum_{k=0}^{|x'|}\left(\beta^k  
\sum_{\begin{smallmatrix}(l,m)\in\mathbb{N}_0\times\mathbb{N}_0: \\l+|k(x',m)|=k              \end{smallmatrix}}\left( \frac{f\left(n(x',m),l,0\right)}{|1x'|-k}\cdot d'_1(k(x',m),y)\right)\right).$$

А это значит, что (вернёмся к запомненному равенству)
$$d'_{\beta}(x,y)- \sum_{i=0}^{\#x}\left( \beta^{|k(x,i)|}d_{\beta}(n(x,i))\cdot d'_1(k(x,i),y)\right)= $$
$$=\sum_{i=0}^{|x'|}\left(\beta^i \frac{f\left(x',i,h(1x',y)\right)}{|1x'|-i}\prod_{j=1}^{d(y)}\frac{\left(g\left(y,j\right)-i\right)}{g\left(y,j\right)}\right)+$$
$$+\beta^{|1x'|} f\left(1x',|1x'|,h(1x',y)\right)\prod_{j=1}^{d(y)}\frac{\left(g\left(y,j\right)-|1x'|\right)}{g\left(y,j\right)}-$$
$$-\sum_{k=0}^{|x'|}\left(\beta^k  
\sum_{\begin{smallmatrix}(l,m)\in\mathbb{N}_0\times\mathbb{N}_0: \\l+|k(x',m)|=k              \end{smallmatrix}}\left( \frac{f\left(n(x',m),l,0\right)}{|1x'|-k}\cdot d'_1(k(x',m),y)\right)\right)-$$
$$-\beta^{|1x'|}\sum_{i=0}^{\#(1x')}  \left( {f\left(n(1x',i),|n(1x',i)|,0\right)}\cdot d'_1(k(1x',i),y)\right)=$$
$$=\beta^{|1x'|} f\left(1x',|1x'|,h(1x',y)\right)\prod_{j=1}^{d(y)}\frac{\left(g\left(y,j\right)-|1x'|\right)}{g\left(y,j\right)}-$$
$$-\beta^{|1x'|}\sum_{i=0}^{\#(1x')}  \left( {f\left(n(1x',i),|n(1x',i)|,0\right)}\cdot d'_1(k(1x',i),y)\right)=$$
$$=\beta^{|1x'|}\left(f\left(1x',|1x'|,h(1x',y)\right)\prod_{j=1}^{d(y)}\frac{\left(g\left(y,j\right)-|1x'|\right)}{g\left(y,j\right)}-\right.$$
$$\left.-\sum_{i=0}^{\#(1x')}  \left( {f\left(n(1x',i),|n(1x',i)|,0\right)}\cdot d'_1(k(1x',i),y)\right)\right).$$

По Утверждению \ref{beta1} при наших $x\in\mathbb{YF}$ и $y\in\mathbb{YF}_\infty$
$$d'_{1}(x,y)= \sum_{i=0}^{\#x} \left(1^{|k(x,i)|}d_{1}(n(x,i))\cdot d'_1(k(x,i),y)\right), $$
а это значит, что (подставим $\beta = 1$ в равенство, к которому мы пришли)
$$0=d'_{1}(x,y)- \sum_{i=0}^{\#x}\left( 1^{|k(x,i)|}d_{1}(n(x,i))\cdot d'_1(k(x,i),y)\right)= $$
$$=1^{|1x'|}\left(f\left(1x',|1x'|,h(1x',y)\right)\prod_{j=1}^{d(y)}\frac{\left(g\left(y,j\right)-|1x'|\right)}{g\left(y,j\right)}-\right.$$
$$\left.-\sum_{i=0}^{\#(1x')}  \left( {f\left(n(1x',i),|n(1x',i)|,0\right)}\cdot d'_1(k(1x',i),y)\right)\right)\Longrightarrow$$
$$\Longrightarrow f\left(1x',|1x'|,h(1x',y)\right)\prod_{j=1}^{d(y)}\frac{\left(g\left(y,j\right)-|1x'|\right)}{g\left(y,j\right)}-$$
$$-\sum_{i=0}^{\#(1x')}  \left( {f\left(n(1x',i),|n(1x',i)|,0\right)}\cdot d'_1(k(1x',i),y)\right)=0.$$

Таким образом, ясно, что
$$d'_{\beta}(x,y)- \sum_{i=0}^{\#x}\left( \beta^{|k(x,i)|}d_{\beta}(n(x,i))\cdot d'_1(k(x,i),y)\right)= $$
$$=\beta^{|1x'|}\left(f\left(1x',|1x'|,h(1x',y)\right)\prod_{j=1}^{d(y)}\frac{\left(g\left(y,j\right)-|1x'|\right)}{g\left(y,j\right)}-\right.$$
$$\left.-\sum_{i=0}^{\#(1x')}  \left( {f\left(n(1x',i),|n(1x',i)|,0\right)}\cdot d'_1(k(1x',i),y)\right)\right)=\beta^{|1x'|}\cdot 0=0\Longrightarrow$$
$$\Longrightarrow d'_{\beta}(x,y)= \sum_{i=0}^{\#x}\left( \beta^{|k(x,i)|}d_{\beta}(n(x,i))\cdot d'_1(k(x,i),y)\right),$$
что и требовалось.

В данном случае \underline{\textbf{Переход}} доказан.

    \item $\alpha_0=2$: 
    
    Вначале посчитаем, чему равна левая часть нашего равенства:
$$d'_{\beta}(x,y)=\sum_{i=0}^{|x|}\left(\beta^i {f\left(x,i,h(x,y)\right)}\prod_{j=1}^{d(y)}\frac{\left(g\left(y,j\right)-i\right)}{g\left(y,j\right)}\right)=$$
$$=\sum_{i=0}^{|2x'|}\left(\beta^i {f\left(2x',i,h(2x',y)\right)}\prod_{j=1}^{d(y)}\frac{\left(g\left(y,j\right)-i\right)}{g\left(y,j\right)}\right)=$$
$$=\sum_{i=0}^{|x'|}\left(\beta^i {f\left(2x',i,h(2x',y)\right)}\prod_{j=1}^{d(y)}\frac{\left(g\left(y,j\right)-i\right)}{g\left(y,j\right)}\right)+$$
$$+\sum_{i=|x'|+1}^{|x'|+1}\left(\beta^i {f\left(2x',i,h(2x',y)\right)}\prod_{j=1}^{d(y)}\frac{\left(g\left(y,j\right)-i\right)}{g\left(y,j\right)}\right)+$$
$$+\sum_{i=|2x'|}^{|2x'|}\left(\beta^i {f\left(2x',i,h(2x',y)\right)}\prod_{j=1}^{d(y)}\frac{\left(g\left(y,j\right)-i\right)}{g\left(y,j\right)}\right)=$$
$$=\sum_{i=0}^{|x'|}\left(\beta^i {f\left(2x',i,h(2x',y)\right)}\prod_{j=1}^{d(y)}\frac{\left(g\left(y,j\right)-i\right)}{g\left(y,j\right)}\right)+$$
$$+\beta^{|x'|+1} {f\left(2x',|x'|+1,h(2x',y)\right)}\prod_{j=1}^{d(y)}\frac{\left(g\left(y,j\right)-|x'|-1\right)}{g\left(y,j\right)}+$$
$$+\beta^{|2x'|} {f\left(2x',|2x'|,h(2x',y)\right)}\prod_{j=1}^{d(y)}\frac{\left(g\left(y,j\right)-|2x'|\right)}{g\left(y,j\right)}.$$

\begin{Prop}[Утверждение 9.3\cite{Evtuh1}, Утверждение 1.11\cite{Evtuh2}]\label{evtuh93}
Пусть $x\in\mathbb{YF},$ $z\in\mathbb{N}_0:$ $z\in\overline{\#(2x)}$. Тогда
$$f(2x,|x|+1,z)=0.$$
\end{Prop}

Применим Утверждение \ref{evtuh93} при $x'\in\mathbb{YF}$, $h(2x',y)\in\mathbb{N}_0$ и поймём, что наше выражение равно следующему:
$$\sum_{i=0}^{|x'|}\left(\beta^i {f\left(2x',i,h(2x',y)\right)}\prod_{j=1}^{d(y)}\frac{\left(g\left(y,j\right)-i\right)}{g\left(y,j\right)}\right)+$$
$$+\beta^{|x'|+1}\cdot {0}\cdot\prod_{j=1}^{d(y)}\frac{\left(g\left(y,j\right)-|x'|-1\right)}{g\left(y,j\right)}+$$
$$+\beta^{|2x'|} {f\left(2x',|2x'|,h(2x',y)\right)}\prod_{j=1}^{d(y)}\frac{\left(g\left(y,j\right)-|2x'|\right)}{g\left(y,j\right)}=$$
$$=\sum_{i=0}^{|x'|}\left(\beta^i {f\left(2x',i,h(2x',y)\right)}\prod_{j=1}^{d(y)}\frac{\left(g\left(y,j\right)-i\right)}{g\left(y,j\right)}\right)+$$
$$+\beta^{|2x'|} {f\left(2x',|2x'|,h(2x',y)\right)}\prod_{j=1}^{d(y)}\frac{\left(g\left(y,j\right)-|2x'|\right)}{g\left(y,j\right)}.$$

Применим Утверждение \ref{evtuh7} при $x'\in\mathbb{YF},$ $i\in \overline{|x'|},$ $h(2x',y)\in\overline{\#x'} $ и $2\in\{1,2\}$ к каждому слагаемому суммы (Действительно верно, что $h(2x',y)\in\overline{\#x'}$, так как в данном случае $h(2x',y)=h(x,y)<\#x\Longrightarrow h(2x',y)\le \#x -1 =\#(2x')-1=1+\#x'-1=\#x'$):

Таким образом, наше выражение равняется следующему:
$$\sum_{i=0}^{|x'|}\left(\beta^i \frac{f\left(x',i,h(2x',y)\right)}{|2x'|-i}\prod_{j=1}^{d(y)}\frac{\left(g\left(y,j\right)-i\right)}{g\left(y,j\right)}\right)+$$
$$+\beta^{|2x'|} f\left(2x',|2x'|,h(2x',y)\right)\prod_{j=1}^{d(y)}\frac{\left(g\left(y,j\right)-|2x'|\right)}{g\left(y,j\right)}.$$

Тут мы посчитали, чему равна левая сторона нашего равенства. Теперь будем считать, чему равна правая:
$$\sum_{i=0}^{\#x} \left(\beta^{|k(x,i)|}d_{\beta}(n(x,i))\cdot d'_1(k(x,i),y)\right)=$$
$$=\sum_{i=0}^{\#x}\left(  \beta^{|k(x,i)|}\left(\sum_{j=0}^{|n(x,i)|}\left(\beta^j {f\left(n(x,i),j,0\right)}\right)\right)d'_1(k(x,i),y)\right)=$$
$$=\sum_{i=0}^{\#(2x')}\left(  \beta^{|k(2x',i)|}\left(\sum_{j=0}^{|n(2x',i)|}\left(\beta^j {f\left(n(2x',i),j,0\right)}\right)\right)d'_1(k(2x',i),y)\right)=$$
$$=\sum_{i=0}^{\#x'}\left(  \beta^{|k(2x',i)|}\left(\sum_{j=0}^{|n(2x',i)|}\left(\beta^j {f\left(n(2x',i),j,0\right)}\right)\right)d'_1(k(2x',i),y)\right)+$$
$$+\sum_{i=\#(2x')}^{\#(2x')}\left(  \beta^{|k(2x',i)|}\left(\sum_{j=0}^{|n(2x',i)|}\left(\beta^j {f\left(n(2x',i),j,0\right)}\right)\right)d'_1(k(2x',i),y)\right)=$$
$$=\sum_{i=0}^{\#x'}\left(  \beta^{|k(2x',i)|}\left(\sum_{j=0}^{|n(2x',i)|}\left(\beta^j {f\left(n(2x',i),j,0\right)}\right)\right)d'_1(k(2x',i),y)\right)+$$
$$ + \beta^{|k(2x',\#(2x'))|}\left(\sum_{j=0}^{|n(2x',\#(2x'))|}\left(\beta^j {f\left(n(2x',\#(2x')),j,0\right)}\right)\right)d'_1(k(2x',\#(2x')),y)=$$
$$=\text{(По Замечанию \ref{mega} при $(2x')\in\mathbb{YF}$)}=$$
$$=\sum_{i=0}^{\#x'}\left(  \beta^{|k(2x',i)|}\left(\sum_{j=0}^{|n(2x',i)|-2}\left(\beta^j {f\left(n(2x',i),j,0\right)}\right)\right)d'_1(k(2x',i),y)\right)+$$
$$+\sum_{i=0}^{\#x'}\left(  \beta^{|k(2x',i)|}\left(\sum_{j=|n(2x',i)|-1}^{|n(2x',i)|-1}\left(\beta^j {f\left(n(2x',i),j,0\right)}\right)\right)d'_1(k(2x',i),y)\right)+$$
$$+\sum_{i=0}^{\#x'}\left(  \beta^{|k(2x',i)|}\left(\sum_{j=|n(2x',i)|}^{|n(2x',i)|}\left(\beta^j {f\left(n(2x',i),j,0\right)}\right)\right)d'_1(k(2x',i),y)\right)+$$
$$ + \beta^{|2x'|}\left(\sum_{j=0}^{|\varepsilon|}\left(\beta^j {f\left(n(2x',\#(2x')),j,0\right)}\right)\right)d'_1(k(2x',\#(2x')),y)=$$
$$=\sum_{i=0}^{\#x'}\left(  \beta^{|k(2x',i)|}\left(\sum_{j=0}^{|n(2x',i)|-2}\left(\beta^j {f\left(n(2x',i),j,0\right)}\right)\right)d'_1(k(2x',i),y)\right)+$$
$$+\sum_{i=0}^{\#x'}\left(  \beta^{|k(2x',i)|}\left(\beta^{|n(2x',i)|-1} {f\left(n(2x',i),|n(2x',i)|-1,0\right)}\right)d'_1(k(2x',i),y)\right)+$$
$$+\sum_{i=0}^{\#x'}\left(  \beta^{|k(2x',i)|}\left(\beta^{|n(2x',i)|} {f\left(n(2x',i),|n(2x',i)|,0\right)}\right)d'_1(k(2x',i),y)\right)+$$
$$ + \beta^{|2x'|}\left(\sum_{j=0}^{0}\left(\beta^j {f\left(n(2x',\#(2x')),j,0\right)}\right)\right)d'_1(k(2x',\#(2x')),y)=$$
$$=\text{(По Утверждению \ref{rofl} при $x'\in\mathbb{YF},$ $i\in\mathbb{N}_0,$ $ 2\in\{1,2\}$ ко всем слагаемым первых строчек)}=$$
$$=\sum_{i=0}^{\#x'}\left(  \beta^{|k(2x',i)|}\left(\sum_{j=0}^{|2n(x',i)|-2}\left(\beta^j {f\left(2n(x',i),j,0\right)}\right)\right)d'_1(k(2x',i),y)\right)+$$
$$+\sum_{i=0}^{\#x'}\left(  \beta^{|k(2x',i)|}\left(\beta^{|n(2x',i)|-1} {f\left(2n(x',i),|2n(x',i)|-1,0\right)}\right)d'_1(k(2x',i),y)\right)+$$
$$+\sum_{i=0}^{\#x'}\left(  \beta^{|k(2x',i)|+|n(2x',i)|} {f\left(n(2x',i),|n(2x',i)|,0\right)}\cdot d'_1(k(2x',i),y)\right)+$$
$$ + \beta^{|2x'|}\cdot\beta^0 {f\left(n(2x',\#(2x')),0,0\right)}\cdot d'_1(k(2x',\#(2x')),y).$$

Применим Утверждение \ref{evtuh7} при $(n(x',i))\in\mathbb{YF},$ $j\in \overline{|2n(x',i)|-2}=\overline{|n(x',i)|},$ $0\in\overline{\#(n(x',i))} $ и $2\in\{1,2\}$ к каждому слагаемому первой строчки.

Кроме того, применим Замечание \ref{prunk} при $(2x')\in\mathbb{YF},$ $i\in\mathbb{N}_0$ к каждому слагаемому третьей строчки и получим, что наше выражение равняется следующему:
$$\sum_{i=0}^{\#x'}\left(  \beta^{|k(2x',i)|}\left(\sum_{j=0}^{|2n(x',i)|-2}\left(\beta^j \frac{f\left(n(x',i),j,0\right)}{|2n(x',i)|-j}\right)\right)d'_1(k(2x',i),y)\right)+$$
$$+\sum_{i=0}^{\#x'}\left(  \beta^{|k(2x',i)|}\left(\beta^{|n(2x',i)|-1} {f\left(2n(x',i),|n(x',i)|+1,0\right)}\right)d'_1(k(2x',i),y)\right)+$$
$$+\sum_{i=0}^{\#x'}\left(  \beta^{|2x'|} {f\left(n(2x',i),|n(2x',i)|,0\right)}\cdot d'_1(k(2x',i),y)\right)+$$
$$ + \beta^{|2x'|} {f\left(n(2x',\#(2x')),0,0\right)}\cdot d'_1(k(2x',\#(2x')),y)=$$
$$=\text{(По Утверждению \ref{rofl} при $x'\in\mathbb{YF},$ $i\in\mathbb{N}_0,$ $ 2\in\{1,2\}$ ко всем слагаемым первой строчки)}=$$
$$\sum_{i=0}^{\#x'}\left(  \beta^{|k(2x',i)|}\left(\sum_{j=0}^{|n(2x',i)|-2}\left(\beta^j \frac{f\left(n(x',i),j,0\right)}{|n(2x',i)|-j}\right)\right)d'_1(k(2x',i),y)\right)+$$
$$+\sum_{i=0}^{\#x'}\left(  \beta^{|k(2x',i)|}\left(\beta^{|n(2x',i)|-1} {f\left(2n(x',i),|n(x',i)|+1,0\right)}\right) d'_1(k(2x',i),y)\right)+$$
$$+\beta^{|2x'|}\sum_{i=0}^{\#x'}\left(   {f\left(n(2x',i),|n(2x',i)|,0\right)}\cdot d'_1(k(2x',i),y)\right)+$$
$$ + \beta^{|2x'|}{f\left(n(2x',\#(2x')),0,0\right)}\cdot d'_1(k(2x',\#(2x')),y).$$

Применим Утверждение \ref{evtuh93} к каждому слагаемому второй строчки при  $n(x',i)\in\mathbb{YF}$, $0\in\mathbb{N}_0$ и поймём, что наше выражение равно следующему:
$$\sum_{i=0}^{\#x'}\left(  \beta^{|k(2x',i)|}\left(\sum_{j=0}^{|n(2x',i)|-2}\left(\beta^j \frac{f\left(n(x',i),j,0\right)}{|n(2x',i)|-j}\right)\right)d'_1(k(2x',i),y)\right)+$$
$$+\sum_{i=0}^{\#x'}\left(  \beta^{|k(2x',i)|}\left(\beta^{|n(2x',i)|-1} \cdot 0\right)d'_1(k(2x',i),y)\right)+$$
$$+ \beta^{|2x'|}\sum_{i=0}^{\#x'}\left(   {f\left(n(2x',i),|n(2x',i)|,0\right)}\cdot d'_1(k(2x',i),y)\right)+$$
$$ + \beta^{|2x'|} {f\left(n(2x',\#(2x')),0,0\right)}\cdot d'_1(k(2x',\#(2x')),y)=$$
$$=\sum_{i=0}^{\#x'}\left(  \beta^{|k(2x',i)|}\left(\sum_{j=0}^{|n(2x',i)|-2}\left(\beta^j \frac{f\left(n(x',i),j,0\right)}{|n(2x',i)|-j}\right)\right)d'_1(k(2x',i),y)\right)+$$
$$+\beta^{|2x'|}\sum_{i=0}^{\#x'}\left(   {f\left(n(2x',i),|n(2x',i)|,0\right)}\cdot d'_1(k(2x',i),y)\right)+$$
$$ + \beta^{|2x'|}{f\left(n(2x',\#(2x')),0,0\right)}\cdot d'_1(k(2x',\#(2x')),y).$$

Очевидно, что $\forall i\in\overline{\#x'}$
$$|k(2x',i)|+(|n(2x',i)|-2)=(|k(2x',i)|+|n(2x',i)|)-2=$$
$$=(\text{По Замечанию \ref{prunk} при $(2x')\in\mathbb{YF},$ $i\in\mathbb{N}_0$})=$$
$$=|2x'|-2=2+|x'|-2=|x'|.$$

А это значит, что наше выражение равняется следующему:
$$\sum_{k=0}^{|x'|}\left(  \beta^{k}\sum_{\begin{smallmatrix}(l,m)\in\mathbb{N}_0\times\mathbb{N}_0: \\l+|k(2x',m)|=k              \end{smallmatrix}}\left( \frac{f\left(n(x',m),l,0\right)}{|n(2x',m)|-l}\cdot d'_1(k(2x',m),y)\right)\right)+$$
$$+\beta^{|2x'|}\sum_{i=0}^{\#x'}\left(   {f\left(n(2x',i),|n(2x',i)|,0\right)}\cdot d'_1(k(2x',i),y)\right)+$$
$$ + \beta^{|2x'|} {f\left(n(2x',\#(2x')),0,0\right)}\cdot d'_1(k(2x',\#(2x')),y).$$

Очевидно, что если у нас есть пара $(l,m)\in\mathbb{N}_0\times\mathbb{N}_0:$ $l+|k(2x',m)|=k$, то по Замечанию \ref{prunk} при $(2x')\in\mathbb{YF},$ $m\in\mathbb{N}_0$ ясно, что $l+|2x'|-|n(2x',m)|=k,$ а это значит, что $|n(2x',m)|-l=|2x'|-k.$ Таким образом, наше выражение равно следующему:
$$\sum_{k=0}^{|x'|}\left(  \beta^{k}\sum_{\begin{smallmatrix}(l,m)\in\mathbb{N}_0\times\mathbb{N}_0: \\l+|k(2x',m)|=k              \end{smallmatrix}}\left( \frac{f\left(n(x',m),l,0\right)}{|2x'|-k}\cdot d'_1(k(2x',m),y)\right)\right)+$$
$$+\beta^{|2x'|}\sum_{i=0}^{\#x'}\left(   {f\left(n(2x',i),|n(2x',i)|,0\right)}\cdot d'_1(k(2x',i),y)\right)+$$
$$ + \beta^{|2x'|} {f\left(n(2x',\#(2x')),0,0\right)}\cdot d'_1(k(2x',\#(2x')),y).$$

Тут у нас есть пары $(l,m)\in\mathbb{N}_0\times\mathbb{N}_0:$ $l+|k(2x',m)|=k$ при $k\le |x'|$. Если $m=\#(2x')$, то 
$$k=l+|k(2x',m)|=l+|k(2x',\#(2x'))|=$$
$$=\text{(По Замечанию \ref{mega} при $(2x')\in\mathbb{YF}$)}=$$
$$=l+|2x'|=l+2+|x'|\ge 0+2+k>k.$$

Противоречие. А это значит, что $m<\#(2x')$, то есть $m\in\overline{\#x'}$. Таким образом, 
$$2x'=n(2x',m)k(2x',m),\; \#(k(2x',m))=m,\; m\in\overline{\#x'}\Longleftrightarrow$$
$$\Longleftrightarrow\text{(По Утверждению \ref{rofl} при $x'\in\mathbb{YF},$ $ m\in\mathbb{N}_0,$ $ 2\in\{1,2\}$)}\Longleftrightarrow$$
$$\Longleftrightarrow2x'=2n(x',m)k(2x',m),\; \#(k(2x',m))=m,\; m\in\overline{\#x'}\Longleftrightarrow$$
$$\Longleftrightarrow x'=n(x',m)k(2x',m), \;\#(k(2x',m))=m,\; m\in\overline{\#x'}\Longleftrightarrow$$
$$\Longleftrightarrow k(x',m)=k(2x',m).$$

Таким образом, наше выражение равно следующему:
$$\sum_{k=0}^{|x'|}\left(\beta^k  
\sum_{\begin{smallmatrix}(l,m)\in\mathbb{N}_0\times\mathbb{N}_0: \\l+|k(x',m)|=k              \end{smallmatrix}}\left( \frac{f\left(n(x',m),l,0\right)}{|2x'|-k}\cdot d'_1(k(x',m),y)\right)\right)+$$
$$+\beta^{|2x'|}\sum_{i=0}^{\#x'}\left(   {f\left(n(2x',i),|n(2x',i)|,0\right)}\cdot d'_1(k(2x',i),y)\right)+$$
$$ + \beta^{|2x'|} {f\left(n(2x',\#(2x')),0,0\right)}\cdot d'_1(k(2x',\#(2x')),y).$$

Заметим, что 
$$0=|\varepsilon|=\text{(По Замечанию \ref{mega} при $(2x')\in\mathbb{YF}$)}=|n(2x',\#(2x'))|.$$
Таким образом, наше выражение равно следующему:
$$\sum_{k=0}^{|x'|}\left(\beta^k  
\sum_{\begin{smallmatrix}(l,m)\in\mathbb{N}_0\times\mathbb{N}_0: \\l+|k(x',m)|=k              \end{smallmatrix}}\left( \frac{f\left(n(x',m),l,0\right)}{|2x'|-k}\cdot d'_1(k(x',m),y)\right)\right)+$$
$$+\beta^{|2x'|}\sum_{i=0}^{\#x'}\left(   {f\left(n(2x',i),|n(2x',i)|,0\right)}\cdot d'_1(k(2x',i),y)\right)+$$
$$ + \beta^{|2x'|} {f\left(n(2x',\#(2x')),|n(2x',\#(2x'))|,0\right)}\cdot d'_1(k(2x',\#(2x')),y)=$$
$$=\sum_{k=0}^{|x'|}\left(\beta^k  
\sum_{\begin{smallmatrix}(l,m)\in\mathbb{N}_0\times\mathbb{N}_0: \\l+|k(x',m)|=k              \end{smallmatrix}}\left( \frac{f\left(n(x',m),l,0\right)}{|2x'|-k}\cdot d'_1(k(x',m),y)\right)\right)+$$
$$+\beta^{|2x'|}\sum_{i=0}^{\#x'}\left(   {f\left(n(2x',i),|n(2x',i)|,0\right)}\cdot d'_1(k(2x',i),y)\right)+$$
$$+\beta^{|2x'|}\sum_{i=\#(2x')}^{\#(2x')}\left(   {f\left(n(2x',i),|n(2x',i)|,0\right)}\cdot d'_1(k(2x',i),y)\right)=$$
$$=\sum_{k=0}^{|x'|}\left(\beta^k  
\sum_{\begin{smallmatrix}(l,m)\in\mathbb{N}_0\times\mathbb{N}_0: \\l+|k(x',m)|=k              \end{smallmatrix}}\left( \frac{f\left(n(x',m),l,0\right)}{|2x'|-k}\cdot d'_1(k(x',m),y)\right)\right)+$$
$$+\beta^{|2x'|}\sum_{i=0}^{\#(2x')}\left(   {f\left(n(2x',i),|n(2x',i)|,0\right)}\cdot d'_1(k(2x',i),y)\right).$$

Тут мы посчитали, чему равна правая сторона нашего равенства. Таким образом, мы поняли, что (вычтем правую сторону равенства из левого)
$$d'_{\beta}(x,y)- \sum_{i=0}^{\#x}\left( \beta^{|k(x,i)|}d_{\beta}(n(x,i))\cdot d'_1(k(x,i),y)\right)= $$
$$=\sum_{i=0}^{|x'|}\left(\beta^i \frac{f\left(x',i,h(2x',y)\right)}{|2x'|-i}\prod_{j=1}^{d(y)}\frac{\left(g\left(y,j\right)-i\right)}{g\left(y,j\right)}\right)+$$
$$+\beta^{|2x'|} f\left(2x',|2x'|,h(2x',y)\right)\prod_{j=1}^{d(y)}\frac{\left(g\left(y,j\right)-|2x'|\right)}{g\left(y,j\right)}-$$
$$-\sum_{k=0}^{|x'|}\left(\beta^k  
\sum_{\begin{smallmatrix}(l,m)\in\mathbb{N}_0\times\mathbb{N}_0: \\l+|k(x',m)|=k              \end{smallmatrix}}\left( \frac{f\left(n(x',m),l,0\right)}{|2x'|-k}\cdot d'_1(k(x',m),y)\right)\right)-$$
$$-\beta^{|2x'|}\sum_{i=0}^{\#(2x')}  \left( {f\left(n(2x',i),|n(2x',i)|,0\right)}\cdot d'_1(k(2x',i),y)\right).$$

Запомним это равенство.

Теперь воспользуемся предположением индукции при $x'\in\mathbb{YF}$:
$$d'_{\beta}(x',y)= \sum_{i=0}^{\#x'} \left(\beta^{|k(x',i)|}d_{\beta}(n(x',i))\cdot d'_1(k(x',i),y) \right)\Longleftrightarrow$$
$$\Longleftrightarrow \sum_{i=0}^{|x'|}\left(\beta^i {f\left(x',i,h(x',y)\right)}\prod_{j=1}^{d(y)}\frac{\left(g\left(y,j\right)-i\right)}{g\left(y,j\right)}\right)=\sum_{i=0}^{\#x'} \left(\beta^{|k(x',i)|}d_{\beta}(n(x',i))\cdot d'_1(k(x',i),y)\right)\Longleftrightarrow$$
$$\Longleftrightarrow(\text{По определению функции $h$, так как в данном случае $h(2x',y)<\#(2x')$})\Longleftrightarrow$$
$$\Longleftrightarrow \sum_{i=0}^{|x'|}\left(\beta^i {f\left(x',i,h(2x',y)\right)}\prod_{j=1}^{d(y)}\frac{\left(g\left(y,j\right)-i\right)}{g\left(y,j\right)}\right)=\sum_{i=0}^{\#x'}\left( \beta^{|k(x',i)|}d_{\beta}(n(x',i))\cdot d'_1(k(x',i),y)\right).$$

Заметим, что по обозначению
$$\sum_{i=0}^{\#x'} \left(\beta^{|k(x',i)|}d_{\beta}(n(x',i))\cdot d'_1(k(x',i),y)\right)=$$
$$=\sum_{i=0}^{\#x'}\left( \beta^{|k(x',i)|}\left(\sum_{j=0}^{|n(x',i)|}\left(\beta^j \left({f\left(n(x',i),j,0\right)}\right)\right)\right) d'_1(k(x',i),y)\right) =$$
$$=(\text{По Замечанию \ref{prunk} при $x\in\mathbb{YF},$ $i\in\mathbb{N}_0$ если $i\in\overline{\#x'}$, то $|k(x',i)|+|n(x',i)|=|x'|$})=$$
$$=\sum_{k=0}^{|x'|}\left( \beta^{k}\sum_{\begin{smallmatrix}(l,m)\in\mathbb{N}_0\times\mathbb{N}_0: \\l+|k(x',m)|=k              \end{smallmatrix}} \left( {f\left(n(x',m),l,0\right)}\cdot d'_1(k(x',m),y)\right)\right).$$

А значит
$$ \sum_{i=0}^{|x'|}\left(\beta^i {f\left(x',i,h(2x',y)\right)}\prod_{j=1}^{d(y)}\frac{\left(g\left(y,j\right)-i\right)}{g\left(y,j\right)}\right)=$$
$$=\sum_{k=0}^{|x'|}\left( \beta^{k}\sum_{\begin{smallmatrix}(l,m)\in\mathbb{N}_0\times\mathbb{N}_0: \\l+|k(x',m)|=k              \end{smallmatrix}} \left( {f\left(n(x',m),l,0\right)}\cdot d'_1(k(x',m),y)\right)\right)\Longleftrightarrow$$
$$\Longleftrightarrow \sum_{i=0}^{|x'|}\left(\beta^i\left( {f\left(x',i,h(2x',y)\right)}\prod_{j=1}^{d(y)}\frac{\left(g\left(y,j\right)-i\right)}{g\left(y,j\right)}-\right.\right.$$
$$\left.\left.- \sum_{\begin{smallmatrix}(l,m)\in\mathbb{N}_0\times\mathbb{N}_0: \\l+|k(x',m)|=i              \end{smallmatrix}} \left( {f\left(n(x',m),l,0\right)}\cdot d'_1(k(x',m),y)\right)\right)\right)=0.$$

И это равенство верно для любого $\beta\in(0,1]$. А значит слева находится многочлен от $\beta$, тождественно равный нулю. А значит любой его коэффициент при $\beta^i$ при $i\in\overline{|x'|}$ равен нулю. То есть
$\forall i\in\overline{|x'|}$ 
$$ {f\left(x',i,h(2x',y)\right)}\prod_{j=1}^{d(y)}\frac{\left(g\left(y,j\right)-i\right)}{g\left(y,j\right)}-\sum_{\begin{smallmatrix}(l,m)\in\mathbb{N}_0\times\mathbb{N}_0: \\l+|k(x',m)|=i            \end{smallmatrix}} \left( {f\left(n(x',m),l,0\right)}\cdot d'_1(k(x',m),y)\right)=0\Longleftrightarrow$$
$$\Longleftrightarrow{f\left(x',i,h(2x',y)\right)}\prod_{j=1}^{d(y)}\frac{\left(g\left(y,j\right)-i\right)}{g\left(y,j\right)}=\sum_{\begin{smallmatrix}(l,m)\in\mathbb{N}_0\times\mathbb{N}_0: \\l+|k(x',m)|=i            \end{smallmatrix}} \left( {f\left(n(x',m),l,0\right)}\cdot d'_1(k(x',m),y)\right).$$

Теперь заметим, что если $i\in\overline{|x'|}$, то $|2x'|-i\ge |2x'|-|x'|=2+|x'|-|x'|=2>0$. А значит $\forall i\in\overline{|x'|}$
$$\beta^i \frac{f\left(x',i,h(2x',y)\right)}{|2x'|-i}\prod_{j=1}^{d(y)}\frac{\left(g\left(y,j\right)-i\right)}{g\left(y,j\right)}=$$
$$=\beta^i\frac{\displaystyle\sum_{\begin{smallmatrix}(l,m)\in\mathbb{N}_0\times\mathbb{N}_0: \\l+|k(x',m)|=i            \end{smallmatrix}} \left( {f\left(n(x',m),l,0\right)}\cdot d'_1(k(x',m),y)\right)}{\displaystyle|2x'|-i}=$$
$$=\beta^i  
\sum_{\begin{smallmatrix}(l,m)\in\mathbb{N}_0\times\mathbb{N}_0: \\l+|k(x',m)|=i              \end{smallmatrix}}\left( \frac{f\left(n(x',m),l,0\right)}{|2x'|-i}\cdot d'_1(k(x',m),y)\right).$$

Просуммируем данное равенство по $i\in\overline{|x'|}$:
$$\sum_{i=0}^{|x'|}\left(\beta^i \frac{f\left(x',i,h(2x',y)\right)}{|2x'|-i}\prod_{j=1}^{d(y)}\frac{\left(g\left(y,j\right)-i\right)}{g\left(y,j\right)}\right)=$$
$$=\sum_{i=0}^{|x'|}\left(\beta^i  
\sum_{\begin{smallmatrix}(l,m)\in\mathbb{N}_0\times\mathbb{N}_0: \\l+|k(x',m)|=i              \end{smallmatrix}}\left( \frac{f\left(n(x',m),l,0\right)}{|2x'|-i}\cdot d'_1(k(x',m),y)\right)\right)=$$
$$=\sum_{k=0}^{|x'|}\left(\beta^k  
\sum_{\begin{smallmatrix}(l,m)\in\mathbb{N}_0\times\mathbb{N}_0: \\l+|k(x',m)|=k              \end{smallmatrix}}\left( \frac{f\left(n(x',m),l,0\right)}{|2x'|-k}\cdot d'_1(k(x',m),y)\right)\right).$$

А это значит, что (вернёмся к запомненному равенству)
$$d'_{\beta}(x,y)- \sum_{i=0}^{\#x}\left( \beta^{|k(x,i)|}d_{\beta}(n(x,i))\cdot d'_1(k(x,i),y)\right)= $$
$$=\sum_{i=0}^{|x'|}\left(\beta^i \frac{f\left(x',i,h(2x',y)\right)}{|2x'|-i}\prod_{j=1}^{d(y)}\frac{\left(g\left(y,j\right)-i\right)}{g\left(y,j\right)}\right)+$$
$$+\beta^{|2x'|} f\left(2x',|2x'|,h(2x',y)\right)\prod_{j=1}^{d(y)}\frac{\left(g\left(y,j\right)-|2x'|\right)}{g\left(y,j\right)}-$$
$$-\sum_{k=0}^{|x'|}\left(\beta^k  
\sum_{\begin{smallmatrix}(l,m)\in\mathbb{N}_0\times\mathbb{N}_0: \\l+|k(x',m)|=k              \end{smallmatrix}}\left( \frac{f\left(n(x',m),l,0\right)}{|2x'|-k}\cdot d'_1(k(x',m),y)\right)\right)-$$
$$-\beta^{|2x'|}\sum_{i=0}^{\#(2x')}  \left( {f\left(n(2x',i),|n(2x',i)|,0\right)}\cdot d'_1(k(2x',i),y)\right)=$$
$$=\beta^{|2x'|} f\left(2x',|2x'|,h(2x',y)\right)\prod_{j=1}^{d(y)}\frac{\left(g\left(y,j\right)-|2x'|\right)}{g\left(y,j\right)}-$$
$$-\beta^{|2x'|}\sum_{i=0}^{\#(2x')}  \left( {f\left(n(2x',i),|n(2x',i)|,0\right)}\cdot d'_1(k(2x',i),y)\right)=$$
$$=\beta^{|2x'|}\left(f\left(2x',|2x'|,h(2x',y)\right)\prod_{j=1}^{d(y)}\frac{\left(g\left(y,j\right)-|2x'|\right)}{g\left(y,j\right)}-\right.$$
$$\left.-\sum_{i=0}^{\#(2x')}  \left( {f\left(n(2x',i),|n(2x',i)|,0\right)}\cdot d'_1(k(2x',i),y)\right)\right).$$

По Утверждению \ref{beta1} при наших $x\in\mathbb{YF}$ и $y\in\mathbb{YF}_\infty$
$$d'_{1}(x,y)= \sum_{i=0}^{\#x} \left(1^{|k(x,i)|}d_{1}(n(x,i))\cdot d'_1(k(x,i),y)\right), $$
а это значит, что (подставим $\beta = 1$ в равенство, к которому мы пришли):
$$0=d'_{1}(x,y)- \sum_{i=0}^{\#x}\left( 1^{|k(x,i)|}d_{1}(n(x,i))\cdot d'_1(k(x,i),y)\right)= $$
$$=1^{|2x'|}\left(f\left(2x',|2x'|,h(2x',y)\right)\prod_{j=1}^{d(y)}\frac{\left(g\left(y,j\right)-|2x'|\right)}{g\left(y,j\right)}-\right.$$
$$\left.-\sum_{i=0}^{\#(2x')}  \left( {f\left(n(2x',i),|n(2x',i)|,0\right)}\cdot d'_1(k(2x',i),y)\right)\right)\Longrightarrow$$
$$\Longrightarrow f\left(2x',|2x'|,h(2x',y)\right)\prod_{j=1}^{d(y)}\frac{\left(g\left(y,j\right)-|2x'|\right)}{g\left(y,j\right)}-$$
$$-\sum_{i=0}^{\#(2x')}  \left( {f\left(n(2x',i),|n(2x',i)|,0\right)}\cdot d'_1(k(2x',i),y)\right)=0.$$

Таким образом, ясно, что
$$d'_{\beta}(x,y)- \sum_{i=0}^{\#x}\left( \beta^{|k(x,i)|}d_{\beta}(n(x,i))\cdot d'_1(k(x,i),y)\right)= $$
$$=\beta^{|2x'|}\left(f\left(2x',|2x'|,h(2x',y)\right)\prod_{j=1}^{d(y)}\frac{\left(g\left(y,j\right)-|2x'|\right)}{g\left(y,j\right)}-\right.$$
$$\left.-\sum_{i=0}^{\#(2x')}  \left( {f\left(n(2x',i),|n(2x',i)|,0\right)}\cdot d'_1(k(2x',i),y)\right)\right)=\beta^{|2x'|}\cdot 0=0\Longrightarrow$$
$$\Longrightarrow d'_{\beta}(x,y)= \sum_{i=0}^{\#x}\left( \beta^{|k(x,i)|}d_{\beta}(n(x,i))\cdot d'_1(k(x,i),y)\right),$$
что и требовалось.

В данном случае \underline{\textbf{Переход}} доказан.

\end{enumerate}
    \item $h(x,y)= \#x$.

    Опять рассмотрим два случая:

\begin{enumerate}
    \item $\alpha_0=1$. 
    
    Вначале посчитаем, чему равна левая часть нашего равенства:
$$d'_{\beta}(x,y)=\sum_{i=0}^{|x|}\left(\beta^i {f\left(x,i,h(x,y)\right)}\prod_{j=1}^{d(y)}\frac{\left(g\left(y,j\right)-i\right)}{g\left(y,j\right)}\right)=$$
$$=\sum_{i=0}^{|x|}\left(\beta^i {f\left(x,i,\#x\right)}\prod_{j=1}^{d(y)}\frac{\left(g\left(y,j\right)-i\right)}{g\left(y,j\right)}\right)=$$
$$=\sum_{i=0}^{|1x'|}\left(\beta^i {f\left(1x',i,\#(1x')\right)}\prod_{j=1}^{d(y)}\frac{\left(g\left(y,j\right)-i\right)}{g\left(y,j\right)}\right)=$$
$$=\sum_{i=0}^{|x'|}\left(\beta^i {f\left(1x',i,\#(1x')\right)}\prod_{j=1}^{d(y)}\frac{\left(g\left(y,j\right)-i\right)}{g\left(y,j\right)}\right)+$$
$$+\sum_{i=|1x'|}^{|1x'|}\left(\beta^i {f\left(1x',i,\#(1x')\right)}\prod_{j=1}^{d(y)}\frac{\left(g\left(y,j\right)-i\right)}{g\left(y,j\right)}\right)=$$
$$=\sum_{i=0}^{|x'|}\left(\beta^i {f\left(1x',i,\#(1x')\right)}\prod_{j=1}^{d(y)}\frac{\left(g\left(y,j\right)-i\right)}{g\left(y,j\right)}\right)+$$
$$+\beta^{|1x'|} {f\left(1x',|1x'|,\#(1x')\right)}\prod_{j=1}^{d(y)}\frac{\left(g\left(y,j\right)-|1x'|\right)}{g\left(y,j\right)}.$$

\begin{Prop}[Утверждение 9.2\cite{Evtuh1}, Утверждение 1.10\cite{Evtuh2}] \label{evtuh92}
Пусть $x\in\mathbb{YF}$, $y\in\mathbb{N}_0:$ $y\in\overline{|x|}.$ Тогда
$$f(1x,y,\#x)=f(1x,y,\#(1x)).$$
\end{Prop}

Применим Утверждение \ref{evtuh92} при $x'\in\mathbb{YF},$ $i\in\mathbb{N}_0$ к каждому слагаемому первой строчки и получим, что наше выражение равняется следующему:
$$\sum_{i=0}^{|x'|}\left(\beta^i {f\left(1x',i,\#x'\right)}\prod_{j=1}^{d(y)}\frac{\left(g\left(y,j\right)-i\right)}{g\left(y,j\right)}\right)+$$
$$+\beta^{|1x'|} {f\left(1x',|1x'|,\#(1x')\right)}\prod_{j=1}^{d(y)}\frac{\left(g\left(y,j\right)-|1x'|\right)}{g\left(y,j\right)}.$$

Применим Утверждение \ref{evtuh7} при $x'\in\mathbb{YF},$ $i\in \overline{|x'|},$ $\#x'\in\overline{\#x'}$ и $1\in\{1,2\}$ к каждому слагаемому суммы  и поймём, что наше выражение равняется следующему:
$$\sum_{i=0}^{|x'|}\left(\beta^i \frac{f\left(x',i,\#x'\right)}{|1x'|-i}\prod_{j=1}^{d(y)}\frac{\left(g\left(y,j\right)-i\right)}{g\left(y,j\right)}\right)+$$
$$+\beta^{|1x'|} f\left(1x',|1x'|,\#(1x')\right)\prod_{j=1}^{d(y)}\frac{\left(g\left(y,j\right)-|1x'|\right)}{g\left(y,j\right)}.$$

Тут мы посчитали, чему равна левая сторона нашего равенства. Теперь будем считать, чему равна правая:
$$\sum_{i=0}^{\#x} \left(\beta^{|k(x,i)|}d_{\beta}(n(x,i))\cdot  d'_1(k(x,i),y)\right)=$$
$$=\sum_{i=0}^{\#x}\left(  \beta^{|k(x,i)|}\left(\sum_{j=0}^{|n(x,i)|}\left(\beta^j {f\left(n(x,i),j,0\right)}\right)\right) d'_1(k(x,i),y)\right)=$$
$$=\sum_{i=0}^{\#(1x')}\left(  \beta^{|k(1x',i)|}\left(\sum_{j=0}^{|n(1x',i)|}\left(\beta^j {f\left(n(1x',i),j,0\right)}\right)\right)d'_1(k(1x',i),y)\right)=$$
$$=\sum_{i=0}^{\#x'}\left(  \beta^{|k(1x',i)|}\left(\sum_{j=0}^{|n(1x',i)|}\left(\beta^j {f\left(n(1x',i),j,0\right)}\right)\right)d'_1(k(1x',i),y)\right)+$$
$$+\sum_{i=\#(1x')}^{\#(1x')}\left(  \beta^{|k(1x',i)|}\left(\sum_{j=0}^{|n(1x',i)|}\left(\beta^j {f\left(n(1x',i),j,0\right)}\right)\right)d'_1(k(1x',i),y)\right)=$$
$$=\sum_{i=0}^{\#x'}\left(  \beta^{|k(1x',i)|}\left(\sum_{j=0}^{|n(1x',i)|}\left(\beta^j {f\left(n(1x',i),j,0\right)}\right)\right)d'_1(k(1x',i),y)\right)+$$
$$ + \beta^{|k(1x',\#(1x'))|}\left(\sum_{j=0}^{|n(1x',\#(1x'))|}\left(\beta^j {f\left(n(1x',\#(1x')),j,0\right)}\right)\right)d'_1(k(1x',\#(1x')),y)=$$
$$=\text{(По Замечанию \ref{mega} при $(1x')\in\mathbb{YF}$)}=$$
$$=\sum_{i=0}^{\#x'}\left(  \beta^{|k(1x',i)|}\left(\sum_{j=0}^{|n(1x',i)|-1}\left(\beta^j {f\left(n(1x',i),j,0\right)}\right)\right)d'_1(k(1x',i),y)\right)+$$
$$+\sum_{i=0}^{\#x'}\left(  \beta^{|k(1x',i)|}\left(\sum_{j=|n(1x',i)|}^{|n(1x',i)|}\left(\beta^j {f\left(n(1x',i),j,0\right)}\right)\right)d'_1(k(1x',i),y)\right)+$$
$$ + \beta^{|1x'|}\left(\sum_{j=0}^{|\varepsilon|}\left(\beta^j {f\left(n(1x',\#(1x')),j,0\right)}\right)\right)d'_1(k(1x',\#(1x')),y)=$$
$$=\sum_{i=0}^{\#x'}\left(  \beta^{|k(1x',i)|}\left(\sum_{j=0}^{|n(1x',i)|-1}\left(\beta^j {f\left(n(1x',i),j,0\right)}\right)\right)d'_1(k(1x',i),y)\right)+$$
$$+\sum_{i=0}^{\#x'}\left(  \beta^{|k(1x',i)|}\left(\beta^{|n(1x',i)|} {f\left(n(1x',i),|n(1x',i)|,0\right)}\right)d'_1(k(1x',i),y)\right)+$$
$$ + \beta^{|1x'|}\left(\sum_{j=0}^{0}\left(\beta^j {f\left(n(1x',\#(1x')),j,0\right)}\right)\right)d'_1(k(1x',\#(1x')),y)=$$
$$=\text{(По Утверждению \ref{rofl} при $x'\in\mathbb{YF},$ $i\in\mathbb{N}_0,$ $1\in\{1,2\}$ ко всем слагаемым первой строчки)}=$$
$$=\sum_{i=0}^{\#x'}\left(  \beta^{|k(1x',i)|}\left(\sum_{j=0}^{|1n(x',i)|-1}\left(\beta^j {f\left(1n(x',i),j,0\right)}\right)\right)d'_1(k(1x',i),y)\right)+$$
$$+\sum_{i=0}^{\#x'}\left(  \beta^{|k(1x',i)|+|n(1x',i)|} {f\left(n(1x',i),|n(1x',i)|,0\right)}\cdot d'_1(k(1x',i),y)\right)+$$
$$ + \beta^{|1x'|}\cdot\beta^0 {f\left(n(1x',\#(1x')),0,0\right)}\cdot d'_1(k(1x',\#(1x')),y).$$

Применим Утверждение \ref{evtuh7} при $(n(x',i))\in\mathbb{YF},$ $j\in \overline{|1n(x',i)|-1}=\overline{|n(x',i)|},$ $0\in\overline{\#(n(x',i))} $ и $1\in\{1,2\}$ к каждому слагаемому первой строчки.

Кроме того, применим Замечание \ref{prunk} при $(1x')\in\mathbb{YF},$ $i\in\mathbb{N}_0$ к каждому слагаемому второй строчки и получим, что наше выражение равняется следующему:
$$\sum_{i=0}^{\#x'}\left(  \beta^{|k(1x',i)|}\left(\sum_{j=0}^{|1n(x',i)|-1}\left(\beta^j \frac{f\left(n(x',i),j,0\right)}{|1n(x',i)|-j}\right)\right) d'_1(k(1x',i),y)\right)+$$
$$+\sum_{i=0}^{\#x'}\left(  \beta^{|1x'|} {f\left(n(1x',i),|n(1x',i)|,0\right)}\cdot d'_1(k(1x',i),y)\right)+$$
$$ + \beta^{|1x'|}{f\left(n(1x',\#(1x')),0,0\right)}\cdot d'_1(k(1x',\#(1x')),y)=$$
$$=\text{(По Утверждению \ref{rofl} при $x'\in\mathbb{YF},$ $i\in\mathbb{N}_0,$ $ 1\in\{1,2\}$ ко всем слагаемым первой строчки)}=$$
$$=\sum_{i=0}^{\#x'}\left(  \beta^{|k(1x',i)|}\left(\sum_{j=0}^{|n(1x',i)|-1}\left(\beta^j \frac{f\left(n(x',i),j,0\right)}{|n(1x',i)|-j}\right)\right)d'_1(k(1x',i),y)\right)+$$
$$+\beta^{|1x'|}\sum_{i=0}^{\#x'}\left(   {f\left(n(1x',i),|n(1x',i)|,0\right)}\cdot d'_1(k(1x',i),y)\right)+$$
$$ + \beta^{|1x'|} {f\left(n(1x',\#(1x')),0,0\right)}\cdot d'_1(k(1x',\#(1x')),y).$$

Очевидно, что $\forall i\in\overline{\#x'}$
$$|k(1x',i)|+(|n(1x',i)|-1)=(|k(1x',i)|+|n(1x',i)|)-1=$$
$$=(\text{По Замечанию \ref{prunk} при $(1x')\in\mathbb{YF},$ $i\in\mathbb{N}_0$})=$$
$$=|1x'|-1=1+|x'|-1=|x'|.$$

А это значит, что наше выражение равняется следующему:
$$\sum_{k=0}^{|x'|}\left(  \beta^{k}\sum_{\begin{smallmatrix}(l,m)\in\mathbb{N}_0\times\mathbb{N}_0: \\l+|k(1x',m)|=k              \end{smallmatrix}}\left( \frac{f\left(n(x',m),l,0\right)}{|n(1x',m)|-l}\cdot d'_1(k(1x',m),y)\right)\right)+$$
$$+\beta^{|1x'|}\sum_{i=0}^{\#x'}\left(   {f\left(n(1x',i),|n(1x',i)|,0\right)}\cdot d'_1(k(1x',i),y)\right)+$$
$$ + \beta^{|1x'|} {f\left(n(1x',\#(1x')),0,0\right)}\cdot d'_1(k(1x',\#(1x')),y).$$

Очевидно, что если у нас есть пара $(l,m)\in\mathbb{N}_0\times\mathbb{N}_0:$ $l+|k(1x',m)|=k$, то по Замечанию \ref{prunk} при $(1x')\in\mathbb{YF},$ $m\in\mathbb{N}_0$ ясно, что $l+|1x'|-|n(1x',m)|=k,$ а это значит, что $|n(1x',m)|-l=|1x'|-k.$ Таким образом, наше выражение равно следующему:
$$\sum_{k=0}^{|x'|}\left(  \beta^{k}\sum_{\begin{smallmatrix}(l,m)\in\mathbb{N}_0\times\mathbb{N}_0: \\l+|k(1x',m)|=k              \end{smallmatrix}}\left( \frac{f\left(n(x',m),l,0\right)}{|1x'|-k}\cdot d'_1(k(1x',m),y)\right)\right)+$$
$$+\beta^{|1x'|}\sum_{i=0}^{\#x'}\left(   {f\left(n(1x',i),|n(1x',i)|,0\right)}\cdot d'_1(k(1x',i),y)\right)+$$
$$ + \beta^{|1x'|}\left( {f\left(n(1x',\#(1x')),0,0\right)}\right)\cdot d'_1(k(1x',\#(1x')),y).$$

Тут у нас есть пары $(l,m)\in\mathbb{N}_0\times\mathbb{N}_0:$ $l+|k(1x',m)|=k$ при $k\le |x'|$. Если $m=\#(1x')$, то 
$$k=l+|k(1x',m)|=l+|k(1x',\#(1x'))|=$$
$$=\text{(По Замечанию \ref{mega} при $(1x')\in\mathbb{YF}$)}=$$
$$=l+|1x'|=l+1+|x'|\ge 0+1+k>k.$$
Противоречие. А это значит, что $m<\#(1x')$, то есть $m\in\overline{\#x'}$. Таким образом, 
$$1x'=n(1x',m)k(1x',m),\; \#(k(1x',m))=m,\; m\in\overline{\#x'}\Longleftrightarrow$$
$$\Longleftrightarrow\text{(По Утверждению \ref{rofl} при $x'\in\mathbb{YF},$ $ m\in\mathbb{N}_0,$ $1\in\{1,2\})$}\Longleftrightarrow$$
$$\Longleftrightarrow 1x'=1n(x',m)k(1x',m),\; \#(k(1x',m))=m,\; m\in\overline{\#x'}\Longleftrightarrow$$
$$\Longleftrightarrow x'=n(x',m)k(1x',m),\; \#(k(1x',m))=m,\; m\in\overline{\#x'}\Longleftrightarrow$$
$$\Longleftrightarrow k(x',m)=k(1x',m).$$

Таким образом, наше выражение равно следующему:
$$\sum_{k=0}^{|x'|}\left(\beta^k  
\sum_{\begin{smallmatrix}(l,m)\in\mathbb{N}_0\times\mathbb{N}_0: \\l+|k(x',m)|=k              \end{smallmatrix}}\left( \frac{f\left(n(x',m),l,0\right)}{|1x'|-k}\cdot d'_1(k(x',m),y)\right)\right)+$$
$$+\beta^{|1x'|}\sum_{i=0}^{\#x'}\left(   {f\left(n(1x',i),|n(1x',i)|,0\right)}\cdot d'_1(k(1x',i),y)\right)+$$
$$ + \beta^{|1x'|} {f\left(n(1x',\#(1x')),0,0\right)}\cdot d'_1(k(1x',\#(1x')),y).$$

Заметим, что 
$$0=|\varepsilon|=\text{(По Замечанию \ref{mega} при $(1x')\in\mathbb{YF}$)}=|n(1x',\#(1x'))|.$$
Таким образом, наше выражение равно следующему:
$$\sum_{k=0}^{|x'|}\left(\beta^k  
\sum_{\begin{smallmatrix}(l,m)\in\mathbb{N}_0\times\mathbb{N}_0: \\l+|k(x',m)|=k              \end{smallmatrix}}\left( \frac{f\left(n(x',m),l,0\right)}{|1x'|-k}\cdot d'_1(k(x',m),y)\right)\right)+$$
$$+\beta^{|1x'|}\sum_{i=0}^{\#x'}\left(   {f\left(n(1x',i),|n(1x',i)|,0\right)}\cdot d'_1(k(1x',i),y)\right)+$$
$$ + \beta^{|1x'|} {f\left(n(1x',\#(1x')),|n(1x',\#(1x'))|,0\right)}\cdot d'_1(k(1x',\#(1x')),y)=$$
$$=\sum_{k=0}^{|x'|}\left(\beta^k  
\sum_{\begin{smallmatrix}(l,m)\in\mathbb{N}_0\times\mathbb{N}_0: \\l+|k(x',m)|=k              \end{smallmatrix}}\left( \frac{f\left(n(x',m),l,0\right)}{|1x'|-k}\cdot d'_1(k(x',m),y)\right)\right)+$$
$$+\beta^{|1x'|}\sum_{i=0}^{\#x'}\left(   {f\left(n(1x',i),|n(1x',i)|,0\right)}\cdot d'_1(k(1x',i),y)\right)+$$
$$ + \beta^{|1x'|} \sum_{i=\#(1x')}^{\#(1x')}\left(  {f\left(n(1x',i),|n(1x',i)|,0\right)}\cdot d'_1(k(1x',i),y)\right)=$$
$$=\sum_{k=0}^{|x'|}\left(\beta^k  
\sum_{\begin{smallmatrix}(l,m)\in\mathbb{N}_0\times\mathbb{N}_0: \\l+|k(x',m)|=k              \end{smallmatrix}}\left( \frac{f\left(n(x',m),l,0\right)}{|1x'|-k}\cdot d'_1(k(x',m),y)\right)\right)+$$
$$+\beta^{|1x'|}\sum_{i=0}^{\#(1x')}\left(   {f\left(n(1x',i),|n(1x',i)|,0\right)}\cdot d'_1(k(1x',i),y)\right).$$

Тут мы посчитали, чему равна правая сторона нашего равенства. Таким образом, мы поняли, что (вычтем правую сторону равенства из левого)
$$d'_{\beta}(x,y)- \sum_{i=0}^{\#x}\left( \beta^{|k(x,i)|}d_{\beta}(n(x,i))\cdot d'_1(k(x,i),y)\right)= $$
$$=\sum_{i=0}^{|x'|}\left(\beta^i \frac{f\left(x',i,\#x'\right)}{|1x'|-i}\prod_{j=1}^{d(y)}\frac{\left(g\left(y,j\right)-i\right)}{g\left(y,j\right)}\right)+$$
$$+\beta^{|1x'|} f\left(1x',|1x'|,\#(1x')\right)\prod_{j=1}^{d(y)}\frac{\left(g\left(y,j\right)-|1x'|\right)}{g\left(y,j\right)}-$$
$$-\sum_{k=0}^{|x'|}\left(\beta^k  
\sum_{\begin{smallmatrix}(l,m)\in\mathbb{N}_0\times\mathbb{N}_0: \\l+|k(x',m)|=k              \end{smallmatrix}}\left( \frac{f\left(n(x',m),l,0\right)}{|1x'|-k}\cdot d'_1(k(x',m),y)\right)\right)-$$
$$-\beta^{|1x'|}\sum_{i=0}^{\#(1x')}  \left( {f\left(n(1x',i),|n(1x',i)|,0\right)}\cdot d'_1(k(1x',i),y)\right).$$

Запомним это равенство.

Теперь воспользуемся предположением индукции при $x'\in\mathbb{YF}$:
$$d'_{\beta}(x',y)= \sum_{i=0}^{\#x'} \left(\beta^{|k(x',i)|}d_{\beta}(n(x',i))\cdot d'_1(k(x',i),y) \right)\Longleftrightarrow$$
$$\Longleftrightarrow \sum_{i=0}^{|x'|}\left(\beta^i {f\left(x',i,h(x',y)\right)}\prod_{j=1}^{d(y)}\frac{\left(g\left(y,j\right)-i\right)}{g\left(y,j\right)}\right)=\sum_{i=0}^{\#x'} \left(\beta^{|k(x',i)|}d_{\beta}(n(x',i))\cdot d'_1(k(x',i),y)\right)\Longleftrightarrow$$
$$\Longleftrightarrow(\text{По определению функции $h$, так как в данном случае $h(1x',y)=\#(1x')$})\Longleftrightarrow$$
$$\Longleftrightarrow \sum_{i=0}^{|x'|}\left(\beta^i {f\left(x',i,\#x'\right)}\prod_{j=1}^{d(y)}\frac{\left(g\left(y,j\right)-i\right)}{g\left(y,j\right)}\right)=\sum_{i=0}^{\#x'}\left( \beta^{|k(x',i)|}d_{\beta}(n(x',i))\cdot d'_1(k(x',i),y)\right).$$

Заметим, что по обозначению
$$\sum_{i=0}^{\#x'} \left(\beta^{|k(x',i)|}d_{\beta}(n(x',i))\cdot d'_1(k(x',i),y)\right)=$$
$$=\sum_{i=0}^{\#x'}\left( \beta^{|k(x',i)|}\left(\sum_{j=0}^{|n(x',i)|}\left(\beta^j \left({f\left(n(x',i),j,0\right)}\right)\right)\right)d'_1(k(x',i),y)\right) =$$
$$=(\text{Так как по Замечанию \ref{prunk} при $x'\in\mathbb{YF},$ $i\in\mathbb{N}_0$, если $i\in\overline{\#x'}$, то $|k(x',i)|+|n(x',i)|=|x'|$})=$$
$$=\sum_{k=0}^{|x'|}\left( \beta^{k}\sum_{\begin{smallmatrix}(l,m)\in\mathbb{N}_0\times\mathbb{N}_0: \\l+|k(x',m)|=k              \end{smallmatrix}} \left( {f\left(n(x',m),l,0\right)}\cdot d'_1(k(x',m),y)\right)\right).$$

А значит
$$ \sum_{i=0}^{|x'|}\left(\beta^i {f\left(x',i,\#x'\right)}\prod_{j=1}^{d(y)}\frac{\left(g\left(y,j\right)-i\right)}{g\left(y,j\right)}\right)=$$
$$=\sum_{k=0}^{|x'|}\left( \beta^{k}\sum_{\begin{smallmatrix}(l,m)\in\mathbb{N}_0\times\mathbb{N}_0: \\l+|k(x',m)|=k              \end{smallmatrix}} \left( {f\left(n(x',m),l,0\right)}\cdot d'_1(k(x',m),y)\right)\right)\Longleftrightarrow$$
$$\Longleftrightarrow \sum_{i=0}^{|x'|}\left(\beta^i\left( {f\left(x',i,\#x'\right)}\prod_{j=1}^{d(y)}\frac{\left(g\left(y,j\right)-i\right)}{g\left(y,j\right)}-\right.\right.-$$
$$-\left.\left. \sum_{\begin{smallmatrix}(l,m)\in\mathbb{N}_0\times\mathbb{N}_0: \\l+|k(x',m)|=i              \end{smallmatrix}} \left( {f\left(n(x',m),l,0\right)}\cdot d'_1(k(x',m),y)\right)\right)\right)=0.$$

И это равенство верно для любого $\beta\in(0,1]$. А значит слева находится многочлен от $\beta$, тождественно равный нулю. А значит любой его коэффициент при $\beta^i$ при $i\in\overline{|x'|}$ равен нулю. То есть
$\forall i\in\overline{|x'|}$ 
$$ {f\left(x',i,\#x'\right)}\prod_{j=1}^{d(y)}\frac{\left(g\left(y,j\right)-i\right)}{g\left(y,j\right)}-\sum_{\begin{smallmatrix}(l,m)\in\mathbb{N}_0\times\mathbb{N}_0: \\l+|k(x',m)|=i            \end{smallmatrix}} \left( {f\left(n(x',m),l,0\right)}\cdot d'_1(k(x',m),y)\right)=0\Longleftrightarrow$$
$$\Longleftrightarrow{f\left(x',i,\#x'\right)}\prod_{j=1}^{d(y)}\frac{\left(g\left(y,j\right)-i\right)}{g\left(y,j\right)}=\sum_{\begin{smallmatrix}(l,m)\in\mathbb{N}_0\times\mathbb{N}_0: \\l+|k(x',m)|=i            \end{smallmatrix}} \left( {f\left(n(x',m),l,0\right)}\cdot d'_1(k(x',m),y)\right).$$

Теперь заметим, что если $i\in\overline{|x'|}$, то $|1x'|-i\ge |1x'|-|x'|=1+|x'|-|x'|=1>0$. А значит $\forall i\in\overline{|x'|}$
$$\beta^i \frac{f\left(x',i,\#x'\right)}{|1x'|-i}\prod_{j=1}^{d(y)}\frac{\left(g\left(y,j\right)-i\right)}{g\left(y,j\right)}=$$
$$=\beta^i\frac{\displaystyle\sum_{\begin{smallmatrix}(l,m)\in\mathbb{N}_0\times\mathbb{N}_0: \\l+|k(x',m)|=i            \end{smallmatrix}} \left( {f\left(n(x',m),l,0\right)}\cdot d'_1(k(x',m),y)\right)}{\displaystyle|1x'|-i}=$$
$$=\beta^i  
\sum_{\begin{smallmatrix}(l,m)\in\mathbb{N}_0\times\mathbb{N}_0: \\l+|k(x',m)|=i              \end{smallmatrix}}\left( \frac{f\left(n(x',m),l,0\right)}{|1x'|-i}\cdot d'_1(k(x',m),y)\right).$$

Просуммируем данное равенство по $i\in\overline{|x'|}$:
$$\sum_{i=0}^{|x'|}\left(\beta^i \frac{f\left(x',i,\#x'\right)}{|1x'|-i}\prod_{j=1}^{d(y)}\frac{\left(g\left(y,j\right)-i\right)}{g\left(y,j\right)}\right)=$$
$$=\sum_{i=0}^{|x'|}\left(\beta^i  
\sum_{\begin{smallmatrix}(l,m)\in\mathbb{N}_0\times\mathbb{N}_0: \\l+|k(x',m)|=i              \end{smallmatrix}}\left( \frac{f\left(n(x',m),l,0\right)}{|1x'|-i}\cdot d'_1(k(x',m),y)\right)\right)=$$
$$=\sum_{k=0}^{|x'|}\left(\beta^k  
\sum_{\begin{smallmatrix}(l,m)\in\mathbb{N}_0\times\mathbb{N}_0: \\l+|k(x',m)|=k              \end{smallmatrix}}\left( \frac{f\left(n(x',m),l,0\right)}{|1x'|-k}\cdot d'_1(k(x',m),y)\right)\right).$$

А это значит, что (вернёмся к запомненному равенству)
$$d'_{\beta}(x,y)- \sum_{i=0}^{\#x}\left( \beta^{|k(x,i)|}d_{\beta}(n(x,i))\cdot d'_1(k(x,i),y)\right)= $$
$$=\sum_{i=0}^{|x'|}\left(\beta^i \frac{f\left(x',i,\#x'\right)}{|1x'|-i}\prod_{j=1}^{d(y)}\frac{\left(g\left(y,j\right)-i\right)}{g\left(y,j\right)}\right)+$$
$$+\beta^{|1x'|} f\left(1x',|1x'|,\#(1x')\right)\prod_{j=1}^{d(y)}\frac{\left(g\left(y,j\right)-|1x'|\right)}{g\left(y,j\right)}-$$
$$-\sum_{k=0}^{|x'|}\left(\beta^k  
\sum_{\begin{smallmatrix}(l,m)\in\mathbb{N}_0\times\mathbb{N}_0: \\l+|k(x',m)|=k              \end{smallmatrix}}\left( \frac{f\left(n(x',m),l,0\right)}{|1x'|-k}\cdot d'_1(k(x',m),y)\right)\right)-$$
$$-\beta^{|1x'|}\sum_{i=0}^{\#(1x')}  \left( {f\left(n(1x',i),|n(1x',i)|,0\right)}\cdot d'_1(k(1x',i),y)\right)=$$
$$=\beta^{|1x'|} f\left(1x',|1x'|,\#(1x')\right)\prod_{j=1}^{d(y)}\frac{\left(g\left(y,j\right)-|1x'|\right)}{g\left(y,j\right)}-$$
$$-\beta^{|1x'|}\sum_{i=0}^{\#(1x')}  \left( {f\left(n(1x',i),|n(1x',i)|,0\right)}\cdot d'_1(k(1x',i),y)\right)=$$
$$=\beta^{|1x'|}\left(f\left(1x',|1x'|,\#(1x')\right)\prod_{j=1}^{d(y)}\frac{\left(g\left(y,j\right)-|1x'|\right)}{g\left(y,j\right)}-\right.$$
$$\left.-\sum_{i=0}^{\#(1x')}  \left( {f\left(n(1x',i),|n(1x',i)|,0\right)}\cdot d'_1(k(1x',i),y)\right)\right).$$

По Утверждению \ref{beta1} при наших $x\in\mathbb{YF}$, $y\in\mathbb{YF}_\infty$
$$d'_{1}(x,y)= \sum_{i=0}^{\#x} \left(1^{|k(x,i)|}d_{1}(n(x,i))\cdot d'_1(k(x,i),y)\right), $$
а это значит, что (подставим $\beta = 1$ в равенство, к которому мы пришли)
$$0=d'_{1}(x,y)- \sum_{i=0}^{\#x}\left( 1^{|k(x,i)|}d_{1}(n(x,i))\cdot d'_1(k(x,i),y)\right)= $$
$$=1^{|1x'|}\left(f\left(1x',|1x'|,\#(1x')\right)\prod_{j=1}^{d(y)}\frac{\left(g\left(y,j\right)-|1x'|\right)}{g\left(y,j\right)}-\right.$$
$$\left.-\sum_{i=0}^{\#(1x')}  \left( {f\left(n(1x',i),|n(1x',i)|,0\right)}\cdot d'_1(k(1x',i),y)\right)\right)\Longrightarrow$$
$$\Longrightarrow f\left(1x',|1x'|,\#(1x')\right)\prod_{j=1}^{d(y)}\frac{\left(g\left(y,j\right)-|1x'|\right)}{g\left(y,j\right)}-$$
$$-\sum_{i=0}^{\#(1x')}  \left( {f\left(n(1x',i),|n(1x',i)|,0\right)}\cdot d'_1(k(1x',i),y)\right)=0.$$

Таким образом, ясно, что
$$d'_{\beta}(x,y)- \sum_{i=0}^{\#x}\left( \beta^{|k(x,i)|}d_{\beta}(n(x,i))\cdot d'_1(k(x,i),y)\right)= $$
$$=\beta^{|1x'|}\left(f\left(1x',|1x'|,\#(1x')\right)\prod_{j=1}^{d(y)}\frac{\left(g\left(y,j\right)-|1x'|\right)}{g\left(y,j\right)}-\right.$$
$$\left.-\sum_{i=0}^{\#(1x')}  \left( {f\left(n(1x',i),|n(1x',i)|,0\right)}\cdot d'_1(k(1x',i),y)\right)\right)=\beta^{|1x'|}\cdot 0=0\Longrightarrow$$
$$\Longrightarrow d'_{\beta}(x,y)= \sum_{i=0}^{\#x}\left( \beta^{|k(x,i)|}d_{\beta}(n(x,i))\cdot d'_1(k(x,i),y)\right),$$
что и требовалось.

В данном случае \underline{\textbf{Переход}} доказан.

    \item $\alpha_0=2$:

Вначале посчитаем, чему равна левая часть нашего равенства:
$$d'_{\beta}(x,y)=\sum_{i=0}^{|x|}\left(\beta^i {f\left(x,i,h(x,y)\right)}\prod_{j=1}^{d(y)}\frac{\left(g\left(y,j\right)-i\right)}{g\left(y,j\right)}\right)=$$
$$=\sum_{i=0}^{|x|}\left(\beta^i {f\left(x,i,\#x\right)}\prod_{j=1}^{d(y)}\frac{\left(g\left(y,j\right)-i\right)}{g\left(y,j\right)}\right)=$$
$$=\sum_{i=0}^{|2x'|}\left(\beta^i {f\left(2x',i,\#(2x')\right)}\prod_{j=1}^{d(y)}\frac{\left(g\left(y,j\right)-i\right)}{g\left(y,j\right)}\right)=$$
$$=\sum_{i=0}^{|x'|}\left(\beta^i {f\left(2x',i,\#(2x')\right)}\prod_{j=1}^{d(y)}\frac{\left(g\left(y,j\right)-i\right)}{g\left(y,j\right)}\right)+$$
$$+\sum_{i=|x'|+1}^{|x'|+1}\left(\beta^i {f\left(2x',i,\#(2x')\right)}\prod_{j=1}^{d(y)}\frac{\left(g\left(y,j\right)-i\right)}{g\left(y,j\right)}\right)+$$
$$+\sum_{i=|2x'|}^{|2x'|}\left(\beta^i {f\left(2x',i,\#(2x')\right)}\prod_{j=1}^{d(y)}\frac{\left(g\left(y,j\right)-i\right)}{g\left(y,j\right)}\right)=$$
$$=\sum_{i=0}^{|x'|}\left(\beta^i {f\left(2x',i,\#(2x')\right)}\prod_{j=1}^{d(y)}\frac{\left(g\left(y,j\right)-i\right)}{g\left(y,j\right)}\right)+$$
$$+\beta^{|x'|+1} {f\left(2x',|x'|+1,\#(2x')\right)}\prod_{j=1}^{d(y)}\frac{\left(g\left(y,j\right)-|x'|-1\right)}{g\left(y,j\right)}+$$
$$+\beta^{|2x'|} {f\left(2x',|2x'|,\#(2x')\right)}\prod_{j=1}^{d(y)}\frac{\left(g\left(y,j\right)-|2x'|\right)}{g\left(y,j\right)}.$$

Применим Утверждение \ref{evtuh93} для $x'\in\mathbb{YF}$, $\#(2x')\in\mathbb{N}_0$ и поймём, что наше выражение равно следующему:
$$\sum_{i=0}^{|x'|}\left(\beta^i {f\left(2x',i,\#(2x')\right)}\prod_{j=1}^{d(y)}\frac{\left(g\left(y,j\right)-i\right)}{g\left(y,j\right)}\right)+$$
$$+\beta^{|x'|+1}\cdot {0}\cdot\prod_{j=1}^{d(y)}\frac{\left(g\left(y,j\right)-|x'|-1\right)}{g\left(y,j\right)}+$$
$$+\beta^{|2x'|} {f\left(2x',|2x'|,\#(2x')\right)}\prod_{j=1}^{d(y)}\frac{\left(g\left(y,j\right)-|2x'|\right)}{g\left(y,j\right)}=$$
$$=\sum_{i=0}^{|x'|}\left(\beta^i {f\left(2x',i,\#(2x')\right)}\prod_{j=1}^{d(y)}\frac{\left(g\left(y,j\right)-i\right)}{g\left(y,j\right)}\right)+$$
$$+\beta^{|2x'|} {f\left(2x',|2x'|,\#(2x')\right)}\prod_{j=1}^{d(y)}\frac{\left(g\left(y,j\right)-|2x'|\right)}{g\left(y,j\right)}.$$

\begin{Prop}[Утверждение 9.1\cite{Evtuh1}, Утверждение 1.9\cite{Evtuh2}] \label{evtuh91}
Пусть $x\in\mathbb{YF}$, $y\in\mathbb{N}_0:$ $y\in\overline{2x}$. Тогда
$$f(2x,y,\#x)=f(2x,y,\#(2x)).$$
\end{Prop}

Применим Утверждение \ref{evtuh91} при $x'\in\mathbb{YF}$, $i\in\mathbb{N}_0$ к каждому слагаемого первой строчки и получим, что наше выражение равняется следующему:
$$\sum_{i=0}^{|x'|}\left(\beta^i {f\left(2x',i,\#x'\right)}\prod_{j=1}^{d(y)}\frac{\left(g\left(y,j\right)-i\right)}{g\left(y,j\right)}\right)+$$
$$+\beta^{|2x'|} {f\left(2x',|2x'|,\#(2x')\right)}\prod_{j=1}^{d(y)}\frac{\left(g\left(y,j\right)-|2x'|\right)}{g\left(y,j\right)}.$$

Применим Утверждение \ref{evtuh7} при $x'\in\mathbb{YF},$ $i\in \overline{|x'|},$ $\#x'\in\overline{\#x'} $ и $2\in\{1,2\}$ к каждому слагаемому суммы и получим, что наше выражение равняется следующему:
$$\sum_{i=0}^{|x'|}\left(\beta^i \frac{f\left(x',i,\#x'\right)}{|2x'|-i}\prod_{j=1}^{d(y)}\frac{\left(g\left(y,j\right)-i\right)}{g\left(y,j\right)}\right)+$$
$$+\beta^{|2x'|} f\left(2x',|2x'|,\#(2x')\right)\prod_{j=1}^{d(y)}\frac{\left(g\left(y,j\right)-|2x'|\right)}{g\left(y,j\right)}.$$

Тут мы посчитали, чему равна левая сторона нашего равенства. Теперь будем считать, чему равна правая:
$$\sum_{i=0}^{\#x} \left(\beta^{|k(x,i)|}d_{\beta}(n(x,i))\cdot  d'_1(k(x,i),y)\right)=$$
$$=\sum_{i=0}^{\#x}\left(  \beta^{|k(x,i)|}\left(\sum_{j=0}^{|n(x,i)|}\left(\beta^j {f\left(n(x,i),j,0\right)}\right)\right)d'_1(k(x,i),y)\right)=$$
$$=\sum_{i=0}^{\#(2x')}\left(  \beta^{|k(2x',i)|}\left(\sum_{j=0}^{|n(2x',i)|}\left(\beta^j {f\left(n(2x',i),j,0\right)}\right)\right)d'_1(k(2x',i),y)\right)=$$
$$=\sum_{i=0}^{\#x'}\left(  \beta^{|k(2x',i)|}\left(\sum_{j=0}^{|n(2x',i)|}\left(\beta^j {f\left(n(2x',i),j,0\right)}\right)\right)d'_1(k(2x',i),y)\right)+$$
$$+\sum_{i=\#(2x')}^{\#(2x')}\left(  \beta^{|k(2x',i)|}\left(\sum_{j=0}^{|n(2x',i)|}\left(\beta^j {f\left(n(2x',i),j,0\right)}\right)\right)d'_1(k(2x',i),y)\right)=$$
$$=\sum_{i=0}^{\#x'}\left(  \beta^{|k(2x',i)|}\left(\sum_{j=0}^{|n(2x',i)|}\left(\beta^j {f\left(n(2x',i),j,0\right)}\right)\right)d'_1(k(2x',i),y)\right)+$$
$$ + \beta^{|k(2x',\#(2x'))|}\left(\sum_{j=0}^{|n(2x',\#(2x'))|}\left(\beta^j {f\left(n(2x',\#(2x')),j,0\right)}\right)\right)d'_1(k(2x',\#(2x')),y)=$$
$$=\text{(По Замечанию \ref{mega} при $(2x')\in\mathbb{YF}$)}=$$
$$=\sum_{i=0}^{\#x'}\left(  \beta^{|k(2x',i)|}\left(\sum_{j=0}^{|n(2x',i)|-2}\left(\beta^j {f\left(n(2x',i),j,0\right)}\right)\right)d'_1(k(2x',i),y)\right)+$$
$$+\sum_{i=0}^{\#x'}\left(  \beta^{|k(2x',i)|}\left(\sum_{j=|n(2x',i)|-1}^{|n(2x',i)|-1}\left(\beta^j {f\left(n(2x',i),j,0\right)}\right)\right)d'_1(k(2x',i),y)\right)+$$
$$+\sum_{i=0}^{\#x'}\left(  \beta^{|k(2x',i)|}\left(\sum_{j=|n(2x',i)|}^{|n(2x',i)|}\left(\beta^j {f\left(n(2x',i),j,0\right)}\right)\right)d'_1(k(2x',i),y)\right)+$$
$$ + \beta^{|2x'|}\left(\sum_{j=0}^{|\varepsilon|}\left(\beta^j {f\left(n(2x',\#(2x')),j,0\right)}\right)\right)d'_1(k(2x',\#(2x')),y)=$$
$$=\sum_{i=0}^{\#x'}\left(  \beta^{|k(2x',i)|}\left(\sum_{j=0}^{|n(2x',i)|-2}\left(\beta^j {f\left(n(2x',i),j,0\right)}\right)\right)d'_1(k(2x',i),y)\right)+$$
$$+\sum_{i=0}^{\#x'}\left(  \beta^{|k(2x',i)|}\left(\beta^{|n(2x',i)|-1} {f\left(n(2x',i),|n(2x',i)|-1,0\right)}\right)d'_1(k(2x',i),y)\right)+$$
$$+\sum_{i=0}^{\#x'}\left(  \beta^{|k(2x',i)|}\left(\beta^{|n(2x',i)|} {f\left(n(2x',i),|n(2x',i)|,0\right)}\right)d'_1(k(2x',i),y)\right)+$$
$$ + \beta^{|2x'|}\left(\sum_{j=0}^{0}\left(\beta^j {f\left(n(2x',\#(2x')),j,0\right)}\right)\right)d'_1(k(2x',\#(2x')),y)=$$
$$=\text{(По Утверждению \ref{rofl} при $x'\in\mathbb{YF},$ $i\in\mathbb{N}_0,$ $ 2\in\{1,2\}$ ко всем слагаемым первых строчек)}=$$
$$=\sum_{i=0}^{\#x'}\left(  \beta^{|k(2x',i)|}\left(\sum_{j=0}^{|2n(x',i)|-2}\left(\beta^j {f\left(2n(x',i),j,0\right)}\right)\right)d'_1(k(2x',i),y)\right)+$$
$$+\sum_{i=0}^{\#x'}\left(  \beta^{|k(2x',i)|}\left(\beta^{|n(2x',i)|-1} {f\left(2n(x',i),|2n(x',i)|-1,0\right)}\right)d'_1(k(2x',i),y)\right)+$$
$$+\sum_{i=0}^{\#x'}\left(  \beta^{|k(2x',i)|+|n(2x',i)|} {f\left(n(2x',i),|n(2x',i)|,0\right)}\cdot d'_1(k(2x',i),y)\right)+$$
$$ + \beta^{|2x'|}\beta^0 {f\left(n(2x',\#(2x')),0,0\right)}\cdot d'_1(k(2x',\#(2x')),y).$$

Применим Утверждение \ref{evtuh7} при $(n(x',i))\in\mathbb{YF},$ $j\in \overline{|2n(x',i)|-2}=\overline{|n(x',i)|},$ $ 0\in\overline{\#(n(x',i))} $ и $2\in\{1,2\}$ к каждому слагаемому первой строчки.

Кроме того, применим Замечание \ref{prunk} при $(2x')\in\mathbb{YF},$ $i\in\mathbb{N}_0$ к каждому слагаемому третьей строчки и получим, что наше выражение равняется следующему:
$$\sum_{i=0}^{\#x'}\left(  \beta^{|k(2x',i)|}\left(\sum_{j=0}^{|2n(x',i)|-2}\left(\beta^j \frac{f\left(n(x',i),j,0\right)}{|2n(x',i)|-j}\right)\right)d'_1(k(2x',i),y)\right)+$$
$$+\sum_{i=0}^{\#x'}\left(  \beta^{|k(2x',i)|}\left(\beta^{|n(2x',i)|-1} {f\left(2n(x',i),|n(x',i)|+1,0\right)}\right)d'_1(k(2x',i),y)\right)+$$
$$+\sum_{i=0}^{\#x'}\left(  \beta^{|2x'|} {f\left(n(2x',i),|n(2x',i)|,0\right)}\cdot d'_1(k(2x',i),y)\right)+$$
$$ + \beta^{|2x'|} {f\left(n(2x',\#(2x')),0,0\right)}\cdot d'_1(k(2x',\#(2x')),y)=$$
$$=\text{(По Утверждению \ref{rofl} при $x'\in\mathbb{YF},$ $ i\in\mathbb{N}_0,$ $ 2\in\{1,2\}$ ко всем слагаемым первой строчки)}=$$
$$=\sum_{i=0}^{\#x'}\left(  \beta^{|k(2x',i)|}\left(\sum_{j=0}^{|n(2x',i)|-2}\left(\beta^j \frac{f\left(n(x',i),j,0\right)}{|n(2x',i)|-j}\right)\right)d'_1(k(2x',i),y)\right)+$$
$$+\sum_{i=0}^{\#x'}\left(  \beta^{|k(2x',i)|}\left(\beta^{|n(2x',i)|-1} {f\left(2n(x',i),|n(x',i)|+1,0\right)}\right)d'_1(k(2x',i),y)\right)+$$
$$+\beta^{|2x'|}\sum_{i=0}^{\#x'}\left(   {f\left(n(2x',i),|n(2x',i)|,0\right)}\cdot d'_1(k(2x',i),y)\right)+$$
$$ + \beta^{|2x'|} {f\left(n(2x',\#(2x')),0,0\right)}\cdot d'_1(k(2x',\#(2x')),y).$$

Применим Утверждение \ref{evtuh93} при $n(x',i)\in\mathbb{YF}$ и $0\in\mathbb{N}_0$  к каждому слагаемому второй строчки и поймём, что наше выражение равняется следующему:
$$\sum_{i=0}^{\#x'}\left(  \beta^{|k(2x',i)|}\left(\sum_{j=0}^{|n(2x',i)|-2}\left(\beta^j \frac{f\left(n(x',i),j,0\right)}{|n(2x',i)|-j}\right)\right)d'_1(k(2x',i),y)\right)+$$
$$+\sum_{i=0}^{\#x'}\left(  \beta^{|k(2x',i)|}\left(\beta^{|n(2x',i)|-1} \cdot 0\right)d'_1(k(2x',i),y)\right)+$$
$$+ \beta^{|2x'|}\sum_{i=0}^{\#x'}\left(   {f\left(n(2x',i),|n(2x',i)|,0\right)}\cdot d'_1(k(2x',i),y)\right)+$$
$$ + \beta^{|2x'|} {f\left(n(2x',\#(2x')),0,0\right)}\cdot d'_1(k(2x',\#(2x')),y)=$$
$$=\sum_{i=0}^{\#x'}\left(  \beta^{|k(2x',i)|}\left(\sum_{j=0}^{|n(2x',i)|-2}\left(\beta^j \frac{f\left(n(x',i),j,0\right)}{|n(2x',i)|-j}\right)\right)d'_1(k(2x',i),y)\right)+$$
$$+\beta^{|2x'|}\sum_{i=0}^{\#x'}\left(   {f\left(n(2x',i),|n(2x',i)|,0\right)}\cdot d'_1(k(2x',i),y)\right)+$$
$$ + \beta^{|2x'|} {f\left(n(2x',\#(2x')),0,0\right)}\cdot d'_1(k(2x',\#(2x')),y).$$

Очевидно, что $\forall i\in\overline{\#x'}$
$$|k(2x',i)|+(|n(2x',i)|-2)=(|k(2x',i)|+|n(2x',i)|)-2=$$
$$=(\text{По Замечанию \ref{prunk} при $(2x')\in\mathbb{YF},$ $i\in\mathbb{N}_0$})=$$
$$=|2x'|-2=2+|x'|-2=|x'|.$$

А это значит, что наше выражение равняется следующему:
$$\sum_{k=0}^{|x'|}\left(  \beta^{k}\sum_{\begin{smallmatrix}(l,m)\in\mathbb{N}_0\times\mathbb{N}_0: \\l+|k(2x',m)|=k              \end{smallmatrix}}\left( \frac{f\left(n(x',m),l,0\right)}{|n(2x',m)|-l}\cdot d'_1(k(2x',m),y)\right)\right)+$$
$$+\beta^{|2x'|}\sum_{i=0}^{\#x'}\left(   {f\left(n(2x',i),|n(2x',i)|,0\right)}\cdot d'_1(k(2x',i),y)\right)+$$
$$ + \beta^{|2x'|} {f\left(n(2x',\#(2x')),0,0\right)}\cdot d'_1(k(2x',\#(2x')),y).$$

Очевидно, что если у нас есть пара $(l,m)\in\mathbb{N}_0\times\mathbb{N}_0:$ $l+|k(2x',m)|=k$, то по Замечанию \ref{prunk} при $(2x')\in\mathbb{YF},$ $m\in\mathbb{N}_0$ ясно, что $l+|2x'|-|n(2x',m)|=k,$ а это значит, что $|n(2x',m)|-l=|2x'|-k.$ Таким образом, наше выражение равно следующему:
$$\sum_{k=0}^{|x'|}\left(  \beta^{k}\sum_{\begin{smallmatrix}(l,m)\in\mathbb{N}_0\times\mathbb{N}_0: \\l+|k(2x',m)|=k              \end{smallmatrix}}\left( \frac{f\left(n(x',m),l,0\right)}{|2x'|-k}\cdot d'_1(k(2x',m),y)\right)\right)+$$
$$+\beta^{|2x'|}\sum_{i=0}^{\#x'}\left(   {f\left(n(2x',i),|n(2x',i)|,0\right)}\cdot d'_1(k(2x',i),y)\right)+$$
$$ + \beta^{|2x'|} {f\left(n(2x',\#(2x')),0,0\right)}\cdot d'_1(k(2x',\#(2x')),y).$$

Тут у нас есть пары $(l,m)\in\mathbb{N}_0\times\mathbb{N}_0:$ $ l+|k(2x',m)|=k$ при $k\le |x'|$. Если $m=\#(2x')$, то 
$$k=l+|k(2x',m)|=l+|k(2x',\#(2x'))|=$$
$$=\text{(По Замечанию \ref{mega} при $(2x')\in\mathbb{YF}$)}=$$
$$=l+|2x'|=l+2+|x'|\ge 0+2+k>k.$$

Противоречие. А это значит, что $m<\#(2x')$, то есть $m\in\overline{\#x'}$. Таким образом, 
$$2x'=n(2x',m)k(2x',m),\; \#(k(2x',m))=m,\; m\in\overline{\#x'}\Longleftrightarrow$$
$$\Longleftrightarrow\text{(По Утверждению \ref{rofl} при $x'\in\mathbb{YF},$ $ m\in\mathbb{N}_0,$ $ 2\in\{1,2\}$)}\Longleftrightarrow$$
$$\Longleftrightarrow 2x'=2n(x',m)k(2x',m),\; \#(k(2x',m))=m,\; m\in\overline{\#x'}\Longleftrightarrow$$
$$\Longleftrightarrow x'=n(x',m)k(2x',m),\; \#(k(2x',m))=m,\; m\in\overline{\#x'}\Longleftrightarrow$$
$$\Longleftrightarrow k(x',m)=k(2x',m).$$

Таким образом, наше выражение равно следующему:
$$\sum_{k=0}^{|x'|}\left(\beta^k  
\sum_{\begin{smallmatrix}(l,m)\in\mathbb{N}_0\times\mathbb{N}_0: \\l+|k(x',m)|=k              \end{smallmatrix}}\left( \frac{f\left(n(x',m),l,0\right)}{|2x'|-k}\cdot d'_1(k(x',m),y)\right)\right)+$$
$$+\beta^{|2x'|}\sum_{i=0}^{\#x'}\left(   {f\left(n(2x',i),|n(2x',i)|,0\right)}\cdot d'_1(k(2x',i),y)\right)+$$
$$ + \beta^{|2x'|} {f\left(n(2x',\#(2x')),0,0\right)}\cdot d'_1(k(2x',\#(2x')),y).$$
Заметим, что 
$$0=|\varepsilon|=\text{(По Замечанию \ref{mega} при $(2x')\in\mathbb{YF}$)}=|n(2x',\#(2x'))|.$$
Таким образом, наше выражение равно следующему:
$$\sum_{k=0}^{|x'|}\left(\beta^k  
\sum_{\begin{smallmatrix}(l,m)\in\mathbb{N}_0\times\mathbb{N}_0: \\l+|k(x',m)|=k              \end{smallmatrix}}\left( \frac{f\left(n(x',m),l,0\right)}{|2x'|-k}\cdot d'_1(k(x',m),y)\right)\right)+$$
$$+\beta^{|2x'|}\sum_{i=0}^{\#x'}\left(   {f\left(n(2x',i),|n(2x',i)|,0\right)}\cdot d'_1(k(2x',i),y)\right)+$$
$$ + \beta^{|2x'|} {f\left(n(2x',\#(2x')),|n(2x',\#(2x'))|,0\right)}\cdot d'_1(k(2x',\#(2x')),y)=$$
$$=\sum_{k=0}^{|x'|}\left(\beta^k  
\sum_{\begin{smallmatrix}(l,m)\in\mathbb{N}_0\times\mathbb{N}_0: \\l+|k(x',m)|=k              \end{smallmatrix}}\left( \frac{f\left(n(x',m),l,0\right)}{|2x'|-k}\cdot d'_1(k(x',m),y)\right)\right)+$$
$$+\beta^{|2x'|}\sum_{i=0}^{\#x'}\left(   {f\left(n(2x',i),|n(2x',i)|,0\right)}\cdot d'_1(k(2x',i),y)\right)+$$
$$+\beta^{|2x'|}\sum_{i=\#(2x')}^{\#(2x')}\left(   {f\left(n(2x',i),|n(2x',i)|,0\right)}\cdot d'_1(k(2x',i),y)\right)=$$
$$=\sum_{k=0}^{|x'|}\left(\beta^k  
\sum_{\begin{smallmatrix}(l,m)\in\mathbb{N}_0\times\mathbb{N}_0: \\l+|k(x',m)|=k              \end{smallmatrix}}\left( \frac{f\left(n(x',m),l,0\right)}{|2x'|-k}\cdot d'_1(k(x',m),y)\right)\right)+$$
$$+\beta^{|2x'|}\sum_{i=0}^{\#(2x')}\left(   {f\left(n(2x',i),|n(2x',i)|,0\right)}d'_1(k(2x',i),y)\right).$$

Тут мы посчитали, чему равна правая сторона нашего равенства. Таким образом, мы поняли, что (вычтем правую сторону равенства из левого)
$$d'_{\beta}(x,y)- \sum_{i=0}^{\#x}\left( \beta^{|k(x,i)|}d_{\beta}(n(x,i))\cdot d'_1(k(x,i),y)\right)= $$
$$=\sum_{i=0}^{|x'|}\left(\beta^i \frac{f\left(x',i,\#x'\right)}{|2x'|-i}\prod_{j=1}^{d(y)}\frac{\left(g\left(y,j\right)-i\right)}{g\left(y,j\right)}\right)+$$
$$+\beta^{|2x'|} f\left(2x',|2x'|,\#(2x')\right)\prod_{j=1}^{d(y)}\frac{\left(g\left(y,j\right)-|2x'|\right)}{g\left(y,j\right)}-$$
$$-\sum_{k=0}^{|x'|}\left(\beta^k  
\sum_{\begin{smallmatrix}(l,m)\in\mathbb{N}_0\times\mathbb{N}_0: \\l+|k(x',m)|=k              \end{smallmatrix}}\left( \frac{f\left(n(x',m),l,0\right)}{|2x'|-k}\cdot d'_1(k(x',m),y)\right)\right)-$$
$$-\beta^{|2x'|}\sum_{i=0}^{\#(2x')}  \left( {f\left(n(2x',i),|n(2x',i)|,0\right)}\cdot d'_1(k(2x',i),y)\right).$$

Запомним это равенство.

Теперь воспользуемся предположением индукции при $x'\in\mathbb{YF}$:
$$d'_{\beta}(x',y)= \sum_{i=0}^{\#x'} \left(\beta^{|k(x',i)|}d_{\beta}(n(x',i))\cdot d'_1(k(x',i),y) \right)\Longleftrightarrow$$
$$\Longleftrightarrow \sum_{i=0}^{|x'|}\left(\beta^i {f\left(x',i,h(x',y)\right)}\prod_{j=1}^{d(y)}\frac{\left(g\left(y,j\right)-i\right)}{g\left(y,j\right)}\right)=\sum_{i=0}^{\#x'} \left(\beta^{|k(x',i)|}d_{\beta}(n(x',i))\cdot d'_1(k(x',i),y)\right)\Longleftrightarrow$$
$$\Longleftrightarrow(\text{По определению функции $h$, так как в данном случае $h(2x',y)=\#(2x')$})\Longleftrightarrow$$
$$\Longleftrightarrow \sum_{i=0}^{|x'|}\left(\beta^i {f\left(x',i,\#x'\right)}\prod_{j=1}^{d(y)}\frac{\left(g\left(y,j\right)-i\right)}{g\left(y,j\right)}\right)=\sum_{i=0}^{\#x'}\left( \beta^{|k(x',i)|}d_{\beta}(n(x',i))\cdot d'_1(k(x',i),y)\right).$$

Заметим, что по обозначению
$$\sum_{i=0}^{\#x'} \left(\beta^{|k(x',i)|}d_{\beta}(n(x',i))\cdot d'_1(k(x',i),y)\right)=$$
$$=\sum_{i=0}^{\#x'}\left( \beta^{|k(x',i)|}\left(\sum_{j=0}^{|n(x',i)|}\left(\beta^j \left({f\left(n(x',i),j,0\right)}\right)\right)\right)d'_1(k(x',i),y)\right) =$$
$$=(\text{По Замечанию \ref{prunk} при $x'\in\mathbb{YF},$ $i\in\mathbb{N}_0$ если $i\in\overline{\#x'}$, то $|k(x',i)|+|n(x',i)|=|x'|$})=$$
$$=\sum_{k=0}^{|x'|}\left( \beta^{k}\sum_{\begin{smallmatrix}(l,m)\in\mathbb{N}_0\times\mathbb{N}_0: \\l+|k(x',m)|=k              \end{smallmatrix}} \left( {f\left(n(x',m),l,0\right)}\cdot d'_1(k(x',m),y)\right)\right).$$

А значит
$$ \sum_{i=0}^{|x'|}\left(\beta^i {f\left(x',i,\#x'\right)}\prod_{j=1}^{d(y)}\frac{\left(g\left(y,j\right)-i\right)}{g\left(y,j\right)}\right)=$$
$$=\sum_{k=0}^{|x'|}\left( \beta^{k}\sum_{\begin{smallmatrix}(l,m)\in\mathbb{N}_0\times\mathbb{N}_0: \\l+|k(x',m)|=k              \end{smallmatrix}} \left( {f\left(n(x',m),l,0\right)}\cdot d'_1(k(x',m),y)\right)\right)\Longleftrightarrow$$
$$\Longleftrightarrow \sum_{i=0}^{|x'|}\left(\beta^i\left( {f\left(x',i,\#x'\right)}\prod_{j=1}^{d(y)}\frac{\left(g\left(y,j\right)-i\right)}{g\left(y,j\right)}-\right.\right.$$
$$-\left.\left.\sum_{\begin{smallmatrix}(l,m)\in\mathbb{N}_0\times\mathbb{N}_0: \\l+|k(x',m)|=i              \end{smallmatrix}} \left( {f\left(n(x',m),l,0\right)}\cdot d'_1(k(x',m),y)\right)\right)\right)=0.$$

И это равенство верно для любого $\beta\in(0,1]$. А значит слева находится многочлен от $\beta$, тождественно равный нулю. А значит любой его коэффициент при $\beta^i$ при $i\in\overline{|x'|}$ равен нулю. То есть
$\forall i\in\overline{|x'|}$ 
$$ {f\left(x',i,\#x'\right)}\prod_{j=1}^{d(y)}\frac{\left(g\left(y,j\right)-i\right)}{g\left(y,j\right)}-\sum_{\begin{smallmatrix}(l,m)\in\mathbb{N}_0\times\mathbb{N}_0: \\l+|k(x',m)|=i            \end{smallmatrix}} \left( {f\left(n(x',m),l,0\right)}\cdot d'_1(k(x',m),y)\right)=0\Longleftrightarrow$$
$$\Longleftrightarrow{f\left(x',i,\#x'\right)}\prod_{j=1}^{d(y)}\frac{\left(g\left(y,j\right)-i\right)}{g\left(y,j\right)}=\sum_{\begin{smallmatrix}(l,m)\in\mathbb{N}_0\times\mathbb{N}_0: \\l+|k(x',m)|=i            \end{smallmatrix}} \left( {f\left(n(x',m),l,0\right)}\cdot d'_1(k(x',m),y)\right).$$

Теперь заметим, что если $i\in\overline{|x'|}$, то $|2x'|-i\ge |2x'|-|x'|=2+|x'|-|x'|=2>0$. А значит $\forall i\in\overline{|x'|}$
$$\beta^i \frac{f\left(x',i,\#x'\right)}{|2x'|-i}\prod_{j=1}^{d(y)}\frac{\left(g\left(y,j\right)-i\right)}{g\left(y,j\right)}=$$
$$=\beta^i\frac{\displaystyle\sum_{\begin{smallmatrix}(l,m)\in\mathbb{N}_0\times\mathbb{N}_0: \\l+|k(x',m)|=i            \end{smallmatrix}} \left( {f\left(n(x',m),l,0\right)}\cdot d'_1(k(x',m),y)\right)}{\displaystyle|2x'|-i}=$$
$$=\beta^i  
\sum_{\begin{smallmatrix}(l,m)\in\mathbb{N}_0\times\mathbb{N}_0: \\l+|k(x',m)|=i              \end{smallmatrix}}\left( \frac{f\left(n(x',m),l,0\right)}{|2x'|-i}\cdot d'_1(k(x',m),y)\right).$$

Просуммируем данное выражение по $i\in\overline{|x'|}$:
$$\sum_{i=0}^{|x'|}\left(\beta^i \frac{f\left(x',i,\#x'\right)}{|2x'|-i}\prod_{j=1}^{d(y)}\frac{\left(g\left(y,j\right)-i\right)}{g\left(y,j\right)}\right)=$$
$$=\sum_{i=0}^{|x'|}\left(\beta^i  
\sum_{\begin{smallmatrix}(l,m)\in\mathbb{N}_0\times\mathbb{N}_0: \\l+|k(x',m)|=i              \end{smallmatrix}}\left( \frac{f\left(n(x',m),l,0\right)}{|2x'|-i}\cdot d'_1(k(x',m),y)\right)\right)=$$
$$=\sum_{k=0}^{|x'|}\left(\beta^k  
\sum_{\begin{smallmatrix}(l,m)\in\mathbb{N}_0\times\mathbb{N}_0: \\l+|k(x',m)|=k              \end{smallmatrix}}\left( \frac{f\left(n(x',m),l,0\right)}{|2x'|-k}\cdot d'_1(k(x',m),y)\right)\right).$$

А это значит, что (вернёмся к запомненному равенству)
$$d'_{\beta}(x,y)- \sum_{i=0}^{\#x}\left( \beta^{|k(x,i)|}d_{\beta}(n(x,i))\cdot d'_1(k(x,i),y)\right)= $$
$$=\sum_{i=0}^{|x'|}\left(\beta^i \frac{f\left(x',i,\#x'\right)}{|2x'|-i}\prod_{j=1}^{d(y)}\frac{\left(g\left(y,j\right)-i\right)}{g\left(y,j\right)}\right)+$$
$$+\beta^{|2x'|} f\left(2x',|2x'|,\#(2x')\right)\prod_{j=1}^{d(y)}\frac{\left(g\left(y,j\right)-|2x'|\right)}{g\left(y,j\right)}-$$
$$-\sum_{k=0}^{|x'|}\left(\beta^k  
\sum_{\begin{smallmatrix}(l,m)\in\mathbb{N}_0\times\mathbb{N}_0: \\l+|k(x',m)|=k              \end{smallmatrix}}\left( \frac{f\left(n(x',m),l,0\right)}{|2x'|-k}\cdot d'_1(k(x',m),y)\right)\right)-$$
$$-\beta^{|2x'|}\sum_{i=0}^{\#(2x')}  \left( {f\left(n(2x',i),|n(2x',i)|,0\right)}\cdot d'_1(k(2x',i),y)\right)=$$
$$=\beta^{|2x'|} f\left(2x',|2x'|,\#(2x')\right)\prod_{j=1}^{d(y)}\frac{\left(g\left(y,j\right)-|2x'|\right)}{g\left(y,j\right)}-$$
$$-\beta^{|2x'|}\sum_{i=0}^{\#(2x')}  \left( {f\left(n(2x',i),|n(2x',i)|,0\right)}\cdot d'_1(k(2x',i),y)\right)=$$
$$=\beta^{|2x'|}\left(f\left(2x',|2x'|,\#(2x')\right)\prod_{j=1}^{d(y)}\frac{\left(g\left(y,j\right)-|2x'|\right)}{g\left(y,j\right)}-\right.$$
$$\left.-\sum_{i=0}^{\#(2x')}  \left( {f\left(n(2x',i),|n(2x',i)|,0\right)}\cdot d'_1(k(2x',i),y)\right)\right).$$

По Утверждению \ref{beta1} при наших $x\in\mathbb{YF}$, $y\in\mathbb{YF}_\infty$
$$d'_{1}(x,y)= \sum_{i=0}^{\#x} \left(1^{|k(x,i)|}d_{1}(n(x,i))\cdot d'_1(k(x,i),y)\right), $$
а это значит, что (подставим $\beta = 1$ в равенство, к которому мы пришли)
$$0=d'_{1}(x,y)- \sum_{i=0}^{\#x}\left( 1^{|k(x,i)|}d_{1}(n(x,i))\cdot d'_1(k(x,i),y)\right)= $$
$$=1^{|2x'|}\left(f\left(2x',|2x'|,\#(2x')\right)\prod_{j=1}^{d(y)}\frac{\left(g\left(y,j\right)-|2x'|\right)}{g\left(y,j\right)}-\right.$$
$$\left.-\sum_{i=0}^{\#(2x')}  \left( {f\left(n(2x',i),|n(2x',i)|,0\right)}\cdot d'_1(k(2x',i),y)\right)\right)\Longrightarrow$$
$$\Longrightarrow f\left(2x',|2x'|,\#(2x')\right)\prod_{j=1}^{d(y)}\frac{\left(g\left(y,j\right)-|2x'|\right)}{g\left(y,j\right)}-$$
$$-\sum_{i=0}^{\#(2x')}  \left( {f\left(n(2x',i),|n(2x',i)|,0\right)}\cdot d'_1(k(2x',i),y)\right)=0.$$

Таким образом, ясно, что
$$d'_{\beta}(x,y)- \sum_{i=0}^{\#x}\left( \beta^{|k(x,i)|}d_{\beta}(n(x,i))\cdot d'_1(k(x,i),y)\right)= $$
$$=\beta^{|2x'|}\left(f\left(2x',|2x'|,\#(2x')\right)\prod_{j=1}^{d(y)}\frac{\left(g\left(y,j\right)-|2x'|\right)}{g\left(y,j\right)}-\right.$$
$$\left.-\sum_{i=0}^{\#(2x')}  \left( {f\left(n(2x',i),|n(2x',i)|,0\right)}\cdot d'_1(k(2x',i),y)\right)\right)=\beta^{|2x'|}\cdot 0=0\Longrightarrow$$
$$\Longrightarrow d'_{\beta}(x,y)= \sum_{i=0}^{\#x}\left( \beta^{|k(x,i)|}d_{\beta}(n(x,i))\cdot d'_1(k(x,i),y)\right),$$
что и требовалось.

В данном случае \underline{\textbf{Переход}} доказан.

\end{enumerate}
\end{enumerate}
В каждом случае \underline{\textbf{Переход}} доказан.

Ясно, что все случаи разобраны, а значит \underline{\textbf{Переход}} доказан.

Утверждение доказано.

\end{proof}

\begin{Oboz}
Пусть $x\in \mathbb{YF}$. Тогда
$$q(x):=\frac{1}{\displaystyle\prod_{i=1}^{\#x} |k(x,i)|}.$$
\end{Oboz}

\begin{Zam}\label{odnoitozhe}
Из определения функции $f$ ясно, что $\forall x\in\mathbb{YF}$
$$q(x)=\frac{1}{\displaystyle\prod_{i=1}^{\#x} |k(x,i)|}=f(x,0,0).$$
\end{Zam}

    \begin{Prop}\label{q}
    Пусть $x,x'\in\mathbb{YF}$, $\alpha_0\in\{1,2\}:$ $x=\alpha_0x'$. Тогда
    $$q(x')=|x|q(x).$$
    \end{Prop}
    \begin{proof}
    По определению функции $q$
    $$q(x')=\frac{1}{\displaystyle\prod_{i=1}^{\#x'}|k(x',i)|}=\frac{|\alpha_0x'|}{\displaystyle\left(\prod_{i=1}^{\#x'}|k(x',i)|\right)|\alpha_0x'|}.$$
    
Ясно, что $\forall i\in\overline{1,\#x'}$
$$\alpha_0x'=n(\alpha_0x',i)k(\alpha_0x',i),\; \#(k(\alpha_0x',i))=i,\; i\in\overline{\#x'}\Longleftrightarrow$$
$$\Longleftrightarrow(\text{По Утверждению \ref{rofl} при $x$ при $x'\in\mathbb{YF}$, $i\in\mathbb{N}_0$, $\alpha_0\in\{1,2\}$})\Longleftrightarrow$$
$$\Longleftrightarrow \alpha_0x'=\alpha_0n(x',i)k(\alpha_0x',i),\; \#(k(\alpha_0x',i))=i,\; i\in\overline{\#x'}\Longleftrightarrow$$
$$\Longleftrightarrow x'=n(x',i)k(\alpha_0x',i),\; \#(k(\alpha_0x',i))=i,\; i\in\overline{\#x'}\Longleftrightarrow$$
$$\Longleftrightarrow k(x',i)=k(\alpha_0x',i).$$

А значит данное выражение равняется следующему:
$$\frac{|\alpha_0x'|}{\displaystyle\left(\prod_{i=1}^{\#x'}|k(\alpha_0x',i)|\right)|\alpha_0x'|}=(\text{По Замечанию \ref{mega} при $x'\in\mathbb{YF}$})=$$
$$=\frac{|\alpha_0x'|}{\displaystyle\left(\prod_{i=1}^{\#x'}|k(\alpha_0x',i)|\right)|k(\alpha_0x',\#(\alpha_0x'))|}=\frac{|\alpha_0x'|}{\displaystyle\left(\prod_{i=1}^{\#x'}|k(\alpha_0x',i)|\right)|k(\alpha_0x',\#x'+1)|}=$$
$$=\frac{|\alpha_0x'|}{\displaystyle\left(\prod_{i=1}^{\#x'}|k(\alpha_0x',i)|\right)\left(\prod_{i=\#x'+1}^{\#x'+1}|k(\alpha_0x',i)|\right)}=$$
$$=\frac{|\alpha_0x'|}{\displaystyle\prod_{i=1}^{\#x'+1}|k(\alpha_0x',i)|}=\frac{|\alpha_0x'|}{\displaystyle\prod_{i=1}^{\#(\alpha_0x')}|k(\alpha_0x',i)|}=\frac{|x|}{\displaystyle\prod_{i=1}^{\#x}|k(x,i)|}=|x|q(x),$$
что и требовалось.

Утверждение доказано.
    \end{proof}

\begin{Prop} \label{delitsa}

Пусть $x\in \mathbb{YF}$. Тогда $d_{\beta}(x)$ делится на $(1-\beta)\quad\# x$ раз и не делится на $(1-\beta)\quad(\# x+1)$ раз.
\end{Prop}
\begin{proof} 

Рассмотрим два случая:
\begin{enumerate}
\item $x\in\mathbb{YF}:$ $\#x=0\Longleftrightarrow x=\varepsilon$.

В данном случае
$$d_{\beta}(x)=d_{\beta}(\varepsilon)=\sum_{i=0}^{|\varepsilon|}\left(\beta^i {f\left(\varepsilon,i,0\right)}\right)=\sum_{i=0}^{0}\left(\beta^i {f\left(\varepsilon,i,0\right)}\right)=\beta^0 {f\left(\varepsilon,0,0\right)}={f\left(\varepsilon,0,0\right)}=1.$$
Ясно, что $1$ как многочлен от $\beta$ делится на $(1-\beta)\quad 0=\#x$ раз и не делится на $(1-\beta)\quad 1=(\#x+1)$ раз, что и требовалось. 

В данном случае Утверждение доказано.

    \item $x\in\mathbb{YF}:\;\#x\ge 1$.

    \renewcommand{\labelenumi}{\arabic{enumi}$^\circ$}
    \renewcommand{\labelenumii}{\alph{enumii}$)$}
    \renewcommand{\labelenumiii}{\arabic{enumi}.\arabic{enumii}.\arabic{enumiii}$^\circ$}

    В данном случаю есть два варианта:
\begin{enumerate}
    \item $\exists x'\in\mathbb{YF}: x=x'1$. 

В данном случае заметим, что по обозначению
$$(d_{\beta}(x))'=(d_{\beta}(x'1))'=\left(\sum_{i=0}^{|x'1|}\left(\beta^i {f\left(x'1,i,0\right)}\right)\right)'=\sum_{i=1}^{|x'1|}\left(i\beta^{i-1} {f\left(x'1,i,0\right)}\right).$$

\begin{Prop}[Утверждение 1.1\cite{Evtuh1}, Утверждение 1.1\cite{Evtuh2}] \label{evtuh11}
Пусть $x\in\mathbb{YF}$, $y\in\mathbb{N}_0:$ $y \in \overline{1,|x1|}$. Тогда 
$$-yf(x1,y,0)=f(x,y-1,0).$$
\end{Prop}
Из Утверждения \ref{evtuh11}, применённого к каждому слагаемому нашей суммы при $x'\in\mathbb{YF},$ $i\in\mathbb{N}_0$, следует, что
$$(d_{\beta}(x))'=(d_{\beta}(x'1))'=\sum_{i=1}^{|x'1|}\left(i\beta^{i-1} {f\left(x'1,i,0\right)}\right)=\sum_{i=1}^{|x'1|}\left(-\beta^{i-1} {f\left(x',i-1,0\right)}\right)=$$
$$=-\sum_{i=1}^{|x'1|}\left(\beta^{i-1} {f\left(x',i-1,0\right)}\right)=-\sum_{i=0}^{|x'|}\left(\beta^{i} {f\left(x',i,0\right)}\right)=-(d_{\beta}(x')).$$

    \item $\exists x'\in\mathbb{YF}: x=x'2$. 

В данном случае
$$(d_{\beta}(x))'=(d_{\beta}(x'2))'=\left(\sum_{i=0}^{|x'2|}\left(\beta^i {f\left(x'2,i,0\right)}\right)\right)'=\sum_{i=1}^{|x'2|}\left(i\beta^{i-1} {f\left(x'2,i,0\right)}\right)=$$
$$=\sum_{i=1}^{1}\left(i\beta^{i-1} {f\left(x'2,i,0\right)}\right)+\sum_{i=2}^{|x'2|}\left(i\beta^{i-1} {f\left(x'2,i,0\right)}\right)=$$
$$=(1\beta^{1-1}f(x'2,1,0))+\sum_{i=2}^{|x'2|}\left(i\beta^{i-1} {f\left(x'2,i,0\right)}\right)=f(x'2,1,0)+\sum_{i=2}^{|x'2|}\left(i\beta^{i-1} {f\left(x'2,i,0\right)}\right).$$

По определению функции $f$ ясно, что $f(x'2,1,0)=0$, а значит наше выражение равняется следующему:
$$\sum_{i=2}^{|x'2|}\left(i\beta^{i-1} {f\left(x'2,i,0\right)}\right).$$

\begin{Prop}[Утверждение 1.2\cite{Evtuh1}, Утверждение 1.2\cite{Evtuh2}] \label{evtuh12}
Пусть $x\in\mathbb{YF}$, $y\in\mathbb{N}_0:$ $y \in \overline{|x2|}$. Тогда 
$$(1-y)f(x11,y,0)=f(x2,y,0).$$ 
\end{Prop}

Посчитаем, воспользовавшись Утверждениями \ref{evtuh11} и \ref{evtuh12}: 
$$\sum_{i=2}^{|x'2|}\left(i\beta^{i-1} {f\left(x'2,i,0\right)}\right)=$$
$$=\left(\text{По Утверждению \ref{evtuh12} ко всем слагаемым при $x'\in\mathbb{YF}$, $i\in\overline{|x'2|}$}\right)=$$
$$=\sum_{i=2}^{|x'2|}\left(i\beta^{i-1} (1-i){f\left(x'11,i,0\right)}\right)=$$
$$=\left(\text{По Утверждению \ref{evtuh11} ко всем слагаемым при $(x'1)\in\mathbb{YF}$, $i\in\overline{1,|x'11|}$}\right)=$$
$$=\sum_{i=2}^{|x'2|}\left({-\beta^{i-1} (1-i){f\left(x'1,i-1,0\right)}}\right)=\sum_{i=1}^{|x'2|-1}\left({-\beta^{i} (-i){f\left(x'1,i,0\right)}}\right)=$$
$$=\left(\text{По Утверждению \ref{evtuh11} ко всем слагаемым при $x'\in\mathbb{YF}$, $i\in\overline{1,|x'1|}$}\right)=$$
$$=\sum_{i=1}^{|x'2|-1}\left({-\beta^{i} {f\left(x',i-1,0\right)}}\right)=\sum_{i=0}^{|x'2|-2}\left({-\beta^{i+1} {f\left(x',i,0\right)}}\right)=$$
$$=\sum_{i=0}^{|x'|}\left(-{\beta^{i+1} {f\left(x',i,0\right)}}\right)=-\beta\sum_{i=0}^{|x'|}\left({\beta^{i} {f\left(x',i,0\right)}}\right)=-\beta(d_{\beta}(x')).$$
\end{enumerate}

Утверждается, что если мы продифференцируем $d_{\beta}(x)$ ровно $b$ раз, где $b\in\overline{\#x}$, то мы получим выражение вида $$\sum_{i=0}^{b}\left(c_i{\beta^{|k(x,i)|-b}\cdot d_\beta(n(x,i))}\right),$$
где $c_i\in\mathbb{R}$, причём $c_b=\pm 1$ и если $ i\in\overline{b}:$ $|k(x,i)|-b<0$, то $c_i=0$.

Докажем это утверждение по индукции по $b$:

\underline{\textbf{База}}: $b=0$:

В данном случае пусть $c_0=1$. Тогда
$$\sum_{i=0}^{b}\left(c_i{\beta^{|k(x,i)|-b}\cdot d_\beta(n(x,i))}\right)=\sum_{i=0}^{0}\left(1{\beta^{|k(x,i)|-b}\cdot d_\beta(n(x,i))}\right)=1{\beta^{|k(x,0)|-0}\cdot d_\beta(n(x,0))}=$$
$$=\text{(По Замечанию \ref{mega} при $x\in\mathbb{YF}$)}=1{\beta^{|\varepsilon|-0}\cdot  d_\beta(x)}={\beta^{0}\cdot d_\beta(x)}=d_\beta(x),$$
что и требовалось.

\underline{\textbf{База}} доказана.

\underline{\textbf{Переход}} к $(b+1)\in\overline{1,\#x} $:

Зафиксируем $i\in\overline{b}\subseteq \overline{\#x-1}$. Ясно, что если $i\in\overline{\#x-1}$, то 
$n(x,i)\ne \varepsilon$. Рассмотрим два случая:
\begin{enumerate}
    \item $\exists n'\in\mathbb{YF}:$ $n(x,i)=n'1.$
    
    В данном случае $$x=n(x,i)k(x,i),\; \#(k(x,i))=i\Longleftrightarrow x=n'1k(x,i), \;\#(k(x,i))=i\Longleftrightarrow$$
    $$\Longleftrightarrow x=n'1k(x,i),\; \#(1k(x,i))=i+1\Longleftrightarrow n(x,i+1)=n',\;k(x,i+1)=1k(x,i).$$
    
    А это, как мы поняли раньше, значит, что 
    $$\left(c_i{\beta^{|k(x,i)|-b}\cdot d_\beta(n(x,i))}\right)'=\left(c_i{\beta^{|k(x,i)|-b}\cdot d_\beta(n'1)}\right)'=$$
    $$=c_i{\beta^{|k(x,i)|-b}\cdot(-d_\beta(n'))}+c_i\left(|k(x,i)|-b\right){\beta^{|k(x,i)|-b-1}\cdot d_\beta(n'1)}=$$
    $$=-c_i{\beta^{|1k(x,i)|-1-b}\cdot d_\beta(n')}+c_i\left(|k(x,i)|-b\right){\beta^{|k(x,i)|-b-1}\cdot d_\beta(n'1)}=$$
    $$=-c_i{\beta^{|k(x,i+1)|-(b+1)}\cdot d_\beta(n(x,i+1))}+c_i\left(|k(x,i)|-b\right){\beta^{|k(x,i)|-(b+1)}\cdot d_\beta(n(x,i))},$$
    кроме того, очевидно, что
    \begin{itemize} 
        \item Если $|k(x,i)|-b>0$, то 
        $$ |k(x,i)|-b\ge 1 \Longrightarrow |k(x,i+1)|-(b+1)=|1k(x,i)|-(b+1)\ge 1, \;|k(x,i)|-(b+1)\ge 0;$$
        \item Если $|k(x,i)|-b=0$, то 
        $$ |k(x,i+1)|-(b+1)=|1k(x,i)|-(b+1)=0,$$
        то есть наше выражение равно следующему:
        $$-c_i{\beta^{0}(d_\beta(n(x,i+1)))}+c_i\left(0\right){\beta^{|k(x,i)|-(b+1)}(d_\beta(n(x,i)))}=-c_i(d_\beta(n(x,i+1)));$$
        \item Если $|k(x,i)|-b<0$, то по предположению индукции $c_i=0$, а это значит, что наше выражение равняется тождественному нулю.
    \end{itemize}
    \item $\exists n'\in\mathbb{YF}:$ $n(x,i)=n'2.$
    
    В данном случае $$x=n(x,i)k(x,i),\; \#(k(x,i))=i\Longleftrightarrow x=n'2k(x,i), \;\#(k(x,i))=i\Longleftrightarrow$$
    $$\Longleftrightarrow x=n'2k(x,i),\; \#(2k(x,i))=i+1\Longleftrightarrow n(x,i+1)=n',\;k(x,i+1)=2k(x,i).$$

    
    А это, как мы поняли, значит, что 
    $$\left(c_i{\beta^{|k(x,i)|-b} \cdot d_\beta(n(x,i))}\right)'=\left(c_i{\beta^{|k(x,i)|-b}\cdot d_\beta(n'2)}\right)'=$$
    $$=c_i{\beta^{|k(x,i)|-b}\cdot(-\beta d_\beta(n'))}+c_i\left(|k(x,i)|-b\right){\beta^{|k(x,i)|-b-1}\cdot d_\beta(n'2)}=$$
    $$=-c_i{\beta^{|2k(x,i)|-2-b+1}\cdot d_\beta(n')}+c_i\left(|k(x,i)|-b\right){\beta^{|k(x,i)|-b-1}\cdot d_\beta(n'2)}=$$
    $$=-c_i{\beta^{|k(x,i+1)|-(b+1)}\cdot d_\beta(n(x,i+1))}+c_i\left(|k(x,i)|-b\right){\beta^{|k(x,i)|-(b+1)}\cdot d_\beta(n(x,i))},$$
     кроме того, очевидно, что
    \begin{itemize} 
        \item Если $|k(x,i)|-b>0$, то 
        $$ |k(x,i)|-b\ge 1 \Longrightarrow  |k(x,i+1)|-(b+1)=|2k(x,i)|-(b+1)\ge 2,\; |k(x,i)|-(b+1)\ge 0;$$
        \item Если $|k(x,i)|-b=0$, то 
        $$|k(x,i+1)|-(b+1)=|2k(x,i)|-(b+1)=1,$$
        то есть наше выражение равно следующему:
        $$-c_i{\beta^{1}\cdot d_\beta(n(x,i+1))}+c_i\left(0\right){\beta^{|k(x,i)|-(b+1)}\cdot d_\beta(n(x,i))}=-c_i\beta(d_\beta(n(x,i+1)));$$
        \item Если $|k(x,i)|-b<0$, то по предположению индукции $c_i=0$, а это значит, что наше выражение равняется тождественному нулю.
 \end{itemize}
\end{enumerate}

Таким образом, мы получаем, что
$$\left(\sum_{i=0}^{b}\left(c_i{\beta^{|k(x,i)|-b}\cdot d_\beta(n(x,i))}\right)\right)'=$$
$$=\sum_{i=0}^{b}\left(-c_i{\beta^{|k(x,i+1)|-(b+1)}\cdot d_\beta(n(x,i+1))}+c_i\left(|k(x,i)|-b\right){\beta^{|k(x,i)|-(b+1)}\cdot d_\beta(n(x,i))}\right)=$$
$$=c_0\left(|k(x,0)|-b\right){\beta^{|k(x,0)|-(b+1)}\cdot d_\beta(n(x,0))}+$$
$$+\sum_{i=1}^{b}\left(-c_{i-1}{\beta^{|k(x,i)|-(b+1)}\cdot d_\beta(n(x,i))}+c_i\left(|k(x,i)|-b\right){\beta^{|k(x,i)|-(b+1)}\cdot d_\beta(n(x,i))}\right)-$$
$$-c_b{\beta^{|k(x,b+1)|-(b+1)}\cdot d_\beta(n(x,b+1))}=$$
$$=c_0\left(|k(x,0)|-b\right){\beta^{|k(x,0)|-(b+1)}\cdot d_\beta(n(x,0))}+$$
$$+\sum_{i=1}^{b}\left((-c_{i-1}+c_i\left(|k(x,i)|-b\right)){\beta^{|k(x,i)|-(b+1)}\cdot d_\beta(n(x,i))}\right)-$$
$$-c_b{\beta^{|k(x,b+1)|-(b+1)}\cdot d_\beta(n(x,b+1))},$$
то есть если
\begin{itemize}
    \item $$c_0'=c_0\left(|k(x,0)|-b\right);$$
    \item $$\text{Если } i \in \overline{1,b}, \text{ то } c_i'=(-c_{i-1}+c_i\left(|k(x,i)|-b\right));$$
    \item  $$c_{b+1}'=-c_b,$$
\end{itemize}
то наше выражение равно следующему:
$$\sum_{i=0}^{b+1}\left(c_i'{\beta^{|k(x,i)|-(b+1)}\cdot d_\beta(n(x,i))}\right),$$
причём, как мы поняли, если $\forall i\in\overline{b}$ выполнялось условие $|k(x,i)|-b<0 \Longrightarrow c_i=0$, то и $\forall i\in\overline{b+1}$ выполняется условие $|k(x,i)|-(b+1)<0 \Longrightarrow c_i'=0$, а также если $c_b=\pm 1$, то $c'_{b+1}=-c_b=\pm 1$.

Таким образом, \underline{\textbf{Переход}} доказан.

Как мы поняли, если мы продифференцируем $d_{\beta}(x)$ ровно $b$ раз, где $b\in\overline{\#x-1}$, и подставим $\beta=1$, то мы получим выражение вида
$$\sum_{i=0}^{b}\left(c_i{1^{|k(x,i)|-b}\cdot d_1(n(x,i))}\right)=\sum_{i=0}^{b}\left(c_i{1^{|k(x,i)|-b}\cdot\sum_{j=0}^{|n(x,i)|} \left({1^jf\left(n(x,j),j,0\right)}\right)}\right)=0,$$
так как в каждом слагаемом нашего выражения $i\in\overline{b}\subseteq\overline{\#x-1}$, что значит, что $n(x,i)\ne\varepsilon$, из чего следует по Утверждению \ref{evtuh5} при $n(x,i)\in\mathbb{YF}:$ $n(x,i)\ne\varepsilon$, что
$$\displaystyle\sum_{j=0}^{|n(x,i)|} \left({1^jf\left(n(x,j),j,0\right)}\right)=\displaystyle\sum_{j=0}^{|n(x,i)|} {f\left(n(x,j),j,0\right)}=0.$$

Кроме того, если мы продифференцируем $d_{\beta}(x)$ ровно $b$ раз, где $b=\#x$, и подставим $\beta=1$, то мы получим выражение вида
$$\sum_{i=0}^{b}\left(c_i{1^{|k(x,i)|-b}\cdot d_1(n(x,i))}\right)=$$
$$=\sum_{i=0}^{\#x}\left(c_i{1^{|k(x,i)|-\#x}\sum_{j=0}^{|n(x,i)|} \left({1^jf\left(n(x,i),j,0\right)}\right)}\right)=$$
$$=\sum_{i=0}^{\#x}\left(c_i{1^{|k(x,i)|-\#x}\sum_{j=0}^{|n(x,i)|}f\left(n(x,i),j,0\right)}\right)=$$
$$=\sum_{i=0}^{\#x-1}\left(c_i{1^{|k(x,i)|-\#x}\sum_{j=0}^{|n(x,i)|} {f\left(n(x,i),j,0\right)}}\right)+$$
$$+\sum_{i=\#x}^{\#x}\left(c_i{1^{|k(x,i)|-\#x}\sum_{j=0}^{|n(x,i)|} {f\left(n(x,i),j,0\right)}}\right).$$

Ясно, что в каждом слагаемом первой строчки $i\in\overline{\#x-1}$, что значит, что $n(x,i)\ne\varepsilon$, из чего следует по Утверждению \ref{evtuh5} при $n(x,i)\in\mathbb{YF}:$ $n(x,i)\ne\varepsilon$, что
$$\displaystyle\sum_{j=0}^{|n(x,i)|} {f\left(n(x,i),j,0\right)}=0.$$

Таким образом, наше выражение равняется следующему:
$$\sum_{i=0}^{\#x-1}\left(c_i{1^{|k(x,i)|-\#x}\cdot0}\right)+\sum_{i=\#x}^{\#x}\left(c_i{1^{|k(x,i)|-\#x}\sum_{j=0}^{|n(x,i)|} {f\left(n(x,i),j,0\right)}}\right)=$$
$$=\sum_{i=\#x}^{\#x}\left(c_i{1^{|k(x,i)|-\#x}\sum_{j=0}^{|n(x,i)|} {f\left(n(x,i),j,0\right)}}\right)=$$
$$=c_{\#x}{1^{|k(x,\#x)|-\#x}\sum_{j=0}^{|n(x,\#x)|} {f\left(n(x,\#x),j,0\right)}}=c_{\#x}{\sum_{j=0}^{|n(x,\#x)|} {f\left(n(x,\#x),j,0\right)}}=$$
$$=(\text{По Замечанию \ref{mega} при $x\in\mathbb{YF}$})=c_{\#x}{\sum_{j=0}^{|\varepsilon|} {f\left(\varepsilon,j,0\right)}}=c_{\#x}{\sum_{j=0}^{0} {f\left(\varepsilon,j,0\right)}}=$$
$$=c_{\#x}{ {f\left(\varepsilon,0,0\right)}}=c_{\#x}=(\text{Так как мы знаем, что $c_{\#x}=\pm 1$})=\pm 1 \ne 0.$$

Мы поняли, что если мы продифференцируем $d_{\beta}(x)$ как многочлен от $\beta$ ровно $b\in\overline{\#x-1}$ раз и подставим $\beta=1$, то мы получим ноль, а если мы продифференцируем $d_{\beta}(x)$ как многочлен от $\beta$ ровно $\#x$ раз и подставим $\beta=1$, то мы получим не ноль, а значит $d_{\beta}(x)$ как многочлен от $\beta$ делится на $(1-\beta) \quad\# x$ раз и не делится на $(1-\beta) \quad\# x+1$ раз, что и требовалось.

Утверждение доказано.
\end{enumerate}

\end{proof}

\begin{Prop} \label{binomische}
Пусть $x\in\mathbb{YF}$, $\beta\in(0,1):$ $\#x\ge1$. Тогда
$$\frac{d_{\beta}(x)}{(1-\beta)^{\#x}}=\sum_{i=0}^{\infty}\left(\beta^i\sum_{j=0}^{\min(i,|x|)}\left(f(x,j,0)\binom{\#x-1+i-j}{\#x-1}\right)\right).$$
\end{Prop}
\begin{proof}
Знаем, что по биному Ньютона
$$\frac{d_{\beta}(x)}{(1-\beta)^{\#x}}=\frac{\displaystyle\sum_{j=0}^{|x|}\left(\beta^j {f\left(x,j,0\right)}\right)}{(1-\beta)^{\#x}}=$$
$$=\left({\displaystyle\sum_{j=0}^{|x|}\left(\beta^j {f\left(x,j,0\right)}\right)}\right)\sum_{i=0}^{\infty}\left(\beta^i\binom{\#x-1+i}{i}\right)=$$
$$=\left({\displaystyle\sum_{j=0}^{|x|}\left(\beta^j {f\left(x,j,0\right)}\right)}\right)\sum_{i=0}^{\infty}\left(\beta^i\binom{\#x-1+i}{\#x-1}\right)=$$
$$=\sum_{i=0}^{\infty}\left(\sum_{j=0}^{|x|}\left(\beta^{i+j} f\left(x,j,0\right)\binom{\#x-1+i}{\#x-1}\right)\right)=$$
$$=\sum_{k=0}^{\infty}\left(\beta^{k}\sum_{j=0}^{\min(k,|x|)}\left( f\left(x,j,0\right)\binom{\#x-1+k-j}{\#x-1}\right)\right)=$$
$$=\sum_{i=0}^{\infty}\left(\beta^i\sum_{j=0}^{\min(i,|x|)}\left(f(x,j,0)\binom{\#x-1+i-j}{\#x-1}\right)\right),$$
что и требовалось.

Утверждение доказано.
\end{proof}
\begin{Col} \label{binomische1}
Пусть $x\in\mathbb{YF},$ $i\in\mathbb{N}_0:$  $\#x\ge 1$, $i\in\left(\overline{|x|}\textbackslash\overline{|x|-\#x}\right)$. Тогда
$$\sum_{j=0}^{i}\left(f(x,j,0)\binom{\#x-1+i-j}{\#x-1}\right)=0.$$
\end{Col}
\begin{proof}
По Утверждению \ref{binomische} если $x\in\mathbb{YF}$, $\beta\in(0,1):$ $\#x\ge 1$, то
$$\frac{d_{\beta}(x)}{(1-\beta)^{\#x}}=\sum_{i=0}^{\infty}\left(\beta^i\sum_{j=0}^{\min(i,|x|)}\left(f(x,j,0)\binom{\#x-1+i-j}{\#x-1}\right)\right).$$
Несложно заметить, что 
$$d_{\beta}(x)=\sum_{i=0}^{|x|}\left(\beta^i {f\left(x,i,0\right)}\right)$$
-- это многочлен от $\beta$ степени не более, чем $|x|$, а значит,
из Утверждения \ref{delitsa} при $x\in\mathbb{YF}$ ясно, что $\frac{d_{\beta}(x)}{(1-\beta)^{\#x}}$ -- многочлен от $\beta$ степени не более $(|x|-\#x)$, а значит и
$$\sum_{i=0}^{\infty}\left(\beta^i\sum_{j=0}^{\min(i,|x|)}\left(f(x,j,0)\binom{\#x-1+i-j}{\#x-1}\right)\right)$$
-- это многочлен от $\beta$ степени не более $(|x|-\#x)$, а значит $\forall i\in\mathbb{N}_0:$ $i>(|x|-\#x)$
$$\sum_{j=0}^{\min(i,|x|)}\left(f(x,j,0)\binom{\#x-1+i-j}{\#x-1}\right)=0,$$
а значит и $\forall i\in\left(\overline{|x|}\textbackslash\overline{|x|-\#x}\right)$
$$\sum_{j=0}^{i}\left(f(x,j,0)\binom{\#x-1+i-j}{\#x-1}\right)=0,$$
что и требовалось.

Следствие доказано.
\end{proof}
\begin{Col} \label{binomische2}
Пусть $x\in\mathbb{YF}$, $\beta\in(0,1):$ $\#x\ge1$. Тогда
$$\frac{d_{\beta}(x)}{(1-\beta)^{\#x}}=\sum_{i=0}^{|x|-\#x}\left(\beta^i\sum_{j=0}^{i}\left(f(x,j,0)\binom{\#x-1+i-j}{\#x-1}\right)\right).$$
\end{Col}
\begin{proof}
По Утверждению \ref{binomische} если $x\in\mathbb{YF},$ $\beta\in(0,1):$ $\#x\ge 1$, то
$$\frac{d_{\beta}(x)}{(1-\beta)^{\#x}}=\sum_{i=0}^{\infty}\left(\beta^i\sum_{j=0}^{\min(i,|x|)}\left(f(x,j,0)\binom{\#x-1+i-j}{\#x-1}\right)\right).$$
Несложно заметить, что 
$$d_{\beta}(x)=\sum_{i=0}^{|x|}\left(\beta^i {f\left(x,i,0\right)}\right)$$
-- это многочлен от $\beta$ степени не более, чем $|x|$, а значит,
из Утверждения \ref{delitsa} при $x\in\mathbb{YF}$ ясно, что $\frac{d_{\beta}(x)}{(1-\beta)^{\#x}}$ -- многочлен от $\beta$ степени не более $(|x|-\#x)$, а значит и
$$\sum_{i=0}^{\infty}\left(\beta^i\sum_{j=0}^{\min(i,|x|)}\left(f(x,j,0)\binom{\#x-1+i-j}{\#x-1}\right)\right)$$
-- это многочлен от $\beta$ степени не более $(|x|-\#x)$, а значит 
$$\frac{d_{\beta}(x)}{(1-\beta)^{\#x}}=\sum_{i=0}^{\infty}\left(\beta^i\sum_{j=0}^{\min(i,|x|)}\left(f(x,j,0)\binom{\#x-1+i-j}{\#x-1}\right)\right)=$$
$$=\sum_{i=0}^{|x|-\#x}\left(\beta^i\sum_{j=0}^{\min(i,|x|)}\left(f(x,j,0)\binom{\#x-1+i-j}{\#x-1}\right)\right)=$$
$$=(\text{так как в каждом слагаемом $i\le |x|-\#x\le  |x|-1$ })=$$
$$=\sum_{i=0}^{|x|-\#x}\left(\beta^i\sum_{j=0}^{i}\left(f(x,j,0)\binom{\#x-1+i-j}{\#x-1}\right)\right),$$
что и требовалось.

Следствие доказано.
\end{proof}

\renewcommand{\labelenumi}{\arabic{enumi}$^\circ$}
\renewcommand{\labelenumii}{\arabic{enumii}$^\circ$}
\renewcommand{\labelenumiii}{\arabic{enumii}.\arabic{enumiii}$^\circ$}

\begin{Prop} \label{schyot}
    Пусть $x\in\mathbb{YF}$, $i\in\mathbb{N}_0:$ $\#x\ge 2,$ $i\in\overline{1,|x|}$. Тогда
    $$\left(\sum_{j=0}^{i}\left(f(x,j,0)\binom{\#x-1+i-j}{\#x-1}\right)\right)(|x|-i)=$$
    $$=\left(\sum_{j=0}^{i-1}\left(f(x,j,0)\binom{\#x-2+i-j}{\#x-1}\right)\right)(|x|-i-\#x+1)+$$
    $$+\sum_{j=0}^{i}\left((|x|-j)f(x,j,0)\binom{\#x-2+i-j}{\#x-2}\right).$$
    \end{Prop}
    \begin{proof}
    Ясно, что нам достаточно доказать, что $\forall j\in\overline{i-1}$
    $$f(x,j,0)\binom{\#x-1+i-j}{\#x-1}(|x|-i)=$$
    $$=f(x,j,0)\binom{\#x-2+i-j}{\#x-1}(|x|-i-\#x+1)+(|x|-j)f(x,j,0)\binom{\#x-2+i-j}{\#x-2},$$
    а также то, что при $j=i$
    $$f(x,j,0)\binom{\#x-1+i-j}{\#x-1}(|x|-i)=(|x|-j)f(x,j,0)\binom{\#x-2+i-j}{\#x-2}$$
    (так как если мы это докажем, то останется просто просуммировать эти равенства по $j\in\overline{i}$).
    
    Давайте доказывать. Начнём со второго равенства. Подставим в него $j=i$ и получим следующее:
    $$f(x,i,0)\binom{\#x-1+i-i}{\#x-1}(|x|-i)=(|x|-i)f(x,i,0)\binom{\#x-2+i-i}{\#x-2}\Longleftarrow$$
    $$\Longleftarrow\binom{\#x-1+i-i}{\#x-1}=\binom{\#x-2+i-i}{\#x-2}\Longleftrightarrow\binom{\#x-1}{\#x-1}=\binom{\#x-2}{\#x-2} \Longleftrightarrow 1=1,$$
    второе равенство доказано.
    
    Теперь перейдём к первому равенству (при $j\in\overline{i-1}$):
    $$f(x,j,0)\binom{\#x-1+i-j}{\#x-1}(|x|-i)=$$
    $$=f(x,j,0)\binom{\#x-2+i-j}{\#x-1}(|x|-i-\#x+1)+(|x|-j)f(x,j,0)\binom{\#x-2+i-j}{\#x-2}\Longleftarrow$$
    $$\Longleftarrow\binom{\#x-1+i-j}{\#x-1}(|x|-i)=$$
    $$=\binom{\#x-2+i-j}{\#x-1}(|x|-i-\#x+1)+(|x|-j)\binom{\#x-2+i-j}{\#x-2}\Longleftrightarrow$$
    $$\Longleftrightarrow\frac{(\#x-1+i-j)!}{(\#x-1)!(i-j)!}(|x|-i)=$$
    $$=\frac{(\#x-2+i-j)!}{(\#x-1)!(i-j-1)!}(|x|-i-\#x+1)+\frac{(\#x-2+i-j)!}{(\#x-2)!(i-j)!}(|x|-j)\Longleftrightarrow$$
    $$\Longleftrightarrow\frac{(\#x-1+i-j)}{(\#x-1)(i-j)}(|x|-i)=\frac{1}{(\#x-1)}(|x|-i-\#x+1)+\frac{1}{(i-j)}(|x|-j)\Longleftrightarrow$$
    $$\Longleftrightarrow(\#x-1+i-j)(|x|-i)=(|x|-i-\#x+1)(i-j)+(|x|-j)(\#x-1)\Longleftrightarrow$$
    $$\Longleftrightarrow \#x|x|-\#xi-|x|+i+i|x|-i^2-j|x|+ji=$$
    $$=|x|i-|x|j-i^2+ij-\#xi+\#xj+i-j+|x|\#x-|x|-j\#x+j\Longleftrightarrow 0=0, $$
    первое равенство также доказано.
    
    А значит и Утверждение доказано.
    
    \end{proof}

\begin{Prop} \label{binom1}
Пусть $ x\in\mathbb{YF},$ $i\in \mathbb{N}_0:$ $\#x\ge 1,$ $i\in\overline{\#x}$. Тогда
$$\sum_{j=0}^{i}\left(f(x,j,0)\binom{\#x-1+i-j}{\#x-1}\right)\le q(x)\binom{\#x}{i}.$$
 \end{Prop}
\begin{proof}
Давайте доказывать Утверждение по индукции по $\#x$, а при равных $\#x$ --- по $i$.

\underline{\textbf{База}}: $x\in\mathbb{YF},$ $i\in\overline{\#x}:$ $\#x=1$:

Ясно, что в данном случае $x\in\{1,2\}$, а $i\in\{0,1\}$. Итак, рассмотрим четыре случая:
\begin{enumerate}
    \item $x=1,$ $i=0$.
    
    В данном случае неравенство имеет следующий вид:
$$\sum_{j=0}^{0}\left(f(1,j,0)\binom{1-1+0-j}{1-1}\right)\le q(1)\binom{1}{0}\Longleftrightarrow f(1,0,0)\binom{1-1+0-0}{1-1}\le q(1)\Longleftrightarrow$$
$$\Longleftrightarrow f(1,0,0)\binom{0}{0}\le q(1)\Longleftrightarrow f(1,0,0)\le q(1)\Longleftrightarrow$$
$$\Longleftrightarrow\text{(По Замечанию \ref{odnoitozhe} при $1\in\mathbb{YF}$)}\Longleftrightarrow f(1,0,0)\le f(1,0,0) \Longleftrightarrow 0\le 0.$$

Мы поняли, что в данном случае неравенство верно, то есть \underline{\textbf{База}} доказана. 
    \item $x=1,$ $i=1$.
    
    В данном случае неравенство имеет следующий вид:
$$\sum_{j=0}^{1}\left(f(1,j,0)\binom{1-1+1-j}{1-1}\right)\le q(1)\binom{1}{1}\Longleftrightarrow$$
$$\Longleftrightarrow f(1,0,0)\binom{1-1+1-0}{1-1}+f(1,1,0)\binom{1-1+1-1}{1-1}\le q(1)\Longleftrightarrow$$
$$\Longleftrightarrow f(1,0,0)\binom{1}{0}+f(1,1,0)\binom{0}{0}\le q(1)\Longleftrightarrow f(1,0,0)+f(1,1,0)\le q(1)\Longleftrightarrow$$
$$\Longleftrightarrow\text{(По Замечанию \ref{odnoitozhe} при $1\in\mathbb{YF}$)}\Longleftrightarrow f(1,0,0)+f(1,1,0)\le$$
$$\le f(1,0,0) \Longleftrightarrow f(1,1,0)\le 0\Longleftrightarrow -1\le 0.$$

Мы поняли, что в данном случае неравенство верно, то есть \underline{\textbf{База}} доказана. 
    \item $x=2,$ $i=0$.
    
    В данном случае неравенство имеет следующий вид:
$$\sum_{j=0}^{0}\left(f(2,j,0)\binom{1-1+0-j}{1-1}\right)\le q(2)\binom{1}{0}\Longleftrightarrow  f(2,0,0)\binom{1-1+0-0}{1-1}\le q(2)\Longleftrightarrow$$
$$\Longleftrightarrow f(2,0,0)\binom{0}{0}\le q(2)\Longleftrightarrow f(2,0,0)\le q(2)\Longleftrightarrow$$
$$\Longleftrightarrow\text{(По Замечанию \ref{odnoitozhe} при $2\in\mathbb{YF}$)}\Longleftrightarrow f(2,0,0)\le f(2,0,0) \Longleftrightarrow  0\le 0.$$

Мы поняли, что в данном случае неравенство верно, то есть \underline{\textbf{База}} доказана. 
    \item $x=2,$ $i=1$.
    
    В данном случае неравенство имеет следующий вид:
$$\sum_{j=0}^{1}\left(f(2,j,0)\binom{1-1+1-j}{1-1}\right)\le q(2)\binom{1}{1}\Longleftrightarrow$$
$$\Longleftrightarrow f(2,0,0)\binom{1-1+1-0}{1-1}+f(2,1,0)\binom{1-1+1-1}{1-1}\le q(2)\Longleftrightarrow$$
$$\Longleftrightarrow f(2,0,0)\binom{1}{0}+f(2,1,0)\binom{0}{0}\le q(2)\Longleftrightarrow f(2,0,0)+f(2,1,0)\le q(2)\Longleftrightarrow$$
$$\Longleftrightarrow\text{(По Замечанию \ref{odnoitozhe} при $2\in\mathbb{YF}$)}\Longleftrightarrow f(2,0,0)+f(2,1,0)\le$$
$$\le f(2,0,0) \Longleftrightarrow f(2,1,0)\le 0\Longleftrightarrow 0\le 0.$$

Мы поняли, что в данном случае неравенство верно, то есть \underline{\textbf{База}} доказана. 

\end{enumerate}

Ясно, что все случаи разобраны.

\underline{\textbf{База}} доказана.

\underline{\textbf{Переход}} к $x\in\mathbb{YF},$ $i\in\overline{\#x}:$ $\#x\ge 2$:

Рассмотрим четыре случая:

\begin{enumerate}
    \item $x\in\mathbb{YF},$ $i\in\{0\}$: $\#x\ge 2$. 
    
    В данном случае неравенство имеет  следующий вид:
    $$\sum_{j=0}^{0}\left(f(x,j,0)\binom{\#x-1+0-j}{\#x-1}\right)\le q(x)\binom{\#x}{0}\Longleftrightarrow$$
    $$\Longleftrightarrow f(x,0,0)\binom{\#x-1+0-0}{\#x-1}\le q(x)\Longleftrightarrow$$
    $$\Longleftrightarrow  f(x,0,0)\binom{\#x-1}{\#x-1}\le q(x)\Longleftrightarrow f(x,0,0)\le q(x).$$
    
    По Замечанию \ref{odnoitozhe} при $x\in\mathbb{YF}$
    $$f(x,0,0)=q(x),$$
    а значит в данном случае \underline{\textbf{Переход}} доказан.

        \item $x\in\mathbb{YF}$, $i\in\mathbb{N}_0:$ $\#x\ge 2,$ $i\in\left(\overline{\#x}\textbackslash\overline{|x|-\#x}\right)$.
        
        Ясно, что если $x\in\mathbb{YF}$, то $\#x\le|x|$, а значит $\left(\overline{\#x}\textbackslash\overline{|x|-\#x}\right)\subseteq\left(\overline{|x|}\textbackslash\overline{|x|-\#x}\right)$, а поэтому в данном случае по Следствию \ref{binomische1} при $x\in\mathbb{YF},$ $i\in\mathbb{N}_0$    $$\sum_{j=0}^{i}\left(f(x,j,0)\binom{\#x-1+i-j}{\#x-1}\right)=0\le \frac{1}{\displaystyle\prod_{i=1}^{\#x}|k(x,i)|}\binom{\#x}{i}=q(x)\binom{\#x}{i},$$
        
        что и требовалось.
  
    В данном случае \underline{\textbf{Переход}} доказан.
    \item

 $x\in\mathbb{YF},$ $i\in\mathbb{N}_0:$ $\#x\ge2,$ $i\in\left(\overline{|x|-\#x}\textbackslash \{0,\#x\}\right).$
    
    Заметим, что $i\ne 0$, что значит, что по предположению индукции наше равенство верно при $x$ и $i-1$, то есть
$$\sum_{j=0}^{i-1}\left(f(x,j,0)\binom{\#x-2+i-j}{\#x-1}\right)\le q(x)\binom{\#x}{i-1}.$$
Давайте заметим, что $i\in\left(\overline{|x|-\#x}\textbackslash \{0,\#x\}\right)
     \Longrightarrow i\le |x|-\#x \Longleftrightarrow |x|-\#x-i\ge 0 \Longleftrightarrow |x|-i-\#x+1\ge 1$, а значит данное неравенство равносильно следующему:
$$\left(\sum_{j=0}^{i-1}\left(f(x,j,0)\binom{\#x-2+i-j}{\#x-1}\right)\right)\left(|x|-i-\#x+1\right)\le q(x)\binom{\#x}{i-1}\left(|x|-i-\#x+1\right).$$

Назовём это первым \underline{запомненным} неравенством.

    Также в данном случае воспользуемся неравенством для $x':$ $x=\alpha_0x'$ (при некотором $\alpha_0\in\{1,2\}$), а также $i\in\mathbb{N}_0:$ $i\in\left(\overline{|x|-\#x}\textbackslash \{0,\#x\}\right)$. Если $i\in\overline{\#x'}$, то оно верно по предположению индукции. Давайте обоснуем, что $i\in\overline{\#x'}$:
    $$i\in\left(\overline{|x|-\#x}\textbackslash \{0,\#x\}\right)\subseteq(\text{Ясно, что }|x|\le 2\#x\Longleftrightarrow |x|-\#x\le \#x)\subseteq$$
    $$\subseteq \overline{\#x}\textbackslash \{0,\#x\}\subseteq \overline{\#x-1}=\overline{\#(\alpha_0x')-1}=\overline{1+\#(x')-1}=\overline{\#(x')}.$$
    
    Обосновали, а значит мы можем воспользоваться неравенством для $x':$ $x=\alpha_0x'$ (при некотором $\alpha_0\in\{1,2\}$), а также $i\in\mathbb{N}_0:$ $i\in\left(\overline{|x|-\#x}\textbackslash \{0,\#x\}\right)$. Давайте воспользуемся:
    $$\sum_{j=0}^{i}\left(f(x',j,0)\binom{\#x'-1+i-j}{\#x'-1}\right)\le q(x')\binom{\#x'}{i}.$$
    
    К каждому слагаемому применим Утверждение \ref{evtuh7} при $x'\in\mathbb{YF}$, $j\in\overline{|x'|}$, $0\in\overline{\#x'}$ и $\alpha_0\in\{1,2\}$ (в каждом слагаемом действительно $j\in\overline{|x'|}$, так как в данном случае $j\in\overline{i}\subseteq\overline{\#x'}\subseteq\overline{|x'|}$) и получим, что наше неравенство равносильно следующему:
   $$\sum_{j=0}^{i}\left((|\alpha_0x'|-j)f(\alpha_0x',j,0)\binom{\#x'-1+i-j}{\#x'-1}\right)\le q(x')\binom{\#x'}{i}\Longleftrightarrow$$
    $$\Longleftrightarrow(\text{Так как ясно, что $\#x'=\#(\alpha_0x')-1=\#x-1$})\Longleftrightarrow$$
    $$\Longleftrightarrow\sum_{j=0}^{i}\left((|x|-j)f(x,j,0)\binom{\#x-2+i-j}{\#x-2}\right)\le q(x')\binom{\#x-1}{i}=$$
    $$=(\text{По Утверждению \ref{q} при $x,x'\in\mathbb{YF}$, $\alpha_0\in\{1,2\}$})=$$
    $$=\sum_{j=0}^{i}\left((|x|-j)f(x,j,0)\binom{\#x-2+i-j}{\#x-2}\right)\le |x|q(x)\binom{\#x-1}{i}.$$
    
    Назовём это вторым \underline{запомненным} неравенством.

    Теперь напишем Утверждение \ref{schyot} при наших $x\in\mathbb{YF},$ $i\in\mathbb{N}_0:$ $\#x\ge 2,$ $i\in\left(\overline{|x|-\#x}\textbackslash \{0,\#x\}\right)\subseteq\overline{1,|x|}:$
    $$\left(\sum_{j=0}^{i}\left(f(x,j,0)\binom{\#x-1+i-j}{\#x-1}\right)\right)(|x|-i)=$$
    $$=\left(\sum_{j=0}^{i-1}\left(f(x,j,0)\binom{\#x-2+i-j}{\#x-1}\right)\right)(|x|-i-\#x+1)+$$
    $$+\sum_{j=0}^{i}\left((|x|-j)f(x,j,0)\binom{\#x-2+i-j}{\#x-2}\right)\Longrightarrow$$
    $$\Longrightarrow(\text{так как $i\le |x|-\#x \Longrightarrow |x|-i\ge\#x\ge2>0$})\Longrightarrow$$
    $$\Longrightarrow\sum_{j=0}^{i}\left(f(x,j,0)\binom{\#x-1+i-j}{\#x-1}\right)=$$
    $$=\frac{\displaystyle\left(\sum_{j=0}^{i-1}\left(f(x,j,0)\binom{\#x-2+i-j}{\#x-1}\right)\right)(|x|-i-\#x+1)}{\displaystyle|x|-i}+$$
    $$+\frac{\displaystyle\sum_{j=0}^{i}\left((|x|-j)f(x,j,0)\binom{\#x-2+i-j}{\#x-2}\right)}{\displaystyle|x|-i}\le$$
    $$\le( \text{По первому \underline{запомненному} неравенству и второму \underline{запомненному} неравенству})\le $$
    $$\le\frac{q(x)\binom{\#x}{i-1}(|x|-i-\#x+1)+|x|q(x)\binom{\#x-1}{i}}{|x|-i}.$$

    Теперь давайте доказывать, что
    $$\frac{q(x)\binom{\#x}{i-1}(|x|-i-\#x+1)+|x|q(x)\binom{\#x-1}{i}}{|x|-i}\le q(x)\binom{\#x}{i}\Longleftrightarrow$$
    $$\Longleftrightarrow\text{(По положительности функции $q$)}\Longleftrightarrow$$
    $$\Longleftrightarrow\frac{\displaystyle\frac{(\#x)!}{(\#x-i+1)!(i-1)!}(|x|-i-\#x+1)+|x|\frac{(\#x-1)!}{(\#x-i-1)!(i)!}}{|x|-i}\le \frac{(\#x)!}{(\#x-i)!(i)!}\Longleftrightarrow$$
    $$\Longleftrightarrow(\text{так как $i\le |x|-\#x \Longrightarrow |x|-i\ge\#x\ge2>0$})\Longleftrightarrow$$
    $$\Longleftrightarrow{\frac{(\#x)!}{(\#x-i+1)!(i-1)!}(|x|-i-\#x+1)+|x|\frac{(\#x-1)!}{(\#x-i-1)!(i)!}}\le \frac{(\#x)!}{(\#x-i)!(i)!}({|x|-i})\Longleftrightarrow$$
    $$\Longleftrightarrow{\frac{(\#x)}{(\#x-i+1)(\#x-i)}(|x|-i-\#x+1)+|x|\frac{1}{(i)}}\le \frac{(\#x)}{(\#x-i)(i)}({|x|-i})\Longleftrightarrow$$
    $$\Longleftrightarrow{{i\#x}(|x|-i-\#x+1)+|x|(\#x-i+1)(\#x-i)}\le {\#x(\#x-i+1)}({|x|-i})\Longleftrightarrow$$
    $$\Longleftrightarrow{i\#x|x|-i\#xi-i(\#x)^2+i\#x+|x|(\#x)^2+|x|i^2-2|x|\#xi+|x|\#x-i|x|}\le$$
    $$\le (\#x)^2|x|-(\#x)^2i-\#xi|x|+\#xi^2+|x|\#x-\#xi\Longleftrightarrow$$
    $$\Longleftrightarrow{-i\#xi+i\#x+|x|i^2-i|x|}\le \#xi^2-\#xi\Longleftrightarrow 0 \le (-2\#x+|x|)(1-i)i.$$
    
    В нашем случае $i\ge 1$, а значит $i>0,$ $(1-i)\le 0$.
    
    Также ясно, что $2\#x\ge|x|$, а значит $(-2\#x+|x|)\le 0$.
    
    Из этого можно понять, что $(-2\#x+|x|)(1-i)i\ge 0,$ что и требовалось.
    
    Можно заметить, что мы доказали, что в данном случае     $$\sum_{j=0}^{i}\left(f(x,j,0)\binom{\#x-1+i-j}{\#x-1}\right)\le$$
    $$\le\frac{q(x)\binom{\#x}{i-1}(|x|-i-\#x+1)+|x|q(x)\binom{\#x-1}{i}}{|x|-i}\le q(x)\binom{\#x}{i}\Longrightarrow$$
    $$\Longrightarrow\sum_{j=0}^{i}\left(f(x,j,0)\binom{\#x-1+i-j}{\#x-1}\right)\le q(x)\binom{\#x}{i},$$
    что и требовалось.
    
    А значит в данном случае \underline{\textbf{Переход}} доказан.

    \item $x\in\mathbb{YF},$ $i\in\mathbb{N}_0:$ $\#x\ge 2$, $i\in\left(\left(\overline{|x|-\#x}\textbackslash \{0\}\right)\cap\{\#x\}\right).$
    
    Заметим, что $i\ne 0$, что значит, что по предположению индукции наше равенство верно при $x$ и $i-1$, то есть
$$\sum_{j=0}^{i-1}\left(f(x,j,0)\binom{\#x-2+i-j}{\#x-1}\right)\le q(x)\binom{\#x}{i-1}.$$
Давайте заметим, что $i\in\left(\overline{|x|-\#x}\textbackslash \{0\}\right)
     \Longrightarrow i\le |x|-\#x \Longleftrightarrow |x|-\#x-i\ge 0 \Longrightarrow |x|-i-\#x+1\ge 1$, а значит данное неравенство равносильно следующему:
    $$\left(\sum_{j=0}^{i-1}\left(f(x,j,0)\binom{\#x-2+i-j}{\#x-1}\right)\right)\left(|x|-i-\#x+1\right)\le q(x)\binom{\#x}{i-1}\left(|x|-i-\#x+1\right).$$

Назовём это \underline{запомненным} неравенством.

Также заметим, что в данном случае $i\in\left(\left(\overline{|x|-\#x}\textbackslash \{0\}\right)\cap\{\#x\}\right),$
     значит $\#x\in\overline{|x|-\#x}\Longrightarrow \#x\le |x|-\#x\Longleftrightarrow 2\#x\le |x| \Longleftrightarrow( \text{так как ясно, что }2\#x\ge |x|)\Longleftrightarrow 2\#x= |x|\Longleftrightarrow$ номер $x$ состоит только из двоек $\Longleftrightarrow x=2^{\#x}$.

Рассмотрим $x':=2^{\#x-1}$ а также наше $i=\#x$. Заметим, что 
\begin{itemize}
    \item $$i=\#x\le\text{(так как в данном случае $\#x\ge 2$)}\le 2\#x-2=2(\#x-1)=\left|2^{\#x-1}\right|=\left|x'\right| ;$$    
    \item $$\left|x'\right|-\#x'= \left|2^{\#x-1}\right|-\#\left(2^{\#x-1}\right)=2(\#x-1)-(\#x-1)=(\#x-1)<\#x=i;$$
    \item $$\#x'=\#\left(2^{\#x-1}\right)=(\#x-1)\ge 2-1\ge 1.$$
\end{itemize}

То есть мы поняли, что $x'\in\mathbb{YF},$ $i\in\mathbb{N}_0$: $\#x'\ge 1,$ $i\in\left(\overline{|x'|}\textbackslash\overline{|x'|-\#x'}\right)$, а это значит, что в данном случае по Следствию \ref{binomische1}
    $$\sum_{j=0}^{i}\left(f(x',j,0)\binom{\#x'-1+i-j}{\#x'-1}\right)=0.$$

    К каждому слагаемому применим Утверждение \ref{evtuh7} при $x'\in\mathbb{YF}$, $j\in\overline{|x'|}$, $0\in\overline{\#x'}$ и $\alpha_0\in\{1,2\}$ (в каждом слагаемом действительно $j\in\overline{|x'|}$, так как в данном случае $j\in\overline{i}\subseteq\overline{|x'|}$) и получим, что наше равенство равносильно следующему:
    $$\sum_{j=0}^{i}\left((|\alpha_0x'|-j)f(\alpha_0x',j,0)\binom{\#x'-1+i-j}{\#x'-1}\right)=0\Longleftrightarrow$$
    $$\Longleftrightarrow(\text{Так как ясно, что $\#x'=\#(\alpha_0x')-1=\#x-1$})\Longleftrightarrow$$
    $$\Longleftrightarrow\sum_{j=0}^{i}\left((|x|-j)f(x,j,0)\binom{\#x-2+i-j}{\#x-2}\right)=0.$$
    Назовём это \underline{запомненным} равенством.

    Итак, давайте считать. Напишем Утверждение \ref{schyot} при наших  $x\in\mathbb{YF},$ $i\in\mathbb{N}_0:$ $\#x\ge 2$, $i\in\left(\left(\overline{|x|-\#x}\textbackslash \{0\}\right)\cap\{\#x\}\right)\subseteq\overline{1,|x|}:$
    $$\left(\sum_{j=0}^{i}\left(f(x,j,0)\binom{\#x-1+i-j}{\#x-1}\right)\right)(|x|-i)=$$
    $$=\left(\sum_{j=0}^{i-1}\left(f(x,j,0)\binom{\#x-2+i-j}{\#x-1}\right)\right)(|x|-i-\#x+1)+$$
    $$+\sum_{j=0}^{i}\left((|x|-j)f(x,j,0)\binom{\#x-2+i-j}{\#x-2}\right)\Longleftrightarrow$$
    $$\Longleftrightarrow(\text{так как $i\le |x|-\#x \Longrightarrow |x|-i\ge\#x\ge2>0$})\Longleftrightarrow$$
    $$\Longleftrightarrow\sum_{j=0}^{i}\left(f(x,j,0)\binom{\#x-1+i-j}{\#x-1}\right)=$$
    $$=\frac{\displaystyle\left(\sum_{j=0}^{i-1}\left(f(x,j,0)\binom{\#x-2+i-j}{\#x-1}\right)\right)(|x|-i-\#x+1)}{\displaystyle|x|-i}+$$
    $$+\frac{\displaystyle\sum_{j=0}^{i}\left((|x|-j)f(x,j,0)\binom{\#x-2+i-j}{\#x-2}\right)}{\displaystyle|x|-i}\le$$
    $$\le\left( \text{По \underline{запомненному} неравенству и \underline{запомненному} равенству}\right)\le $$
    $$\le\frac{\displaystyle q(x)\binom{\#x}{i-1}(|x|-i-\#x+1)+0}{\displaystyle|x|-i}=$$
    $$=\left(\text{так как мы уже поняли, что в данном случае $x=2^{\#x}$, $|x|=2\#x,$ $i=\#x$}\right)=$$
    $$=\frac{\displaystyle q(x)\binom{\#x}{\#x-1}(2\#x-\#x-\#x+1)}{\displaystyle 2\#x-\#x}=\frac{\displaystyle q(x)\#x}{\displaystyle\#x}=q(x)=q(x)\binom{\#x}{\#x}=q(x)\binom{\#x}{i},$$
    что и требовалось, а значит в данном случае \underline{\textbf{Переход}} доказан.
    
\end{enumerate}

Ясно, что все случаи разобраны, то есть \underline{\textbf{Переход}} доказан.
    
Утверждение доказано.

\end{proof}
\begin{Prop} \label{mamka2}
Пусть $x\in \mathbb{YF}$, $\beta\in(0,1)$. Тогда
$$d_{\beta}(x)\le q(x)\left(1-\beta^2\right)^{\#x}.$$
\end{Prop}
\begin{proof}

Рассмотрим два случая:
\begin{enumerate}
    \item $\#x=0\Longleftrightarrow x=\varepsilon$.
    
    В данном случае $$d_{\beta}(\varepsilon)\le q(\varepsilon)\left(1-\beta^2\right)^{\#\varepsilon}\Longleftrightarrow \sum_{i=0}^{|\varepsilon|}\left(\beta^i {f\left(\varepsilon,i,0\right)}\right)\le\frac{1}{\displaystyle\prod_{i=1}^{\#\varepsilon}|k(\varepsilon,i)|}\left(1-\beta^2\right)^{0}\Longleftrightarrow$$
$$\Longleftrightarrow \sum_{i=0}^{0}\left(\beta^i {f\left(\varepsilon,i,0\right)}\right)\le\frac{1}{\displaystyle\prod_{i=1}^{0}|k(\varepsilon,i)|}\left(1-\beta^2\right)^{0}\Longleftrightarrow\beta^0 {f\left(\varepsilon,0,0\right)}\le 1\cdot1\Longleftrightarrow 1\le 1 ,$$
что и требовалось.

В данном случае Утверждение доказано.

\item $\#x\ge 1$.

Ясно, что (так как $\beta\in(0,1)$)
$${d_{\beta}(x)}\le q(x)\left(1-\beta^2\right)^{\#x} \Longleftrightarrow \frac{d_{\beta}(x)}{(1-\beta)^{\#x}}\le q(x)\left(1+\beta\right)^{\#x}.$$

По Следствию \ref{binomische2} при $x\in\mathbb{YF}$, $\beta\in(0,1)$
$$\frac{d_{\beta}(x)}{(1-\beta)^{\#x}}=\sum_{i=0}^{|x|-\#x}\left(\beta^i\sum_{j=0}^{i}\left(f(x,j,0)\binom{\#x-1+i-j}{\#x-1}\right)\right).$$

По Утверждению \ref{binom1} при $ x\in\mathbb{YF},$ $i\in \mathbb{N}_0$ применённому к каждому слагаемому (это законно, так как если $ x\in\mathbb{YF}$, то $2\#x\ge |x|\Longleftrightarrow \#x\ge |x|-\#x$, а также так как $\beta\in(0,1)$) получаем, что 
$$\sum_{i=0}^{|x|-\#x}\left(\beta^i\sum_{j=0}^{i}\left(f(x,j,0)\binom{\#x-1+i-j}{\#x-1}\right)\right)\le \sum_{i=0}^{|x|-\#x}\left(\beta^iq(x)\binom{\#x}{i}\right)\le$$
$$\le(\text{так как если $x\in\mathbb{YF}$, то $2\#x\ge |x|\Longleftrightarrow\#x\ge |x|-\#x,$ а также так как $q(x)\ge 0$ и $\beta\in(0,1)$})\le$$
$$\le \sum_{i=0}^{\#x}\left(\beta^iq(x)\binom{\#x}{i}\right)= q(x)\sum_{i=0}^{\#x}\left(\beta^i\binom{\#x}{i}\right)=q(x)(1+\beta)^{\#x},$$
так как тут написан бином Ньютона.

То есть мы доказали, что при $x\in \mathbb{YF}$, $\beta\in(0,1)$: $\#x\ge 1$
$$\frac{d_{\beta}(x)}{(1-\beta)^{\#x}}\le q(x)(1+\beta)^{\#x}.$$
А это (так как $\beta\in(0,1)$) равносильно тому, что 
$${d_{\beta}(x)}\le q(x)\left(1-\beta^2\right)^{\#x},$$
что и требовалось.

В данном случае Утверждение доказано.
\end{enumerate}
Ясно, что все случаи разобраны.

Утверждение доказано.

\end{proof}

\newpage
\section{Волшебные таблицы}

\begin{Def}
Пусть $w\in \mathbb{YF}_\infty$, $\beta\in(0,1]$, $n\in\mathbb{N}_0$.
Тогда волшебной таблицей $T_{w,\beta,n}(x,y)$ с параметрами $w,\beta$ и $n$ назовём функцию 
$$T_{w,\beta,n}:\mathbb{YF}_n\times \overline{n} \to \mathbb{R},$$
определённую следующим образом:

\begin{equation*}
T_{w,\beta,n}(x,y)=
 \begin{cases}
   $$\displaystyle d(\varepsilon,x)\cdot  q(x(y))\cdot d_1'(x'(y),w)\cdot \beta^y\cdot \left(1-\beta^2\right)^{\#(x(y))}$$ &\text {если $x\in K(n,y)$}\\
   0 &\text{если $x\in \overline{K}(n,y)$.
}
 \end{cases}
\end{equation*}

или, что то же самое,

\begin{equation*}
T_{w,\beta,n}(x,y)=
 \begin{cases}
   $$\displaystyle d(\varepsilon,x)\cdot q(n(x,i))\cdot d_1'(k(x,i),w)\cdot \beta^y\cdot \left(1-\beta^2\right)^{\#(n(x,i))}$$ &\text {если $\exists i\in\overline{\#x}:\; |k(x,i)|=y$}\\
   0 &\text{если $\nexists i\in\overline{\#x}:\; |k(x,i)|=y$.}
 \end{cases}
\end{equation*}
\end{Def}

\begin{Zam}
Из всех обозначений очевидно, что эти определения равносильны.
\end{Zam}

Визуализируем данную функцию мы следующим образом (отсюда и название):

\begin{Ex}
$w\in\mathbb{YF}_\infty$, $\beta\in(0,1]$, $n=5:$ 
\tiny
\begin{center}
\begin{tabular}{ | m{0.5cm} || m{2.2cm} | m{2.2cm} | m{2.3cm} | m{2.4cm} | m{2.4cm} | m{1.7cm} | } 
  \hline
 & $$0$$ & $$1$$ & $$2$$ & $$3$$ & $$4$$ & $$5$$\\
  \hline \hline
$$122$$ & $$\frac{3}{40}{d_1'(\varepsilon,w)}(1-\beta^2)^3$$ & $$0$$ & $$\frac{3}{6}{d_1'(2,w)}\beta^2(1-\beta^2)^2$$ &$$0$$ &$$\frac{3}{1}{d_1'(22,w)}\beta^4(1-\beta^2)$$ &$$\frac{3}{1}{d_1'(122,w)}\beta^5$$\\
  \hline
$$212$$ & $$\frac{4}{30}{d_1'(\varepsilon,w)}(1-\beta^2)^3$$ & $$0$$ & $$\frac{4}{3}{d_1'(2,w)}\beta^2(1-\beta^2)^2$$ &$$\frac{4}{2}{d_1'(12,w)}\beta^3(1-\beta^2)$$ &$$0$$ &$$\frac{4}{1}{d_1'(212,w)}\beta^5$$\\
  \hline
$$1112$$ & $$\frac{1}{120}{d_1'(\varepsilon,w)}(1-\beta^2)^4$$ & $$0$$ & $$\frac{1}{6}{d_1'(2,w)}\beta^2(1-\beta^2)^3$$ &$$\frac{1}{2}{d_1'(12,w)}\beta^3(1-\beta^2)^2$$ &$$\frac{1}{1}{d_1'(112,w)}\beta^4(1-\beta^2)$$ &$$\frac{1}{1}{d_1'(1112,w)}\beta^5$$\\
  \hline
$$221$$ & $$\frac{8}{15}{d_1'(\varepsilon,w)}(1-\beta^2)^3$$ & $$\frac{8}{8}{d_1'(1,w)}\beta(1-\beta^2)^2$$ & $$0$$ &$$\frac{8}{2}{d_1'(21,w)}\beta^3(1-\beta^2)$$  &$$0$$ &$$\frac{8}{1}{d_1'(221,w)}\beta^5$$\\
  \hline
$$1121$$ & $$\frac{2}{60}{d_1'(\varepsilon,w)}(1-x^2)^4$$ & $$\frac{2}{24}{d_1'(1,w)}\beta(1-\beta^2)^3$$ & $$0$$ &$$\frac{2}{2}{d_1'(21,w)}\beta^3(1-\beta^2)^2$$  &$$\frac{2}{1}{d_1'(121,w)}\beta^4(1-\beta^2)$$ &$$\frac{2}{1}{d_1'(1121,w)}\beta^5$$\\
  \hline
$$1211$$ & $$\frac{3}{40}{d_1'(\varepsilon,w)}(1-\beta^2)^4$$ & $$\frac{3}{12}{d_1'(1,w)}\beta(1-\beta^2)^3$$ & $$\frac{3}{6}{d_1'(11,w)}\beta^2(1-\beta^2)^2$$ &$$0$$ &$$\frac{3}{1}{d_1'(211,w)}\beta^4(1-\beta^2)$$ &$$\frac{3}{1}{d_1'(1211,w)}\beta^5$$\\
  \hline
$$2111$$ & $$\frac{4}{30}{d_1'(\varepsilon,w)}(1-\beta^2)^4$$ & $$\frac{4}{8}{d_1'(1,w)}\beta(1-\beta^2)^3$$ & $$\frac{4}{3}{d_1'(11,w)}\beta^2(1-\beta^2)^2$$ &$$\frac{4}{2}{d_1'(111,w)}\beta^3(1-\beta^2)$$  &$$0$$ &$$\frac{4}{1}{d_1'(2111,w)}\beta^5$$\\
  \hline
$$11111$$ & $$\frac{1}{120}{d_1'(\varepsilon,w)}(1-\beta^2)^5$$ & $$\frac{1}{24}{d_1'(1,w)}\beta(1-\beta^2)^4$$ & $$\frac{1}{6}{d_1'(11,w)}\beta^2(1-\beta^2)^3$$ &$$\frac{1}{2}{d_1'(111,w)}\beta^3(1-\beta^2)^2$$  &$$\frac{1}{1}{d_1'(1111,w)}\beta^4(1-\beta^2)$$ &$$\frac{1}{1}{d_1'(11111,w)}\beta^5$$\\
  \hline
\end{tabular}
\end{center}
\normalsize
\end{Ex}

\begin{Prop}\label{limitstrih}
Пусть $x\in \mathbb{YF}$, $y\in\mathbb{YF}_{\infty}$. Тогда
$$d'_1(x,y)=\lim_{m \to \infty} {\frac{d(x,y_m)}{d(\varepsilon,y_m)}}.$$

\end{Prop}
\begin{proof}
Посчитаем (всегда считаем, что $m\ge |x|,$ что значит, что $|y_m|\ge \#y_m=m\ge |x|$):
$$\lim_{m \to \infty} {\frac{d(x,y_m)}{d(\varepsilon,y_m)}}=(\text{По Теореме \ref{evtuh} при $x,y_m\in\mathbb{YF}$})=$$
$$=\lim_{m \to \infty} {\frac{\displaystyle\sum_{i=0}^{|x|}\left({ f\left(x,i,h(x,y_m)\right)}\prod_{j=1}^{d(y_m)}\left(g\left(y_m,j\right)-i\right)\right)}{\displaystyle \prod_{j=1}^{d(y_m)}g\left(y_m,j\right)}}=$$
$$=\lim_{m \to \infty}\left( {\sum_{i=0}^{|x|}\left( { f\left(x,i,h(x,y_m)\right)}\prod_{j=1}^{d(y_m)}\frac{\left(g\left(y_m,j\right)-i\right)}{g\left(y_m,j\right)}\right)}\right)=$$
$$=\text{(Ясно, что $\#y_m=m\ge |x|\ge\#x\Longrightarrow h(x,y_m)=h(x,y)$)}=$$
$$=\lim_{m \to \infty}\left( {\sum_{i=0}^{|x|}\left( { f\left(x,i,h(x,y)\right)}\prod_{j=1}^{d(y_m)}\frac{\left(g\left(y_m,j\right)-i\right)}{g\left(y_m,j\right)}\right)}\right)=$$
$$={\sum_{i=0}^{|x|}\left( { f\left(x,i,h(x,y)\right)}\cdot\lim_{m \to \infty} \prod_{j=1}^{d(y_m)}\frac{\left(g\left(y_m,j\right)-i\right)}{g\left(y_m,j\right)}\right)}=$$
$$=(\text{По определению функции $g$})=$$
$$= {\sum_{i=0}^{|x|}\left( { f\left(x,i,h(x,y)\right)}\prod_{j=1}^{d(y)}\frac{\left(g\left(y,j\right)-i\right)}{g\left(y,j\right)}\right)}=$$
$$= {\sum_{i=0}^{|x|}\left( {1^i f\left(x,i,h(x,y)\right)}\prod_{j=1}^{d(y)}\frac{\left(g\left(y,j\right)-i\right)}{g\left(y,j\right)}\right)}=d'_1(x,y),$$
что и требовалось.

Утверждение доказано.
\end{proof}

\begin{Col}\label{neo}
Пусть $x\in \mathbb{YF}$, $y\in\mathbb{YF}_{\infty}$. Тогда
$$d'_1(x,y)=\lim_{m \to \infty} {\frac{d(x,y_m)}{d(\varepsilon,y_m)}}\ge 0.$$

\end{Col}

\begin{Col}\label{granatakerambita}
Пусть $w\in\mathbb{YF}_\infty,$ $v\in\mathbb{YF}$. Тогда
$$\mu_{w,1}(v)=\mu_w(v).$$
\end{Col}
\begin{proof}
Воспользуется Утверждением \ref{limitstrih} и обозначениями:
$$\mu_{w,1}(v)=d(\varepsilon,v)\cdot d_1'(v,w)=d(\varepsilon,v)\cdot\lim_{m\to\infty}\frac{d(v,w_m)}{d(\varepsilon,w_m)}=\lim_{m\to\infty}\frac{d(\varepsilon,v)d(v,w_m)}{d(\varepsilon,w_m)}=\mu_w(v),$$
что и требовалось.

Следствие доказано.
\end{proof}

\begin{Zam}
Пусть $w\in \mathbb{YF}_\infty$, $\beta\in(0,1]$, $n\in\mathbb{N}_0$.
Тогда функция $T_{w,\beta,n}(x,y)$ неотрицательна.
\end{Zam}

\begin{Prop}[Утверждение  6\cite{Evtuh3}]  \label{razbivaem}
Пусть $x,x',x''\in\mathbb{YF}:$ $x=x'x''$. Тогда
$$d(\varepsilon,x)=d(\varepsilon,x'')d\left(\varepsilon,x'1^{\left|x''\right|}\right).$$
\end{Prop}

\begin{Prop}[Утверждение  2\cite{Evtuh3}] \label{mera}
Пусть $w\in\mathbb{YF}_\infty$, $m,n\in\mathbb{N}_0:$ $|w_m|\ge n$. Тогда
$$\sum_{v\in\mathbb{YF}_n}\frac{d(\varepsilon,v)d(v,w_m)}{d(\varepsilon,w_m)}=1. $$
\end{Prop}

\begin{Prop} \label{sum}
Пусть $w\in\mathbb{YF}_\infty$, $\beta\in(0,1]$, $n,y\in\mathbb{N}_0:$ $y\le n$. Тогда
$$\sum_{x\in\mathbb{YF}_n} T_{w,\beta,n}(x,y)=\sum_{x'\in\mathbb{YF}_{n-y} }\left(q(x')\cdot d(\varepsilon,x'1^y)\cdot \beta^y\cdot \left(1-\beta^2\right)^{\#x'}\right).$$
\end{Prop}
\begin{proof} По Замечанию \ref{pruzhinka} при $n,y\in\mathbb{N}_0$

$$\sum_{x\in\mathbb{YF}_n} T_{w,\beta,n}(x,y)=\left(\sum_{x\in K(n,y)} T_{w,\beta,n}(x,y)\right)+\left(\sum_{x\in \overline{K}(n,y)} T_{w,\beta,n}(x,y)\right)=$$
$$= \text{(По определению функции $T$)}=$$
$$=\left(\sum_{x\in K(n,y)} T_{w,\beta,n}(x,y)\right)+0=\sum_{x\in K(n,y)} T_{w,\beta,n}(x,y)=$$
$$=\sum_{x\in K(n,y)}\left(d(\varepsilon,x)\cdot q(x(y))\cdot {d_1'(x'(y),w)}\cdot \beta^y\cdot \left(1-\beta^2\right)^{\#(x(y))}\right).$$

Ясно, что в каждом слагаемом по Замечанию \ref{kerambus} при  $x\in\mathbb{YF}$, $n,y\in\mathbb{N}_0$
$$x=x(y)x'(y).$$

А значит к каждому слагаемому можно применить Утверждение \ref{razbivaem} при $x,x(y),x'(y)\in\mathbb{YF}$ и получить, что наше выражение равняется следующему:
$$\sum_{x\in K(n,y)}\left(d(\varepsilon,x'(y))\cdot  d\left(\varepsilon,x(y)1^{|x'(y)|}\right)\cdot q(x(y))\cdot {d_1'(x'(y),w)}\cdot \beta^y\cdot \left(1-\beta^2\right)^{\#(x(y))}\right)=$$
$$=(\text{По обозначению $x'(y)$})=$$
$$=\sum_{x\in K(n,y)}\left(d(\varepsilon,x'(y))\cdot  d\left(\varepsilon,x(y)1^y\right)\cdot q(x(y))\cdot {d_1'(x'(y),w)}\cdot \beta^y\cdot \left(1-\beta^2\right)^{\#(x(y))}\right).$$

Заметим, что по обозначениям при $n,y\in\mathbb{N}_0:$ $y\le n$ 
\begin{itemize}
    \item если $x\in K(n,y),$ то $x=x(y)x'(y)$, причём $x(y)\in\mathbb{YF}_{n-y},$ $ x'(y)\in\mathbb{YF}_{y}$;
    \item если $x_1,x_2\in K(n,y)$: $x_1\ne x_2$, то $x_1(y)\ne x_2(y)$ или
    $x_1'(y)\ne x'_2(y)$;
    \item если $x''\in \mathbb{YF}_{n-y},$ $ x'''\in\mathbb{YF}_{y}$, то $\left(x''x'''\right)\in K(n,y)$, $\left(x''x'''\right)(y)=x'',$ $\left(x''x'''\right)'(y)=x'''$.
\end{itemize}

А это значит, что при всех $x\in K(n,y)$, пара $(x(y),x'(y))$ принимает все возможные значения в $\mathbb{YF}_{n-y}\times \mathbb{YF}_{y}$, причём ровно по одному разу.

Таким образом, наше выражение равняется следующему:
$$\sum_{x''\in \mathbb{YF}_{n-y}}\left(\sum_{x'''\in \mathbb{YF}_{y}}\left(d(\varepsilon,x''')\cdot  d\left(\varepsilon,x''1^{y}\right)\cdot q(x'')\cdot d_1'(x''',w)\cdot \beta^y\cdot \left(1-\beta^2\right)^{\#x''}\right)\right)=$$
$$=\left(\sum_{x''\in \mathbb{YF}_{n-y}}\left(q(x'')\cdot d\left(\varepsilon,x''1^{y}\right)\cdot \beta^y\cdot \left(1-\beta^2\right)^{\#x''}\right)\right)\left(\sum_{x'''\in \mathbb{YF}_{y}}\left(d(\varepsilon,x''')\cdot d_1'(x''',w) \right)\right)=$$
$$=(\text{По Утверждению \ref{limitstrih} при $x'''\in\mathbb{YF}$, $w\in\mathbb{YF}_{\infty}$})=$$
$$=\left(\sum_{x''\in \mathbb{YF}_{n-y}}\left(q(x'')\cdot d\left(\varepsilon,x''1^{y}\right)\cdot \beta^y\cdot \left(1-\beta^2\right)^{\#x''}\right)\right)\left(\sum_{x'''\in \mathbb{YF}_{y}}\left(d(\varepsilon,x''')\cdot\lim_{m \to \infty}{\frac{d(x''',w_m)}{d(\varepsilon,w_m)}}\right)\right)=$$
$$=\left(\sum_{x''\in \mathbb{YF}_{n-y}}\left(q(x'')\cdot d\left(\varepsilon,x''1^{y}\right)\cdot\beta^y\cdot\left(1-\beta^2\right)^{\#x''}\right)\right)\lim_{m \to \infty}\left(\sum_{x'''\in \mathbb{YF}_{y}}\left(d(\varepsilon,x'''){\frac{d(x''',w_m)}{d(\varepsilon,w_m)}}\right)\right).$$

Заметим, что по Утверждению \ref{mera} при $w\in\mathbb{YF}_\infty,$ $m,y\in\mathbb{N}_0$ если $|w_m|\ge y$, то 
$$\sum_{x'''\in\mathbb{YF}_y } \left(d(\varepsilon,x''')\frac{d(x''',w_m)}{d(\varepsilon,w_m)}\right)=1.$$
А значит если $m\ge y$, то $|w_m|\ge m\ge y $, то есть
$$\sum_{x'''\in\mathbb{YF}_y } \left(d(\varepsilon,x''')\frac{d(x''',w_m)}{d(\varepsilon,w_m)}\right)=1,$$
а значит
$$\lim_{m\to\infty}\left(\sum_{x'''\in\mathbb{YF}_y } \left(d(\varepsilon,x''')\frac{d(x''',w_m)}{d(\varepsilon,w_m)}\right)\right)=1.$$
Таким образом, наше выражение равняется следующему:
$$\sum_{x''\in \mathbb{YF}_{n-y}}\left(q(x'')\cdot  d\left(\varepsilon,x''1^{y}\right)\cdot \beta^y\cdot \left(1-\beta^2\right)^{\#x''}\right)\cdot 1=$$
$$=\sum_{x''\in \mathbb{YF}_{n-y}}\left(q(x'')\cdot d\left(\varepsilon,x''1^{y}\right)\cdot \beta^y\cdot \left(1-\beta^2\right)^{\#x''}\right)=$$
$$=\sum_{x'\in \mathbb{YF}_{n-y}}\left(q(x')\cdot d\left(\varepsilon,x'1^{y}\right)\cdot \beta^y\cdot \left(1-\beta^2\right)^{\#x'}\right),$$
что и требовалось.

Утверждение доказано.
\end{proof}

\begin{Prop} \label{dostalo}
Пусть $n,y \in \mathbb{N}_0:$ $ y\le n$. Тогда
$$\sum_{x'\in\mathbb{YF}_{n-y}}\left(q(x')\cdot d(\varepsilon,x'1^y)\right)= \prod_{i=1}^{\left\lfloor \frac{n-y}{2} \right\rfloor} \frac{2i+y}{2i}.$$
\end{Prop}
\begin{proof}
Будем доказывать это Утверждение по индукции по $n$. 

\underline{\textbf{База}}: $n=0$:

В данном случае ясно, что $y=0$. А это значит, что равенство принимает следующий вид:
$$\sum_{x'\in\mathbb{YF}_{0-0}}\left(q(x')\cdot d\left(\varepsilon,x'1^0\right)\right)= \prod_{i=1}^{\left\lfloor \frac{0-0}{2} \right\rfloor} \frac{2i+0}{2i}\Longleftrightarrow\sum_{x'\in\mathbb{YF}_{0}}\left(q(x')\cdot d(\varepsilon,x')\right)= \prod_{i=1}^{0} \frac{2i}{2i}\Longleftrightarrow$$
$$\Longleftrightarrow\sum_{x'\in\{\varepsilon\}}\left(q(x')\cdot d(\varepsilon,x')\right)=1\Longleftrightarrow q(\varepsilon)\cdot d(\varepsilon,\varepsilon)=1\Longleftrightarrow$$
$$\Longleftrightarrow\frac{1}{\displaystyle\prod_{i=1}^{\#\varepsilon}|k(\varepsilon,i)|}\cdot 1
=1\Longleftrightarrow\frac{1}{\displaystyle\prod_{i=1}^{0}|k(\varepsilon,i)|}
=1\Longleftrightarrow1=1.$$

\underline{\textbf{База}} доказана.

\underline{\textbf{Переход}} к $n\in\mathbb{N}_0:$ $n\ge 1$:

Рассмотрим три случая:
\begin{enumerate}
    \item $n,y \in \mathbb{N}_0:$ $ (n-y)\ge 2$.
    
    Давайте считать:
    $$\sum_{x'\in\mathbb{YF}_{n-y}}\left(q(x')\cdot d(\varepsilon,x'1^y)\right)=$$
    $$= \left(\sum_{(1x'')\in\mathbb{YF}_{n-y} }\left(q(1x'')\cdot d(\varepsilon,1x''1^y)\right)\right)+\left(\sum_{(2x'')\in\mathbb{YF}_{n-y}}\left(q(2x'')\cdot d(\varepsilon,2x''1^y)\right)\right)=$$
    $$=\left( \sum_{x''\in\mathbb{YF}_{n-y-1} }\left(q(1x'')\cdot d(\varepsilon,1x''1^y)\right)\right)+\left(\sum_{x''\in\mathbb{YF}_{n-y-2}}\left(q(2x'')\cdot d(\varepsilon,2x''1^y)\right)\right)=$$
    \begin{center}
    $$=(\text{По Утверждению \ref{q} при $(1x''),x''\in\mathbb{YF}$, $1\in\{1,2\}$ к каждому слагаемому первой суммы}$$
    $$\text{и при  $(2x''),x''\in\mathbb{YF}$, $2\in\{1,2\}$ к каждому слагаемому второй суммы})=$$
    \end{center}
    $$= \left(\sum_{x''\in\mathbb{YF}_{n-y-1}}\left(q(x'')\frac{1}{|1x''|}d(\varepsilon,1x''1^y)\right)\right)+\left(\sum_{x''\in\mathbb{YF}_{n-y-2}}\left(q(x'')\frac{1}{|2x''|}d(\varepsilon,2x''1^y)\right)\right)=$$    
    $$= \left(\sum_{x''\in\mathbb{YF}_{n-y-1}}\left(q(x'')\frac{1}{n-y}d(\varepsilon,1x''1^y)\right)\right)+\left(\sum_{x''\in\mathbb{YF}_{n-y-2}}\left(q(x'')\frac{1}{n-y}d(\varepsilon,2x''1^y)\right)\right).$$    
    Применим Утверждение \ref{razbivaem} при $(1x''1^y),1,(x''1^y)\in\mathbb{YF}$ к каждому слагаемому первой суммы и при $(2x''1^y),2,(x''1^y)\in\mathbb{YF}$ к каждому слагаемому второй суммы  и получим, что наше выражение равняется следующему:
    $$\sum_{x''\in\mathbb{YF}_{n-y-1}}\left(q(x'')\frac{1}{n-y} d(\varepsilon,x''1^y) d\left(\varepsilon,11^{|x''1^y|}\right)\right)+$$
    $$+ \sum_{x''\in\mathbb{YF}_{n-y-2}}\left(q(x'')\frac{1}{n-y} d(\varepsilon,x''1^y) d\left(\varepsilon,21^{|x''1^y|}\right)\right)=$$
    $$= \sum_{x''\in\mathbb{YF}_{n-y-1}}\left(q(x'')\frac{1}{n-y}d(\varepsilon,x''1^y) d\left(\varepsilon,1^{1+|x''|+y}\right)\right)+$$
    $$+\sum_{x''\in\mathbb{YF}_{n-y-2}}\left(q(x'')\frac{1}{n-y}d(\varepsilon,x''1^y) d\left(\varepsilon,21^{|x''|+y}\right)\right)=$$
    $$= \sum_{x''\in\mathbb{YF}_{n-y-1}}\left(q(x'')\frac{1}{n-y}d(\varepsilon,x''1^y)\prod_{i=1}^{d\left(1^{1+|x''|+y}\right)}g\left(1^{1+|x''|+y},i\right)\right)+$$
    $$+\sum_{x''\in\mathbb{YF}_{n-y-2}}\left(q(x'')\frac{1}{n-y}d(\varepsilon,x''1^y)\prod_{i=1}^{d\left(21^{|x''|+y}\right)}g\left(21^{|x''|+y},i\right)\right)=$$
    $$= \sum_{x''\in\mathbb{YF}_{n-y-1}}\left(q(x'')\frac{1}{n-y}d(\varepsilon,x''1^y)\prod_{i=1}^{0}g\left(1^{1+|x''|+y},i\right)\right)+$$
    $$+\sum_{x''\in\mathbb{YF}_{n-y-2}}\left(q(x'')\frac{1}{n-y}d(\varepsilon,x''1^y)\prod_{i=1}^{1}g\left(21^{(n-y-2)+y},i\right)\right)=$$
    $$= \sum_{x''\in\mathbb{YF}_{n-y-1}}\left(q(x'')\frac{1}{n-y}d(\varepsilon,x''1^y)\prod_{i=1}^{0}g\left(1^{1+|x''|+y},i\right)\right)+$$
    $$+\sum_{x''\in\mathbb{YF}_{n-y-2}}\left(q(x'')\frac{1}{n-y}d(\varepsilon,x''1^y)\prod_{i=1}^{1}g\left(21^{n-2},i\right)\right)=$$
    $$=(\text{По определению функции $g$})=$$
    $$= \left(\sum_{x''\in\mathbb{YF}_{n-y-1}}\left(q(x'')\frac{1}{n-y}d(\varepsilon,x''1^y)\right)\right)+\left(\sum_{x''\in\mathbb{YF}_{n-y-2}}\left(q(x'')\frac{1}{n-y}d(\varepsilon,x''1^y)\cdot(n-1)\right)\right)=$$
    $$=\left(\frac{1}{n-y} \sum_{x''\in\mathbb{YF}_{n-y-1}}\left(q(x'')\cdot d(\varepsilon,x''1^y)\right)\right)+\left(\frac{n-1}{n-y}\sum_{x''\in\mathbb{YF}_{n-y-2}}\left(q(x'')\cdot d(\varepsilon,x''1^y)\right)\right)=$$
    $$=\text{(По предположению индукции при $n-1$ и $n-2$ (ясно, что $n-1\ge y$  и $n-2\ge y$))}=$$
    $$=\left(\frac{1}{n-y}\prod_{i=1}^{\left\lfloor \frac{n-y-1}{2} \right\rfloor} \frac{2i+y}{2i}\right)+\left(\frac{n-1}{n-y}\prod_{i=1}^{\left\lfloor \frac{n-y-2}{2} \right\rfloor} \frac{2i+y}{2i}\right).$$

Рассмотрим два подслучая:
\begin{enumerate}
    \item $(n-y) \;mod\; 2 = 0$.
    
    В данном подслучае ясно, что
    $$\left(\frac{1}{n-y}\prod_{i=1}^{\left\lfloor \frac{n-y-1}{2} \right\rfloor} \frac{2i+y}{2i}\right)+\left(\frac{n-1}{n-y}\prod_{i=1}^{\left\lfloor \frac{n-y-2}{2} \right\rfloor} \frac{2i+y}{2i}\right)=$$
    $$=\left(\frac{1}{n-y}\prod_{i=1}^{\frac{n-y}{2} -1} \frac{2i+y}{2i}\right)+\left(\frac{n-1}{n-y}\prod_{i=1}^{\frac{n-y}{2} -1} \frac{2i+y}{2i}\right)=$$
    $$=\left(\prod_{i=1}^{\frac{n-y}{2} -1} \frac{2i+y}{2i}\right)\frac{n}{n-y}=\left(\prod_{i=1}^{\frac{n-y}{2} -1} \frac{2i+y}{2i}\right)\left(\prod_{i=\frac{n-y}{2}}^{\frac{n-y}{2}} \frac{2i+y}{2i}\right)=$$
    $$=\prod_{i=1}^{ \frac{n-y}{2} } \frac{2i+y}{2i}=\prod_{i=1}^{\left\lfloor \frac{n-y}{2} \right\rfloor} \frac{2i+y}{2i}.$$
    
    Таким образом, мы поняли, что в данном случае
    $$\sum_{x'\in\mathbb{YF}_{n-y}}\left(q(x')\cdot d(\varepsilon,x'1^y)\right)=\prod_{i=1}^{\left\lfloor \frac{n-y}{2} \right\rfloor} \frac{2i+y}{2i},$$
    что и требовалось.
    
    В данном случае \underline{\textbf{Переход}} доказан.    
    
    \item $(n-y) \;mod\; 2 = 1$.
    
    В данном подслучае ясно, что
    $$\left(\frac{1}{n-y}\prod_{i=1}^{\left\lfloor \frac{n-y-1}{2} \right\rfloor} \frac{2i+y}{2i}\right)+\left(\frac{n-1}{n-y}\prod_{i=1}^{\left\lfloor \frac{n-y-2}{2} \right\rfloor} \frac{2i+y}{2i}\right)=$$
    $$=\left(\frac{1}{n-y}\prod_{i=1}^{ \frac{n-y-1}{2} } \frac{2i+y}{2i}\right)+\left(\frac{n-1}{n-y}\prod_{i=1}^{\frac{n-y-1}{2}-1 } \frac{2i+y}{2i}\right)=$$
    $$=\frac{1}{n-y} \left(\prod_{i=1}^{ \frac{n-y-1}{2} -1} \frac{2i+y}{2i}\right)\left(\prod_{i=\frac{n-y-1}{2} }^{ \frac{n-y-1}{2} } \frac{2i+y}{2i}\right)+\frac{n-1}{n-y}\left(\prod_{i=1}^{ \frac{n-y-1}{2} -1} \frac{2i+y}{2i}\right)=$$
    $$=\frac{1}{n-y} \left(\prod_{i=1}^{ \frac{n-y-1}{2} -1} \frac{2i+y}{2i}\right) \frac{n-1}{n-y-1} +\frac{n-1}{n-y}\left(\prod_{i=1}^{ \frac{n-y-1}{2} -1} \frac{2i+y}{2i}\right)=$$
    $$=\frac{n-1}{n-y}\left(\frac{1}{n-y-1}+1\right)\prod_{i=1}^{\frac{n-y-1}{2}-1} \frac{2i+y}{2i}=\frac{n-1}{n-y}\cdot\frac{n-y}{n-y-1}\prod_{i=1}^{ \frac{n-y-1}{2} -1} \frac{2i+y}{2i}=$$
    $$=\left(\prod_{i=1}^{ \frac{n-y-1}{2} -1} \frac{2i+y}{2i}\right)\frac{n-1}{n-y-1}=\left(\prod_{i=1}^{ \frac{n-y-1}{2} -1} \frac{2i+y}{2i}\right)\left(\prod_{i= \frac{n-y-1}{2}}^{ \frac{n-y-1}{2}} \frac{2i+y}{2i}\right)=$$
    $$=\prod_{i=1}^{ \frac{n-y-1}{2}} \frac{2i+y}{2i}=\prod_{i=1}^{\lfloor \frac{n-y}{2} \rfloor} \frac{2i+y}{2i}.$$
    Таким образом, мы поняли, что в данном случае
    $$\sum_{x'\in\mathbb{YF}_{n-y}}\left(q(x')\cdot d(\varepsilon,x'1^y)\right)=\prod_{i=1}^{\left\lfloor \frac{n-y}{2} \right\rfloor} \frac{2i+y}{2i},$$
    что и требовалось.
    
    В данном случае \underline{\textbf{Переход}} доказан.    
\end{enumerate}
    
    \item $(n-y)=1$.
    
    В данном случае
    $$\sum_{x'\in\mathbb{YF}_{n-y}}\left(q(x')\cdot d(\varepsilon,x'1^y)\right)= \prod_{i=1}^{\left\lfloor \frac{n-y}{2} \right\rfloor} \frac{2i+y}{2i}\Longleftrightarrow \sum_{x'\in\mathbb{YF}_{1}}\left(q(x')\cdot d(\varepsilon,x'1^y)\right)= \prod_{i=1}^{\left\lfloor \frac{1}{2} \right\rfloor} \frac{2i+y}{2i}\Longleftrightarrow$$
    $$\Longleftrightarrow \sum_{x'\in\{1\}}\left(q(x')\cdot d(\varepsilon,x'1^y)\right)= \prod_{i=1}^{0} \frac{2i+y}{2i}\Longleftrightarrow q(1)\cdot d(\varepsilon,11^{y})=1\Longleftrightarrow q(1)\cdot d(\varepsilon,1^{y+1})=1\Longleftrightarrow$$
    $$\Longleftrightarrow \frac{1}{\displaystyle\prod_{i=1}^{\#1}|k(1,i)|}\prod_{i=1}^{d(1^{y+1})}g(1^{y+1},i)=1\Longleftrightarrow \frac{1}{\displaystyle\prod_{i=1}^{1}|k(1,i)|}\prod_{i=1}^{0}g(1^{y+1},i)=1\Longleftrightarrow$$
    $$\Longleftrightarrow \frac{1}{\displaystyle|k(1,1)|}\cdot 1=1\Longleftrightarrow \frac{1}{\displaystyle|1|}\cdot 1=1\Longleftrightarrow 1 = 1.$$
    То есть в данном случае \underline{\textbf{Переход}} снова доказана. 
    
    \item $(n-y)=0$.
    
    В данном случае 
    $$\sum_{x'\in\mathbb{YF}_{n-y}}\left(q(x')\cdot d(\varepsilon,x'1^y)\right)= \prod_{i=1}^{\left\lfloor \frac{n-y}{2} \right\rfloor} \frac{2i+y}{2i}\Longleftrightarrow \sum_{x'\in\mathbb{YF}_{0}}\left(q(x')\cdot d(\varepsilon,x'1^y)\right)= \prod_{i=1}^{\left\lfloor \frac{0}{2} \right\rfloor} \frac{2i+y}{2i}\Longleftrightarrow$$
    $$\Longleftrightarrow \sum_{x'\in\{\varepsilon\}}\left(q(x')\cdot d(\varepsilon,x'1^y)\right)= \prod_{i=1}^{0} \frac{2i+y}{2i}\Longleftrightarrow q(\varepsilon)\cdot d(\varepsilon,1^{y})=1\Longleftrightarrow \frac{1}{\displaystyle\prod_{i=1}^{\#\varepsilon}|k(\varepsilon,i)|}\prod_{i=1}^{d(1^{y})}g(1^{y},i)=1\Longleftrightarrow$$
    $$\Longleftrightarrow \frac{1}{\displaystyle\prod_{i=1}^{0}|k(\varepsilon,i)|}\prod_{i=1}^{0}g(1^{y},i)=1\Longleftrightarrow 1\cdot 1=1\Longleftrightarrow 1 = 1.$$
    
   То есть в данном случае \underline{\textbf{Переход}} опять же доказан. 

\end{enumerate}
    Ясно, что все случаи разобраны, во всех \underline{\textbf{Переход}} доказан. 
    
    Утверждение доказано.

\end{proof}

\begin{Prop} \label{stolb}
Пусть $w\in\mathbb{YF}_\infty$, $\beta\in(0,1]$, $n,y\in\mathbb{N}_0:$ $y\le n$. Тогда
$$\sum_{x\in\mathbb{YF}_n} T_{w,\beta,n}(x,y)\le \left(\prod_{i=1}^{\left\lfloor \frac{n-y}{2} \right\rfloor} \frac{2i+y}{2i}\right)\beta^{y}\left(1-\beta^2\right)^{\left\lfloor \frac{n-y}{2} \right\rfloor}.$$
\end{Prop}
\begin{proof}
По Утверждению \ref{sum} при $w\in\mathbb{YF}_\infty$, $\beta\in(0,1]$, $n,y\in\mathbb{N}_0:$ $y\le n$
$$\sum_{x\in\mathbb{YF}_n} T_{w,\beta,n}(x,y)=\sum_{x'\in\mathbb{YF}_{n-y}}\left(q(x')\cdot  d(\varepsilon,x'1^y)\cdot \beta^y\cdot \left(1-\beta^2\right)^{\#x'}\right)\le$$
$$\le(\text{Так как $\beta\in(0,1]$ и если $x'\in\mathbb{YF}$, то $2\#x'\ge |x'|$})\le$$
$$\le\sum_{x'\in\mathbb{YF}_{n-y}}\left(q(x')\cdot d(\varepsilon,x'1^y)\cdot \beta^y\cdot \left(1-\beta^2\right)^{ \frac{|x'|}{2} }\right)\le$$
$$\le\sum_{x'\in\mathbb{YF}_{n-y}}\left(q(x')\cdot d(\varepsilon,x'1^y)\cdot \beta^y\cdot \left(1-\beta^2\right)^{\left\lfloor \frac{|x'|}{2} \right\rfloor}\right)=$$
$$=\sum_{x'\in\mathbb{YF}_{n-y}}\left(q(x')\cdot d(\varepsilon,x'1^y)\cdot \beta^y\cdot \left(1-\beta^2\right)^{\left\lfloor \frac{n-y}{2} \right\rfloor}\right)=$$
$$=\left(\sum_{x'\in\mathbb{YF}_{n-y}}\left(q(x')\cdot d(\varepsilon,x'1^y)\right)\right)\beta^y\left(1-\beta^2\right)^{\left\lfloor \frac{n-y}{2} \right\rfloor}.$$

Таким образом, нам достаточно доказать, что 
$$\left(\sum_{x'\in\mathbb{YF}_{n-y}}\left(q(x')\cdot d(\varepsilon,x'1^y)\right)\right)\beta^y\left(1-\beta^2\right)^{\left\lfloor \frac{n-y}{2} \right\rfloor}\le \left(\prod_{i=1}^{\left\lfloor \frac{n-y}{2} \right\rfloor} \frac{2i+y}{2i}\right)\beta^{y}\left(1-\beta^2\right)^{\left\lfloor \frac{n-y}{2} \right\rfloor} \Longleftarrow$$
$$\Longleftarrow(\text{Так как $\beta\in(0,1]$})\Longleftarrow \sum_{x'\in\mathbb{YF}_{n-y}}\left(q(x')\cdot d(\varepsilon,x'1^y)\right)\le \prod_{i=1}^{\left\lfloor \frac{n-y}{2} \right\rfloor} \frac{2i+y}{2i}.$$

А это верно по Утверждению \ref{dostalo} при наших $n,y\in\mathbb{YF}:$ $ y\le n$.

Таким образом, мы доказали, что по Утверждению \ref{dostalo} при наших $n,y\in\mathbb{YF}:$ $ y\le n$
$$ \sum_{x'\in\mathbb{YF}_{n-y}}\left(q(x')\cdot d(\varepsilon,x'1^y)\right)= \prod_{i=1}^{\left\lfloor \frac{n-y}{2} \right\rfloor} \frac{2i+y}{2i}\Longrightarrow (\text{Так как $\beta\in(0,1]$})\Longrightarrow$$
$$\Longrightarrow  \left(\sum_{x'\in\mathbb{YF}_{n-y}}\left(q(x')\cdot d(\varepsilon,x'1^y)\right)\right)\beta^y\left(1-\beta^2\right)^{\left\lfloor \frac{n-y}{2} \right\rfloor}= \left(\prod_{i=1}^{\left\lfloor \frac{n-y}{2} \right\rfloor} \frac{2i+y}{2i}\right)\beta^{y}\left(1-\beta^2\right)^{\left\lfloor \frac{n-y}{2} \right\rfloor} \Longrightarrow$$
$$\Longrightarrow(\text{Как мы уже доказали})\Longrightarrow$$
$$\Longrightarrow \sum_{x\in\mathbb{YF}_n} T_{w,\beta,n}(x,y)\le \left(\sum_{x'\in\mathbb{YF}_{n-y}}\left(q(x')\cdot d(\varepsilon,x'1^y)\right)\right)\beta^y\left(1-\beta^2\right)^{\left\lfloor \frac{n-y}{2} \right\rfloor}= \left(\prod_{i=1}^{\left\lfloor \frac{n-y}{2} \right\rfloor} \frac{2i+y}{2i}\right)\beta^{y}\left(1-\beta^2\right)^{\left\lfloor \frac{n-y}{2} \right\rfloor},$$
что и требовалось.

Утверждение доказано.
\end{proof}

\begin{Oboz}
Пусть $n,a,b\in\mathbb{N}_0:$ $a\ge 2,$ $b\in \overline{a-1}$. Тогда
$$\overline{n}(a,b):=\{c\in\overline{n}: c\;mod\;{a} = b\}.$$
\end{Oboz}

\begin{Prop} \label{lehamed}
Пусть $\beta\in(0,1)$, $n\in\mathbb{N}_0$. Тогда
$$\sum_{y=0}^{n} \left(\left(\prod_{i=1}^{\left\lfloor \frac{n-y}{2} \right\rfloor} \frac{2i+y}{2i}\right)\beta^{y}\left(1-\beta^2\right)^{\left\lfloor \frac{n-y}{2} \right\rfloor}\right)\le 1+\frac{1}{\beta}.$$
\end{Prop}
\begin{proof}
Сначала рассмотрим только чётные игреки:
$$\sum_{y\in \overline{n}(2,0)} \left(\left(\prod_{i=1}^{\left\lfloor \frac{n-y}{2} \right\rfloor} \frac{2i+y}{2i}\right)\beta^{y}\left(1-\beta^2\right)^{\left\lfloor \frac{n-y}{2} \right\rfloor}\right)=\sum_{y'=0}^{\left\lfloor\frac{n}{2}\right\rfloor} \left(\left(\prod_{i=1}^{\left\lfloor \frac{n-2y'}{2} \right\rfloor} \frac{2i+2y'}{2i}\right)\beta^{2y'}\left(1-\beta^2\right)^{\left\lfloor \frac{n-2y'}{2} \right\rfloor}\right)=$$
$$=\sum_{y'=0}^{\left\lfloor\frac{n}{2}\right\rfloor} \left(\left(\prod_{i=1}^{\left\lfloor \frac{n}{2} \right\rfloor-y'} \frac{i+y'}{i}\right)\beta^{2y'}\left(1-\beta^2\right)^{\left\lfloor \frac{n}{2} \right\rfloor-y'}\right)=\sum_{y'=0}^{\left\lfloor\frac{n}{2}\right\rfloor} \left(\left(\frac{\displaystyle\prod_{i=y'+1}^{\left\lfloor\frac{n}{2}\right\rfloor}i}{\displaystyle\prod_{i=1}^{\left\lfloor\frac{n}{2}\right\rfloor-y'}i}\right)\left(\beta^2\right)^{y'}\left(1-\beta^2\right)^{\left\lfloor \frac{n}{2} \right\rfloor-y'}\right)=$$
$$=\sum_{y'=0}^{\left\lfloor\frac{n}{2}\right\rfloor} \left(\binom{\left\lfloor\frac{n}{2}\right\rfloor}{y'}\left(\beta^2\right)^{y'}\left(1-\beta^2\right)^{\left\lfloor \frac{n}{2} \right\rfloor-y'}\right)=(\text{так как это Бином Ньютона})=\left(\beta^2+\left(1-\beta^2\right)\right)^{{\left\lfloor \frac{n}{2} \right\rfloor}}=1.$$

Теперь рассмотрим нечётные игреки:

Рассмотрим два случая:
\begin{enumerate}
    \item $n\in\mathbb{N}_0:$ $n\;mod\;{2}=0$.
    
    Посчитаем, помня, что $\beta\in(0,1)$:    
    $$\sum_{y\in\overline{n}(2,1)} \left(\left(\prod_{i=1}^{\left\lfloor \frac{n-y}{2} \right\rfloor} \frac{2i+y}{2i}\right)\beta^{y}\left(1-\beta^2\right)^{\left\lfloor \frac{n-y}{2} \right\rfloor}\right)\le $$
    $$\le\sum_{y\in\overline{n}(2,1)} \left(\left(\prod_{i=1}^{\left\lfloor \frac{n-y}{2} \right\rfloor} \frac{2i+y+1}{2i}\right)\beta^{y}\left(1-\beta^2\right)^{\left\lfloor \frac{n-y}{2} \right\rfloor}\right)=$$
    $$=\left(\text{Так как ясно, что если $n\; mod\;2=0$ и $y\;mod\;2=1$, то $\left\lfloor \frac{n-y}{2} \right\rfloor=\frac{n-y-1}{2}$}\right)=$$
    $$=\frac{1}{\beta}\sum_{y\in\overline{n}(2,1)} \left(\left(\prod_{i=1}^{ \frac{n-y-1}{2}} \frac{2i+y+1}{2i}\right)\beta^{y+1}\left(1-\beta^2\right)^{ \frac{n-y-1}{2}}\right)=$$
    $$=\frac{1}{\beta}\sum_{y\in\overline{n}(2,1)} \left(\left(\prod_{i=1}^{ \frac{n-2\frac{y+1}{2}}{2}} \frac{2i+2\frac{y+1}{2}}{2i}\right)\beta^{2\frac{y+1}{2}}\left(1-\beta^2\right)^{ \frac{n-2\frac{y+1}{2}}{2}}\right).$$
    Ясно, что если $y$ пробегает все значения в множестве $\overline{n}(2,1),$ при $n\in\mathbb{N}_0:$ $n\;mod\;{2}=0$, то $\frac{y+1}{2}$ пробегает все значения в множестве $\left\{1,\ldots,\frac{n}{2}\right\}$, то есть наше выражение равняется следующему: 
    $$\frac{1}{\beta}\sum_{y'=1}^{\frac{n}{2}} \left(\left(\prod_{i=1}^{ \frac{n-2y'}{2}} \frac{2i+2y'}{2i}\right)\beta^{2y'}\left(1-\beta^2\right)^{ \frac{n-2y'}{2}}\right)=$$
    $$=\frac{1}{\beta}\sum_{y'=1}^{\frac{n}{2}} \left(\left(\prod_{i=1}^{\frac{n}{2}-y' } \frac{i+y'}{i}\right)\beta^{2y'}\left(1-\beta^2\right)^{ \frac{n}{2}-y' }\right)\le \text{(Так как $\beta\in(0,1)$)} \le$$
    $$\le\frac{1}{\beta}\sum_{y'=1}^{\frac{n}{2}} \left(\left(\prod_{i=1}^{\frac{n}{2}-y' } \frac{i+y'}{i}\right)\beta^{2y'}\left(1-\beta^2\right)^{ \frac{n}{2}-y' }\right)+\frac{1}{\beta} \left(\prod_{i=1}^{\frac{n}{2} } \frac{i}{i}\right)\beta^{0}\left(1-\beta^2\right)^{ \frac{n}{2} }=$$
    $$=\frac{1}{\beta}\sum_{y'=1}^{\frac{n}{2}} \left(\left(\prod_{i=1}^{\frac{n}{2}-y' } \frac{i+y'}{i}\right)\beta^{2y'}\left(1-\beta^2\right)^{ \frac{n}{2}-y' }\right)+\frac{1}{\beta}\sum_{y'=0}^{0} \left(\left(\prod_{i=1}^{\frac{n}{2}-y' } \frac{i+y'}{i}\right)\beta^{2y'}\left(1-\beta^2\right)^{ \frac{n}{2}-y' }\right)=$$
    $$=\frac{1}{\beta}\sum_{y'=0}^{\frac{n}{2}} \left(\left(\prod_{i=1}^{\frac{n}{2}-y' } \frac{i+y'}{i}\right)\beta^{2y'}\left(1-\beta^2\right)^{ \frac{n}{2}-y' }\right)=$$
    $$=\frac{1}{\beta}\sum_{y'=0}^{\frac{n}{2}} \left(\left(\frac{\displaystyle\prod_{i=y'+1}^{\frac{n}{2}}i}{\displaystyle\prod_{i=1}^{\frac{n}{2}-y'}i}\right)\left(\beta^2\right)^{y'}\left(1-\beta^2\right)^{ \frac{n}{2}-y'}\right)=$$
    $$=\frac{1}{\beta}\sum_{y'=0}^{\frac{n}{2}} \left(\binom{\frac{n}{2}}{y'}\left(\beta^2\right)^{y'}\left(1-\beta^2\right)^{ \frac{n}{2} -y'}\right)=$$
    $$=(\text{так как это Бином Ньютона})=$$
    $$=\frac{1}{\beta}\left(\beta^2+\left(1-\beta^2\right)\right)^{ \frac{n}{2} }=\frac{1}{\beta}.$$

    В данном случае сумма по чётным игрекам равна единице, а сумма по нечётным игрекам оценивается сверху как $\frac{1}{\beta}$, а значит вся сумма не больше, чем $1+\frac{1}{\beta}$, то есть
    $$\sum_{y=0}^{n} \left(\left(\prod_{i=1}^{\left\lfloor \frac{n-y}{2} \right\rfloor} \frac{2i+y}{2i}\right)\beta^{y}\left(1-\beta^2\right)^{\left\lfloor \frac{n-y}{2} \right\rfloor}\right)=$$
    $$=\sum_{y\in\overline{n}(2,0)} \left(\left(\prod_{i=1}^{\left\lfloor \frac{n-y}{2} \right\rfloor} \frac{2i+y}{2i}\right)\beta^{y}\left(1-\beta^2\right)^{\left\lfloor \frac{n-y}{2} \right\rfloor}\right)+\sum_{y\in\overline{n}(2,1)} \left(\left(\prod_{i=1}^{\left\lfloor \frac{n-y}{2} \right\rfloor} \frac{2i+y}{2i}\right)\beta^{y}\left(1-\beta^2\right)^{\left\lfloor \frac{n-y}{2} \right\rfloor}\right)\le1+\frac{1}{\beta},$$
    что и требовалось.
    \item $n \;mod\;{2}=1$:
    
    Посчитаем, помня, что $\beta\in(0,1)$:    
    $$\sum_{y\in\overline{n}(2,1)} \left(\left(\prod_{i=1}^{\left\lfloor \frac{n-y}{2} \right\rfloor} \frac{2i+y}{2i}\right)\beta^{y}\left(1-\beta^2\right)^{\left\lfloor \frac{n-y}{2} \right\rfloor}\right)\le \sum_{y\in\overline{n}(2,1)} \left(\left(\prod_{i=1}^{\left\lfloor \frac{n-y}{2} \right\rfloor} \frac{2i+y+1}{2i}\right)\beta^{y}\left(1-\beta^2\right)^{\left\lfloor \frac{n-y}{2} \right\rfloor}\right)=$$
    $$=\left(\text{Так как ясно, что если $n\; mod\;{2}=1$ и $y\;mod\;{2}=1$, то $\left\lfloor \frac{n-y}{2} \right\rfloor=\frac{n-y}{2}$}\right)=$$
    $$=\frac{1}{\beta}\sum_{y\in\overline{n}(2,1)} \left(\left(\prod_{i=1}^{ \frac{n-y}{2}} \frac{2i+y+1}{2i}\right)\beta^{y+1}\left(1-\beta^2\right)^{ \frac{n-y}{2}}\right)=$$
    $$=\frac{1}{\beta}\sum_{y\in\overline{n}(2,1)} \left(\left(\prod_{i=1}^{ \frac{n+1-2\frac{y+1}{2}}{2}} \frac{2i+2\frac{y+1}{2}}{2i}\right)\beta^{2\frac{y+1}{2}}\left(1-\beta^2\right)^{ \frac{n+1-2\frac{y+1}{2}}{2}}\right).$$
    Ясно, что если $y$ пробегает все значения в множестве $\overline{n}(2,1)$ при $n\in\mathbb{N}_0:$ $n\;mod\;{2}=1$, то $\frac{y+1}{2}$ пробегает все значения в множестве $\left\{1,\ldots,\frac{n+1}{2}\right\}$, то есть наше выражение равняется следующему: 
    $$\frac{1}{\beta}\sum_{y'=1}^{\frac{n+1}{2}} \left(\left(\prod_{i=1}^{ \frac{n+1-2y'}{2}} \frac{2i+2y'}{2i}\right)\beta^{2y'}\left(1-\beta^2\right)^{ \frac{n+1-2y'}{2}}\right)=$$
    $$=\frac{1}{\beta}\sum_{y'=1}^{\frac{n+1}{2}} \left(\left(\prod_{i=1}^{\frac{n+1}{2}-y' } \frac{i+y'}{i}\right)\beta^{2y'}\left(1-\beta^2\right)^{ \frac{n+1}{2}-y' }\right)\le \text{(Так как $\beta\in(0,1)$)} \le$$
    $$\le\frac{1}{\beta}\sum_{y'=1}^{\frac{n+1}{2}} \left(\left(\prod_{i=1}^{\frac{n+1}{2}-y' } \frac{i+y'}{i}\right)\beta^{2y'}\left(1-\beta^2\right)^{ \frac{n+1}{2}-y' }\right)+\frac{1}{\beta} \left(\prod_{i=1}^{\frac{n+1}{2} } \frac{i}{i}\right)\beta^{0}\left(1-\beta^2\right)^{ \frac{n+1}{2} }=$$
    $$=\frac{1}{\beta}\sum_{y'=1}^{\frac{n+1}{2}} \left(\left(\prod_{i=1}^{\frac{n+1}{2}-y' } \frac{i+y'}{i}\right)\beta^{2y'}\left(1-\beta^2\right)^{ \frac{n+1}{2}-y' }\right)+\frac{1}{\beta}\sum_{y'=0}^{0} \left(\left(\prod_{i=1}^{\frac{n+1}{2}-y' } \frac{i+y'}{i}\right)\beta^{2y'}\left(1-\beta^2\right)^{ \frac{n+1}{2}-y' }\right)=$$
    $$=\frac{1}{\beta}\sum_{y'=0}^{\frac{n+1}{2}} \left(\left(\prod_{i=1}^{\frac{n+1}{2}-y' } \frac{i+y'}{i}\right)\beta^{2y'}\left(1-\beta^2\right)^{ \frac{n+1}{2}-y' }\right)=$$
    $$=\frac{1}{\beta}\sum_{y'=0}^{\frac{n+1}{2}} \left(\left(\frac{\displaystyle\prod_{i=y'+1}^{\frac{n+1}{2}}i}{\displaystyle\prod_{i=1}^{\frac{n+1}{2}-y'}i}\right)\left(\beta^2\right)^{y'}\left(1-\beta^2\right)^{ \frac{n+1}{2}-y'}\right)=$$
    $$=\frac{1}{\beta}\sum_{y'=0}^{\frac{n+1}{2}} \left(\binom{\frac{n+1}{2}}{y'}\left(\beta^2\right)^{y'}\left(1-\beta^2\right)^{ \frac{n+1}{2} -y'}\right)=$$
    $$=(\text{так как это Бином Ньютона})=$$
    $$=\frac{1}{\beta}\left(\beta^2+\left(1-\beta^2\right)\right)^{ \frac{n+1}{2} }=\frac{1}{\beta}.$$

    В данном случае сумма по чётным игрекам равна единице, а сумма по нечётным игрекам оценивается сверху как $\frac{1}{\beta}$, а значит вся сумма не больше, чем $1+\frac{1}{\beta}$, то есть
    $$\sum_{y=0}^{n} \left(\left(\prod_{i=1}^{\left\lfloor \frac{n-y}{2} \right\rfloor} \frac{2i+y}{2i}\right)\beta^{y}\left(1-\beta^2\right)^{\left\lfloor \frac{n-y}{2} \right\rfloor}\right)=$$
    $$=\sum_{y\in\overline{n}(2,0)} \left(\left(\prod_{i=1}^{\left\lfloor \frac{n-y}{2} \right\rfloor} \frac{2i+y}{2i}\right)\beta^{y}\left(1-\beta^2\right)^{\left\lfloor \frac{n-y}{2} \right\rfloor}\right)+\sum_{y\in\overline{n}(2,1)} \left(\left(\prod_{i=1}^{\left\lfloor \frac{n-y}{2} \right\rfloor} \frac{2i+y}{2i}\right)\beta^{y}\left(1-\beta^2\right)^{\left\lfloor \frac{n-y}{2} \right\rfloor}\right)\le 1+\frac{1}{\beta},$$
    что и требовалось.
   \end{enumerate}

В обоих случаях Утверждение доказано.
    \end{proof}

\begin{Col} \label{lehamed1}
Пусть $w\in\mathbb{YF}_\infty$, $\beta\in(0,1)$, $n\in\mathbb{N}_0$. Тогда
$$\sum_{y=0}^n\left(\sum_{x\in\mathbb{YF}_n} T_{w,\beta,n}(x,y)\right)\le 1+\frac{1}{\beta}.$$
\end{Col}
\begin{proof}
Возьмём Утверждение \ref{stolb} при наших $w\in\mathbb{YF}_\infty$, $\beta\in(0,1]$, $n\in\mathbb{N}_0$ и просуммируем его по $y\in\overline{n}$:
$$\sum_{y=0}^n\left(\sum_{x\in\mathbb{YF}_n} T_{w,\beta,n}(x,y)\right)\le\sum_{y=0}^{n} \left(\left(\prod_{i=1}^{\left\lfloor \frac{n-y}{2} \right\rfloor} \frac{2i+y}{2i}\right)\beta^{y}\left(1-\beta^2\right)^{\left\lfloor \frac{n-y}{2} \right\rfloor}\right)\le$$
$$\le\text{(По Утверждению \ref{lehamed} при наших $\beta\in(0,1)$ и  $n\in\mathbb{N}_0$)}\le 1+\frac{1}{\beta},$$
что и требовалось. 

Следствие доказано.
\end{proof}

\begin{Prop} \label{zabe}
Пусть $ w\in\mathbb{YF}_\infty$, $\beta\in(0,1)$, $n\in\mathbb{N}_0$, $v\in\mathbb{YF}$. Тогда
$$\mu_{w,\beta}(v) \le \sum_{y=0}^{|v|} T_{w,\beta,n}(v,y).$$

\end{Prop}
\begin{proof}
По обозначению
$$\mu_{w,\beta}(v)=d(\varepsilon,v)\cdot d'_\beta(v,w)=\text{ (По Утверждению \ref{kusok} при $w\in\mathbb{YF}_\infty,$ $v\in\mathbb{YF}$ и $\beta\in(0,1]$) }=$$
$$=d(\varepsilon,v)\sum_{i=0}^{\#v} \left(\beta^{|k(v,i)|}d_{\beta}(n(v,i))\cdot d'_1(k(v,i),w)\right).$$

Таким образом, наше неравенство равносильно следующему:
$$d(\varepsilon,v)\sum_{i=0}^{\#v} \left(\beta^{|k(v,i)|}d_{\beta}(n(v,i))\cdot d'_1(k(v,i),w)\right)\le\sum_{y=0}^{|v|} T_{w,\beta,n}(v,y).$$

По определению волшебных таблиц:
\begin{equation*}
T_{w,\beta,n}(x,y)=
 \begin{cases}
   $$\displaystyle d(\varepsilon,x)\cdot q(n(x,i))\cdot {{d_1'(k(x,i),w)}}\cdot \beta^y\cdot \left(1-\beta^2\right)^{\#(n(x,i))}$$ &\text {если $\exists i\in\overline{\#x}:\; |k(x,i)|=y$}\\
   0 &\text{иначе}
 \end{cases}.
\end{equation*}
Ясно, что для $v$ определение можно написать следующим образом:
\begin{equation*}
T_{w,\beta,n}(v,y)=
 \begin{cases}
   $$\displaystyle d(\varepsilon,v)\cdot q(n(v,i))\cdot {{d_1'(k(v,i),w)}}\cdot \beta^{|k(v,i)|}\cdot \left(1-\beta^2\right)^{\#(n(v,i))}$$ &\text {если $\exists i\in\overline{\#v}:\; |k(v,i)|=y$}\\
   0 &\text{иначе}
 \end{cases}.
\end{equation*}
По определению функции $k(v,i)$ ясно, что
\begin{itemize}
    \item если  $i\in \overline{\#v}$, то $|k(v,i)|\in\overline{|v|}$;
    \item если  $i,j\in \overline{\#v}:$ $i\ne j$, то $|k(v,i)|\ne |k(v,j)|$.
\end{itemize} 
А из этого ясно, что
$$\sum_{y=0}^{|v|}T_{w,\beta,n}(v,y)=\sum_{i=0}^{\#v}\left(  d(\varepsilon,v)\cdot q(n(v,i))\cdot {d_1'(k(v,i),w)}\cdot \beta^{|k(v,i)|}\cdot \left(1-\beta^2\right)^{\#(n(v,i))}\right).$$

Таким образом, наше неравенство равносильно следующему:
$$ d(\varepsilon,v)\sum_{i=0}^{\#v}\left( \beta^{|k(v,i)|}d_{\beta}(n(v,i))\cdot d'_1(k(v,i),w)\right)\le$$
$$\le\sum_{i=0}^{\#v}\left(  d(\varepsilon,v)\cdot  q(n(v,i))\cdot {d_1'(k(v,i),w)}\cdot \beta^{|k(v,i)|}\cdot \left(1-\beta^2\right)^{\#(n(v,i))}\right)\Longleftrightarrow$$
$$\Longleftrightarrow \sum_{i=0}^{\#v}\left( \beta^{|k(v,i)|}d_{\beta}(n(v,i))\cdot d'_1(k(v,i),w)\right)\le\sum_{i=0}^{\#v}\left( q(n(v,i))\cdot {d_1'(k(v,i),w)}\cdot \beta^{|k(v,i)|}\cdot \left(1-\beta^2\right)^{\#(n(v,i))}\right).$$

Ясно, что чтобы доказать это неравенство, достаточно доказать, что $\forall i\in\overline{\#v}$
$$ \beta^{|k(v,i)|}d_{\beta}(n(v,i))\cdot d'_1(k(v,i),w)\le q(n(v,i))\cdot {d_1'(k(v,i),w)}\cdot \beta^{|k(v,i)|}\cdot \left(1-\beta^2\right)^{\#(n(v,i))}.$$

Давайте докажем. Пусть $i\in\overline{\#v}$. Тогда
$$ \beta^{|k(v,i)|}d_{\beta}(n(v,i))\cdot d'_1(k(v,i),w)\le q(n(v,i))\cdot {d_1'(k(v,i),w)}\cdot \beta^{|k(v,i)|}\cdot \left(1-\beta^2\right)^{\#(n(v,i))}\Longleftrightarrow$$
$$\Longleftrightarrow\text{(Так как $\beta\in(0,1)$)}\Longleftrightarrow$$
$$\Longleftrightarrow d_{\beta}(n(v,i))\cdot d'_1(k(v,i),w)\le q(n(v,i))\cdot {d_1'(k(v,i),w)}\cdot \left(1-\beta^2\right)^{\#(n(v,i))}\Longleftarrow$$
$$\Longleftarrow(\text{По Следствию \ref{neo} при $k(v,i)\in\mathbb{YF},$ $w\in\mathbb{YF}_\infty$})\Longleftarrow$$
$$\Longleftarrow d_{\beta}(n(v,i))\le q(n(v,i))\cdot \left(1-\beta^2\right)^{\#(n(v,i))},$$
а это в точности Утверждение \ref{mamka2} при $n(v,i)\in\mathbb{YF},$ $\beta\in(0,1)$.

То есть мы доказали, что $\forall i\in\overline{\#v}$
$$ \beta^{|k(v,i)|}d_{\beta}(n(v,i))\cdot d'_1(k(v,i),w)\le q(n(v,i))\cdot {d_1'(k(v,i),w)}\cdot \beta^{|k(v,i)|}\cdot \left(1-\beta^2\right)^{\#(n(v,i))},$$
а значит
$$\sum_{i=0}^{\#v}\left( \beta^{|k(v,i)|}d_{\beta}(n(v,i))\cdot d'_1(k(v,i),w)\right)\le\sum_{i=0}^{\#v}\left( q(n(v,i))\cdot {d_1'(k(v,i),w)}\cdot \beta^{|k(v,i)|}\cdot \left(1-\beta^2\right)^{\#(n(v,i))}\right)\Longleftrightarrow$$
$$\Longleftrightarrow d(\varepsilon,v)\sum_{i=0}^{\#v}\left( \beta^{|k(v,i)|}d_{\beta}(n(v,i))\cdot d'_1(k(v,i),w)\right)\le$$
$$\le\sum_{i=0}^{\#v}\left(  d(\varepsilon,v)\cdot  q(n(v,i))\cdot {d_1'(k(v,i),w)}\cdot \beta^{|k(v,i)|}\cdot \left(1-\beta^2\right)^{\#(n(v,i))}\right)\Longleftrightarrow$$
$$\Longleftrightarrow \mu_{w,\beta}(v) \le \sum_{y=0}^{|v|} T_{w,\beta,n}(v,y),$$
что и требовалось.

Утверждение доказано.
\end{proof}

\newpage
\section{Доказательство первой ключевой теоремы}

\begin{Oboz}
Пусть $w\in\mathbb{YF}_\infty$, $n,l\in\mathbb{N}_0$. Тогда
\begin{itemize}
    \item $$Q(w,n,l):=\left\{v\in\mathbb{YF}_n:\; {h'(v,w)}\ge l\right\};$$
    \item $$ \overline{Q}(w,n,l):=\left\{v\in\mathbb{YF}_n:\; {h'(v,w)}< l\right\}.$$
\end{itemize}
\end{Oboz}

\renewcommand{\labelenumi}{\arabic{enumi}$)$}
\renewcommand{\labelenumii}{\arabic{enumi}.\arabic{enumii}$^\circ$}
\renewcommand{\labelenumiii}{\arabic{enumi}.\arabic{enumii}.\arabic{enumiii}$^\circ$}

\begin{theorem} [Следствие 3\cite{Evtuh3}] \label{main1}
Пусть $w\in\mathbb{YF}_\infty^+$, $l\in\mathbb{N}_0$. Тогда
\begin{enumerate}
    \item $$ \lim_{n \to \infty}{\sum_{v\in \overline{Q}(w,n,l)}\mu_w(v)=0};$$
    \item $$\lim_{n \to \infty}{\sum_{v\in Q(w,n,l)}\mu_w(v)=1}.$$
\end{enumerate}

\renewcommand{\labelenumi}{\arabic{enumi}$^\circ$}
\renewcommand{\labelenumii}{\arabic{enumi}.\arabic{enumii}$^\circ$}
\renewcommand{\labelenumiii}{\arabic{enumi}.\arabic{enumii}.\arabic{enumiii}$^\circ$}

\end{theorem}

\renewcommand{\labelenumi}{\arabic{enumi}$)$}
\renewcommand{\labelenumii}{\arabic{enumi}.\arabic{enumii}$^\circ$}
\renewcommand{\labelenumiii}{\arabic{enumi}.\arabic{enumii}.\arabic{enumiii}$^\circ$}

\begin{theorem}\label{t2}
Пусть $w\in \mathbb{YF}_\infty^+$, $\beta\in(0,1)$, $l \in\mathbb{N}_0$. Тогда
\begin{enumerate}
    \item $$ \lim_{n \to \infty}{\sum_{v\in \overline{Q}(w,n,l)}\mu_{w,\beta}(v)=0};$$
    \item $$\lim_{n \to \infty}{\sum_{v\in Q(w,n,l)}\mu_{w,\beta}(v)=1}.$$
\end{enumerate}
\end{theorem}
\renewcommand{\labelenumi}{\arabic{enumi}$^\circ$}
\renewcommand{\labelenumii}{\arabic{enumi}.\arabic{enumii}$^\circ$}
\renewcommand{\labelenumiii}{\arabic{enumi}.\arabic{enumii}.\arabic{enumiii}$^\circ$}

\begin{proof}

\begin{Prop} \label{kerambit}
Пусть $w\in\mathbb{YF}_\infty,$ $\beta\in(0,1]$, $n,y,l\in\mathbb{N}_0:$ $n\ge y \ge 2l$. Тогда
$${\displaystyle\sum_{v\in \overline{Q}(w,n,l)} T_{w,\beta,n}(v,y)}=\left({\displaystyle\sum_{v\in\mathbb{YF}_n} T_{w,\beta,n}(v,y)}\right)\left({\sum_{v \in \overline{Q}(w,y,l)}\mu_{w}(v)}\right).$$
\end{Prop}
\begin{proof}По Замечанию \ref{pruzhinka} при $n,y\in\mathbb{N}_0 $
$${\displaystyle\sum_{v\in \overline{Q}(w,n,l)} T_{w,\beta,n}(v,y)}={\displaystyle\left(\sum_{v\in \overline{Q}(w,n,l)\cap K(n,y)} T_{w,\beta,n}(v,y)\right)+\left(\sum_{v\in \overline{Q}(w,n,l)\cap \overline{K}(n,y)} T_{w,\beta,n}(v,y)\right)}=$$
$$= \text{(По определению функции $T$)}=$$
$$={\displaystyle\left(\sum_{v\in \overline{Q}(w,n,l)\cap K(n,y)} T_{w,\beta,n}(v,y)\right)+0}={\displaystyle\sum_{v\in \overline{Q}(w,n,l)\cap K(n,y)} T_{w,\beta,n}(v,y)}=$$
$$={\displaystyle\sum_{v\in \overline{Q}(w,n,l)\cap K(n,y)}\left(d(\varepsilon,v)\cdot q(v(y))\cdot {{d_1'(v'(y),w)}}\cdot \beta^y\cdot \left(1-\beta^2\right)^{\#(v(y))}\right)}.$$

Ясно, что в каждом слагаемом по Замечанию \ref{kerambus} при  $v\in\mathbb{YF}$, $n,y\in\mathbb{N}_0$
$$v=v(y)v'(y).$$

А значит можно воспользоваться Утверждением \ref{razbivaem} при $v,v(y),v'(y)\in\mathbb{YF}$ и получить, что наше выражение равняется следующему:
$${\displaystyle\sum_{v\in \overline{Q}(w,n,l)\cap K(n,y)}\left(d(\varepsilon,v'(y))\cdot d\left(\varepsilon,v(y)1^{|v'(y)|}\right)\cdot q(v(y))\cdot {{d_1'(v'(y),w)}\cdot }\beta^y\cdot \left(1-\beta^2\right)^{\#(v(y))}\right)}=$$
$$=(\text{По обозначению $v'(y)$})=$$
$$={\displaystyle\sum_{v\in \overline{Q}(w,n,l)\cap K(n,y)}\left(d(\varepsilon,v'(y))\cdot d\left(\varepsilon,v(y)1^{y}\right)\cdot q(v(y))\cdot {{d_1'(v'(y),w)}\cdot }\beta^y\cdot \left(1-\beta^2\right)^{\#(v(y))}\right)}.$$

Заметим, что при $w\in\mathbb{YF}_\infty$, $n,y,l\in\mathbb{N}_0:$ $n\ge y \ge 2l$:
\begin{itemize}
    \item если $v\in \overline{Q}(w,n,l)\cap K(n,y),$ то $v=v(y)v'(y)$, причём $v(y)\in\mathbb{YF}_{n-y},$ $v'(y)\in\overline{Q}(w,y,l)$ (так как $y\ge 2l$, а значит $h(w,v)=h(w,v'(y))$);
    \item если $v_1,v_2\in \overline{Q}(w,n,l)\cap K(n,y)$: $v_1\ne v_2$, то $v_1(y)\ne v_2(y)$ или
    $v'_1(y)\ne v'_2(y)$;
    \item если $v''\in \mathbb{YF}_{n-y},$ $ v'''\in \overline{Q}(w,y,l)$, то $\left(v''v'''\right)\in \overline{Q}(w,n,l)\cap K(n,y)$, $\left(v''v'''\right)(y)=v'',$ $\left(v''v'''\right)'(y)=v'''$ (так как $y\ge 2l$, а значит $h(w,v''')=h(w,v''v''')$).
\end{itemize}

А это значит, что при всех $v\in \overline{Q}(w,n,l)\cap K(n,y)$, пара $(v(y),v'(y))$ принимает все значения в $\mathbb{YF}_{n-y}\times \overline{Q}(w,y,l)$, причём ровно по одному разу.

А это, в свою очередь, значит, что наше выражение равняется следующему:
$${\displaystyle\sum_{v''\in \mathbb{YF}_{n-y}}\left(\sum_{v'''\in \overline{Q}(w,y,l)}\left(d(\varepsilon,v''')\cdot  d\left(\varepsilon,v''1^{y}\right)\cdot q(v'')\cdot {{d_1'(v''',w)}}\cdot \beta^y\cdot \left(1-\beta^2\right)^{\#v''}\right)\right)}=$$
$$={\displaystyle \left(\sum_{v''\in \mathbb{YF}_{n-y}}\left(q(v'')\cdot d\left(\varepsilon,v''1^{y}\right)\cdot \beta^y\cdot \left(1-\beta^2\right)^{\#v''}\right)\right)\sum_{v'''\in \overline{Q}(w,y,l)}\left(d(\varepsilon,v''')\cdot {{d_1'(v''',w)}}\right)}=$$
$$=(\text{По Утверждению \ref{limitstrih} при $v'''\in\mathbb{YF}$, $w\in\mathbb{YF}_{\infty}$})=$$
$$={\displaystyle \left(\sum_{v''\in \mathbb{YF}_{n-y}}\left(q(v'')\cdot d\left(\varepsilon,v''1^{y}\right)\cdot \beta^y\cdot \left(1-\beta^2\right)^{\#v''}\right)\right)\sum_{v'''\in \overline{Q}(w,y,l)}\left(d(\varepsilon,v''')\lim_{m \to \infty}{\frac{d(v''',w_m)}{d(\varepsilon,w_m)}}\right)}=$$
$$=\left(\text{По Утверждению \ref{sum} при $w\in\mathbb{YF}_\infty$, $\beta\in(0,1]$, $n,y\in\mathbb{N}_0$}\right)=$$
$$={\displaystyle \left({\displaystyle\sum_{x\in\mathbb{YF}_n} T_{w,\beta,n}(x,y)}\right)\sum_{v'''\in \overline{Q}(w,y,l)}\left(\lim_{m \to \infty}{\frac{d(\varepsilon,v''')d(v''',w_m)}{d(\varepsilon,w_m)}}\right)}=\left(\text{По обозначению}\right)=$$
$$=\left({\displaystyle\sum_{x\in\mathbb{YF}_n} T_{w,\beta,n}(x,y)}\right)\left({\displaystyle \sum_{v'''\in \overline{Q}(w,y,l)}\mu_w(v''')}\right)=\left({\displaystyle\sum_{v\in\mathbb{YF}_n} T_{w,\beta,n}(v,y)}\right)\left({\displaystyle \sum_{v\in \overline{Q}(w,y,l)}\mu_w(v)}\right),$$
что и требовалось.

Утверждение доказано.

\end{proof}

\begin{Lemma} \label{deepexsense}
Пусть $w\in\mathbb{YF}_\infty^+$, $\beta\in(0,1)$, $l\in\mathbb{N}_0$, $\varepsilon\in\mathbb{R}_{>0}$. Тогда $\exists Y\in\mathbb{N}_0:$ $ Y\ge 1,$ $\forall n\in\mathbb{N}_0:$ $n\ge Y$ 
$${\sum_{y= Y}^{n}\left({\sum_{v\in \overline{Q}(w,y,l)} T_{w,\beta,n}(v,y)}\right) }<\varepsilon.$$

\end{Lemma}
\begin{proof}
По Утверждению \ref{kerambit} при $w\in\mathbb{YF}_\infty,$ $\beta\in(0,1]$, $n,y,l\in\mathbb{N}_0:$ $n\ge y\ge 2l$
$${\displaystyle\sum_{v\in \overline{Q}(w,n,l)} T_{w,\beta,n}(v,y)}=\left({\displaystyle\sum_{v\in\mathbb{YF}_n} T_{w,\beta,n}(v,y)}\right)\left({\sum_{v \in \overline{Q}(w,y,l)}\mu_{w}(v)}\right),$$
а значит если $n,Y\in\mathbb{N}_0$: $n\ge Y\ge 2l$, то мы можем просуммировать данное выражение по $y\in\overline{Y,n}$. Просуммируем:
$$\sum_{y=Y}^n\left({\sum_{v\in \overline{Q}(w,n,l)} T_{w,\beta,n}(v,y)}\right) =\sum_{y=Y}^n\left(\left({\displaystyle\sum_{v\in\mathbb{YF}_n} T_{w,\beta,n}(v,y)}\right)\left({\sum_{v \in \overline{Q}(w,y,l)}\mu_{w}(v)}\right)\right).$$

По Теореме \ref{main1} при $w\in\mathbb{YF}_\infty^+$ и $l\in\mathbb{N}_0$
$$ \lim_{y \to \infty}{\sum_{v\in \overline{Q}(w,y,l)}\mu_w(v)=0},$$
то есть по определению предела для $\varepsilon'=\frac{\varepsilon}{\left(1+\frac{1}{\beta}\right)}$ $\exists Y\in\mathbb{N}_0:$ при $y\in\mathbb{N}_0:$ $ y\ge Y$
$$ {\sum_{v\in \overline{Q}(w,y,l)}\mu_w(v)}<\varepsilon',$$
а из этого ясно, что для $\varepsilon'=\frac{\varepsilon}{\left(1+\frac{1}{\beta}\right)}$ $\exists Y\in\mathbb{N}_0:$ $Y\ge \max(1,2l)$ и  при $y\in\mathbb{N}_0:$ $ y\ge Y$ 
$$ {\sum_{v\in \overline{Q}(w,y,l)}\mu_w(v)}<\varepsilon'.$$

Зафиксируем данный $Y$. Ясно, что $Y\ge 1$. Кроме того, как мы уже поняли, $\forall n\in\mathbb{N}_0:$ $ n\ge Y$ 
$$\sum_{y=Y}^n\left({\sum_{v\in \overline{Q}(w,n,l)} T_{w,\beta,n}(v,y)}\right) =\sum_{y=Y}^n\left(\left({\displaystyle\sum_{v\in\mathbb{YF}_n} T_{w,\beta,n}(v,y)}\right)\left({\sum_{v \in \overline{Q}(w,y,l)}\mu_{w}(v)}\right)\right)< $$
$$<\text{(так как в каждом большом слагаемом $y\ge Y$)}<$$
$$<\sum_{y=Y}^n\left(\left({\sum_{v\in\mathbb{YF}_n} T_{w,\beta,n}(v,y)}\right)    \varepsilon'\right)=\varepsilon'\sum_{y=Y}^n\left({\sum_{v\in\mathbb{YF}_n} T_{w,\beta,n}(v,y)}    \right)\le$$
$$\le\text{(так как функция $T$ неотрицательна)}\le$$
$$\le\varepsilon'\sum_{y=0}^n\left({\sum_{v\in\mathbb{YF}_n} T_{w,\beta,n}(v,y)}    \right)\le \text{ (По Следствию \ref{lehamed1} при $w\in\mathbb{YF}_\infty$, $\beta\in(0,1)$, $n\in\mathbb{N}_0$) } \le$$
$$\le\varepsilon'\left(1+\frac{1}{\beta}\right)=\frac{\varepsilon}{\left(1+\frac{1}{\beta}\right)}\left(1+\frac{1}{\beta}\right)=\varepsilon,$$
что и требовалось.

Лемма доказана.
\end{proof}

\begin{Oboz}
Пусть $a,b\in\mathbb{N}_0:$ $a\ge 2,$ $b\in \overline{a-1}$, $f(n):\mathbb{N}\to \mathbb{R}$ -- последовательность вещественных чисел.
Тогда 
$$\lim_{n\to\infty}^{(a,b)}f(n):=\lim_{n\to\infty}f(an+b)$$
\end{Oboz}

\begin{Lemma} \label{piem}
Пусть $w\in\mathbb{YF}_\infty$, $\beta\in(0,1)$, $l,Y\in\mathbb{N}_0:$ $Y\ge 1$. Тогда
$${\sum_{y= 0}^{Y-1}\left({\sum_{v\in \overline{Q}(w,n,l)} T_{w,\beta,n}(v,y)}\right) \xrightarrow{n\to \infty}0}.$$
\end{Lemma}
\begin{proof}

Давайте во всём доказательстве рассматривать только $n\in\mathbb{N}_0:$ $n\ge Y$ (ясно, что так можно).

Заметим, что функция $T$ неотрицательна, а также то, что по обозначению $\forall w\in\mathbb{YF}_\infty$, $n,l\in\mathbb{N}_0$ 
$$\overline{Q}(w,n,l)\subseteq \mathbb{YF}_n.$$

Это значит, что
$$0\le {\sum_{y= 0}^{Y-1}\left({\sum_{v\in \overline{Q}(w,n,l)} T_{w,\beta,n}(v,y)}\right)}\le{\sum_{y= 0}^{Y-1}\left({\sum_{x\in\mathbb{YF}_n} T_{w,\beta,n}(x,y)}\right)}. $$

Давайте оценивать сверху правую часть данного неравенства. Для начала оценим чётные игреки. Если $n,y\in\mathbb{N}_0:$ $n \ge Y > y,$ то по Утверждению \ref{stolb} при $w\in\mathbb{YF}_\infty$, $\beta\in(0,1]$, $n,y\in\mathbb{N}_0$
$$\sum_{x\in\mathbb{YF}_n} T_{w,\beta,n}(x,y)\le \left(\prod_{i=1}^{\left\lfloor \frac{n-y}{2} \right\rfloor} \frac{2i+y}{2i}\right)\beta^{y}\left(1-\beta^2\right)^{\left\lfloor \frac{n-y}{2} \right\rfloor}.$$
Ясно, что мы можем просуммировать это выражение по $y\in\overline{Y-1}(2,0)$. Просуммируем:
$$\sum_{y\in\overline{Y-1}(2,0)}\left(\sum_{x\in\mathbb{YF}_n} T_{w,\beta,n}(x,y)\right)\le\sum_{y\in\overline{Y-1}(2,0)}\left(\left( \prod_{i=1}^{\left\lfloor \frac{n-y}{2} \right\rfloor} \frac{2i+y}{2i}\right)\beta^{y}\left(1-\beta^2\right)^{\left\lfloor \frac{n-y}{2} \right\rfloor}\right)=$$
$$=\sum_{y'= 0}^{\left\lfloor\frac{Y-1}{2} \right\rfloor}\left(\left( \prod_{i=1}^{\left\lfloor \frac{n-2y'}{2} \right\rfloor} \frac{2i+2y'}{2i}\right)\beta^{2y'}\left(1-\beta^2\right)^{\left\lfloor \frac{n-2y'}{2} \right\rfloor}\right)=\sum_{y'= 0}^{\left\lfloor\frac{Y-1}{2} \right\rfloor}\left(\left( \prod_{i=1}^{\left\lfloor \frac{n}{2} \right\rfloor-y'} \frac{i+y'}{i}\right)\beta^{2y'}\left(1-\beta^2\right)^{\left\lfloor \frac{n}{2} \right\rfloor-y'}\right)=$$
$$=\sum_{y'=0}^{\left\lfloor\frac{Y-1}{2}\right\rfloor} \left(\left(\frac{\displaystyle\prod_{i=y'+1}^{\left\lfloor\frac{n}{2}\right\rfloor}i}{\displaystyle\prod_{i=1}^{\left\lfloor\frac{n}{2}\right\rfloor-y'}i}\right)\left(\beta^2\right)^{y'}\left(1-\beta^2\right)^{\left\lfloor \frac{n}{2} \right\rfloor-y'}\right)=\sum_{y'=0}^{\left\lfloor\frac{Y-1}{2}\right\rfloor} \left(\binom{\left\lfloor\frac{n}{2}\right\rfloor}{y'}\left(\beta^2\right)^{y'}\left(1-\beta^2\right)^{\left\lfloor \frac{n}{2} \right\rfloor-y'}\right).$$

Ясно, что 
$$\left\lfloor\frac{n}{2}\right\rfloor\xrightarrow{n\to\infty}\infty,$$
а значит
$$\lim_{n\to\infty}\left(\sum_{y'=0}^{\left\lfloor\frac{Y-1}{2}\right\rfloor} \left(\binom{\left\lfloor\frac{n}{2}\right\rfloor}{y'}\left(\beta^2\right)^{y'}\left(1-\beta^2\right)^{\left\lfloor \frac{n}{2} \right\rfloor-y'}\right)\right)=\lim_{n'\to\infty}\left(\sum_{y'=0}^{\left\lfloor\frac{Y-1}{2}\right\rfloor} \left(\binom{n'}{y'}\left(\beta^2\right)^{y'}\left(1-\beta^2\right)^{n'-y'}\right)\right)=0.$$
(Этот предел действительно равен нулю при $\beta\in(0,1)$ по закону распределения биномиальных коэффициентов).

В силу доказанного выше, а также неотрицительности функции $T$, ясно, что при $n\in\mathbb{N}_0$
$$0\le \sum_{{y\in\overline{Y-1}(2,0)}}\left(\sum_{x\in\mathbb{YF}_n} T_{w,\beta,n}(x,y)\right)\le \sum_{y'=0}^{\left\lfloor\frac{Y-1}{2}\right\rfloor} \left(\binom{\left\lfloor\frac{n}{2}\right\rfloor}{y'}\left(\beta^2\right)^{y'}\left(1-\beta^2\right)^{\left\lfloor \frac{n}{2} \right\rfloor-y'}\right),$$
а значит, по Лемме о двух полицейских,
$$\lim_{n\to\infty}\left(\sum_{{y\in\overline{Y-1}(2,0)}}\left(\sum_{x\in\mathbb{YF}_n} T_{w,\beta,n}(x,y)\right)\right)= 0.$$

Теперь будем оценивать нечётные игреки. Рассмотрим две подпоследовательности:

\renewcommand{\labelenumi}{\alph{enumi}$)$}
\renewcommand{\labelenumii}{\arabic{enumi}.\arabic{enumii}$^\circ$}
\renewcommand{\labelenumiii}{\arabic{enumi}.\arabic{enumii}.\arabic{enumiii}$^\circ$}

\begin{enumerate}
    \item Подпоследовательность $n\in\mathbb{N}_0:n\;mod\;{2}=0$.
    
    Зафиксируем какое-то $n\in\mathbb{N}_0:n\;mod\;{2}=0$.
    
    Если $n,y\in\mathbb{N}_0:$ $n \ge  Y > y,$ то по Утверждению \ref{stolb} при $w\in\mathbb{YF}_\infty$, $\beta\in(0,1]$, $n,y\in\mathbb{N}_0$
    $$\sum_{x\in\mathbb{YF}_n} T_{w,\beta,n}(x,y)\le \left(\prod_{i=1}^{\left\lfloor \frac{n-y}{2} \right\rfloor} \frac{2i+y}{2i}\right)\beta^{y}\left(1-\beta^2\right)^{\left\lfloor \frac{n-y}{2} 
    \right\rfloor}.$$
    
    Ясно, что мы можем просуммировать это выражение по $y\in\overline{Y-1}(2,1)$. Просуммируем и посчитаем, помня, что $\beta\in(0,1)$:
    $$\sum_{y\in\overline{Y-1}(2,1)}\left(\sum_{x\in\mathbb{YF}_n} T_{w,\beta,n}(x,y)\right)\le$$
    $$\le\sum_{y\in\overline{Y-1}(2,1)}\left(\left( \prod_{i=1}^{\left\lfloor \frac{n-y}{2} \right\rfloor} \frac{2i+y}{2i}\right)\beta^{y}\left(1-\beta^2\right)^{\left\lfloor \frac{n-y}{2} \right\rfloor}\right)\le$$
    $$\le \sum_{y\in\overline{Y-1}(2,1)} \left(\left(\prod_{i=1}^{\left\lfloor \frac{n-y}{2} \right\rfloor} \frac{2i+y+1}{2i}\right)\beta^{y}\left(1-\beta^2\right)^{\left\lfloor \frac{n-y}{2} \right\rfloor}\right)=$$
    $$=\left(\text{Так как ясно, что если $n \;mod\;{2}=0$ и $y\;mod\;{2}=1$, то $\left\lfloor \frac{n-y}{2} \right\rfloor=\frac{n-y-1}{2}$}\right)=$$
    $$=\frac{1}{\beta}\sum_{y\in\overline{Y-1}(2,1)} \left(\left(\prod_{i=1}^{ \frac{n-y-1}{2}} \frac{2i+y+1}{2i}\right)\beta^{y+1}\left(1-\beta^2\right)^{ \frac{n-y-1}{2}}\right)=$$
    $$=\frac{1}{\beta}\sum_{y\in\overline{Y-1}(2,1)} \left(\left(\prod_{i=1}^{ \frac{n-2\frac{y+1}{2}}{2}} \frac{2i+2\frac{y+1}{2}}{2i}\right)\beta^{2\frac{y+1}{2}}\left(1-\beta^2\right)^{ \frac{n-2\frac{y+1}{2}}{2}}\right).$$
    
    Ясно, что если $y$ пробегает все значения в множестве $\overline{Y-1}(2,1)$ при $Y\in\mathbb{N}_0:$ $Y\ge 1$, то $\frac{y+1}{2}$ пробегает все значения в множестве $\left\{1,\ldots,\left\lfloor\frac{Y}{2}\right\rfloor\right\}$, то есть наше выражение равняется следующему: 
    $$\frac{1}{\beta}\sum_{y'=1}^{\left\lfloor\frac{Y}{2}\right\rfloor} \left(\left(\prod_{i=1}^{ \frac{n-2y'}{2}} \frac{2i+2y'}{2i}\right)\beta^{2y'}\left(1-\beta^2\right)^{ \frac{n-2y'}{2}}\right)=$$
    $$=\frac{1}{\beta}\sum_{y'=1}^{\left\lfloor\frac{Y}{2}\right\rfloor} \left(\left(\prod_{i=1}^{\frac{n}{2}-y' } \frac{i+y'}{i}\right)\beta^{2y'}\left(1-\beta^2\right)^{ \frac{n}{2}-y' }\right)\le \text{(Так как $\beta\in(0,1)$)} \le$$
    $$\le\frac{1}{\beta}\sum_{y'=1}^{\left\lfloor\frac{Y}{2}\right\rfloor} \left(\left(\prod_{i=1}^{\frac{n}{2}-y' } \frac{i+y'}{i}\right)\beta^{2y'}\left(1-\beta^2\right)^{ \frac{n}{2}-y' }\right)+\frac{1}{\beta} \left(\prod_{i=1}^{\frac{n}{2} } \frac{i}{i}\right)\beta^{0}\left(1-\beta^2\right)^{ \frac{n}{2} }=$$
    $$=\frac{1}{\beta}\sum_{y'=1}^{\left\lfloor\frac{Y}{2}\right\rfloor} \left(\left(\prod_{i=1}^{\frac{n}{2}-y' } \frac{i+y'}{i}\right)\beta^{2y'}\left(1-\beta^2\right)^{ \frac{n}{2}-y' }\right)+\frac{1}{\beta}\sum_{y'=0}^{0} \left(\left(\prod_{i=1}^{\frac{n}{2}-y' } \frac{i+y'}{i}\right)\beta^{2y'}\left(1-\beta^2\right)^{ \frac{n}{2}-y' }\right)=$$
    $$=\frac{1}{\beta}\sum_{y'=0}^{\left\lfloor\frac{Y}{2}\right\rfloor} \left(\left(\prod_{i=1}^{\frac{n}{2}-y' } \frac{i+y'}{i}\right)\beta^{2y'}\left(1-\beta^2\right)^{ \frac{n}{2}-y' }\right)=$$
    $$=\frac{1}{\beta}\sum_{y'=0}^{\left\lfloor\frac{Y}{2}\right\rfloor} \left(\left(\frac{\displaystyle\prod_{i=y'+1}^{\frac{n}{2}}i}{\displaystyle\prod_{i=1}^{\frac{n}{2}-y'}i}\right)\left(\beta^2\right)^{y'}\left(1-\beta^2\right)^{ \frac{n}{2}-y'}\right)=$$
    $$=\frac{1}{\beta}\sum_{y'=0}^{\left\lfloor\frac{Y}{2}\right\rfloor} \left(\binom{\frac{n}{2}}{y'}\left(\beta^2\right)^{y'}\left(1-\beta^2\right)^{ \frac{n}{2} -y'}\right).$$
    
    Ясно, что
    
    $$\lim_{n\to\infty}^{(2,0)}\left(\frac{1}{\beta}\sum_{y'=0}^{\left\lfloor\frac{Y}{2}\right\rfloor} \left(\binom{\frac{n}{2}}{y'}\left(\beta^2\right)^{y'}\left(1-\beta^2\right)^{ \frac{n}{2} -y'}\right)\right)=$$
    $$=\lim_{n'\to\infty}\left(\frac{1}{\beta}\sum_{y'=0}^{\left\lfloor\frac{Y}{2}\right\rfloor} \left(\binom{n'}{y'}\left(\beta^2\right)^{y'}\left(1-\beta^2\right)^{n' -y'}\right)\right)=0.$$
    (Этот предел действительно равен нулю при $\beta\in(0,1)$ по закону распределения биномиальных коэффициентов).
    
    В силу доказанного выше, а также неотрицительности функции $T$, ясно, что при $n\in\mathbb{N}_0: n \; mod \; 2=0$
    $$0\le \sum_{{y\in\overline{Y-1}(2,1)}}\left(\sum_{x\in\mathbb{YF}_n} T_{w,\beta,n}(x,y)\right)\le\frac{1}{\beta}\sum_{y'=0}^{\left\lfloor\frac{Y}{2}\right\rfloor} \left(\binom{\frac{n}{2}}{y'}\left(\beta^2\right)^{y'}\left(1-\beta^2\right)^{ \frac{n}{2} -y'}\right),$$
    а значит, по Лемме о двух полицейских,
    $$\lim_{n\to\infty}^{(2,0)}\left(\sum_{{y\in\overline{Y-1}(2,1)}}\left(\sum_{x\in\mathbb{YF}_n} T_{w,\beta,n}(x,y)\right)\right)= 0.$$

   \item Подпоследовательность $n\in\mathbb{N}_0:n\;mod\;{2}=1$.
    
    Зафиксируем какое-то $n\in\mathbb{N}_0:n\;mod\;{2}=1$.

    Если $n,y\in\mathbb{N}_0:$ $n \ge Y > y,$ то по Утверждению \ref{stolb} при $w\in\mathbb{YF}_\infty$, $\beta\in(0,1]$, $n,y\in\mathbb{N}_0$
    $$\sum_{x\in\mathbb{YF}_n} T_{w,\beta,n}(x,y)\le \left(\prod_{i=1}^{\left\lfloor \frac{n-y}{2} \right\rfloor} \frac{2i+y}{2i}\right)\beta^{y}\left(1-\beta^2\right)^{\left\lfloor \frac{n-y}{2} \right\rfloor}.$$
    
    Ясно, что мы можем просуммировать это выражение по $y\in\overline{Y-1}(2,1)$. Просуммируем и посчитаем, помня, что $\beta\in(0,1)$:
    $$\sum_{y\in\overline{Y-1}(2,1)}\left(\sum_{x\in\mathbb{YF}_n} T_{w,\beta,n}(x,y)\right)\le$$
    $$\le\sum_{y\in\overline{Y-1}(2,1)}\left(\left( \prod_{i=1}^{\left\lfloor \frac{n-y}{2} \right\rfloor} \frac{2i+y}{2i}\right)\beta^{y}\left(1-\beta^2\right)^{\left\lfloor \frac{n-y}{2} \right\rfloor}\right)\le$$
    $$\le \sum_{y\in\overline{Y-1}(2,1)} \left(\left(\prod_{i=1}^{\left\lfloor \frac{n-y}{2} \right\rfloor} \frac{2i+y+1}{2i}\right)\beta^{y}\left(1-\beta^2\right)^{\left\lfloor \frac{n-y}{2} \right\rfloor}\right)=$$
    $$=\left(\text{Так как ясно, что если $n \;mod\;{2}=1$ и $y\;mod\;{2}=1$, то $\left\lfloor \frac{n-y}{2} \right\rfloor=\frac{n-y}{2}$}\right)=$$
    $$=\frac{1}{\beta}\sum_{y\in\overline{Y-1}(2,1)} \left(\left(\prod_{i=1}^{ \frac{n-y}{2}} \frac{2i+y+1}{2i}\right)\beta^{y+1}\left(1-\beta^2\right)^{ \frac{n-y}{2}}\right)=$$
    $$=\frac{1}{\beta}\sum_{y\in\overline{Y-1}(2,1)} \left(\left(\prod_{i=1}^{ \frac{n+1-2\frac{y+1}{2}}{2}} \frac{2i+2\frac{y+1}{2}}{2i}\right)\beta^{2\frac{y+1}{2}}\left(1-\beta^2\right)^{ \frac{n+1-2\frac{y+1}{2}}{2}}\right).$$
    
    Ясно, что если $y$ пробегает все значения в множестве $\overline{Y-1}(2,1),$ при $Y\in\mathbb{N}_0:$ $Y\ge 1$, то $\frac{y+1}{2}$ пробегает все значения в множестве $\left\{1,\ldots,\left\lfloor\frac{Y}{2}\right\rfloor\right\}$, то есть наше выражение равняется следующему: 
    $$\frac{1}{\beta}\sum_{y'=1}^{\left\lfloor\frac{Y}{2}\right\rfloor} \left(\left(\prod_{i=1}^{ \frac{n+1-2y'}{2}} \frac{2i+2y'}{2i}\right)\beta^{2y'}\left(1-\beta^2\right)^{ \frac{n+1-2y'}{2}}\right)=$$
    $$=\frac{1}{\beta}\sum_{y'=1}^{\left\lfloor\frac{Y}{2}\right\rfloor} \left(\left(\prod_{i=1}^{\frac{n+1}{2}-y' } \frac{i+y'}{i}\right)\beta^{2y'}\left(1-\beta^2\right)^{ \frac{n+1}{2}-y' }\right)\le \text{(Так как $\beta\in(0,1)$)} \le$$
    $$\le\frac{1}{\beta}\sum_{y'=1}^{\left\lfloor\frac{Y}{2}\right\rfloor} \left(\left(\prod_{i=1}^{\frac{n+1}{2}-y' } \frac{i+y'}{i}\right)\beta^{2y'}\left(1-\beta^2\right)^{ \frac{n+1}{2}-y' }\right)+\frac{1}{\beta} \left(\prod_{i=1}^{\frac{n+1}{2} } \frac{i}{i}\right)\beta^{0}\left(1-\beta^2\right)^{ \frac{n+1}{2} }=$$
    $$=\frac{1}{\beta}\sum_{y'=1}^{\left\lfloor\frac{Y}{2}\right\rfloor} \left(\left(\prod_{i=1}^{\frac{n+1}{2}-y' } \frac{i+y'}{i}\right)\beta^{2y'}\left(1-\beta^2\right)^{ \frac{n+1}{2}-y' }\right)+\frac{1}{\beta}\sum_{y'=0}^{0} \left(\left(\prod_{i=1}^{\frac{n+1}{2}-y' } \frac{i+y'}{i}\right)\beta^{2y'}\left(1-\beta^2\right)^{ \frac{n+1}{2}-y' }\right)=$$
    $$=\frac{1}{\beta}\sum_{y'=0}^{\left\lfloor\frac{Y}{2}\right\rfloor} \left(\left(\prod_{i=1}^{\frac{n+1}{2}-y' } \frac{i+y'}{i}\right)\beta^{2y'}\left(1-\beta^2\right)^{ \frac{n+1}{2}-y' }\right)=$$
    $$=\frac{1}{\beta}\sum_{y'=0}^{\left\lfloor\frac{Y}{2}\right\rfloor} \left(\left(\frac{\displaystyle\prod_{i=y'+1}^{\frac{n+1}{2}}i}{\displaystyle\prod_{i=1}^{\frac{n+1}{2}-y'}i}\right)\left(\beta^2\right)^{y'}\left(1-\beta^2\right)^{ \frac{n+1}{2}-y'}\right)=$$
    $$=\frac{1}{\beta}\sum_{y'=0}^{\left\lfloor\frac{Y}{2}\right\rfloor} \left(\binom{\frac{n+1}{2}}{y'}\left(\beta^2\right)^{y'}\left(1-\beta^2\right)^{ \frac{n+1}{2} -y'}\right).$$
    
    Ясно, что 
    $$\lim_{n\to\infty}^{(2,1)}\left(\frac{1}{\beta}\sum_{y'=0}^{\left\lfloor\frac{Y}{2}\right\rfloor} \left(\binom{\frac{n+1}{2}}{y'}\left(\beta^2\right)^{y'}\left(1-\beta^2\right)^{ \frac{n+1}{2} -y'}\right)\right)=$$
    $$=\lim_{n'\to\infty}\left(\frac{1}{\beta}\sum_{y'=0}^{\left\lfloor\frac{Y}{2}\right\rfloor} \left(\binom{n'+1}{y'}\left(\beta^2\right)^{y'}\left(1-\beta^2\right)^{n'+1 -y'}\right)\right)=$$
    $$=\lim_{n'\to\infty}\left(\frac{1}{\beta}\sum_{y'=0}^{\left\lfloor\frac{Y}{2}\right\rfloor} \left(\binom{n'}{y'}\left(\beta^2\right)^{y'}\left(1-\beta^2\right)^{n' -y'}\right)\right)=0.$$
    (Этот предел действительно равен нулю при $\beta\in(0,1)$ по закону распределения биномиальных коэффициентов).

    В силу доказанного выше, а также неотрицительности функции $T$, ясно, что при $n\in\mathbb{N}_0: n \; mod \; 2=1$
    $$0\le \sum_{{y\in\overline{Y-1}(2,1)}}\left(\sum_{x\in\mathbb{YF}_n} T_{w,\beta,n}(x,y)\right)\le\frac{1}{\beta}\sum_{y'=0}^{\left\lfloor\frac{Y}{2}\right\rfloor} \left(\binom{\frac{n+1}{2}}{y'}\left(\beta^2\right)^{y'}\left(1-\beta^2\right)^{ \frac{n+1}{2} -y'}\right),$$
    а значит, по Лемме о двух полицейских,
    $$\lim_{n\to\infty}^{(2,1)}\left(\sum_{{y\in\overline{Y-1}(2,1)}}\left(\sum_{x\in\mathbb{YF}_n} T_{w,\beta,n}(x,y)\right)\right)= 0.$$

    \end{enumerate}

Итак, подытожив написанное выше, получаем, что мы разбиваем последовательность $n\in\mathbb{N}_0$ на две подпоследовательности, а также, что при наших $w\in\mathbb{YF}_\infty$, $\beta\in(0,1)$, $l,Y\in\mathbb{N}_{0}:$ $Y\ge 1$
\begin{itemize}
\item $\forall n\in\mathbb{N}_0$
$$0\le {\sum_{y= 0}^{Y-1}\left({\sum_{v\in \overline{Q}(w,n,l)} T_{w,\beta,n}(v,y)}\right)}\le{\sum_{y= 0}^{Y-1}\left({\sum_{x\in\mathbb{YF}_n} T_{w,\beta,n}(x,y)}\right)}= $$
$$=\left(\sum_{y\in\overline{Y-1}(2,0)}\left(\sum_{x\in\mathbb{YF}_n} T_{w,\beta,n}(x,y)\right)\right)+\left(\sum_{y\in\overline{Y-1}(2,1)}\left(\sum_{x\in\mathbb{YF}_n} T_{w,\beta,n}(x,y)\right)\right);$$
\item
$$\lim_{n\to\infty}\left(\sum_{{y\in\overline{Y-1}(2,0)}}\left(\sum_{x\in\mathbb{YF}_n} T_{w,\beta,n}(x,y)\right)\right)= 0.$$
\item
$$\lim_{n\to\infty}^{(2,0)}\left(\sum_{{y\in\overline{Y-1}(2,1)}}\left(\sum_{x\in\mathbb{YF}_n} T_{w,\beta,n}(x,y)\right)\right)= 0.$$
\item
$$\lim_{n\to\infty}^{(2,1)}\left(\sum_{{y\in\overline{Y-1}(2,1)}}\left(\sum_{x\in\mathbb{YF}_n} T_{w,\beta,n}(x,y)\right)\right)= 0.$$
    
\end{itemize}
А значит, по Лемме о двух полицейских
\begin{itemize}
    \item $$\lim_{n\to\infty}^{(2,0)}\left({\sum_{y= 0}^{Y-1}\left({\sum_{v\in \overline{Q}(w,n,l)} T_{w,\beta,n}(v,y)}\right)}\right)= 0;$$
    \item $$\lim_{n\to\infty}^{(2,1)}\left({\sum_{y= 0}^{Y-1}\left({\sum_{v\in \overline{Q}(w,n,l)} T_{w,\beta,n}(v,y)}\right)}\right)= 0.$$
\end{itemize}
а из этого ясно, что если $w\in\mathbb{YF}_\infty$, $\beta\in(0,1)$, $l,Y\in\mathbb{N}_0:$ $Y\ge 1$, то
$${\sum_{y= 0}^{Y-1}\left({\sum_{v\in \overline{Q}(w,n,l)} T_{w,\beta,n}(v,y)}\right) \xrightarrow{n\to \infty}0},$$
что и требовалось.

Лемма доказана.
\end{proof}

Итак, для завершения доказательства теоремы осталось собрать всё доказанное ранее. Сейчас самое время вспомнить формулировку: 

\renewcommand{\labelenumi}{\arabic{enumi}$)$}
\renewcommand{\labelenumii}{\arabic{enumi}.\arabic{enumii}$^\circ$}
\renewcommand{\labelenumiii}{\arabic{enumi}.\arabic{enumii}.\arabic{enumiii}$^\circ$}

Пусть $w\in \mathbb{YF}_\infty^+$, $\beta\in(0,1)$, $l \in\mathbb{N}_0$. Тогда
\begin{enumerate}
    \item $$ \lim_{n \to \infty}{\sum_{v\in \overline{Q}(w,n,l)}\mu_{w,\beta}(v)=0};$$
    \item $$\lim_{n \to \infty}{\sum_{v\in Q(w,n,l)}\mu_{w,\beta}(v)=1}.$$
\end{enumerate}

Мы знаем, что при $w\in \mathbb{YF}_\infty^+$, $\beta\in(0,1)$, $l \in\mathbb{N}_0$ и $n\in\mathbb{N}_0$
$$0\le {\sum_{v\in \overline{Q}(w,n,l)}\mu_{w,\beta}(v)}\le$$
$$\le \text{ (По Утверждению \ref{zabe} при $ w\in\mathbb{YF}_\infty$, $\beta\in(0,1)$, $n\in\mathbb{N}_0$, $v\in\mathbb{YF}$ для всех $v\in \overline{Q}(w,n,l)$) } \le$$
$$\le \sum_{v\in \overline{Q}(w,n,l)}\left(\sum_{y=0}^{|v|} T_{w,\beta,n}(v,y)\right)=\left(\text{так как $\overline{Q}(w,n,l)\subseteq \mathbb{YF}_n$}\right)=$$
$$=\sum_{v\in \overline{Q}(w,n,l)}\left(\sum_{y=0}^{n} T_{w,\beta,n}(v,y)\right)=\sum_{y=0}^{n}\left(\sum_{v\in \overline{Q}(w,n,l)} T_{w,\beta,n}(v,y)\right).$$

Зафиксируем произвольный $\varepsilon\in\mathbb{R}_{>0}$.

По Лемме \ref{deepexsense} при $w\in\mathbb{YF}_\infty^+$, $\beta\in(0,1)$, $l\in\mathbb{N}_0$ и $\displaystyle\frac{\varepsilon}{2}\in\mathbb{R}_{>0}$
$\exists Y\in\mathbb{N}_0:$ $Y\ge 1$, $\forall n\ge Y$
$${\sum_{y= Y}^{n}\left({\sum_{v\in \overline{Q}(w,y,l)} T_{w,\beta,n}(v,y)}\right) }<\frac{\varepsilon}{2}.$$

Зафиксируем данный $Y\in\mathbb{N}_0:$ $Y\ge 1$.

Также из Леммы \ref{piem} при $w\in\mathbb{YF}_\infty$, $\beta\in(0,1)$, $l,Y\in\mathbb{N}_0$ следует то, что при нашем $\varepsilon\in\mathbb{R}_{>0}$ $\exists N\in\mathbb{N}_0:$ при $n\ge N$
$${\sum_{y= 0}^{Y-1}\left({\sum_{v\in \overline{Q}(w,n,l)} T_{w,\beta,n}(v,y)}\right)} <\frac{\varepsilon}{2}.$$

Теперь зафиксируем данное $N\in\mathbb{N}_0$.

Мы поняли, что при наших $w\in\mathbb{YF}_\infty^+$, $\beta\in(0,1)$, $l\in\mathbb{N}_0$ и $\varepsilon\in\mathbb{R}_{>0}$ для любого $n\ge\max\left(Y,N\right)$
$$\sum_{y=0}^{n}\left(\sum_{v\in \overline{Q}(w,n,l)} T_{w,\beta,n}(v,y)\right)=\left(\sum_{y=0}^{Y-1}\left(\sum_{v\in \overline{Q}(w,n,l)} T_{w,\beta,n}(v,y)\right)\right)+\left(\sum_{y=Y}^{n}\left(\sum_{v\in \overline{Q}(w,n,l)} T_{w,\beta,n}(v,y)\right)\right)<\varepsilon,$$
то есть в силу неотрицательности функции $T$
$$ \lim_{n \to \infty}\left(\sum_{y=0}^{n}\left(\sum_{v\in \overline{Q}(w,n,l)} T_{w,\beta,n}(v,y)\right)\right)=0.$$

Мы уже поняли, что при $w\in \mathbb{YF}_\infty^+$, $\beta\in(0,1)$, $l \in\mathbb{N}_0$ и $n\in\mathbb{N}_0$
$$0\le {\sum_{v\in \overline{Q}(w,n,l)}\mu_{w,\beta}(v)}\le \sum_{y=0}^{n}\left(\sum_{v\in \overline{Q}(w,n,l)} T_{w,\beta,n}(v,y)\right),$$
а значит, по Лемме о двух полицейских ясно, что при наших $w\in \mathbb{YF}_\infty^+$, $\beta\in(0,1)$ и $l \in\mathbb{N}_0$
$$ \lim_{n \to \infty}\sum_{v\in \overline{Q}(w,n,l)}\mu_{w,\beta}(v)=0,$$
что доказывает первый пункт.

Кроме того, ясно, что
\begin{itemize}
    \item $$\overline{Q}\left(w,n,l\right)\cup Q\left(w,n,l\right)=\left\{v\in\mathbb{YF}_n:\; {h'(v,w)}< l\right\}\cup \left\{v\in\mathbb{YF}_n:\; {h'(v,w)}\ge l\right\}=\mathbb{YF}_n;$$
    \item $$\overline{Q}\left(w,n,l\right)\cap Q\left(w,n,l\right)=\left\{v\in\mathbb{YF}_n:\; {h'(v,w)}< l\right\}\cap \left\{v\in\mathbb{YF}_n: \;{h'(v,w)}\ge l\right\}=\varnothing;$$
    \item (Следствие \ref{mera1}) $\forall w\in\mathbb{YF}_\infty^+,$ $\beta\in(0,1]$, $n\in\mathbb{N}_0$
    $${\sum_{v\in \mathbb{YF}_n}\mu_{w,\beta}(v)}=1.$$
\end{itemize}
А из этого очевидно, что
$$\lim_{n \to \infty}{\sum_{v\in Q(w,n,l)}\mu_{w,\beta}(v)}=1,$$
что доказывает второй пункт.

Таким образом, оба пункта доказаны.

Теорема доказана.
\end{proof}

\newpage
\section{Доказательство второй ключевой теоремы}

\renewcommand{\labelenumi}{\arabic{enumi}$)$}
\renewcommand{\labelenumii}{\arabic{enumi}.\arabic{enumii}$^\circ$}
\renewcommand{\labelenumiii}{\arabic{enumi}.\arabic{enumii}.\arabic{enumiii}$^\circ$}

\begin{Oboz}
Пусть $w\in\mathbb{YF}_\infty$, $n\in\mathbb{N}_0$, $\varepsilon\in\mathbb{R}_{>0}$. Тогда
\begin{itemize}
    \item $$R(w,n,\varepsilon):=\left\{v\in\mathbb{YF}_n:\; \pi(v)\in(\pi(w)(1-\varepsilon),\pi(w)(1+\varepsilon)) \right\};$$
    \item $$\overline{R}(w,n,\varepsilon):=\left\{v\in\mathbb{YF}_n:\; \pi(v)\notin(\pi(w)(1-\varepsilon),\pi(w)(1+\varepsilon)) \right\}.$$
\end{itemize}
\end{Oboz}

\begin{theorem} [Следствие 4\cite{Evtuh3}] \label{main2}
Пусть $w\in\mathbb{YF}_\infty^+$, $\varepsilon\in\mathbb{R}_{>0}$. Тогда
\begin{enumerate}
    \item $$ \lim_{n \to \infty}{\sum_{v\in \overline{R}(w,n,\varepsilon)}\mu_w(v)=0};$$
    \item $$\lim_{n \to \infty}{\sum_{v\in R(w,n,\varepsilon)}\mu_w(v)=1}.$$
\end{enumerate}
\end{theorem}

\begin{Oboz}
Пусть $w\in\mathbb{YF}_\infty$, $\beta\in(0,1]$, $n\in\mathbb{N}_0$, $\varepsilon\in\mathbb{R}_{>0}$. Тогда
\begin{itemize}
    \item $$R(w,\beta,n,\varepsilon):=\left\{v\in\mathbb{YF}_n:\; \pi(v)\in(\pi(w)(\beta-\varepsilon),\pi(w)(\beta+\varepsilon)) \right\};$$
    \item $$\overline{R}(w,\beta,n,\varepsilon):=\left\{v\in\mathbb{YF}_n:\; \pi(v)\notin(\pi(w)(\beta-\varepsilon),\pi(w)(\beta+\varepsilon)) \right\}.$$
\end{itemize}
\end{Oboz}

\begin{Zam} \label{final}
Пусть $w\in\mathbb{YF}_\infty$, $n\in\mathbb{N}_0$, $\varepsilon\in\mathbb{R}_{>0}$. Тогда
\begin{itemize}
    \item $$R(w,1,n,\varepsilon)=\left\{v\in\mathbb{YF}_n:\; \pi(v)\in(\pi(w)(1-\varepsilon),\pi(w)(1+\varepsilon)) \right\}=R(w,n,\varepsilon);$$
    \item $$\overline{R}(w,1,n,\varepsilon):=\left\{v\in\mathbb{YF}_n:\; \pi(v)\notin(\pi(w)(1-\varepsilon),\pi(w)(1+\varepsilon)) \right\}=\overline{R}(w,n,\varepsilon).$$
\end{itemize}

\end{Zam}

\begin{theorem} \label{t3}
Пусть  $w\in \mathbb{YF}_\infty^+$, $\beta \in(0,1)$, $ \varepsilon\in\mathbb{R}_{>0}$. Тогда
\begin{enumerate}
    \item $$\lim_{n \to \infty}{\sum_{v\in \overline{R}(w,\beta,n,\varepsilon)}\mu_{w,\beta}(v)}=0;$$
    \item $$\lim_{n \to \infty}{\sum_{v\in R(w,\beta,n,\varepsilon)}\mu_{w,\beta}(v)}=1.$$
\end{enumerate}
\end{theorem}

\begin{proof}

Для начала нам надо ввести очень много обозначений:

\begin{Oboz}
Пусть $n\in\mathbb{N}_0,$ $\beta \in(0,1]$, $\varepsilon\in\mathbb{R}_{>0}$, $a,b\in\mathbb{N}_0:$ $a\ge 2$, $b\in \overline{a-1}$. Тогда
\begin{itemize} 
    \item $$\overline{n}[\beta,\varepsilon]:=\overline{n}\cap\left(n\left(\beta^2-\varepsilon\right),n\left(\beta^2+\varepsilon\right)\right);$$
    \item $$\overline{n}\{\beta,\varepsilon\}:= \overline{n}\textbackslash\left(n\left(\beta^2-\varepsilon\right),n\left(\beta^2+\varepsilon\right)\right);$$
    \item $$\overline{n}[\beta,\varepsilon](a,b):=\{c\in\overline{n}[\beta,\varepsilon]:\; c \;mod\;{a} = b\};$$
    \item $$\overline{n}\{\beta,\varepsilon\}(a,b):=\{c\in\overline{n}\{\beta,\varepsilon\}:\; c \;mod\;{a} = b\};$$
    \item $$\overline{n}_{00}\{\beta,\varepsilon\}:= \overline{\frac{n}{2}}\textbackslash\left(\frac{n\left(\beta^2-\varepsilon\right)}{2},\frac{n\left(\beta^2+\varepsilon\right)}{2}\right);$$
    \item $$\overline{n}_{00,2}\{\beta,\varepsilon\}:= \overline{{n}}\textbackslash\left(n\left(\beta^2-\varepsilon\right),n\left(\beta^2+\varepsilon\right)\right);$$
    \item $$\overline{n}_{10}\{\beta,\varepsilon\}:= \overline{\frac{n-1}{2}}\textbackslash\left(\frac{n\left(\beta^2-\varepsilon\right)}{2},\frac{n\left(\beta^2+\varepsilon\right)}{2}\right);$$
    \item $$\overline{n}_{10,2}\{\beta,\varepsilon\}:= \overline{n}\textbackslash\left(\frac{(2n+1)\left(\beta^2-\varepsilon\right)}{2},\frac{(2n+1)\left(\beta^2+\varepsilon\right)}{2}\right);$$
    \item $$\overline{n}_{01}\{\beta,\varepsilon\}:= \overline{\frac{n}{2}}\textbackslash\left(\frac{n\left(\beta^2-\varepsilon\right)+1}{2},\frac{n\left(\beta^2+\varepsilon\right)+1}{2}\right);$$
    \item $$\overline{n}_{01,2}\{\beta,\varepsilon\}:= \overline{{n}}\textbackslash\left(n\left(\beta^2-\varepsilon\right)+\frac{1}{2},n\left(\beta^2+\varepsilon\right)+\frac{1}{2}\right);$$
    \item $$\overline{n}_{11}\{\beta,\varepsilon\}:= \overline{\frac{n+1}{2}}\textbackslash\left(\frac{n\left(\beta^2-\varepsilon\right)+1}{2},\frac{n\left(\beta^2+\varepsilon\right)+1}{2}\right);$$
    \item $$\overline{n}_{11,2}\{\beta,\varepsilon\}:= \overline{n}\textbackslash\left(\frac{(2n-1)\left(\beta^2-\varepsilon\right)+1}{2},\frac{(2n-1)\left(\beta^2+\varepsilon\right)+1}{2}\right).$$
    
\end{itemize}
\end{Oboz}

\begin{Zam} \label{popa}
Пусть $n\in\mathbb{N}_0$, $\beta\in(0,1]$, $\varepsilon\in\mathbb{R}_{>0}$. Тогда
\begin{itemize}
    \item 
    $$\overline{n}[\beta,\varepsilon]\cup \overline{n}\{\beta,\varepsilon\}=\overline{n};$$
    \item
    $$\overline{n}[\beta,\varepsilon]\cap \overline{n}\{\beta,\varepsilon\}=\varnothing.$$
\end{itemize}

\end{Zam}


\begin{Oboz} 
Пусть $v\in\mathbb{YF}$, $y\in\mathbb{N}_0$. Тогда
\begin{itemize}
    \item 
    \begin{equation*}
\pi_y(v):= \begin{cases}
$$\pi(v'1^{y})$$
    &\text {если $\exists v',v''\in\mathbb{YF}:$ $ v=v'v''$ и $|v''|=y$
}\\
   \text{не определено} &\text{иначе}
 \end{cases};
\end{equation*}
    \item 
    \begin{equation*}
\pi'_y(v):= \begin{cases}
$$\pi(v'')$$
    &\text {если $\exists v',v''\in\mathbb{YF}:$ $v=v'v''$ и $|v''|=y$
}\\
   \text{не определено} &\text{иначе}
 \end{cases}.
\end{equation*}
\end{itemize}
\end{Oboz}

\begin{Zam}
Пусть $v\in\mathbb{YF},$ $n,y\in\mathbb{N}_0:$ $v\in\mathbb{YF}_n$, $y\le n$.  Тогда
\begin{itemize}
    \item $$v(y) \text{ определено} \Longleftrightarrow v'(y) \text{ определено} \Longleftrightarrow \pi_y(v) \text{ определено}\Longleftrightarrow$$
    $$\Longleftrightarrow \pi'_y(v) \text{ определено}\Longleftrightarrow v\in K(n,y) ;$$ 
    \item $$v(y) \text{ не определено} \Longleftrightarrow v'(y) \text{ не определено}\Longleftrightarrow$$
    $$\Longleftrightarrow \pi_y(v) \text{ не определено}\Longleftrightarrow \pi'_y(v) \text{ не определено} \Longleftrightarrow v\in \overline{K}(n,y) .$$ 
\end{itemize}
\end{Zam}

\renewcommand{\labelenumi}{\arabic{enumi}$^\circ$}
\renewcommand{\labelenumii}{\arabic{enumi}.\arabic{enumii}$^\circ$}
\renewcommand{\labelenumiii}{\arabic{enumi}.\arabic{enumii}.\arabic{enumiii}$^\circ$}

\begin{Prop} \label{meexy} 
Пусть $v\in\mathbb{YF}$, $y\in\mathbb{N}_0:$ $\exists v',v''\in\mathbb{YF}:$ $v=v'v''$ и $|v''|=y$. Тогда
$$\pi(v)=\pi_y(v)\pi'_y(v).$$
\end{Prop}
\begin{proof}

\begin{Prop}[Утверждение 7\cite{Evtuh3}] \label{meexy1} 
Пусть $v,v',v''\in\mathbb{YF}$: $v=v'v''$. Тогда
$$\pi(v)=\pi(v'')\pi\left(v'1^{|v''|}\right).$$
\end{Prop}

По Утверждению \ref{meexy1} при $v,v',v''\in\mathbb{YF}\in\mathbb{YF}$
$$\pi(v)=\pi(v'')\pi\left(v'1^{|v''|}\right)=\pi(v'')\pi\left(v'1^{y}\right)=(\text{По обозначению})=\pi'_y(v)\pi_y(v)=\pi_y(v)\pi'_y(v).$$

\end{proof}

\begin{Oboz}
Пусть $w\in\mathbb{YF}_\infty$, $\beta\in(0,1]$, $n\in\mathbb{N}_0$, $\varepsilon\in\mathbb{R}_{>0}$, $y\in\mathbb{N}_0$, $\varepsilon_3\in\mathbb{R}_{>0}:$ $y\le n$. Тогда
\begin{itemize}
    \item $$R(w,\beta,n,\varepsilon,y,\varepsilon_3):=\left\{v\in K(n,y):\; \pi(v)\in(\pi(w)(\beta-\varepsilon),\pi(w)(\beta+\varepsilon)),\pi_y(v)\in (\beta-\varepsilon_3,\beta+\varepsilon_3) \right\};$$
    \item $$\overline{R}(w,\beta,n,\varepsilon,y,\varepsilon_3):=\left\{v\in K(n,y):\; \pi(v)\notin(\pi(w)(\beta-\varepsilon),\pi(w)(\beta+\varepsilon)),\pi_y(v)\in (\beta-\varepsilon_3,\beta+\varepsilon_3) \right\};$$
    \item $$R'(w,\beta,n,\varepsilon,y,\varepsilon_3):=\left\{v\in K(n,y):\; \pi(v)\in(\pi(w)(\beta-\varepsilon),\pi(w)(\beta+\varepsilon)),\pi_y(v)\notin (\beta-\varepsilon_3,\beta+\varepsilon_3) \right\};$$
    \item $$\overline{R'}(w,\beta,n,\varepsilon,y,\varepsilon_3):=\left\{v\in K(n,y):\; \pi(v)\notin(\pi(w)(\beta-\varepsilon),\pi(w)(\beta+\varepsilon)),\pi_y(v)\notin (\beta-\varepsilon_3,\beta+\varepsilon_3) \right\};$$
    \item $$R''(w,\beta,n,\varepsilon,y):=\left\{v\in \overline{K}(n,y):\; \pi(v)\in(\pi(w)(\beta-\varepsilon),\pi(w)(\beta+\varepsilon)) \right\};$$
    \item $$\overline{R''}(w,\beta,n,\varepsilon,y):=\left\{v\in \overline{K}(n,y):\; \pi(v)\notin(\pi(w)(\beta-\varepsilon),\pi(w)(\beta+\varepsilon))\right\}.$$
\end{itemize}
\end{Oboz}

\begin{Zam} \label{zhopa}
Пусть $w\in\mathbb{YF}_\infty$, $\beta\in(0,1]$, $n\in\mathbb{N}_0$, $\varepsilon\in\mathbb{R}_{>0}$, $y\in\mathbb{N}_0$, $\varepsilon_3\in\mathbb{R}_{>0}:$ $y\le n$. Тогда
\begin{itemize}
    \item пересечение любых двух множеств среди 
    $$R(w,\beta, n,\varepsilon,y,\varepsilon_3), \overline{R}(w,\beta,n,\varepsilon,y,\varepsilon_3), R'(w,\beta,n,\varepsilon,y,\varepsilon_3),$$
    $$ \overline{R'}(w,\beta,n,\varepsilon,y,\varepsilon_3), R''(w,\beta,n,\varepsilon,y), \overline{R''}(w,\beta,n,\varepsilon,y)$$
    пусто;
    \item 
    $$R(w,\beta, n,\varepsilon,y,\varepsilon_3)\cup \overline{R}(w,\beta,n,\varepsilon,y,\varepsilon_3)\cup R'(w,\beta,n,\varepsilon,y,\varepsilon_3)\cup$$
    $$\cup \overline{R'}(w,\beta,n,\varepsilon,y,\varepsilon_3)\cup R''(w,\beta,n,\varepsilon,y)\cup \overline{R''}(w,\beta,n,\varepsilon,y)=\mathbb{YF}_n;$$
    
    \item
    $$\overline{R}(w,\beta,n,\varepsilon,y,\varepsilon_3)\cup \overline{R'}(w,\beta,n,\varepsilon,y,\varepsilon_3)\cup \overline{R''}(w,\beta,n,\varepsilon,y)=$$
    $$=\left\{v\in K(n,y):\; \pi(v)\notin(\pi(w)(\beta-\varepsilon),\pi(w)(\beta+\varepsilon))\right\}\cup$$
    $$\cup\left\{v\in \overline{K}(n,y):\; \pi(v)\notin(\pi(w)(\beta-\varepsilon),\pi(w)(\beta+\varepsilon))\right\}=$$
    $$=\left\{v\in\mathbb{YF}_n:\; \pi(v)\notin(\pi(w)(\beta-\varepsilon),\pi(w)(\beta+\varepsilon))\right\}=\overline{R}(w,\beta,n,\varepsilon);$$
    \item
    $${R}(w,\beta,n,\varepsilon,y,\varepsilon_3)\cup {R'}(w,\beta,n,\varepsilon,y,\varepsilon_3)\cup {R''}(w,\beta,n,\varepsilon,y)=$$
    $$=\left\{v\in K(n,y):\; \pi(v)\in(\pi(w)(\beta-\varepsilon),\pi(w)(\beta+\varepsilon))\right\}\cup$$
    $$\cup\left\{v\in \overline{K}(n,y):\; \pi(v)\in(\pi(w)(\beta-\varepsilon),\pi(w)(\beta+\varepsilon))\right\}=$$
    $$=\left\{v\in\mathbb{YF}_n:\; \pi(v)\in(\pi(w)(\beta-\varepsilon),\pi(w)(\beta+\varepsilon))\right\}={R}(w,\beta,n,\varepsilon).$$
     
\end{itemize}
\end{Zam}

\begin{Oboz}
Пусть $w\in\mathbb{YF}_\infty$, $n,y\in\mathbb{N}_0$, $\varepsilon_5\in\mathbb{R}_{>0}:$ $y\le n$. Тогда
\begin{itemize}
    \item $$\widetilde{R}(w,n,y,\varepsilon_5):=\left\{v\in K(n,y):\; \pi_y'(v)\notin (\pi(w)(1-\varepsilon_5),\pi(w)(1+\varepsilon_5)) \right\}.$$
\end{itemize}
\end{Oboz}

Теперь давайте что-то докажем:

\begin{Prop} \label{kerambit1}
Пусть $w\in\mathbb{YF}_\infty$, $\beta\in(0,1]$, $\varepsilon_5\in\mathbb{R}_{>0}$, $n,y\in\mathbb{N}_0:$ $y\le n$. Тогда
$${\displaystyle\sum_{v\in \widetilde{R}(w,n,y,\varepsilon_5)} T_{w,\beta,n}(v,y)}=\left({\displaystyle\sum_{v\in\mathbb{YF}_n} T_{w,\beta,n}(v,y)}\right)\left({\sum_{v \in \overline{R}(w,y,\varepsilon_5)}\mu_{w}(v)}\right).$$
\end{Prop}
\begin{proof}

По определению функции $T$ ясно, что (так как если $w\in\mathbb{YF}_\infty$, $n,y\in\mathbb{N}_0$, $\varepsilon_5\in\mathbb{R}_{>0}:$ $y\le n$, то $\widetilde{R}(w,n,y,\varepsilon_5)\subseteq K(n,y)$)
$${\displaystyle\sum_{v\in \widetilde{R}(w,n,y,\varepsilon_5)} T_{w,\beta,n}(v,y)}={\displaystyle\sum_{v\in \widetilde{R}(w,n,y,\varepsilon_5)}\left(d(\varepsilon,v)\cdot q(v(y))\cdot {{d_1'(v'(y),w)}}\cdot \beta^y\cdot \left(1-\beta^2\right)^{\#(v(y))}\right)}.$$

Ясно, что в каждом слагаемом по Замечанию \ref{kerambus} при  $v\in\mathbb{YF}$, $n,y\in\mathbb{N}_0$
$$v=v(y)v'(y).$$

А значит к каждому слагаемому можно применить Утверждение \ref{razbivaem} при $v,v(y),v'(y)\in\mathbb{YF}$ и получить, что наше выражение равняется следующему:
$${\displaystyle\sum_{v\in \widetilde{R}(w,n,y,\varepsilon_5)}\left(d(\varepsilon,v'(y))\cdot d\left(\varepsilon,v(y)1^{|v'(y)|}\right)\cdot q(v(y))\cdot {{d_1'(v'(y),w)}}\cdot \beta^y\cdot \left(1-\beta^2\right)^{\#(v(y))}\right)}=$$
$$=(\text{По обозначению $v'(y)$})=$$
$$={\displaystyle\sum_{v\in \widetilde{R}(w,n,y,\varepsilon_5)}\left(d(\varepsilon,v'(y))\cdot d\left(\varepsilon,v(y)1^{y}\right)\cdot q(v(y))\cdot {{d_1'(v'(y),w)}}\cdot \beta^y\cdot \left(1-\beta^2\right)^{\#(v(y))}\right)}.$$

Заметим, что при $w\in\mathbb{YF}_\infty$, $n,y\in\mathbb{N}_0,$ $\varepsilon_5\in\mathbb{R}_{>0}:$ $n\ge y$:
\begin{itemize}
    \item если $v\in \widetilde{R}(w,n,y,\varepsilon_5),$ то $v=v(y)v'(y)$, причём $v(y)\in\mathbb{YF}_{n-y},$ $ v'(y)\in\overline{R}(w,y,\varepsilon_5)$ (очевидно из определений $\widetilde{R}(w,n,y,\varepsilon_5)$ и $\overline{R}(w,y,\varepsilon_5)$);
    \item если $v_1,v_2\in \widetilde{R}(w,n,y,\varepsilon_5)$: $v_1\ne v_2$, то либо $v_1(y)\ne v_2(y)$ или
    $v_1'(y)\ne v'_2(y)$;
    \item если $v''\in \mathbb{YF}_{n-y},$ $v'''\in\overline{R}(w,y,\varepsilon_5)$, то $\left(v''v'''\right)\in \widetilde{R}(w,n,y,\varepsilon_5)$, $\left(v''v'''\right)(y)=v'',$ $\left(v''v'''\right)'(y)=v'''$ (очевидно из определений $\widetilde{R}(w,n,y,\varepsilon_5)$ и $\overline{R}(w,y,\varepsilon_5)$).
\end{itemize}

А это значит, что при всех $v\in \widetilde{R}(w,n,y,\varepsilon_5)$, пара $(v(y),v'(y))$ принимает все значения в $\mathbb{YF}_{n-y}\times \overline{R}(w,y,\varepsilon_5)$, причём ровно по одному разу.

А это значит, что наше выражение равняется следующему:
$${\displaystyle\sum_{v''\in \mathbb{YF}_{n-y}}\left(\sum_{v'''\in \overline{R}(w,y,\varepsilon_5)}\left(d(\varepsilon,v''')\cdot  d\left(\varepsilon,v''1^{y}\right)\cdot q(v'')\cdot {{d_1'(v''',w)}}\cdot \beta^y\cdot \left(1-\beta^2\right)^{\#v''}\right)\right)}=$$
$$={\displaystyle \left(\sum_{v''\in \mathbb{YF}_{n-y}}\left(q(v'')\cdot d\left(\varepsilon,v''1^{y}\right)\cdot \beta^y\cdot \left(1-\beta^2\right)^{\#v''}\right)\right)\left(\sum_{v'''\in \overline{R}(w,y,\varepsilon_5)}\left(d(\varepsilon,v''')\cdot{{d_1'(v''',w)}}\right)\right)}=$$
$$=(\text{По Утверждению \ref{limitstrih} при $v'''\in\mathbb{YF}$, $w\in\mathbb{YF}_{\infty}$})=$$
$$={\displaystyle \left(\sum_{v''\in \mathbb{YF}_{n-y}}\left(q(v'')\cdot d\left(\varepsilon,v''1^{y}\right)\cdot \beta^y\cdot \left(1-\beta^2\right)^{\#v''}\right)\right)\left(\sum_{v'''\in \overline{R}(w,y,\varepsilon_5)}\left(d(\varepsilon,v''')\cdot \lim_{m \to \infty}{\frac{d(v''',w_m)}{d(\varepsilon,w_m)}}\right)\right)}=$$
$$=\left(\text{По Утверждению \ref{sum} при $w\in\mathbb{YF}_\infty$, $\beta\in(0,1]$, $n,y\in\mathbb{N}_0$}\right)=$$
$$={\displaystyle\left(\sum_{x\in\mathbb{YF}_n} T_{w,\beta,n}(x,y)\right)}\left(\sum_{v'''\in \overline{R}(w,y,\varepsilon_5)}\left(\lim_{m \to \infty}{\frac{d(\varepsilon,v''')d(v''',w_m)}{d(\varepsilon,w_m)}}\right)\right)=(\text{По обозначению})=$$
$$=\left({\displaystyle\sum_{x\in\mathbb{YF}_n} T_{w,\beta,n}(x,y)}\right)\left(\sum_{v'''\in \overline{R}(w,y,\varepsilon_5)}\mu_w(v''')\right)=\left({\displaystyle\sum_{v\in\mathbb{YF}_n} T_{w,\beta,n}(v,y)}\right)\left(\sum_{v\in \overline{R}(w,y,\varepsilon_5)}\mu_w(v)\right),$$

что и требовалось.

Утверждение доказано.

\end{proof}

\begin{Prop}\label{halloween}
Пусть $w\in\mathbb{YF}_\infty^+$, $\beta\in(0,1)$, $\varepsilon_1',\varepsilon_5,\overline{\varepsilon'}\in\mathbb{R}_{>0}:$ $\varepsilon_1'\in\left(0,\beta^2\right)$. Тогда $\exists N'\in\mathbb{N}_0:$ $\forall n \in\mathbb{N}_0:$ $n\ge N':$
$${\sum_{y\in\overline{n}[\beta,\varepsilon_1']}\left({\sum_{v\in \widetilde{R}(w,n,y,\varepsilon_5)} T_{w,\beta,n}(v,y)}\right) <\overline{\varepsilon'}}.$$
\end{Prop}
\begin{proof}

По Утверждению \ref{kerambit1} при $w\in\mathbb{YF}_\infty,$ $\beta\in(0,1],$ $\varepsilon_5\in\mathbb{R}_{>0}$, $n,y\in\mathbb{N}_0:$ $y\le n$
$${\displaystyle\sum_{v\in \widetilde{R}(w,n,y,\varepsilon_5)} T_{w,\beta,n}(v,y)}=\left({\displaystyle\sum_{v\in\mathbb{YF}_n} T_{w,\beta,n}(v,y)}\right)\left({\sum_{v \in \overline{R}(w,y,\varepsilon_5)}\mu_{w}(v)}\right),$$
а значит если $n\in\mathbb{N}_0$, то мы можем просуммировать данное выражение по $y\in\overline{n}[\beta,\varepsilon_1']$ (ясно, что $\overline{n}[\beta,\varepsilon_1']\subseteq\overline{n}$). Просуммируем и получим, при наших $w\in\mathbb{YF}_\infty^+$, $\beta\in(0,1]$, $\varepsilon_1',\varepsilon_5,\overline{\varepsilon'}\in\mathbb{R}_{>0}:$ $\varepsilon_1'\in\left(0,\beta^2\right)$
$${\sum_{y\in\overline{n}[\beta,\varepsilon_1']}\left({\sum_{v\in \widetilde{R}(w,n,y,\varepsilon_5)} T_{w,\beta,n}(v,y)}\right)} =\sum_{y\in\overline{n}[\beta,\varepsilon_1']}\left(\left({\displaystyle\sum_{v\in\mathbb{YF}_n} T_{w,\beta,n}(v,y)}\right)\left({\sum_{v \in \overline{R}(w,y,\varepsilon_5)}\mu_{w}(v)}\right)\right).$$

Зафиксируем $\overline{\varepsilon'}\in\mathbb{R}_{>0}$. 

По Теореме \ref{main2} при нашем $w\in\mathbb{YF}_\infty^+$ и $\varepsilon=\varepsilon_5\in\mathbb{R}_{>0}$
$$ \lim_{y \to \infty}{\sum_{v\in \overline{R}(w,y,\varepsilon_5)}\mu_w(v)=0},$$
то есть по определению предела для $\varepsilon'=\frac{\overline{\varepsilon'}}{\left(1+\frac{1}{\beta}\right)}$ $\exists Y\in\mathbb{N}_0:$ при $y\in\mathbb{N}_0:$ $y\ge Y$
$$ {\sum_{v\in \overline{R}(w,y,\varepsilon_5)}\mu_w(v)}<\varepsilon'.$$

Зафиксируем данный $Y\in\mathbb{N}_0$. Пусть $N'=\left\lceil\frac{\displaystyle Y}{\displaystyle\beta^2-\varepsilon_1'}\right\rceil$ (помним, что $\varepsilon_1'\in\left(0,\beta^2\right)$). Тогда если  $y\in\overline{n}[\beta,\varepsilon_1']\Longleftrightarrow y\in\left(\overline{n}\cap\left(n\left(\beta^2-\varepsilon_1'\right),n\left(\beta^2+\varepsilon_1'\right)\right)\right)$ при $n\in\mathbb{N}_0:$ $n\ge N'$, то 
$$y\ge n\left(\beta^2-\varepsilon_1'\right)\ge N'\left(\beta^2-\varepsilon_1'\right)=\left\lceil\frac{\displaystyle Y}{\displaystyle\beta^2-\varepsilon_1'}\right\rceil\left(\beta^2-\varepsilon_1'\right)\ge\left(\frac{\displaystyle Y}{\displaystyle\beta^2-\varepsilon_1'}\right)\left(\beta^2-\varepsilon_1'\right)=Y.$$

Таким образом, мы получаем, что при наших $w\in\mathbb{YF}_\infty^+$, $\beta\in(0,1)$, $\varepsilon_1',\varepsilon_5,\overline{\varepsilon'}\in\mathbb{R}_{>0}:$ $\varepsilon_1'\in\left(0,\beta^2\right)$ и выбранном $N'\in\mathbb{N}_0$ $\forall n \in\mathbb{N}_0:$ $n\ge N'$
$${\sum_{y\in\overline{n}[\beta,\varepsilon_1']}\left({\sum_{v\in \widetilde{R}(w,n,y,\varepsilon_5)} T_{w,\beta,n}(v,y)}\right)} =\sum_{y\in\overline{n}[\beta,\varepsilon_1']}\left(\left({\displaystyle\sum_{v\in\mathbb{YF}_n} T_{w,\beta,n}(v,y)}\right)\left({\sum_{v \in \overline{R}(w,y,\varepsilon_5)}\mu_{w}(v)}\right)\right)<$$
 $$<\text{(Так как в каждом слагаемом $y\ge Y$)}<$$
$$<\sum_{y\in \overline{n}[\beta,\varepsilon_1']}\left(\left({\sum_{v\in\mathbb{YF}_n} T_{w,\beta,n}(v,y)}\right)\cdot    \varepsilon'\right)=\varepsilon'\sum_{y\in \overline{n}[\beta,\varepsilon_1']}\left({\sum_{v\in\mathbb{YF}_n} T_{w,\beta,n}(v,y)}    \right)\le$$
$$\le\text{(так как функция $T$ неотрицательна и $\overline{n}[\beta,\varepsilon_1']\subseteq\overline{n}$)}\le$$
$$\le\varepsilon'\sum_{y=0}^n\left({\sum_{v\in\mathbb{YF}_n} T_{w,\beta,n}(v,y)}    \right)\le \text{(По Следствию \ref{lehamed1} при $w\in\mathbb{YF}_\infty$, $\beta\in(0,1)$, $n\in\mathbb{N}_0$)} \le$$
$$\le \varepsilon'\left(1+\frac{1}{\beta}\right)=\frac{\overline{\varepsilon'}}{\left(1+\frac{1}{\beta}\right)}\left(1+\frac{1}{\beta}\right)=\overline{\varepsilon'},$$
что и требовалось.

Утверждение доказано.

\end{proof}

\begin{Lemma} \label{pohoronil}
Пусть $w\in\mathbb{YF}_\infty^+$, $\beta\in(0,1)$, $\varepsilon\in\mathbb{R}_{>0}$. Тогда $\exists\varepsilon_3\in\mathbb{R}_{>0}:$ $\forall \varepsilon_1'\in\left(0,\beta^2\right),$ $\overline{\varepsilon'}\in\mathbb{R}_{>0}$ $\exists N'\in\mathbb{N}_0:$ $\forall n\in\mathbb{N}_0:$ $n\ge N'$
$${\sum_{y\in \overline{n}[\beta,\varepsilon_1']}\left({\sum_{v\in \overline{R}(w,\beta,n,\varepsilon,y,\varepsilon_3)} T_{w,\beta,n}(v,y)}\right) <\overline{\varepsilon'}}.$$
\end{Lemma}
\begin{proof}

Ясно, что при наших $w\in\mathbb{YF}_\infty^+$, $\beta\in(0,1)$, $\varepsilon\in\mathbb{R}_{>0}$ существуют такие $\varepsilon_3,$ $\varepsilon_5\in\mathbb{R}_{>0}$, что $\forall n,y\in\mathbb{N}_0,$ $v\in\mathbb{YF}:$ $y\le n,$ $v\in K(n,y)$ если $\pi(v)\notin(\pi(w)(\beta-\varepsilon),\pi(w)(\beta+\varepsilon))$ и $\pi_y(v)\in(\beta-\varepsilon_3,\beta+\varepsilon_3)$, то $\pi'_y(v)\notin(\pi(w)(1-\varepsilon_5),\pi(w)(1+\varepsilon_5))$ (так как по Утверждению \ref{meexy} при $v\in\mathbb{YF}$ и $y\in\mathbb{N}_0$ если $v\in K(n,y)$, то $\pi(v)=\pi_y(v)\pi'_y(v)$).

Таким образом, из определений $\overline{R}(w,\beta,n,\varepsilon,y,\varepsilon_3)$ и $\widetilde{R}(w,n,y,\varepsilon_5)$ ясно, что при наших $w\in\mathbb{YF}_\infty$, $\beta\in(0,1)$, $\varepsilon\in\mathbb{R}_{>0}$ существуют такие $\varepsilon_3,\varepsilon_5\in\mathbb{R}_{>0}$, что $\forall n,y\in\mathbb{N}_0,$ $v\in\mathbb{YF}:$ $y\le n$ если $v\in \overline{R}(w,\beta,n,\varepsilon,y,\varepsilon_3)$, то $v\in \widetilde{R}(w,n,y,\varepsilon_5)$.

А это значит, что при наших $w\in\mathbb{YF}_\infty^+$, $\beta\in(0,1)$, $\varepsilon\in\mathbb{R}_{>0}$ существуют такие $\varepsilon_3,\varepsilon_5\in\mathbb{R}_{>0}$, что $\forall n,y\in\mathbb{N}_0:$ $y\le n$  $$\overline{R}(w,\beta,n,\varepsilon,y,\varepsilon_3)\subseteq\widetilde{R}(w,n,y,\varepsilon_5).$$ 

Из этого делаем следующий вывод (из-за неотрицательности функции $T$): при наших $w\in\mathbb{YF}_\infty^+$, $\beta\in(0,1)$, $\varepsilon\in\mathbb{R}_{>0}$ существуют такие $\varepsilon_3,\varepsilon_5\in\mathbb{R}_{>0}$, что $\forall n,y\in\mathbb{N}_0:$ $y\le n$
$${\sum_{v\in \overline{R}(w,\beta,n,\varepsilon,y,\varepsilon_3)} T_{w,\beta,n}(v,y)}\le {\sum_{v\in \widetilde{R}(w,n,y,\varepsilon_5)} T_{w,\beta,n}(v,y)}.$$

Зафиксируем эти $\varepsilon_3,\varepsilon_5\in\mathbb{R}_{>0}$

Теперь зафиксируем произвольные $\varepsilon_1'\in\left(0,\beta^2\right),$ $\overline{\varepsilon'}\in\mathbb{R}_{>0}$.

Ясно, что если $n\in\mathbb{N}_0$, то мы можем просуммировать наше выражение по $y\in\overline{n}[\beta,\varepsilon_1']$ (ясно, что $\overline{n}[\beta,\varepsilon_1']\subseteq\overline{n}$). Просуммируем и получим, что при наших $w\in\mathbb{YF}_\infty^+$, $\beta\in(0,1)$, $\varepsilon\in\mathbb{R}_{>0},$ $\varepsilon_1'\in\left(0,\beta^2\right)$
$$\sum_{y\in \overline{n}[\beta,\varepsilon_1']}\left({\sum_{v\in \overline{R}(w,\beta,n,\varepsilon,y,\varepsilon_3)} T_{w,\beta,n}(v,y)}\right)\le \sum_{y\in \overline{n}[\beta,\varepsilon_1']}\left({\sum_{v\in \widetilde{R}(w,n,y,\varepsilon_5)} T_{w,\beta,n}(v,y)}\right).$$

По Утверждению \ref{halloween} при наших $w\in\mathbb{YF}_\infty^+$, $\beta\in(0,1)$, $\varepsilon_1',\varepsilon_5,\overline{\varepsilon'}\in\mathbb{R}_{>0}:$ $\varepsilon_1'\in\left(0,\beta^2\right)$ $\exists N'\in\mathbb{N}_0:$ $\forall n \in\mathbb{N}_0:$ $n\ge N'$
$${\sum_{y\in\overline{n}[\beta,\varepsilon_1']}\left({\sum_{v\in \widetilde{R}(w,n,y,\varepsilon_5)} T_{w,\beta,n}(v,y)}\right) <\overline{\varepsilon'}}.$$

Зафиксируем этот $N'\in\mathbb{N}_0$.

Таким образом, мы поняли, что при наших $w\in\mathbb{YF}_\infty^+$, $\beta\in(0,1)$, $\varepsilon\in\mathbb{R}_{>0}$ мы выбрали $\varepsilon_3,\varepsilon_5\in\mathbb{R}_{>0}$ так, что $\forall \varepsilon_1'\in\left(0,\beta^2\right),$ $\overline{\varepsilon'}\in\mathbb{R}_{>0}$ $\exists N'\in\mathbb{N}_0$: $\forall n\in\mathbb{N}_0:$ $n\ge N'$
$$\sum_{y\in \overline{n}[\beta,\varepsilon_1']}\left({\sum_{v\in \overline{R}(w,\beta,n,\varepsilon,y,\varepsilon_3)} T_{w,\beta,n}(v,y)}\right)\le \sum_{y\in \overline{n}[\beta,\varepsilon_1']}\left({\sum_{v\in \widetilde{R}(w,n,y,\varepsilon_5)} T_{w,\beta,n}(v,y)}\right)<\overline{\varepsilon'},$$
что и требовалось.

Лемма доказана.
\end{proof}

\begin{Oboz}
Пусть $\beta\in(0,1)$, $\varepsilon_1\in\mathbb{R}_{>0}$, $N\in\mathbb{N}_0:$ $N\ge 1$. Тогда 
$$U(\beta,\varepsilon_1,N):=\left\{(a,b)\in\mathbb{N}_0\times\mathbb{N}_0:\; a+b\ge N,\; \beta^2-\varepsilon_1<\frac{a}{a+b}<\beta^2+\varepsilon_1\right\}.$$
\end{Oboz}

\begin{Oboz}

$\;$

\begin{itemize}
    \item
    $$N_1: \left\{(\beta,\varepsilon_1,A):\;\beta\in(0,1),\;\varepsilon_1\in\left(0,\beta^2\right),\;A\in\mathbb{N}_0\right\}\to \mathbb{N}_0$$
    -- это функция, определённая следующим образом:
    $$N_1(\beta,\varepsilon_1,A):=\left\lceil\frac{A}{\beta^2-\varepsilon_1}\right\rceil;$$
    
    \item $$N_2: \;\left\{(\beta,\varepsilon_1,B):\;\beta\in(0,1),\;\varepsilon_1\in\left(0,1-\beta^2\right),\;B\in\mathbb{N}_0\right\}\to \mathbb{N}_0$$
    -- это функция, определённая следующим образом:
    $$N_2(\beta,\varepsilon_1,B):=\left\lceil\frac{B}{1-\beta^2-\varepsilon_1}\right\rceil.$$
    
\end{itemize}
\end{Oboz}

\renewcommand{\labelenumi}{\arabic{enumi}$)$}
\renewcommand{\labelenumii}{\arabic{enumi}.\arabic{enumii}$^\circ$}
\renewcommand{\labelenumiii}{\arabic{enumi}.\arabic{enumii}.\arabic{enumiii}$^\circ$}

\begin{Prop} \label{kunibes}

$\;$

\begin{enumerate} 
    \item Пусть $\beta\in(0,1),$ $\varepsilon_1\in\mathbb{R}_{>0},$ $A,a,b,N\in\mathbb{N}_0:$ $\varepsilon_1\in \left(0,\beta^2\right)$, $(a,b)\in U(\beta,\varepsilon_1,N),$ $N\ge \max(1,N_1(\beta,\varepsilon_1,A))$. Тогда 
    $$a\ge A.$$
    \item Пусть $\beta\in(0,1),$ $\varepsilon_1\in\mathbb{R}_{>0},$ $B,a,b,N\in\mathbb{N}_0:$ $\varepsilon_1\in \left(0,{1-\beta^2}\right)$, $(a,b)\in U(\beta,\varepsilon_1,N),$ $N\ge \max(1,N_2(\beta,\varepsilon_1,B))$. Тогда
    $$b\ge B.$$
    
\end{enumerate}

\end{Prop}
\begin{proof}

$\;$

\begin{enumerate}
    \item 
     Из обозначения множества $U$ в данном случае ясно, что в если $a,b,N\in\mathbb{N}_0:$ $(a,b)\in U(\beta,\varepsilon_1,N),$ $N\ge \max(1,N_1(\beta,\varepsilon_1,A))$, то $$\displaystyle\beta^2-\varepsilon_1<\frac{a}{a+b},$$
     а значит, так как $a+b\ge N\ge 1$ $\left(\text{при счёте помним, что $\varepsilon_1\in\left(0,\beta^2\right)$}\right)$,
     $$a>(a+b)(\beta^2-\varepsilon_1)\ge N(\beta^2-\varepsilon_1)\ge N_1(\beta,\varepsilon_1,A)(\beta^2-\varepsilon_1)=$$
     $$= \left\lceil\frac{A}{\beta^2-\varepsilon_1}\right\rceil(\beta^2-\varepsilon_1)\ge\left(\frac{A}{\beta^2-\varepsilon_1}\right)(\beta^2-\varepsilon_1)=A,$$
     что и требовалось.
     
     Первый пункт Утверждения доказан.

    \item 
     Из обозначения множества $U$ в данном случае ясно, что в если $a,b,N\in\mathbb{N}_0:$ $(a,b)\in U(\beta,\varepsilon_1,N),$ $N\ge \max(1,N_2(\beta,\varepsilon_1,B))$, то $$\displaystyle\frac{a}{a+b}<\beta^2+\varepsilon_1\Longleftrightarrow \displaystyle 1-\frac{b}{a+b}<\beta^2+\varepsilon_1\Longleftrightarrow 1-\beta^2-\varepsilon_1<\frac{b}{a+b},$$
     а значит, так как $a+b\ge N\ge 1$ $\left(\text{при счёте помним, что $\varepsilon_1\in\left(0,1-\beta^2\right)$}\right)$,
    $${b}>(a+b)(1-\beta^2-\varepsilon_1)\ge N(1-\beta^2-\varepsilon_1) \ge N_2(\beta,\varepsilon_1,B)(1-\beta^2-\varepsilon_1)=$$
    $$=\left\lceil\frac{B}{1-\beta^2-\varepsilon_1}\right\rceil(1-\beta^2-\varepsilon_1)\ge\frac{B}{1-\beta^2-\varepsilon_1}(1-\beta^2-\varepsilon_1)=B,$$
    что и требовалось.
    
    Второй пункт Утверждения доказан.

\end{enumerate}

Оба пункта Утверждения доказаны.

\end{proof}

\begin{Prop} \label{betasvyaz1}
Пусть $\beta\in(0,1)$, $\varepsilon_2\in\mathbb{R}_{>0}$. Тогда $\exists \varepsilon_1\in\mathbb{R}_{>0},$ $N\in\mathbb{N}_0:$ $N\ge 1$, $\forall a,b\in\mathbb{N}_0:$ $a,b\ge 1,$ $(a,b)\in U(\beta,\varepsilon_1,N)$
$$\beta-\varepsilon_2<\prod_{i=a}^{a+b-1}\frac{2i-1}{2i}< \beta+\varepsilon_2.$$
\end{Prop}
\begin{proof}

Пусть $\varepsilon_1\in \left(0,\min\left({\beta^2},{1-\beta^2}\right)\right)$, $N\in\mathbb{N}_0:$ $N\ge\max(1,N_1(\beta,\varepsilon_1,2))$. Тогда по Утверждению \ref{kunibes} (пункт 1) при  $\beta\in(0,1),$ $\varepsilon_1\in\mathbb{R}_{>0},$ $2,a,b,N\in\mathbb{N}_0$ ясно, что если $a,b\in\mathbb{N}_0:$ $a,b\ge 1,$ $(a,b)\in U(\beta,\varepsilon_1,N)$, то $a\ge 2$.

А значит если $\varepsilon_1\in \left(0,\min\left({\beta^2},{1-\beta^2}\right)\right)$, $N\in\mathbb{N}_0:$ $N\ge\max(1,N_1(\beta,\varepsilon_1,2))$ и $a,b\in\mathbb{N}_0:$ $a,b\ge 1,$ $(a,b)\in U(\beta,\varepsilon_1,N)$, то
$$\prod_{i=a}^{a+b-1}\frac{2i-1}{2i}=\prod_{i=a}^{a+b-1}\frac{(2i-1)2i}{(2i)^2}=\frac{\displaystyle\prod_{i=2a-1}^{2a+2b-2}i}{\displaystyle\left(2^b\prod_{i=a}^{a+b-1}i\right)^2}=$$
$$=\frac{\displaystyle \frac{(2a+2b-2)!}{(2a-2)!}}{\displaystyle 2^{2b}\left(\frac{(a+b-1)!}{(a-1)!}\right)^2}=\frac{1}{2^{2b}}\left(\frac{(a-1)!}{(a+b-1)!}\right)^2\frac{(2a+2b-2)!}{(2a-2)!}.$$

А значит если $\varepsilon_1\in \left(0,\min\left({\beta^2},{1-\beta^2}\right)\right)$, $N\in\mathbb{N}_0:$ $N\ge\max(1,N_1(\beta,\varepsilon_1,2))$ и $(a,b)\in\mathbb{N}_0:$ $a,b\ge 1,$ $(a,b)\in U(\beta,\varepsilon_1,N)$, то по формуле Стирлинга данное выражение равняется следующему выражению при некоторых $\theta_{a-1},\theta_{a+b-1},\theta_{2a+2b-2},\theta_{2a-2}\in(0,1)$ (тут важно, что в данном случае $a\ge 2$ и $b\ge 1$, что, в свою очередь, значит, что $a-1\ge 1,$ $a+b-1\ge 1,$ $2a+2b-2\ge 1,$ $2a-2\ge 1$):
$$\frac{1}{2^{2b}}\frac{\displaystyle{2\pi(a-1)}\left(\frac{(a-1)}{e}\right)^{2a-2}\left(\exp{\frac{\theta_{a-1}}{12(a-1)}}\right)^2}{\displaystyle{2\pi(a+b-1)}\left(\frac{(a+b-1)}{e}\right)^{2a+2b-2}\left(\exp{\frac{\theta_{a+b-1}}{12(a+b-1)}}\right)^2}\cdot$$
$$\cdot\frac{\displaystyle\sqrt{2\pi(2a+2b-2)}\left(\frac{(2a+2b-2)}{e}\right)^{2a+2b-2}\exp{\frac{\theta_{2a+2b-2}}{12(2a+2b-2)}}}{\displaystyle\sqrt{2\pi(2a-2)}\left(\frac{(2a-2)}{e}\right)^{2a-2}\exp{\frac{\theta_{2a-2}}{12(2a-2)}}}=$$
$$=\frac{1}{2^{2b}}\frac{\displaystyle{(a-1)}\left({a-1}\right)^{2a-2}\left(\exp{\frac{\theta_{a-1}}{12(a-1)}}\right)^2}{\displaystyle{(a+b-1)}\left({a+b-1}\right)^{2a+2b-2}\left(\exp{\frac{\theta_{a+b-1}}{12(a+b-1)}}\right)^2}\cdot$$
$$\cdot\frac{\displaystyle\sqrt{(2a+2b-2)}\left({2a+2b-2}\right)^{2a+2b-2}\exp{\frac{\theta_{2a+2b-2}}{12(2a+2b-2)}}}{\displaystyle\sqrt{(2a-2)}\left({2a-2}\right)^{2a-2}\exp{\frac{\theta_{2a-2}}{12(2a-2)}}}.$$

Для начала рассмотрим
$$\frac{\left(\exp{\frac{\theta_{a-1}}{12(a-1)}}\right)^2}{\left(\exp{\frac{\theta_{a+b-1}}{12(a+b-1)}}\right)^2}\cdot\frac{\exp{\frac{\theta_{2a+2b-2}}{12(2a+2b-2)}}}{\exp{\frac{\theta_{2a-2}}{12(2a-2)}}}.$$

Ясно, что $\forall \varepsilon_1'\in\mathbb{R}_{>0}$ $\exists A_1'\in\mathbb{N}_0:$ $\forall a,b\in\mathbb{N}_0:$ $a\ge A_1'$:
\begin{itemize}
    \item $0<{\frac{\theta_{a-1}}{12(a-1)}}<\varepsilon_1'$;
    \item $0<{\frac{\theta_{a+b-1}}{12(a+b-1)}}<\varepsilon_1'$;
    \item $0<{\frac{\theta_{2a+2b-2}}{12(2a+2b-2)}}<\varepsilon_1'$;
    \item $0<{\frac{\theta_{2a-2}}{12(2a-2)}}<\varepsilon_1'$.
\end{itemize}

А это значит, что $\forall \varepsilon_2'\in\mathbb{R}_{>0}$ $\exists A_2'\in\mathbb{N}_0:$ $\forall a,b\in\mathbb{N}_0:$ $a\ge A_2'$:
\begin{itemize}
    \item $1<\exp{\frac{\theta_{a-1}}{12(a-1)}}<1+\varepsilon_2'$ ;
    \item $1<\exp{\frac{\theta_{a+b-1}}{12(a+b-1)}}<1+\varepsilon_2'$;
    \item $1<\exp{\frac{\theta_{2a+2b-2}}{12(2a+2b-2)}}<1+\varepsilon_2'$;
    \item $1<\exp{\frac{\theta_{2a-2}}{12(2a-2)}}<1+\varepsilon_2'$.
\end{itemize}

А это, в свою очередь, значит, что $\forall \varepsilon'\in\mathbb{R}_{>0}$ $\exists A'(\varepsilon')\in\mathbb{N}_0:$ $\forall a,b\in\mathbb{N}_0:$ $a\ge A'(\varepsilon')$
$$1-\varepsilon'<\frac{\left(\exp{\frac{\theta_{a-1}}{12(a-1)}}\right)^2}{\left(\exp{\frac{\theta_{a+b-1}}{12(a+b-1)}}\right)^2}\cdot\frac{\exp{\frac{\theta_{2a+2b-2}}{12(2a+2b-2)}}}{\exp{\frac{\theta_{2a-2}}{12(2a-2)}}}< 1+\varepsilon'.$$

Таким образом, $\forall \varepsilon'\in\mathbb{R}_{>0}$ $\exists A'(\varepsilon')\in\mathbb{N}_0:$ если $\varepsilon_1\in \left(0,\min\left({\beta^2},{1-\beta^2}\right)\right)$,
$N\in\mathbb{N}_0:$ $N\ge\max(1,N_1(\beta,\varepsilon_1,\max(2,A'(\varepsilon'))))$ и $a,b\in\mathbb{N}_0:$ $a,b\ge 1,$ $(a,b)\in U(\beta,\varepsilon_1,N)$, то по Утверждению \ref{kunibes} (пункт 1) при $\beta\in(0,1),$ $\varepsilon_1\in\mathbb{R}_{>0},$ $A'(\varepsilon'),a,b,N\in\mathbb{N}_0$ наше выражение равняется следующему выражению при некотором $c'\in(1-\varepsilon',1+\varepsilon'):$
$$c'\frac{1}{2^{2b}}\frac{\displaystyle{(a-1)}\left({a-1}\right)^{2a-2}}{\displaystyle{(a+b-1)}\left({a+b-1}\right)^{2a+2b-2}}\cdot\frac{\displaystyle\sqrt{(2a+2b-2)}\left({2a+2b-2}\right)^{2a+2b-2}}{\displaystyle\sqrt{(2a-2)}\left({2a-2}\right)^{2a-2}}=$$
$$=c'\frac{\displaystyle{(a-1)}\left({a-1}\right)^{2a-2}}{\displaystyle{(a+b-1)}\left({a+b-1}\right)^{2a+2b-2}}\cdot\frac{\displaystyle\sqrt{(a+b-1)}\left({a+b-1}\right)^{2a+2b-2}}{\displaystyle\sqrt{(a-1)}\left({a-1}\right)^{2a-2}}=c'\frac{a-1}{a+b-1}\cdot \frac{\sqrt{a+b-1}}{\sqrt{a-1}}=$$
$$=c'\frac{\sqrt{a-1}}{\sqrt{a+b-1}}=c'\frac{\sqrt{a-1}}{\sqrt{a}}\cdot \frac{\sqrt{a+b}}{\sqrt{a+b-1}}\cdot \frac{\sqrt{a}}{\sqrt{a+b}}=c'\sqrt{\frac{{a-1}}{{a}}}\sqrt{\frac{{a+b}}{{a+b-1}}}\sqrt{\frac{{a}}{{a+b}}}.$$

Ясно, что $\forall \varepsilon_1''\in\mathbb{R}_{>0}$ $\exists A_1''\in\mathbb{N}_0:$ $\forall a,b\in\mathbb{N}_0:$ $a\ge A_1''$:
\begin{itemize}
    \item $1-\varepsilon_1''<\sqrt{\frac{{a-1}}{{a}}}<1+\varepsilon_1''$ ;
    \item $1-\varepsilon_1''<\sqrt{\frac{{a+b}}{{a+b-1}}}<1+\varepsilon_1''$.
\end{itemize}

А это, в свою очередь, значит, что $\forall \varepsilon''\in\mathbb{R}_{>0}$ $\exists A''(\varepsilon'')\in\mathbb{N}_0:$ $\forall a,b\in\mathbb{N}_0:$ $a\ge A''(\varepsilon'')$
$$1-\varepsilon''\sqrt{\frac{{a-1}}{{a}}}\sqrt{\frac{{a+b}}{{a+b-1}}}<1+\varepsilon''.$$

Таким образом, $\forall \varepsilon',\varepsilon''\in\mathbb{R}_{>0}$ $\exists A'(\varepsilon'),A''(\varepsilon'')\in\mathbb{N}_0:$ если $\varepsilon_1\in \left(0,\min\left({\beta^2},{1-\beta^2}\right)\right)$,
$N\in\mathbb{N}_0:$ $N\ge\max(1,N_1(\beta,\varepsilon_1,\max(2,A'(\varepsilon'),A''(\varepsilon''))))$ и $a,b\in\mathbb{N}_0:$ $a,b\ge 1,$ $(a,b)\in U(\beta,\varepsilon_1,N)$, то по Утверждению \ref{kunibes} (пункт 1) при $\beta\in(0,1),$ $\varepsilon_1\in\mathbb{R}_{>0},$ $A''(\varepsilon''),a,b,N\in\mathbb{N}_0$ наше выражение равняется следующему выражению при некоторых $c'\in(1-\varepsilon',1+\varepsilon'),$ $c''\in(1-\varepsilon'',1+\varepsilon'')$:
$$c'c''\sqrt{\frac{{a}}{{a+b}}}=c'c''\sqrt{c_1}$$
при некотором $c_1\in(\beta^2-\varepsilon_1,\beta^2+\varepsilon_1)$.

А из этого ясно, что $\forall \varepsilon',\varepsilon''\in(0,1)$ и $\varepsilon_1\in \left(0,\min\left({\beta^2},{1-\beta^2}\right)\right)$ $\exists N'(\varepsilon',\varepsilon'',\varepsilon_1)=\max(1,N_1(\beta,\varepsilon_1,\max(2,A'(\varepsilon'),A''(\varepsilon''))))\in\mathbb{N}_0:$ если $N\in\mathbb{N}_0:$ $N\ge N'(\varepsilon',\varepsilon'',\varepsilon_1)$ и $a,b\in\mathbb{N}_0:$ $a,b\ge 1,$ $(a,b)\in U(\beta,\varepsilon_1,N)$, то наше выражение равняется следующему выражению при некоторых $c'\in(1-\varepsilon',1+\varepsilon'),$ $c''\in(1-\varepsilon'',1+\varepsilon'')$, $c_1\in(\beta^2-\varepsilon_1,\beta^2+\varepsilon_1)$:
$$c'c''\sqrt{c_1}.$$

Ясно, что мы можем выбрать $\varepsilon',\varepsilon''\in(0,1)$ и $\varepsilon_1\in \left(0,\min\left({\beta^2},{1-\beta^2}\right)\right)$ так, что:
\begin{itemize}
    \item $$(1-\varepsilon')(1-\varepsilon'')\sqrt{\beta^2-\varepsilon_1}>{\beta-\varepsilon_2};$$
    \item $$(1+\varepsilon')(1+\varepsilon'')\sqrt{\beta^2+\varepsilon_1}<{\beta+\varepsilon_2}.$$
\end{itemize}

Очевидно, что при выбранных  $\varepsilon',\varepsilon''\in(0,1)$ и $\varepsilon_1\in \left(0,\min\left({\beta^2},{1-\beta^2}\right)\right)$ если $c'\in(1-\varepsilon',1+\varepsilon'),$ $c''\in(1-\varepsilon'',1+\varepsilon'')$, $c_1\in(\beta^2-\varepsilon_1,\beta^2+\varepsilon_1)$, то
$$\beta-\varepsilon_2<c'c''\sqrt{c_1}<\beta+\varepsilon_2.$$

Теперь пусть $N=N'(\varepsilon',\varepsilon'',\varepsilon_1)$ при выбранных $\varepsilon',\varepsilon''\in(0,1)$ и $\varepsilon_1\in \left(0,\min\left({\beta^2},{1-\beta^2}\right)\right)$.

Из того, что мы уже поняли, очевидно следует, что $N\in\mathbb{N}_0,$ $N\ge 1$ и что $\forall a,b\in\mathbb{N}_0:$ $a,b\ge 1,$  $(a,b)\in U(\beta,\varepsilon_1,N)$
$$\beta-\varepsilon_2<\prod_{i=a}^{a+b-1}\frac{2-i}{2i}< \beta+\varepsilon_2,$$
что и требовалось.

Утверждение доказано. 
\end{proof}

\begin{Prop} \label{abbalbisk}
Пусть $\varepsilon_3\in\mathbb{R}_{>0}$, $\beta\in(0,1)$. Тогда $\exists \varepsilon_1\in\mathbb{R}_{>0}:$ $\forall d\in\mathbb{N}_0$ $\exists N\in\mathbb{N}_0:$ $N\ge 1,$ $\forall a,b\in\mathbb{N}_0,$ $x\in\mathbb{YF}:$ $a,b\in\mathbb{N}_0:$ $a,b\ge 1,$ $(a,b)\in U(\beta,\varepsilon_1,N),$ $|x|\in \{2b,2b+1\}$
        $$e(x)\ge d\text{ или } \beta-\varepsilon_3<\pi\left(x1^{2a}\right)<\beta+\varepsilon_3. $$
\end{Prop}
\begin{proof}

По Утверждению \ref{betasvyaz1} при нашем $\beta\in(0,1)$ и $\forall\varepsilon_2\in\mathbb{R}_{>0}$ $\exists \varepsilon_1'\in\mathbb{R}_{>0},$ $N'\in\mathbb{N}_0$: $N'\ge 1$, $\forall a,b\in\mathbb{N}_0:$ $a,b\ge 1,$ $ (a,b)\in U(\beta,\varepsilon_1',N')$
$$\beta-\varepsilon_2<\prod_{i=a}^{a+b-1}\frac{2i-1}{2i}< \beta+\varepsilon_2.$$

А это по определению функций $\pi$ и $g$ равносильно тому, что при нашем $\beta\in(0,1)$ и $\forall\varepsilon_2\in\mathbb{R}_{>0}$ $\exists \varepsilon_1'\in\mathbb{R}_{>0},$ $N'\in\mathbb{N}_0$: $N'\ge 1$, $\forall a,b\in\mathbb{N}_0:$ $a,b\ge 1,$ $(a,b)\in U(\beta,\varepsilon_1',N')$
$$ \beta-\varepsilon_2<\pi\left(2^b1^{2a-1}\right)<\beta+\varepsilon_2.$$

Теперь докажем, что при нашем $\beta\in(0,1)$ и $\forall\varepsilon_4\in\mathbb{R}_{>0}$ $\exists \varepsilon_1''\in\mathbb{R}_{>0}:$ $\forall d\in\mathbb{N}_0$ $\exists N''\in\mathbb{N}_0:$ $N''\ge 1,$ $\forall a,b\in\mathbb{N}_0,$ $x\in\mathbb{YF} :$ $a,b\ge 1,$ $(a,b)\in U(\beta,\varepsilon_1'',N''),$ $|x|\in \{2b,2b+1\}$
$$e(x)\ge d\text{ или } 1-\varepsilon_4<\frac{\pi\left(x1^{2a}\right)}{\pi\left(2^b1^{2a-1}\right)}<1+\varepsilon_4. $$
        
Зафиксируем произвольное $\varepsilon_4\in\mathbb{R}_{>0}$. Пусть $\varepsilon_1''=\min\left(\frac{\beta^2}{2},\frac{1-\beta^2}{2}\right)\in \left(0,\min\left({\beta^2},{1-\beta^2}\right)\right)\subseteq \mathbb{R}_{>0}$. После чего зафиксируем произвольное $d\in\mathbb{N}_0$. Пусть $N_1''=\left\lceil\frac{1}{\varepsilon_4}\right\rceil$. Ясно, что $N_1''\in\mathbb{N}_0$ и $N_1''\ge 1$.

Давайте рассмотрим $a,b\in\mathbb{N}_0$: $a,b\ge 1,$ $\displaystyle (a,b)\in U(\beta,\varepsilon_1'',N_1'')$. Заметим, что

$$\displaystyle\frac{\pi\left(x1^{2a}\right)}{\pi\left(2^{b}1^{2a-1}\right)}\ge\text{(так как $a\ge 1$ и $|x|\le 2b+1$)}\ge \frac{\pi\left(2^{b+1}1^{2a}\right)}{\pi\left(2^{b}1^{2a-1}\right)}=$$
$$=(\text{По определению функций $\pi$ и $g$})=$$
$$=\frac{\displaystyle\prod_{i=a}^{a+b}\frac{2i}{2i+1}}{\displaystyle\prod_{i=a}^{a+b-1}\frac{2i-1}{2i}}\ge\frac{\displaystyle\prod_{i=a}^{a+b}\frac{2i-1}{2i}}{\displaystyle\prod_{i=a}^{a+b-1}\frac{2i-1}{2i}}=\prod_{i=a+b}^{a+b}\frac{2i-1}{2i}=\frac{2a+2b-1}{2a+2b}=1-\frac{1}{2a+2b}\ge$$
$$\ge 1-\frac{1}{2N_1''}> 1-\frac{1}{N_1''}=1-\frac{1}{\left\lceil\frac{1}{\displaystyle\varepsilon_4}\right\rceil}\ge 1-\frac{1}{\frac{1}{\displaystyle\varepsilon_4}}=1-\varepsilon_4.$$

Теперь давайте рассмотрим $a,b\in\mathbb{N}_0$: $a,b\ge 1,$ $\displaystyle(a,b)\in U(\beta,\varepsilon_1'',N_2'')$ ($N_2''$ мы выберем позднее), а также $x\in\mathbb{YF}:$ $|x|\in \{2b,2b+1\}$ и $e(x)\le d$.

Помним, что $\varepsilon_1''\in\left(0,\min\left({\beta^2},{1-\beta^2}\right)\right)$. Пусть $N_2''\in\mathbb{N}_0:$ $N_2''\ge\max\left(1,N_2\left(\beta,\varepsilon_1'',\left\lceil\frac{d}{2}\right\rceil+2\right)\right)$. Тогда по Утверждению \ref{kunibes} (пункт 2) при $\beta\in(0,1),$ $\varepsilon_1''\in\mathbb{R}_{>0},$ $\left(\left\lceil\frac{d}{2}\right\rceil+2\right),a,b,N_2''\in\mathbb{N}_0$ ясно, что если $(a,b)\in U(\beta,\varepsilon_1,N)$, то $b\ge \left\lceil\frac{d}{2}\right\rceil+2$.

Кроме того, ясно, что если $e(x)\le d$, то (так как $|x|\ge 2b$) $d(x)\ge b-\left\lceil\frac{d}{2}\right\rceil$, также ясно, что (так как $a\ge 1$) максимальное значение функции $\pi$ может быть достигнуто, если двойки располагаются слева (из определений функций $\pi$ и $g$). Таким образом (так как $|x|\le 2b+1$), при выбранном $\varepsilon_1''\in \left(0,\min\left({\beta^2},{1-\beta^2}\right)\right)$ если $N_2''\in\mathbb{N}_0:$ $N_2''\ge\max\left(1,N_2\left(\beta,\varepsilon_1'',\left\lceil\frac{d}{2}\right\rceil+2\right)\right)$ и $a,b\in\mathbb{N}_0:$ $a,b\ge 1,$ $(a,b)\in U(\beta,\varepsilon_1'',N_2'')$, то
$$ \frac{\pi\left(x1^{2a}\right)}{\pi\left(2^{b}1^{2a-1}\right)}\le \frac{\pi\left(2^{b-\left\lceil\frac{d}{2}\right\rceil}1^{2\left\lceil\frac{d}{2}\right\rceil+1}1^{2a}\right)}{\pi\left(2^{b}1^{2a-1}\right)}\le$$
$$\le\left(\text{Из определений функций $\pi$ и $g$ и так как $b-\left\lceil\frac{d}{2}\right\rceil\ge 2$}\right)\le$$
$$ \le\frac{\pi\left(2^{b-\left\lceil\frac{d}{2}\right\rceil-1}1^{2\left\lceil\frac{d}{2}\right\rceil+1}1^{2a}\right)}{\pi\left(2^{b}1^{2a-1}\right)}=\frac{\pi\left(2^{b-\left\lceil\frac{d}{2}\right\rceil-1}1^{2\left\lceil\frac{d}{2}\right\rceil+2}1^{2a-1}\right)}{\pi\left(2^{b}1^{2a-1}\right)} =$$
$$=\text{(Из определений функций $\pi$ и $g$)}=$$
$$=
\frac{\displaystyle\prod_{i=a+\left\lceil\frac{d}{2}\right\rceil+1}^{a+b-1}\frac{2i-1}{2i}}{\displaystyle\prod_{i=a}^{a+b-1}\frac{2i-1}{2i}}=\frac{1}{\displaystyle \prod_{i=a}^{a+\left\lceil\frac{d}{2}\right\rceil}\frac{2i-1}{2i}} =\prod_{i=a}^{a+\left\lceil\frac{d}{2}\right\rceil}\frac{2i}{2i-1}\le\left(\frac{2a}{2a-1}\right)^{\left\lceil\frac{d}{2}\right\rceil+1}.$$

Ясно, что $\exists A(\varepsilon_4)\in\mathbb{N}_0:$ при $a\ge A(\varepsilon_4)$
$$\left(\frac{2a}{2a-1}\right)^{\left\lceil\frac{d}{2}\right\rceil+1}<1+\varepsilon_4.$$

Таким образом, $\forall \varepsilon_4\in\mathbb{R}_{>0}$ $\exists A(\varepsilon_4)\in\mathbb{N}_0:$ при выбранных $\varepsilon_1''\in \left(0,\min\left({\beta^2},{1-\beta^2}\right)\right)$,
$N_2''=\max\left(1,N_1(\beta,\varepsilon_1'',A(\varepsilon_4)),N_2\left(\beta,\varepsilon_1'',\left\lceil\frac{d}{2}\right\rceil+2\right)\right)\in\mathbb{N}_0$ и $a,b\in\mathbb{N}_0:$ $a,b\ge 1,$ $(a,b)\in U(\beta,\varepsilon_1'',N_2'')$ по Утверждению \ref{kunibes} (пункт 1) при $\beta\in(0,1),$ $\varepsilon_1''\in\mathbb{R}_{>0},$ $A(\varepsilon_4),a,b,N_2''\in\mathbb{N}_0$ наше выражение меньше, чем $1+\varepsilon_4$.

Таким образом, мы доказали, что при нашем $\beta\in(0,1)$ и $\forall\varepsilon_4\in\mathbb{R}_{>0}$ $\exists \varepsilon_1''\in\mathbb{R}_{>0}:$ $\forall d\in\mathbb{N}_0$ $\exists N''=\max(N_1'',N_2'')\in\mathbb{N}_0:$ $ N''\ge 1,$  $\forall a,b\in\mathbb{N}_0,$ $x\in\mathbb{YF} :$ $a,b\ge 1,$ $(a,b)\in U(\beta,\varepsilon_1'',N''),$ $ |x|\in \{2b,2b+1\}$, $e(x)\le d$
$$1-\varepsilon_4 < \frac{\pi\left(x1^{2a}\right)}{\pi\left(2^{b}1^{2a-1}\right)}<1+\varepsilon_4,$$
что и требовалось.

То есть у нас доказано, что
\begin{itemize}
    \item При нашем $\beta\in(0,1)$ и $\forall\varepsilon_2\in\mathbb{R}_{>0}$ $\exists \varepsilon_1'\in\mathbb{R}_{>0},$ $N'\in\mathbb{N}_0$: $N'\ge 1$, $\forall a,b\in\mathbb{N}_0:$ $a,b\ge 1,$ $ (a,b)\in U(\beta,\varepsilon_1',N')$
$$ \beta-\varepsilon_2<\pi\left(2^b1^{2a-1}\right)<\beta+\varepsilon_2;$$

    \item При нашем $\beta\in(0,1)$ и $\forall\varepsilon_4\in\mathbb{R}_{>0}$ $\exists \varepsilon_1''\in\mathbb{R}_{>0}:$ $\forall d\in\mathbb{N}_0$ $\exists N''\in\mathbb{N}_0:$ $N''\ge 1,$ $\forall a,b\in\mathbb{N}_0,$ $x\in\mathbb{YF} :$ $a,b\ge 1,$ $(a,b)\in U(\beta,\varepsilon_1'',N''),$ $ |x|\in \{2b,2b+1\}$
        $$e(x)\ge d\text{ или } 1-\varepsilon_4<\frac{\pi\left(x1^{2a}\right)}{\pi\left(2^b1^{2a-1}\right)}<1+\varepsilon_4. $$
\end{itemize}

Ясно, что при наших $\varepsilon_3\in\mathbb{R}_{>0}$ и $\beta\in(0,1)$ мы можем выбрать $\varepsilon_2\in(0,\beta),$ $\varepsilon_4\in(0,1)$ так, что 
\begin{itemize}
    \item $$(\beta-\varepsilon_2)(1-\varepsilon_4)>{\beta-\varepsilon_3};$$
    \item $$(\beta+\varepsilon_2)(1+\varepsilon_4)<{\beta+\varepsilon_3}.$$
\end{itemize}

Воспользуемся этими двумя фактами при только что выбранных $\varepsilon_2\in(0,\beta),$ $\varepsilon_4\in(0,1)$ и поймём, что при наших $\beta\in(0,1)$ и $\varepsilon_3\in\mathbb{R}_{>0}$  $\exists  \varepsilon_1=\min(\varepsilon_1',\varepsilon_1'')\in\mathbb{R}_{>0}:$ $\forall d\in\mathbb{N}_0$ $\exists N=\max(N',N'')\in\mathbb{N}_0:$ $ N\ge 1,$ $\forall a,b\in\mathbb{N}_0,$ $x\in\mathbb{YF} :$ $a,b\ge 1,$ $(a,b)\in U(\beta,\varepsilon_1,N),$ $ |x|\in \{2b,2b+1\}$
        $$e(x)\ge d\text{ или } \pi\left(2^b1^{2a-1}\right)\cdot\frac{\displaystyle\pi\left(x1^{2a}\right)}{\displaystyle\pi\left(2^{b}1^{2a-1}\right)}=\pi\left(x1^{2a}\right)\in (\beta-\varepsilon_3,\beta+\varepsilon_3), $$
что и требовалось.

Утверждение доказано.

\end{proof}

\begin{Prop} \label{abbalbisk1}
Пусть $\varepsilon_3\in\mathbb{R}_{>0}$, $\beta\in(0,1)$. Тогда $\exists \varepsilon_1\in\mathbb{R}_{>0}:$ $\forall d\in\mathbb{N}_0$ $\exists N\in\mathbb{N}_0:$ $ N\ge 1,$ $\forall a,b\in\mathbb{N}_0,$ $x\in\mathbb{YF} :$ $a,b\ge 1, $ $(a,b)\in U(\beta,\varepsilon_1,N),$ $ |x|\in \{2b,2b+1\}$
        $$e(x)\ge d\text{ или } \beta-\varepsilon_3<\pi\left(x1^{2a-1}\right)<\beta+\varepsilon_3. $$
\end{Prop}
\begin{proof}

По Утверждению \ref{betasvyaz1} при нашем $\beta\in(0,1)$ и $\forall\varepsilon_2\in\mathbb{R}_{>0}$ $\exists \varepsilon_1'\in\mathbb{R}_{>0} ,$ $N'\in\mathbb{N}_0$: $N'\ge 1$, $\forall a,b\in\mathbb{N}_0:$ $a,b\ge 1,$ $(a,b)\in U(\beta,\varepsilon_1,N')$
$$\beta-\varepsilon_2<\prod_{i=a}^{a+b-1}\frac{2i-1}{2i}< \beta+\varepsilon_2.$$

А это по определению функций $\pi$ и $g$ равносильно тому, что при нашем $\beta\in(0,1)$ и $\forall\varepsilon_2\in\mathbb{R}_{>0}$ $\exists \varepsilon_1'\in\mathbb{R}_{>0} ,$ $N'\in\mathbb{N}_0$: $N'\ge 1$, $\forall a,b\in\mathbb{N}_0:$ $a,b\ge 1,$ $ (a,b)\in U(\beta,\varepsilon_1',N')$
$$ \beta-\varepsilon_2<\pi\left(2^b1^{2a-1}\right)<\beta+\varepsilon_2.$$

Теперь докажем, что при нашем $\beta\in(0,1)$ и $\forall\varepsilon_4\in\mathbb{R}_{>0}$ $\exists \varepsilon_1''\in\mathbb{R}_{>0}:$ $\forall d\in\mathbb{N}_0$ $\exists N''\in\mathbb{N}_0:$ $ N''\ge 1,$ $\forall a,b\in\mathbb{N}_0,$ $x\in\mathbb{YF} :$ $a,b\ge 1,$ $(a,b)\in U(\beta,\varepsilon_1'',N''),$ $ |x|\in \{2b,2b+1\}$
        $$e(x)\ge d\text{ или } 1-\varepsilon_4<\frac{\pi\left(x1^{2a-1}\right)}{\pi\left(2^b1^{2a-1}\right)}<1+\varepsilon_4. $$
        
Зафиксируем произвольное $\varepsilon_4\in\mathbb{R}_{>0}$. Пусть $\varepsilon_1''=\min\left(\frac{\beta^2}{2},\frac{1-\beta^2}{2}\right)\in \left(0,\min\left({\beta^2},{1-\beta^2}\right)\right)\subseteq\mathbb{R}_{>0}$. После чего зафиксируем произвольное $d\in\mathbb{N}_0$. Пусть $N_1''=\left\lceil\frac{1}{\varepsilon_4}\right\rceil$. Ясно, что $N_1''\in\mathbb{N}_0$ и $N_1''\ge 1$.

Давайте рассмотрим $a,b\in\mathbb{N}_0$: $a,b\ge 1,$ $\displaystyle(a,b)\in U(\beta,\varepsilon_1'',N_1'')$. Заметим, что
$$\displaystyle\frac{\pi\left(x1^{2a-1}\right)}{\pi\left(2^{b}1^{2a-1}\right)}\ge\text{(так как $a\ge 1$ и $|x|\le 2b+1$)}\ge \frac{\pi\left(2^{b+1}1^{2a-1}\right)}{\pi\left(2^{b}1^{2a-1}\right)}=$$
$$=(\text{По определению функций $\pi$ и $g$})=$$
$$=\frac{\displaystyle\prod_{i=a}^{a+b}\frac{2i-1}{2i}}{\displaystyle\prod_{i=a}^{a+b-1}\frac{2i-1}{2i}}=\prod_{i=a+b}^{a+b}\frac{2i-1}{2i}=\frac{2a+2b-1}{2a+2b}=1-\frac{1}{2a+2b}\ge$$
$$\ge 1-\frac{1}{2N_1''}> 1-\frac{1}{N_1''}=1-\frac{1}{\left\lceil\frac{1}{\displaystyle\varepsilon_4}\right\rceil}\ge 1-\frac{1}{\frac{1}{\displaystyle\varepsilon_4}}=1-\varepsilon_4.$$

Теперь давайте рассмотрим $a,b\in\mathbb{N}_0$: $a,b\ge 1,$ $\displaystyle(a,b)\in U(\beta,\varepsilon_1'',N_2'')$ ($N_2''$ мы выберем позднее), а также $x\in\mathbb{YF}:$ $|x|\in \{2b,2b+1\}$ и $e(x)\le d$.

Помним, что $\varepsilon_1''\in\left(0,\min\left({\beta^2},{1-\beta^2}\right)\right)$. Пусть $N_2''\in\mathbb{N}_0:$ $N_2''\ge\max\left(1,N_2\left(\beta,\varepsilon_1'',\left\lceil\frac{d}{2}\right\rceil+2\right)\right)$. Тогда по Утверждению \ref{kunibes} (пункт 2) при $\beta\in(0,1),$ $\varepsilon_1''\in\mathbb{R}_{>0},$ $\left(\left\lceil\frac{d}{2}\right\rceil+2\right),a,b,N_2''\in\mathbb{N}_0$ ясно, что если $a,b\in\mathbb{N}_0:$ $a,b\ge 1,$ $(a,b)\in U(\beta,\varepsilon_1,N)$, то $b\ge \left\lceil\frac{d}{2}\right\rceil+2$.

Кроме того, ясно, что если $e(x)\le d$, то (так как $|x|\ge 2b$) $d(x)\ge b-\left\lceil\frac{d}{2}\right\rceil$, также ясно, что (так как $a\ge 1$) максимальное значение функции $\pi$ может быть достигнуто, если двойки располагаются слева (из определений функций $\pi$ и $g$). Таким образом (так как $|x|\le 2b+1$), при выбранном $\varepsilon_1''\in \left(0,\min\left({\beta^2},{1-\beta^2}\right)\right)$ если $N_2''\in\mathbb{N}_0:$ $N_2''\ge\max\left(1,N_2\left(\beta,\varepsilon_1'',\left\lceil\frac{d}{2}\right\rceil+2\right)\right)$ и $a,b\in\mathbb{N}_0:$ $a,b\ge 1,$ $(a,b)\in U(\beta,\varepsilon_1'',N_2'')$, то
$$ \frac{\pi\left(x1^{2a-1}\right)}{\pi\left(2^{b}1^{2a-1}\right)}\le \frac{\pi\left(2^{b-\left\lceil\frac{d}{2}\right\rceil}1^{2\left\lceil\frac{d}{2}\right\rceil+1}1^{2a-1}\right)}{\pi\left(2^{b}1^{2a-1}\right)}\le\frac{\pi\left(2^{b-\left\lceil\frac{d}{2}\right\rceil}1^{2\left\lceil\frac{d}{2}\right\rceil+2}1^{2a-1}\right)}{\pi\left(2^{b}1^{2a-1}\right)} \le$$
$$\le\left(\text{Из определений функций $\pi$ и $g$ и так как $b-\left\lceil\frac{d}{2}\right\rceil\ge 2$}\right)\le$$
$$ \le\frac{\pi\left(2^{b-\left\lceil\frac{d}{2}\right\rceil-1}1^{2\left\lceil\frac{d}{2}\right\rceil+2}1^{2a-1}\right)}{\pi\left(2^{b}1^{2a-1}\right)} =\text{(Из определений функций $\pi$ и $g$)}=$$
$$=
\frac{\displaystyle\prod_{i=a+\left\lceil\frac{d}{2}\right\rceil+1}^{a+b-1}\frac{2i-1}{2i}}{\displaystyle\prod_{i=a}^{a+b-1}\frac{2i-1}{2i}}=\frac{1}{\displaystyle \prod_{i=a}^{a+\left\lceil\frac{d}{2}\right\rceil}\frac{2i-1}{2i}} =\prod_{i=a}^{a+\left\lceil\frac{d}{2}\right\rceil}\frac{2i}{2i-1}\le\left(\frac{2a}{2a-1}\right)^{\left\lceil\frac{d}{2}\right\rceil+1}.$$

Ясно, что $\exists A(\varepsilon_4)\in\mathbb{N}_0:$ при $a\ge A(\varepsilon_4)$
$$\left(\frac{2a}{2a-1}\right)^{\left\lceil\frac{d}{2}\right\rceil+1}<1+\varepsilon_4.$$

Таким образом, $\forall \varepsilon_4\in\mathbb{R}_{>0}$ $\exists A(\varepsilon_4)\in\mathbb{N}_0:$ при выбранных $\varepsilon_1''\in \left(0,\min\left({\beta^2},{1-\beta^2}\right)\right)$,
$N_2''=\max\left(1,N_1(\beta,\varepsilon_1'',A(\varepsilon_4)),N_2\left(\beta,\varepsilon_1'',\left\lceil\frac{d}{2}\right\rceil+2\right)\right)\in\mathbb{N}_0$ и $(a,b)\in U(\beta,\varepsilon_1'',N_2'')$ по Утверждению \ref{kunibes} (пункт 1) при $\beta\in(0,1),$ $\varepsilon_1''\in\mathbb{R}_{>0},$ $A(\varepsilon_4),a,b,N_2''\in\mathbb{N}_0$ наше выражение меньше, чем $1+\varepsilon_4$.

Таким образом, мы доказали, что при нашем $\beta\in(0,1)$ и $\forall\varepsilon_4\in\mathbb{R}_{>0}$ $\exists \varepsilon_1''\in\mathbb{R}_{>0}:$ $\forall d\in\mathbb{N}_0$ $\exists N''=\max(N_1'',N_2'')\in\mathbb{N}_0:$ $ N''\ge 1,$ $\forall a,b\in\mathbb{N}_0,$ $x\in\mathbb{YF} :$ $a,b\ge 1,$ $(a,b)\in U(\beta,\varepsilon_1'',N''),$ $ |x|\in \{2b,2b+1\}$, $e(x)\le d$
$$1-\varepsilon_4 < \frac{\pi\left(x1^{2a-1}\right)}{\pi\left(2^{b}1^{2a-1}\right)}<1+\varepsilon_4,$$
что и требовалось.

То есть у нас доказано, что:
\begin{itemize}
    \item При нашем $\beta\in(0,1)$ и $\forall\varepsilon_2\in\mathbb{R}_{>0}$ $\exists \varepsilon_1'\in\mathbb{R}_{>0},$ $N'\in\mathbb{N}_0$: $N'\ge 1$, $\forall a,b\in\mathbb{N}_0:$ $a,b\ge 1,$ $(a,b)\in U(\beta,\varepsilon_1',N')$
    $$ \beta-\varepsilon_2<\pi\left(2^b1^{2a-1}\right)<\beta+\varepsilon_2;$$

    \item При нашем $\beta\in(0,1)$ и $\forall\varepsilon_4\in\mathbb{R}_{>0}$ $\exists \varepsilon_1''\in\mathbb{R}_{>0}:$ $\forall d\in\mathbb{N}_0$ $\exists N''\in\mathbb{N}_0:$ $ N''\ge 1,$ $\forall a,b\in\mathbb{N}_0,$ $x\in\mathbb{YF}:$ $a,b\ge 1,$ $(a,b)\in U(\beta,\varepsilon_1'',N''),$ $|x|\in \{2b,2b+1\}$
    $$e(x)\ge d\text{ или } 1-\varepsilon_4<\frac{\pi\left(x1^{2a-1}\right)}{\pi\left(2^b1^{2a-1}\right)}<1+\varepsilon_4. $$

\end{itemize}

Ясно, что при наших $\varepsilon_3\in\mathbb{R}_{>0}$ и $\beta\in(0,1)$ мы можем выбрать $\varepsilon_2\in(0,\beta),$ $\varepsilon_4\in(0,1)$ так, что 
\begin{itemize}
    \item $$(\beta-\varepsilon_2)(1-\varepsilon_4)>{\beta-\varepsilon_3};$$
    \item $$(\beta+\varepsilon_2)(1+\varepsilon_4)<{\beta+\varepsilon_3}.$$
\end{itemize}

Воспользуемся этими двумя фактами при только что выбранных $\varepsilon_2\in(0,\beta),$ $\varepsilon_4\in(0,1)$ и поймём, что при наших $\beta\in(0,1)$ и $\varepsilon_3\in\mathbb{R}_{>0}$  $\exists \varepsilon_1=\min(\varepsilon_1',\varepsilon_1'')\in\mathbb{R}_{>0}:$ $\forall d\in\mathbb{N}_0$ $\exists N=\max(N',N'')\in\mathbb{N}_0:$ $ N\ge 1,$ $\forall a,b\in\mathbb{N}_0,$ $x\in\mathbb{YF}:$ $a,b\ge 1,$ $(a,b)\in U(\beta,\varepsilon_1,N),$ $|x|\in \{2b,2b+1\}$
    $$e(x)\ge d\text{ или } \pi\left(2^b1^{2a-1}\right)\cdot\frac{\displaystyle\pi\left(x1^{2a-1}\right)}{\displaystyle\pi\left(2^{b}1^{2a-1}\right)}=\pi\left(x1^{2a-1}\right)\in (\beta-\varepsilon_3,\beta+\varepsilon_3), $$
что и требовалось.

Утверждение доказано.
\end{proof}

\begin{Prop} \label{sum1}
Пусть $ w\in\mathbb{YF}_\infty$, $\beta\in(0,1]$, $n,y\in\mathbb{N}_0:$ $y\le n$. Тогда
$${\sum_{x\in K(n,y)} \left(d(\varepsilon,x)\cdot  q(x(y))\cdot {d_1'(x'(y),w)}\cdot \beta^y\cdot \left(1-\beta^2\right)^{{\frac{n-y}{2}}}\right)}=\sum_{x''\in\mathbb{YF}_{n-y} }\left(q(x'')\cdot d(\varepsilon,x''1^y)\cdot \beta^y  \cdot \left(1-\beta^2\right)^{\frac{n-y}{2}}\right).$$
\end{Prop}
\begin{proof}
Рассмотрим выражение
$${\sum_{x\in K(n,y)} \left(d(\varepsilon,x)\cdot  q(x(y))\cdot {d_1'(x'(y),w)}\cdot \beta^y\cdot \left(1-\beta^2\right)^{{\frac{n-y}{2}}}\right)}.$$

Ясно, что в каждом слагаемом по Замечанию \ref{kerambus} при  $x\in\mathbb{YF}$, $n,y\in\mathbb{N}_0$
$$x=x(y)x'(y).$$

А значит можно воспользоваться Утверждением \ref{razbivaem} при $x,x(y),x'(y)\in\mathbb{YF}$ и получить, что наше выражение равняется следующему:
$$\sum_{x\in K(n,y)}\left(d(\varepsilon,x'(y))\cdot d\left(\varepsilon,x(y)1^{|x'(y)|}\right)\cdot q(x(y))\cdot {{d_1'(x'(y),w)}}\cdot \beta^y\left(1-\beta^2\right)^{\frac{n-y}{2}}\right)=$$
$$=(\text{По обозначению $x'(y)$})=$$
$$=\sum_{x\in K(n,y)}\left(d(\varepsilon,x'(y))\cdot d\left(\varepsilon,x(y)1^{y}\right)\cdot q(x(y))\cdot {{d_1'(x'(y),w)}}\cdot \beta^y\left(1-\beta^2\right)^{\frac{n-y}{2}}\right).$$

Заметим, что при $n,y\in\mathbb{N}_0:$ $ y\le n$ 
\begin{itemize}
    \item если $x\in K(n,y),$ то $x=x(y)x'(y)$, причём $x(y)\in\mathbb{YF}_{n-y},$ $x'(y)\in\mathbb{YF}_{y}$;
    \item если $x_1,x_2\in K(n,y)$: $x_1\ne x_2$, то $x_1(y)\ne x_2(y)$ или
    $x_1'(y)\ne x'_2(y)$;
    \item если $x''\in \mathbb{YF}_{n-y},$ $ x'''\in\mathbb{YF}_{y}$, то $\left(x''x'''\right)\in K(n,y)$, $\left(x''x'''\right)(y)=x'',$ $\left(x''x'''\right)'(y)=x'''$.
\end{itemize}

А это значит, что при всех $x\in K(n,y)$ пара $(x(y),x'(y))$ принимает все значения в $\mathbb{YF}_{n-y}\times \mathbb{YF}_{y}$, причём ровно по одному разу.

А это значит, что наше выражение равняется следующему:
$$\sum_{x''\in \mathbb{YF}_{n-y}}\left(\sum_{x'''\in \mathbb{YF}_{y}}\left(d(\varepsilon,x''')\cdot d\left(\varepsilon,x''1^{y}\right)\cdot q(x'')\cdot {{d_1'(x''',w)}}\cdot \beta^y\cdot \left(1-\beta^2\right)^{\frac{n-y}{2}}\right)\right)=$$
$$=\left(\sum_{x''\in \mathbb{YF}_{n-y}}\left(q(x'')\cdot d\left(\varepsilon,x''1^{y}\right)\cdot \beta^y\left(1-\beta^2\right)^{\frac{n-y}{2}}\right)\right)\sum_{x'''\in \mathbb{YF}_{y}}\left(d(\varepsilon,x''')\cdot {{d_1'(x''',w)}}\right)=$$
$$=(\text{По Утверждению \ref{limitstrih} при $x'''\in\mathbb{YF}$, $w\in\mathbb{YF}_{\infty}$})=$$
$$=\left(\sum_{x''\in \mathbb{YF}_{n-y}}\left(q(x'')\cdot d\left(\varepsilon,x''1^{y}\right)\cdot \beta^y\cdot \left(1-\beta^2\right)^{\frac{n-y}{2}}\right)\right)\sum_{x'''\in \mathbb{YF}_{y}}\left(d(\varepsilon,x''')\cdot \lim_{m \to \infty}{\frac{d(x''',w_m)}{d(\varepsilon,w_m)}}\right)=$$
$$=\left(\sum_{x''\in \mathbb{YF}_{n-y}}\left(q(x'')\cdot d\left(\varepsilon,x''1^{y}\right)\cdot \beta^y\cdot \left(1-\beta^2\right)^{\frac{n-y}{2}}\right)\right)\lim_{m \to \infty}\left(\sum_{x'''\in \mathbb{YF}_{y}}\left(d(\varepsilon,x'''){\frac{d(x''',w_m)}{d(\varepsilon,w_m)}}\right)\right).$$

Заметим, что по Утверждению \ref{mera}, при $w\in\mathbb{YF}_\infty,$ $y,m\in\mathbb{N}_0$ если $|w_m|\ge y$, то 
$$\sum_{x'''\in\mathbb{YF}_y } \left(d(\varepsilon,x''')\frac{d(x''',w_m)}{d(\varepsilon,w_m)}\right)=1.$$

А значит если $m\ge y$, то $|w_m|\ge m\ge y $, то есть
$$\sum_{x'''\in\mathbb{YF}_y } \left(d(\varepsilon,x''')\frac{d(x''',w_m)}{d(\varepsilon,w_m)}\right)=1,$$
а значит
$$\lim_{m\to\infty}\left(\sum_{x'''\in\mathbb{YF}_y } \left(d(\varepsilon,x''')\frac{d(x''',w_m)}{d(\varepsilon,w_m)}\right)\right)=1.$$

Таким образом, наше выражение равняется следующему:
$$\left(\sum_{x''\in \mathbb{YF}_{n-y}}\left(q(x'')\cdot d\left(\varepsilon,x''1^{y}\right)\cdot \beta^y\cdot \left(1-\beta^2\right)^{\frac{n-y}{2}}\right)\right)\cdot1=$$
$$=\sum_{x''\in \mathbb{YF}_{n-y}}\left(q(x'')\cdot d\left(\varepsilon,x''1^{y}\right)\cdot \beta^y\cdot \left(1-\beta^2\right)^{\frac{n-y}{2}}\right),$$
что и требовалось.

Утверждение доказано.
\end{proof}

\begin{Lemma} \label{planedead}
Пусть $w\in\mathbb{YF}_\infty$, $\beta\in(0,1)$, $\varepsilon,\varepsilon_3\in\mathbb{R}_{>0}$. Тогда $\exists\varepsilon_1'\in\left(0,\beta^2\right):$ $\forall \overline{\varepsilon'}\in\mathbb{R}_{>0}$ $\exists N''\in\mathbb{N}_0:$ $\forall n\in\mathbb{N}_0:$ $n\ge N''$
$${\sum_{y\in \overline{n}[\beta,\varepsilon_1']}\left(\left({\sum_{v\in \overline{R'}(w,\beta,n,\varepsilon,y,\varepsilon_3)} T_{w,\beta,n}(v,y)}\right)+\left({\sum_{v\in \overline{R''}(w,\beta,n,\varepsilon,y)} T_{w,\beta,n}(v,y)}\right)\right) <\overline{\varepsilon'}}.$$
\end{Lemma}
\begin{proof}

Зафиксируем $n\in\mathbb{N}_0$ и $\varepsilon_1'\in \left(0,\beta^2\right)$ и воспользуемся определением волшебных таблиц:
$${\sum_{y\in \overline{n}[\beta,\varepsilon_1']}\left(\left({\sum_{v\in \overline{R'}(w,\beta,n,\varepsilon,y,\varepsilon_3)} T_{w,\beta,n}(v,y)}\right)+\left({\sum_{v\in \overline{R''}(w,\beta,n,\varepsilon,y)} T_{w,\beta,n}(v,y)}\right)\right)}=$$
$$=\left(\text{так как $\overline{R''}(w,\beta,n,\varepsilon,y)\subseteq \overline{K}(n,y)$}\right)=$$
$$={\sum_{y\in \overline{n}[\beta,\varepsilon_1']}\left(\left({\sum_{v\in \overline{R'}(w,\beta,n,\varepsilon,y,\varepsilon_3)} T_{w,\beta,n}(v,y)}\right)+0\right)}={\sum_{y\in \overline{n}[\beta,\varepsilon_1']}\left({\sum_{v\in \overline{R'}(w,\beta,n,\varepsilon,y,\varepsilon_3)} T_{w,\beta,n}(v,y)}\right)}=$$
$$=\left(\text{так как $\overline{R'}(w,\beta,n,\varepsilon,y,\varepsilon_3)\subseteq {K}(n,y)$}\right)=$$
$$={\sum_{y\in \overline{n}[\beta,\varepsilon_1']}\left({\sum_{v\in \overline{R'}(w,\beta,n,\varepsilon,y,\varepsilon_3)} \left(d(\varepsilon,v)\cdot q(v(y))\cdot {d_1'(v'(y),w)}\cdot \beta^y\cdot \left(1-\beta^2\right)^{\#v(y)}\right)}\right)}.$$

Ясно, что если $v\in\overline{R'}(w,\beta,n,\varepsilon,y,\varepsilon_3)$, то $|v(y)|=n-y$, а значит $\displaystyle n-y=|v(y)|=2d(v(y))+e(v(y))=2d(v(y))+2e(v(y))-e(v(y))=2(d(v(y))+e(v(y)))-e(v(y))=2\#(v(y))-e(v(y)) \Longrightarrow \#(v(y))=\frac{n-y+e(v(y))}{2}$. Таким образом, наше выражение равняется следующему:
$${\sum_{y\in \overline{n}[\beta,\varepsilon_1']}\left({\sum_{v\in \overline{R'}(w,\beta,n,\varepsilon,y,\varepsilon_3)} \left(d(\varepsilon,v)\cdot q(v(y))\cdot {d_1'(v'(y),w)}\cdot \beta^y\cdot \left(1-\beta^2\right)^{\frac{n-y+e(v(y))}{2}}\right)}\right)}=$$
$$={\sum_{y\in \overline{n}[\beta,\varepsilon_1']}\left({\sum_{v\in \overline{R'}(w,\beta,n,\varepsilon,y,\varepsilon_3)} \left(d(\varepsilon,v)\cdot q(v(y))\cdot {d_1'(v'(y),w)}\cdot \beta^y\cdot \left(1-\beta^2\right)^{{\frac{n-y}{2}}}\cdot \left(1-\beta^2\right)^{{\frac{e(v(y))}{2}}}\right)}\right)}.$$

\begin{Oboz}
Пусть $w\in\mathbb{YF}_\infty$, $\beta\in(0,1]$, $\varepsilon,\varepsilon_3,\varepsilon_1'\in\mathbb{R}_{>0}$, $n\in\mathbb{N}_0$. Тогда:
\begin{itemize}
    \item Если $\exists y\in\mathbb{N}_0,$ $v\in\mathbb{YF}:$ $y\in \overline{n}[\beta,\varepsilon_1']$, $v\in \overline{R'}(w,\beta,n,\varepsilon,y,\varepsilon_3)$, то
    $$\max^0_{\begin{smallmatrix}y\in \overline{n}[\beta,\varepsilon_1'],\\{v\in \overline{R'}(w,\beta,n,\varepsilon,y,\varepsilon_3)}\end{smallmatrix}} \left(1-\beta^2\right)^{{\frac{e(v(y))}{2}}}=\max_{\begin{smallmatrix}y\in \overline{n}[\beta,\varepsilon_1'],\\{v\in \overline{R'}(w,\beta,n,\varepsilon,y,\varepsilon_3)}\end{smallmatrix}} \left(1-\beta^2\right)^{{\frac{e(v(y))}{2}}};$$
    \item Если $\nexists y\in\mathbb{N}_0,$ $v\in\mathbb{YF}:$ $y\in \overline{n}[\beta,\varepsilon_1']$, $v\in \overline{R'}(w,\beta,n,\varepsilon,y,\varepsilon_3)$, то
    $$\max^0_{\begin{smallmatrix}y\in \overline{n}[\beta,\varepsilon_1'],\\{v\in \overline{R'}(w,\beta,n,\varepsilon,y,\varepsilon_3)}\end{smallmatrix}} \left(1-\beta^2\right)^{{\frac{e(v(y))}{2}}}=0.$$
\end{itemize}
\end{Oboz}

\renewcommand{\labelenumi}{\alph{enumi}$)$}
\renewcommand{\labelenumii}{\arabic{enumi}.\arabic{enumii}$^\circ$}
\renewcommand{\labelenumiii}{\arabic{enumi}.\arabic{enumii}.\arabic{enumiii}$^\circ$}

Далее рассмотрим два случая. Ясно, что:
\begin{enumerate}
    \item  Если $\nexists y\in\mathbb{N}_0,$ $v\in\mathbb{YF}$: $y\in \overline{n}[\beta,\varepsilon_1']$, $v\in \overline{R'}(w,\beta,n,\varepsilon,y,\varepsilon_3)$, то
    $${\sum_{y\in \overline{n}[\beta,\varepsilon_1']}\left({\sum_{v\in \overline{R'}(w,\beta,n,\varepsilon,y,\varepsilon_3)} \left(d(\varepsilon,v)\cdot q(v(y))\cdot {d_1'(v'(y),w)}\cdot \beta^y\cdot \left(1-\beta^2\right)^{{\frac{n-y}{2}}}\cdot \left(1-\beta^2\right)^{{\frac{e(v(y))}{2}}}\right)}\right)}=$$
    $$=0=0\cdot0={\sum_{y\in \overline{n}[\beta,\varepsilon_1']}\left({\sum_{v\in \overline{R'}(w,\beta,n,\varepsilon,y,\varepsilon_3)} \left(d(\varepsilon,v)\cdot q(v(y))\cdot {d_1'(v'(y),w)}\cdot \beta^y\cdot \left(1-\beta^2\right)^{{\frac{n-y}{2}}}\right)}\right)}\cdot$$
    $$\cdot\max^0_{\begin{smallmatrix}y\in \overline{n}[\beta,\varepsilon_1'],\\{v\in \overline{R'}(w,\beta,n,\varepsilon,y,\varepsilon_3)}\end{smallmatrix}} \left(1-\beta^2\right)^{{\frac{e(v(y))}{2}}}.$$

    \item  Если $\exists y\in\mathbb{N}_0,$ $v\in\mathbb{YF}:$ $y\in \overline{n}[\beta,\varepsilon_1']$, $v\in \overline{R'}(w,\beta,n,\varepsilon,y,\varepsilon_3)$, то    
    $${\sum_{y\in \overline{n}[\beta,\varepsilon_1']}\left({\sum_{v\in \overline{R'}(w,\beta,n,\varepsilon,y,\varepsilon_3)} \left(d(\varepsilon,v)\cdot q(v(y))\cdot {d_1'(v'(y),w)}\cdot \beta^y\cdot \left(1-\beta^2\right)^{{\frac{n-y}{2}}}\cdot \left(1-\beta^2\right)^{{\frac{e(v(y))}{2}}}\right)}\right)}\le$$
    $$\le{\sum_{y\in \overline{n}[\beta,\varepsilon_1']}\left({\sum_{v\in \overline{R'}(w,\beta,n,\varepsilon,y,\varepsilon_3)} \left(d(\varepsilon,v)\cdot q(v(y))\cdot {d_1'(v'(y),w)}\cdot \beta^y\cdot \left(1-\beta^2\right)^{{\frac{n-y}{2}}} \right)}\right)}\cdot$$
    $$\cdot\max_{\begin{smallmatrix}y\in \overline{n}[\beta,\varepsilon_1'],\\{v\in \overline{R'}(w,\beta,n,\varepsilon,y,\varepsilon_3)}\end{smallmatrix}} \left(1-\beta^2\right)^{{\frac{e(v(y))}{2}}}=$$
    $$={\sum_{y\in \overline{n}[\beta,\varepsilon_1']}\left({\sum_{v\in \overline{R'}(w,\beta,n,\varepsilon,y,\varepsilon_3)} \left(d(\varepsilon,v)\cdot q(v(y))\cdot {d_1'(v'(y),w)}\cdot \beta^y\cdot \left(1-\beta^2\right)^{{\frac{n-y}{2}}} \right)}\right)}\cdot$$
    $$\cdot\max^0_{\begin{smallmatrix}y\in \overline{n}[\beta,\varepsilon_1'],\\{v\in \overline{R'}(w,\beta,n,\varepsilon,y,\varepsilon_3)}\end{smallmatrix}} \left(1-\beta^2\right)^{{\frac{e(v(y))}{2}}}.$$
\end{enumerate}

А значит в любом из двух случаев (ясно, что других нет) наше выражение не превосходит следующее:
    $${\sum_{y\in \overline{n}[\beta,\varepsilon_1']}\left({\sum_{v\in \overline{R'}(w,\beta,n,\varepsilon,y,\varepsilon_3)} \left(d(\varepsilon,v)\cdot q(v(y))\cdot {d_1'(v'(y),w)}\cdot \beta^y\cdot \left(1-\beta^2\right)^{{\frac{n-y}{2}}} \right)}\right)}\cdot$$
    $$\cdot\max^0_{\begin{smallmatrix}y\in \overline{n}[\beta,\varepsilon_1'],\\{v\in \overline{R'}(w,\beta,n,\varepsilon,y,\varepsilon_3)}\end{smallmatrix}} \left(1-\beta^2\right)^{{\frac{e(v(y))}{2}}}\le$$
    $$\le\left(\text{так как если $w\in\mathbb{YF}_\infty$, $\beta\in(0,1)$, $\varepsilon,\varepsilon_3\in\mathbb{R}_{>0}$,  $n,y\in\mathbb{N}_0:$ $y\le n$, то $\overline{R'}(w,\beta,n,\varepsilon,y,\varepsilon_3)\subseteq K(n,y)$}\right)\le$$
    $$\le{\sum_{y\in \overline{n}[\beta,\varepsilon_1']}\left({\sum_{v\in K(n,y)} \left(d(\varepsilon,v)\cdot q(v(y))\cdot {d_1'(v'(y),w)}\cdot \beta^y\cdot \left(1-\beta^2\right)^{{\frac{n-y}{2}}} \right)}\right)}
    \cdot\max^0_{\begin{smallmatrix}y\in \overline{n}[\beta,\varepsilon_1'],\\{v\in \overline{R'}(w,\beta,n,\varepsilon,y,\varepsilon_3)}\end{smallmatrix}} \left(1-\beta^2\right)^{{\frac{e(v(y))}{2}}} \le$$
$$\le\left(\text{так как если $n\in\mathbb{N}_0,$ $\beta\in(0,1)$, $\varepsilon_1'\in\mathbb{R}_{>0}$, то $\overline{n}[\beta,\varepsilon_1']\subseteq \overline{n}$}\right)\le$$
    $$\le{\sum_{y=0}^n\left({\sum_{v\in K(n,y)} \left(d(\varepsilon,v)\cdot q(v(y))\cdot {d_1'(v'(y),w)}\cdot \beta^y\cdot \left(1-\beta^2\right)^{{\frac{n-y}{2}}} \right)}\right)}
    \cdot\max^0_{\begin{smallmatrix}y\in \overline{n}[\beta,\varepsilon_1'],\\{v\in \overline{R'}(w,\beta,n,\varepsilon,y,\varepsilon_3)}\end{smallmatrix}} \left(1-\beta^2\right)^{{\frac{e(v(y))}{2}}} =$$
    $$=\left(\text{По Утверждению \ref{sum1} при наших $w\in\mathbb{YF}_\infty$, $\beta\in(0,1]$, $n\in\mathbb{N}_0$, просуммированному по  $y\in\overline{n}$}\right)=$$
$$=\sum_{y=0}^n\left(\sum_{x''\in\mathbb{YF}_{n-y} }\left(q(x'')\cdot  d(\varepsilon,x''1^y)\cdot \beta^y\cdot   \left(1-\beta^2\right)^{\frac{n-y}{2}}\right)\right)\cdot\max^0_{\begin{smallmatrix}y\in \overline{n}[\beta,\varepsilon_1'],\\{v\in \overline{R'}(w,\beta,n,\varepsilon,y,\varepsilon_3)}\end{smallmatrix}} \left(1-\beta^2\right)^{{\frac{e(v(y))}{2}}} =$$
$$=\sum_{y=0}^n\left(\beta^y\cdot  \left(1-\beta^2\right)^{\frac{n-y}{2}}\cdot\sum_{x''\in\mathbb{YF}_{n-y} }\left(q(x'')d(\varepsilon,x''1^y)\right)\right)\cdot\max^0_{\begin{smallmatrix}y\in \overline{n}[\beta,\varepsilon_1'],\\{v\in \overline{R'}(w,\beta,n,\varepsilon,y,\varepsilon_3)}\end{smallmatrix}} \left(1-\beta^2\right)^{{\frac{e(v(y))}{2}}}.$$

По Утверждению \ref{dostalo} при зафиксированном нами $n\in\mathbb{N}_0$, просуммированному по $y\in\overline{n}$, это выражение равняется следующему:
$$\sum_{y=0}^n\left(\beta^y \cdot \left(1-\beta^2\right)^{\frac{n-y}{2}}\cdot\prod_{i=1}^{\left\lfloor \frac{n-y}{2} \right\rfloor} \frac{2i+y}{2i}\right)\cdot\max^0_{\begin{smallmatrix}y\in \overline{n}[\beta,\varepsilon_1'],\\{v\in \overline{R'}(w,\beta,n,\varepsilon,y,\varepsilon_3)}\end{smallmatrix}} \left(1-\beta^2\right)^{{\frac{e(v(y))}{2}}}=$$
$$=\sum_{y=0}^n\left(\left(\prod_{i=1}^{\left\lfloor \frac{n-y}{2} \right\rfloor} \frac{2i+y}{2i}\right)\beta^y\cdot  \left(1-\beta^2\right)^{\frac{n-y}{2}}\right)\cdot\max^0_{\begin{smallmatrix}y\in \overline{n}[\beta,\varepsilon_1'],\\{v\in \overline{R'}(w,\beta,n,\varepsilon,y,\varepsilon_3)}\end{smallmatrix}} \left(1-\beta^2\right)^{{\frac{e(v(y))}{2}}}\le$$
$$\le\sum_{y=0}^n\left(\left(\prod_{i=1}^{\left\lfloor \frac{n-y}{2} \right\rfloor} \frac{2i+y}{2i}\right)\beta^y\cdot  \left(1-\beta^2\right)^{\left\lfloor\frac{n-y}{2}\right\rfloor}\right)\cdot\max^0_{\begin{smallmatrix}y\in \overline{n}[\beta,\varepsilon_1'],\\{v\in \overline{R'}(w,\beta,n,\varepsilon,y,\varepsilon_3)}\end{smallmatrix}} \left(1-\beta^2\right)^{{\frac{e(v(y))}{2}}}\le$$
$$\le \text{(По Утверждению \ref{lehamed} при наших $\beta\in (0,1)$ и $n\in\mathbb{N}_0$)}\le $$
$$\le\left(1+\frac{1}{\beta}\right)\cdot\max^0_{\begin{smallmatrix}y\in \overline{n}[\beta,\varepsilon_1'],\\{v\in \overline{R'}(w,\beta,n,\varepsilon,y,\varepsilon_3)}\end{smallmatrix}} \left(1-\beta^2\right)^{{\frac{e(v(y))}{2}}}.$$

    Что значит, что $v\in \overline{R'}(w,\beta,n,\varepsilon,y,\varepsilon_3)$ при $y\in \overline{n}[\beta,\varepsilon_1']$?
    
    Это значит, что $v\in K(n,y)$, $\pi(v)\notin(\pi(w)(\beta-\varepsilon),\pi(w)(\beta+\varepsilon))$, $\pi_y(v)\notin(\beta-\varepsilon_3,\beta+\varepsilon_3)$, $y\in\left(\overline{n}\cap \left(n\left(\beta^2-\varepsilon_1'\right),n\left(\beta^2+\varepsilon_1'\right)\right)\right)$.
    
    Сначала рассмотрим чётные игреки. 

        Мы знаем, что по Утверждению \ref{abbalbisk} при наших $\varepsilon_3\in\mathbb{R}_{>0}$ и $\beta\in(0,1)$ $\exists \varepsilon_1\in\mathbb{R}_{>0}:$ $\forall d\in\mathbb{N}_0$ $\exists N\in\mathbb{N}_0:$ $ N\ge 1,$ $\forall a,b\in\mathbb{N}_0,$ $x\in\mathbb{YF} :$ $a,b\ge 1,$ $(a,b)\in U(\beta,\varepsilon_1,N),$ $ |x|\in \{2b,2b+1\}$
        $$e(x)\ge d\text{ или } \beta-\varepsilon_3<\pi\left(x1^{2a}\right)<\beta+\varepsilon_3. $$
        
        Рассмотрим $\varepsilon_1\in\mathbb{R}_{>0}$, удовлетворяющее условию Утверждения, зафиксируем какое-нибудь $d\in\mathbb{N}_0$, для него рассмотрим $N\in\mathbb{N}_0:$ $N\ge 1$, удовлетворяющее условию Утверждения. 
        
        Теперь пусть $\varepsilon_{10}=\min\left(\frac{\beta^2}{2},\frac{1-\beta^2}{2},\frac{\varepsilon_{1}}{2}\right)\in\left(0,\beta^2\right)$, $n\in\mathbb{N}_0:$ $ n\ge N_0:= \max\left(2N+2,\left\lceil\frac{\displaystyle 2\left(\beta^2+\frac{\displaystyle\varepsilon_1}{\displaystyle 2}\right)}{\displaystyle \frac{\displaystyle\varepsilon_1}{\displaystyle 2}}+3\right\rceil,\frac{\displaystyle 4}{\displaystyle 1-\beta^2}\right)$ и $v\in \overline{R'}(w,\beta,n,\varepsilon,y,\varepsilon_3)$ при $y\in \overline{n}[\beta,\varepsilon_{10}](2,0)$.
        
        
        Наконец, пусть $a=\frac{y}{2}$, $b=\left\lfloor\frac{n-y}{2}\right\rfloor$, $x=v(y)$ ($v\in \overline{R'}(w,\beta,n,\varepsilon,y,\varepsilon_3)\subseteq K(n,y)\Longrightarrow v(y)$ существует).
        
        Тогда заметим, что 
        \begin{itemize}
            \item $$a= \frac{y}{2}\ge  \left(\text{так как $y\in \overline{n}[\beta,\varepsilon_{10}]\Longleftrightarrow y\in \left(\overline{n}\cap \left(n\left(\beta^2-\varepsilon_{10}\right),n\left(\beta^2+\varepsilon_{10}\right)\right)\right)$}\right) \ge$$
            $$\ge n\left(\beta^2-\varepsilon_{10}\right)\ge  n\left(\beta^2-\frac{\beta^2}{2}\right)=\frac{n\beta^2}{2}>0\Longrightarrow\left(\text{так как $a\in\mathbb{N}_0$}\right)\Longrightarrow a\ge 1;$$
            \item $$b=\left\lfloor\frac{n-y}{2}\right\rfloor \ge  \left(\text{так как $y\in \overline{n}[\beta,\varepsilon_{10}]\Longleftrightarrow y\in \left(\overline{n}\cap \left(n\left(\beta^2-\varepsilon_{10}\right),n\left(\beta^2+\varepsilon_{10}\right)\right)\right)$}\right) \ge$$
            $$\ge \left\lfloor\frac{n-n\left(\beta^2+\varepsilon_{10}\right)}{2}\right\rfloor\ge \left\lfloor\frac{n\left(1-\beta^2-\varepsilon_{10}\right)}{2}\right\rfloor\ge \left\lfloor\frac{n\left(1-\beta^2-\frac{1-\beta^2}{2}\right)}{2}\right\rfloor=$$
            $$=\left\lfloor n\frac{\left(1-\beta^2\right)}{4}\right\rfloor\ge \left\lfloor\frac{4}{1-\beta^2}\cdot\frac{\left(1-\beta^2\right)}{4}\right\rfloor=\left\lfloor1\right\rfloor=1;$$
            \item $$a+b=\frac{y}{2}+\left\lfloor\frac{n-y}{2}\right\rfloor\ge\frac{y}{2}+\left(\frac{n-y}{2}-1\right)=\frac{n}{2}-1\ge \frac{N_0}{2}-1\ge \frac{2N+2}{2}-1  =N;$$
            \item $$\frac{a}{a+b}=\frac{\frac{y}{2}}{\frac{y}{2}+\left\lfloor\frac{n-y}{2}\right\rfloor}=\frac{y}{y+2\left\lfloor\frac{n-y}{2}\right\rfloor}\ge \frac{y}{y+2\left(\frac{n-y}{2}\right)}=\frac{y}{n}\ge$$
            $$\ge \left(\text{так как $y\in \overline{n}[\beta,\varepsilon_{10}]\Longleftrightarrow y\in \left(\overline{n}\cap \left(n\left(\beta^2-\varepsilon_{10}\right),n\left(\beta^2+\varepsilon_{10}\right)\right)\right)$}\right) \ge $$
            $$\ge \frac{n\left(\beta^2-\varepsilon_{10}\right)}{n}=\beta^2-\varepsilon_{10}\ge\beta^2-\frac{\varepsilon_{1}}{2}>\beta^2-{\varepsilon_{1}};$$
            \item $$\frac{a}{a+b}=\frac{\frac{y}{2}}{\frac{y}{2}+\left\lfloor\frac{n-y}{2}\right\rfloor}=\frac{y}{y+2\left\lfloor\frac{n-y}{2}\right\rfloor}\le \frac{y}{y+2\left(\frac{n-y}{2}-1\right)}=\frac{y}{n-2}\le$$
            $$\le \left(\text{так как $y\in \overline{n}[\beta,\varepsilon_{10}]\Longleftrightarrow y\in \left(\overline{n}\cap \left(n\left(\beta^2-\varepsilon_{10}\right),n\left(\beta^2+\varepsilon_{10}\right)\right)\right)$}\right) \le $$
            $$\le \frac{n\left(\beta^2+\varepsilon_{10}\right)}{n-2}=\frac{(n-2)\left(\beta^2+\varepsilon_{10}\right)}{n-2}+\frac{2\left(\beta^2+\varepsilon_{10}\right)}{n-2}=\beta^2+\varepsilon_{10}+\frac{2\left(\beta^2+\varepsilon_{10}\right)}{n-2}\le$$
            $$\le\beta^2+\frac{ \displaystyle\varepsilon_1}{\displaystyle 2}+\frac{2\left(\beta^2+\frac{ \displaystyle\varepsilon_1}{\displaystyle 2}\right)}{n-2}\le \beta^2+\frac{ \displaystyle\varepsilon_1}{\displaystyle 2}+\frac{2\left(\beta^2+\frac{ \displaystyle\varepsilon_1}{\displaystyle 2}\right)}{N_0-2} \le  \beta^2+\frac{ \displaystyle\varepsilon_1}{\displaystyle 2}+\frac{2\left(\beta^2+\frac{ \displaystyle\varepsilon_1}{\displaystyle 2}\right)}{\frac{\displaystyle 2\left(\beta^2+\frac{\displaystyle\varepsilon_1}{\displaystyle 2}\right)}{\displaystyle \frac{\displaystyle\varepsilon_1}{\displaystyle 2}}+3-2}<$$
            $$<\beta^2+\frac{ \displaystyle\varepsilon_1}{\displaystyle 2}+\frac{2\left(\beta^2+\frac{ \displaystyle\varepsilon_1}{\displaystyle 2}\right)}{\frac{\displaystyle 2\left(\beta^2+\frac{\displaystyle\varepsilon_1}{\displaystyle 2}\right)}{\displaystyle \frac{\displaystyle\varepsilon_1}{\displaystyle 2}}}=\beta^2+\frac{ \displaystyle\varepsilon_1}{\displaystyle 2}+\frac{ \displaystyle\varepsilon_1}{\displaystyle 2}=\beta^2+\varepsilon_1;$$
            \item  Если $(n-y) \;mod\;{2} = 0$, то 
            $$|x|=n-y=2\frac{n-y}{2}=2\left\lfloor\frac{n-y}{2}\right\rfloor=2b;$$
            \item Если $(n-y) \;mod\;{2} = 1$, то 
            $$|x|=n-y=2\frac{n-y-1}{2}+1=2\left\lfloor\frac{n-y}{2}\right\rfloor+1=2b+1.$$
        \end{itemize}
        
        То есть $a,b\in\mathbb{N}_0,$ $x\in\mathbb{YF} :$ $a,b\ge 1,$ $(a,b)\in U(\beta,\varepsilon_1,N),$ $ |x|\in \{2b,2b+1\}$.
                
        А это по Утверждению \ref{abbalbisk} значит, что 
        $$e(x)\ge d\text{ или } \beta-\varepsilon_3<\pi\left(x1^{2a}\right)<\beta+\varepsilon_3, $$
        то есть
        $$e(v(y))\ge d\text{ или } \beta-\varepsilon_3<\pi\left(v(y)1^{2a}\right)<\beta+\varepsilon_3\Longleftrightarrow $$
        $$\Longleftrightarrow e(v(y))\ge d\text{ или } \beta-\varepsilon_3<\pi\left(v(y)1^{2\frac{y}{2}}\right)<\beta+\varepsilon_3\Longleftrightarrow $$
        $$\Longleftrightarrow e(v(y))\ge d\text{ или } \beta-\varepsilon_3<\pi\left(v(y)1^{y}\right)<\beta+\varepsilon_3\Longleftrightarrow $$
        $$\Longleftrightarrow\text{(По определению $\pi_y(v)$)}\Longleftrightarrow$$
        $$\Longleftrightarrow e(v(y))\ge d\text{ или } \beta-\varepsilon_3<\pi_y(v)<\beta+\varepsilon_3\Longleftrightarrow $$
        $$\Longleftrightarrow\left(\text{так как $v\in\overline{R'}(w,\beta,n,\varepsilon,y,\varepsilon_3$)}\right)\Longleftrightarrow$$
        $$\Longleftrightarrow e(v(y))\ge d.$$
        
        Таким образом, мы поняли, что 
при наших $w\in\mathbb{YF}_\infty$, $\beta\in(0,1)$, $\varepsilon,\varepsilon_3\in\mathbb{R}_{>0}$ $\exists\varepsilon_{10}\in\left(0,\beta^2\right)$: $\forall d\in\mathbb{N}_0$ $\exists N_0\in\mathbb{N}_0:$ $\forall v\in \overline{R'}(w,\beta,n,\varepsilon,y,\varepsilon_3)$ при $n\in\mathbb{N}_0:$ $n\ge N'$, $y\in \overline{n}[\beta,\varepsilon_{10}](2,0)$
$$e(v(y))\ge d.$$

    Теперь рассмотрим нечётные игреки: 

        Мы знаем, что по Утверждению \ref{abbalbisk1} при наших $\varepsilon_3\in\mathbb{R}_{>0}$ и $\beta\in(0,1)$ $\exists \varepsilon_1\in\mathbb{R}_{>0}:$ $\forall d\in\mathbb{N}_0$ $\exists N\in\mathbb{N}_0:$ $ N\ge 1,$ $\forall a,b\in\mathbb{N}_0,$ $x\in\mathbb{YF} :$ $(a,b)\in U(\beta,\varepsilon_1,N), $ $|x|\in \{2b,2b+1\}$
        $$e(x)\ge d\text{ или } \beta-\varepsilon_3<\pi\left(x1^{2a-1}\right)<\beta+\varepsilon_3. $$
        
        Рассмотрим $\varepsilon_1\in\mathbb{R}_{>0}$, удовлетворяющее условию Утверждения, зафиксируем какое-нибудь $d\in\mathbb{N}_0$, для него рассмотрим $N\in\mathbb{N}_0:$ $N\ge 1$, удовлетворяющее условию Утверждения. 
        
        Теперь пусть $\varepsilon_{11}=\min\left(\frac{\beta^2}{2},\frac{1-\beta^2}{2},\frac{\varepsilon_1}{2}\right)\in\left(0,\beta^2\right)$, $n\in\mathbb{N}_0:$ $ n\ge N_1:= \max\left(2N+2,\left\lceil\frac{\displaystyle \beta^2+\frac{\displaystyle\varepsilon_1}{\displaystyle 2}+1}{\displaystyle \frac{\displaystyle\varepsilon_1}{\displaystyle 2}}+2\right\rceil,\frac{\displaystyle 4}{\displaystyle 1-\beta^2}\right)$ и $v\in \overline{R'}(w,\beta,n,\varepsilon,y,\varepsilon_3)$ при $y\in \overline{n}[\beta,\varepsilon_{11}](2,1)$.
        
        Наконец, пусть $a=\frac{y+1}{2}$, $b=\left\lfloor\frac{n-y}{2}\right\rfloor$, $x=v(y)$ ($v\in \overline{R'}(w,\beta,n,\varepsilon,y,\varepsilon_3)\subseteq K(n,y)\Longrightarrow v(y)$ существует).
        
        Тогда заметим, что 
        \begin{itemize}
            \item $$a= \frac{y+1}{2}\ge \frac{0+1}{2}=\frac{1}{2}>0 \Longrightarrow\left(\text{так как $a\in\mathbb{N}_0$}\right)\Longrightarrow a\ge 1;$$
            \item $$b=\left\lfloor\frac{n-y}{2}\right\rfloor \ge  \left(\text{так как $y\in \overline{n}[\beta,\varepsilon_{10}]\Longleftrightarrow y\in \left(\overline{n}\cap \left(n\left(\beta^2-\varepsilon_{10}\right),n\left(\beta^2+\varepsilon_{10}\right)\right)\right)$}\right) \ge$$
            $$\ge \left\lfloor\frac{n-n\left(\beta^2+\varepsilon_{10}\right)}{2}\right\rfloor\ge \left\lfloor\frac{n\left(1-\beta^2-\varepsilon_{10}\right)}{2}\right\rfloor\ge \left\lfloor\frac{n\left(1-\beta^2-\frac{1-\beta^2}{2}\right)}{2}\right\rfloor=$$
            $$=\left\lfloor n\frac{\left(1-\beta^2\right)}{4}\right\rfloor\ge \left\lfloor\frac{4}{1-\beta^2}\cdot\frac{\left(1-\beta^2\right)}{4}\right\rfloor=\left\lfloor1\right\rfloor=1;$$
            \item $$a+b=\frac{y+1}{2}+\left\lfloor\frac{n-y}{2}\right\rfloor\ge\frac{y}{2}+\left(\frac{n-y}{2}-1\right)=\frac{n}{2}-1\ge\frac{N_1}{2}-1\ge \frac{2N+2}{2}-1  =N;$$
            \item $$\frac{a}{a+b}=\frac{\frac{y+1}{2}}{\frac{y+1}{2}+\left\lfloor\frac{n-y}{2}\right\rfloor}=\frac{y+1}{y+1+2\left\lfloor\frac{n-y}{2}\right\rfloor}\ge \frac{y+1}{y+1+2\left(\frac{n-y}{2}\right)}\ge \frac{y}{y+2\left(\frac{n-y}{2}\right)}=\frac{y}{n}\ge$$
            $$\ge \left(\text{так как $y\in \overline{n}[\beta,\varepsilon_{11}]\Longleftrightarrow y\in \left(\overline{n}\cap \left(n\left(\beta^2-\varepsilon_{11}\right),n\left(\beta^2+\varepsilon_{11}\right)\right)\right)$}\right) \ge $$
            $$\ge \frac{n\left(\beta^2-\varepsilon_{11}\right)}{n}=\beta^2-\varepsilon_{11}\ge\beta^2-\frac{\varepsilon_{1}}{2}>\beta^2-{\varepsilon_{1}};$$
            \item $$\frac{a}{a+b}=\frac{\frac{y+1}{2}}{\frac{y+1}{2}+\left\lfloor\frac{n-y}{2}\right\rfloor}=\frac{y+1}{y+1+2\left\lfloor\frac{n-y}{2}\right\rfloor}\le \frac{y+1}{y+1+2\left(\frac{n-y}{2}-1\right)}=\frac{y+1}{n-1}\le$$
            $$\le \left(\text{так как $y\in \overline{n}[\beta,\varepsilon_{11}]\Longleftrightarrow y\in \left(\overline{n}\cap \left(n\left(\beta^2-\varepsilon_{11}\right),n\left(\beta^2+\varepsilon_{11}\right)\right)\right)$}\right) \le $$
            $$\le \frac{n\left(\beta^2+\varepsilon_{11}\right)+1}{n-1}=\frac{(n-1)\left(\beta^2+\varepsilon_{11}\right)}{n-1}+\frac{\left(\beta^2+\varepsilon_{11}\right)+1}{n-1}=\beta^2+\varepsilon_{11}+\frac{\left(\beta^2+\varepsilon_{11}\right)+1}{n-1}\le$$
            $$\le\beta^2+\frac{ \displaystyle\varepsilon_1}{\displaystyle 2}+\frac{\beta^2+\frac{ \displaystyle\varepsilon_1}{\displaystyle 2}+1}{n-1}\le \beta^2+\frac{ \displaystyle\varepsilon_1}{\displaystyle 2}+\frac{\beta^2+\frac{ \displaystyle\varepsilon_1}{\displaystyle 2}+1}{N_1-1}\le \beta^2+\frac{ \displaystyle\varepsilon_1}{\displaystyle 2}+\frac{\beta^2+\frac{ \displaystyle\varepsilon_1}{\displaystyle 2}+1}{\frac{\displaystyle \beta^2+\frac{\displaystyle\varepsilon_1}{\displaystyle 2}+1}{\displaystyle \frac{\displaystyle\varepsilon_1}{\displaystyle 2}}+2-1}<$$
            $$<\beta^2+\frac{ \displaystyle\varepsilon_1}{\displaystyle 2}+\frac{\beta^2+\frac{ \displaystyle\varepsilon_1}{\displaystyle 2}+1}{\frac{\displaystyle \beta^2+\frac{\displaystyle\varepsilon_1}{\displaystyle 2}+1}{\displaystyle \frac{\displaystyle\varepsilon_1}{\displaystyle 2}}}=\beta^2+\frac{ \displaystyle\varepsilon_1}{\displaystyle 2}+\frac{ \displaystyle\varepsilon_1}{\displaystyle 2}=\beta^2+\varepsilon_1;$$
            \item  Если $(n-y) \;mod\;{2} = 0$, то 
            $$|x|=n-y=2\frac{n-y}{2}=2\left\lfloor\frac{n-y}{2}\right\rfloor=2b;$$
            \item Если $(n-y) \;mod\; {2} = 1$, то 
            $$|x|=n-y=2\frac{n-y-1}{2}+1=2\left\lfloor\frac{n-y}{2}\right\rfloor+1=2b+1.$$
        \end{itemize}
        
        То есть $a,b\in\mathbb{N}_0,$ $x\in\mathbb{YF}:$ $a,b\ge 1,$ $(a,b)\in U(\beta,\varepsilon_1,N),$ $ |x|\in \{2b,2b+1\}$.
        
        А это по Утверждению \ref{abbalbisk1} значит, что 
        $$e(x)\ge d\text{ или } \beta-\varepsilon_3<\pi\left(x1^{2a-1}\right)<\beta+\varepsilon_3, $$
        то есть
        $$e(v(y))\ge d\text{ или } \beta-\varepsilon_3<\pi\left(v(y)1^{2a-1}\right)<\beta+\varepsilon_3\Longleftrightarrow $$
        $$\Longleftrightarrow e(v(y))\ge d\text{ или } \beta-\varepsilon_3<\pi\left(v(y)1^{2\frac{y+1}{2}-1}\right)<\beta+\varepsilon_3\Longleftrightarrow $$
        $$\Longleftrightarrow e(v(y))\ge d\text{ или } \beta-\varepsilon_3<\pi\left(v(y)1^{y}\right)<\beta+\varepsilon_3\Longleftrightarrow $$
        $$\Longleftrightarrow\text{(По определению $\pi_y(v)$)}\Longleftrightarrow$$
        $$\Longleftrightarrow e(v(y))\ge d\text{ или } \beta-\varepsilon_3<\pi_y(v)<\beta+\varepsilon_3\Longleftrightarrow $$
        $$\Longleftrightarrow\left(\text{так как $v\in\overline{R'}(w,\beta,n,\varepsilon,y,\varepsilon_3)$}\right)\Longleftrightarrow$$
        $$\Longleftrightarrow e(v(y))\ge d.$$
        
        Таким образом, мы поняли, что 
при наших $w\in\mathbb{YF}_\infty$, $\beta\in(0,1)$, $\varepsilon,\varepsilon_3\in\mathbb{R}_{>0}$ $\exists\varepsilon_{11}\in\left(0,\beta^2\right)$: $\forall d\in\mathbb{N}_0$ $\exists N_1\in\mathbb{N}_0:$ $\forall v\in \overline{R'}(w,\beta,n,\varepsilon,y,\varepsilon_3)$ при $n\in\mathbb{N}_0:$ $n\ge N'$, $y\in \overline{n}[\beta,\varepsilon_{11}](2,1)$
$$e(v(y))\ge d.$$

Итак, объединяем информацию про чётные и нечётные игреки:

Мы поняли, что 
при наших $w\in\mathbb{YF}_\infty$,  $\beta\in(0,1)$, $\varepsilon,\varepsilon_3\in\mathbb{R}_{>0}$ $\exists\varepsilon_1'=\min(\varepsilon_{10},\varepsilon_{11})\in \left(0,\beta^2\right)$: $\forall d\in\mathbb{N}_0$ $\exists N'=\max(N_0,N_1)\in\mathbb{N}_0:$ $\forall v\in \overline{R'}(w,\beta,n,\varepsilon,y,\varepsilon_3)$ при $n\in\mathbb{N}_0:$ $ n\ge N'$, $y\in \overline{n}[\beta,\varepsilon_{1}']$
$$e(v(y))\ge d.$$

Таким образом, при наших $w\in\mathbb{YF}_\infty$,  $\beta\in(0,1)$, $\varepsilon,\varepsilon_3\in\mathbb{R}_{>0}$ $\exists\varepsilon_{1}'\in\left(0,\beta^2\right):$ $\forall d\in\mathbb{N}_0$ $\exists N'\in\mathbb{N}_0:$ при $n\in\mathbb{N}_0:$ $ n\ge N'$
$$\max_{\begin{smallmatrix}y\in \overline{n}[\beta,\varepsilon_1'],\\{v\in \overline{R'}(w,\beta,n,\varepsilon,y,\varepsilon_3)}\end{smallmatrix}} \left(1-\beta^2\right)^{{\frac{e(v(y))}{2}}}\le \left(1-\beta^2\right)^{{\frac{d}{2}}}.$$

А значит при наших $w\in\mathbb{YF}_\infty$, $\beta\in(0,1)$, $\varepsilon,\varepsilon_3\in\mathbb{R}_{>0}$ $\exists\varepsilon_{1}'\in\left(0,\beta^2\right)$: $\forall d\in\mathbb{N}_0$ $\exists N'\in\mathbb{N}_0$: при $n\in\mathbb{N}_0:$ $ n\ge N'$
$$\max^0_{\begin{smallmatrix}y\in \overline{n}[\beta,\varepsilon_1'],\\{v\in \overline{R'}(w,\beta,n,\varepsilon,y,\varepsilon_3)}\end{smallmatrix}} \left(1-\beta^2\right)^{{\frac{e(v(y))}{2}}}\le \left(1-\beta^2\right)^{{\frac{d}{2}}}.$$

Ясно, что наших $\beta\in(0,1)$ и $\overline{\varepsilon'}\in\mathbb{R}_{>0}$ $\exists d\in\mathbb{N}_0:$
$$\left(1-\beta^2\right)^{{\frac{d}{2}}}<\frac{\overline{\varepsilon'}}{1+\frac{1}{\beta}}.$$

Зафиксируем это $d$.

Как мы поняли при наших $w\in\mathbb{YF}_\infty$, $\beta\in\mathbb{R}_{>0}$, $\varepsilon,\varepsilon_3\in\mathbb{R}_{>0}$ $\exists\varepsilon_{1}'\in\left(0,\beta^2\right)$: при только что зафиксированном $d$ $\exists N''\in\mathbb{N}_0$: при $n\in\mathbb{N}_0:$ $ n\ge N''$
$${\sum_{y\in \overline{n}[\beta,\varepsilon_1']}\left(\left({\sum_{v\in \overline{R'}(w,\beta,n,\varepsilon,y,\varepsilon_3)} T_{w,\beta,n}(v,y)}\right)+\left({\sum_{v\in \overline{R''}(w,\beta,n,\varepsilon,y)} T_{w,\beta,n}(v,y)}\right)\right)}\le$$
$$\le\left(1+\frac{1}{\beta}\right)\cdot\max^0_{\begin{smallmatrix}y\in \overline{n}[\beta,\varepsilon_1'],\\{v\in \overline{R'}(w,\beta,n,\varepsilon,y,\varepsilon_3)}\end{smallmatrix}} \left(1-\beta^2\right)^{{\frac{e(v(y))}{2}}}<\left(1+\frac{1}{\beta}\right)\cdot\frac{\overline{\varepsilon'}}{1+\frac{1}{\beta}}=\overline{\varepsilon'},$$
что и требовалось.

        Лемма доказана.
\end{proof}

\begin{Lemma} \label{hyde}
Пусть $w\in\mathbb{YF}_\infty$, $\beta\in(0,1)$, $\varepsilon,\varepsilon_1'\in\mathbb{R}_{>0}$. Тогда
$${\sum_{y\in \overline{n}\{\beta,\varepsilon_1'\}}\left({\sum_{v\in \overline{R}(w,\beta,n,\varepsilon)} T_{w,\beta,n}(v,y)}\right) \xrightarrow{n\to \infty}0}.$$
\end{Lemma}
\begin{proof}

Заметим, что функция $T$ неотрицательна, а также то, что $\forall w\in\mathbb{YF}_\infty$, $\beta\in(0,1),$ $n\in\mathbb{N}_0$, $\varepsilon\in\mathbb{R}_{>0}$ 
$$\overline{R}(w,\beta,n,\varepsilon)\subseteq \mathbb{YF}_n.$$

\renewcommand{\labelenumi}{\alph{enumi}$)$}
\renewcommand{\labelenumii}{\arabic{enumii}$^\circ$}
\renewcommand{\labelenumiii}{\arabic{enumi}.\arabic{enumii}.\arabic{enumiii}$^\circ$}

Это значит, что
$$0\le{\sum_{y\in \overline{n}\{\beta,\varepsilon_1'\}}\left({\sum_{v\in \overline{R}(w,\beta,n,\varepsilon)} T_{w,\beta,n}(v,y)}\right)}\le{\sum_{y\in \overline{n}\{\beta,\varepsilon_1'\}}\left({\sum_{v\in\mathbb{YF}_n} T_{w,\beta,n}(v,y)}\right)}. $$

Рассмотрим две подпоследовательности:
\begin{enumerate}
    \item Подпоследовательность $n\in\mathbb{N}_0:$ $n\;mod\;{2}=0$.
    
    Для начала рассмотрим чётные игреки.
        
    Зафиксируем какое-то $n\in\mathbb{N}_0:$ $n\;mod\;{2}=0$.
    
    Если $y\in\mathbb{N}_0:$ $y \le n$, то по Утверждению \ref{stolb} при $w\in\mathbb{YF}_\infty$, $\beta\in(0,1]$, $n,y\in\mathbb{N}_0$
    $$\sum_{x\in\mathbb{YF}_n} T_{w,\beta,n}(x,y)\le \left(\prod_{i=1}^{\left\lfloor \frac{n-y}{2} \right\rfloor} \frac{2i+y}{2i}\right)\beta^{y}\left(1-\beta^2\right)^{\left\lfloor \frac{n-y}{2} \right\rfloor}.$$
    
    Ясно, что мы можем просуммировать это выражение по $y\in{\overline{n}\{\beta,\varepsilon_1'\}(2,0)}$. Просуммируем:
    $$\sum_{{y\in\overline{n}\{\beta,\varepsilon_1'\}(2,0)}}\left(\sum_{x\in\mathbb{YF}_n} T_{w,\beta,n}(x,y)\right)\le$$
    $$\le\sum_{{y\in\overline{n}\{\beta,\varepsilon_1'\}(2,0)}}\left(\left( \prod_{i=1}^{\left\lfloor \frac{n-y}{2} \right\rfloor} \frac{2i+y}{2i}\right)\beta^{y}\left(1-\beta^2\right)^{\left\lfloor \frac{n-y}{2} \right\rfloor}\right)=$$
    $$=\left(\text{Так как ясно, что если $n \;mod\;{2}=0$ и $y\;mod\;{2}=0$, то $\left\lfloor \frac{n-y}{2} \right\rfloor=\frac{n-y}{2}$}\right)=$$
    $$=\sum_{{y\in\overline{n}\{\beta,\varepsilon_1'\}(2,0)}} \left(\left(\prod_{i=1}^{ \frac{n-y}{2}} \frac{2i+y}{2i}\right)\beta^{y}\left(1-\beta^2\right)^{ \frac{n-y}{2}}\right)=$$
    $$=\sum_{{y\in\overline{n}\{\beta,\varepsilon_1'\}(2,0)}} \left(\left(\prod_{i=1}^{ \frac{n-2\frac{y}{2}}{2}} \frac{2i+2\frac{y}{2}}{2i}\right)\beta^{2\frac{y}{2}}\left(1-\beta^2\right)^{ \frac{n-2\frac{y}{2}}{2}}\right).$$
    
    Ясно, что если $y$ пробегает все значения в множестве ${{\overline{n}\{\beta,\varepsilon_1'\}(2,0)}},$ при $n\in\mathbb{N}_0:$ $n\;mod\;{2}=0$, то $\frac{y}{2}$ пробегает все значения в множестве $\overline{n}_{00}\{\beta,\varepsilon_1'\}$ (просто по определению этого множества), то есть данное выражение равняется следующему: 
    $$\sum_{y'\in \overline{n}_{00}\{\beta,\varepsilon_1'\}} \left(\left(\prod_{i=1}^{ \frac{n-2y'}{2}} \frac{2i+2y'}{2i}\right)\beta^{2y'}\left(1-\beta^2\right)^{ \frac{n-2y'}{2}}\right)=$$
    $$=\sum_{y'\in \overline{n}_{00}\{\beta,\varepsilon_1'\}} \left(\left(\prod_{i=1}^{\frac{n}{2}-y' } \frac{i+y'}{i}\right)\beta^{2y'}\left(1-\beta^2\right)^{ \frac{n}{2}-y' }\right)=$$
    $$=\sum_{y'\in \overline{n}_{00}\{\beta,\varepsilon_1'\}} \left(\left(\frac{\displaystyle\prod_{i=y'+1}^{\frac{n}{2}}i}{\displaystyle\prod_{i=1}^{\frac{n}{2}-y'}i}\right)\left(\beta^2\right)^{y'}\left(1-\beta^2\right)^{ \frac{n}{2}-y'}\right)=$$
    $$=\sum_{y'\in \overline{n}_{00}\{\beta,\varepsilon_1'\}  } \left(\binom{\frac{n}{2}}{y'}\left(\beta^2\right)^{y'}\left(1-\beta^2\right)^{ \frac{n}{2} -y'}\right).$$
    
    Ясно, что 
    $$\lim_{n\to\infty}^{(2,0)}\left(\sum_{y'\in \overline{n}_{00}\{\beta,\varepsilon_1'\}  } \left(\binom{\frac{n}{2}}{y'}\left(\beta^2\right)^{y'}\left(1-\beta^2\right)^{ \frac{n}{2} -y'}\right)\right)=$$
    $$=\lim_{n\to\infty}\left(\sum_{y'\in \overline{2n}_{00}\{\beta,\varepsilon_1'\}} \left(\binom{n}{y'}\left(\beta^2\right)^{y'}\left(1-\beta^2\right)^{n -y'}\right)\right)=$$
    $$=\text{(просто по определению множества $\overline{n}_{00,2}\{\beta,\varepsilon_1'\}$)}=$$
    $$=\lim_{n\to\infty}\left(\sum_{y'\in \overline{n}_{00,2}\{\beta,\varepsilon_1'\}} \left(\binom{n}{y'}\left(\beta^2\right)^{y'}\left(1-\beta^2\right)^{n -y'}\right)\right).$$
    
    Мы знаем, что
   $$\overline{n}_{00,2}\{\beta,\varepsilon_1'\}= \overline{{n}}\textbackslash\left(n\left(\beta^2-\varepsilon_1'\right),n\left(\beta^2+\varepsilon_1'\right)\right).$$
    
    А значит при $\beta\in(0,1)$, по закону распределения биномиальных коэффициентов,
    $$\lim_{n\to\infty}\left(\sum_{y'\in \overline{n}_{00,2}\{\beta,\varepsilon_1'\}} \left(\binom{n}{y'}\left(\beta^2\right)^{y'}\left(1-\beta^2\right)^{n -y'}\right)\right)=0.$$
    
    А значит
    $$\lim_{n\to\infty}^{(2,0)}\sum_{y'\in \overline{n}_{00}\{\beta,\varepsilon_1'\}  } \left(\binom{\frac{n}{2}}{y'}\left(\beta^2\right)^{y'}\left(1-\beta^2\right)^{ \frac{n}{2} -y'}\right)=0.$$
    
    В силу доказанного выше, а также неотрицительности функции $T$, ясно, что при $n\in\mathbb{N}_0: n \; mod \; 2=0$
    $$0\le \sum_{{y\in\overline{n}\{\beta,\varepsilon_1'\}(2,0)}}\left(\sum_{x\in\mathbb{YF}_n} T_{w,\beta,n}(x,y)\right)\le \sum_{y'\in \overline{n}_{00}\{\beta,\varepsilon_1'\}  } \left(\binom{\frac{n}{2}}{y'}\left(\beta^2\right)^{y'}\left(1-\beta^2\right)^{ \frac{n}{2} -y'}\right),$$
    а значит, по Лемме о двух полицейских,
    $$\lim_{n\to\infty}^{(2,0)}\left(\sum_{{y\in\overline{n}\{\beta,\varepsilon_1'\}(2,0)}}\left(\sum_{x\in\mathbb{YF}_n} T_{w,\beta,n}(x,y)\right)\right)= 0.$$
    
    Теперь рассмотрим нечётные игреки.
    
    Зафиксируем какое-то $n\in\mathbb{N}_0:n\;mod \;2=0$.

    Если $y\in\mathbb{N}_0:$ $y \le n,$ то по Утверждению \ref{stolb} при $w\in\mathbb{YF}_\infty$, $\beta\in(0,1]$, $n,y\in\mathbb{N}_0$
    $$\sum_{x\in\mathbb{YF}_n} T_{w,\beta,n}(x,y)\le \left(\prod_{i=1}^{\left\lfloor \frac{n-y}{2} \right\rfloor} \frac{2i+y}{2i}\right)\beta^{y}\left(1-\beta^2\right)^{\left\lfloor \frac{n-y}{2} \right\rfloor}.$$
    Ясно, что мы можем просуммировать это выражение по $y\in{\overline{n}\{\beta,\varepsilon_1'\}(2,1)}$. Просуммируем:
    $$\sum_{y\in{\overline{n}\{\beta,\varepsilon_1'\}(2,1)}}\left(\sum_{x\in\mathbb{YF}_n} T_{w,\beta,n}(x,y)\right)\le$$
    $$\le\sum_{{y\in\overline{n}\{\beta,\varepsilon_1'\}(2,1)}}\left(\left( \prod_{i=1}^{\left\lfloor \frac{n-y}{2} \right\rfloor} \frac{2i+y}{2i}\right)\beta^{y}\left(1-\beta^2\right)^{\left\lfloor \frac{n-y}{2} \right\rfloor}\right)\le$$
    $$\le \sum_{{y\in\overline{n}\{\beta,\varepsilon_1'\}(2,1)}} \left(\left(\prod_{i=1}^{\left\lfloor \frac{n-y}{2} \right\rfloor} \frac{2i+y+1}{2i}\right)\beta^{y}\left(1-\beta^2\right)^{\left\lfloor \frac{n-y}{2} \right\rfloor}\right)=$$
    $$=\left(\text{Так как ясно, что если $n \;mod\;{2}=0$ и $y\;mod\;{2}=1$, то $\left\lfloor \frac{n-y}{2} \right\rfloor=\frac{n-y-1}{2}$}\right)=$$
    $$=\frac{1}{\beta}\sum_{{\overline{n}\{\beta,\varepsilon_1'\}(2,1)}} \left(\left(\prod_{i=1}^{ \frac{n-y-1}{2}} \frac{2i+y+1}{2i}\right)\beta^{y+1}\left(1-\beta^2\right)^{ \frac{n-y-1}{2}}\right)=$$
    $$=\frac{1}{\beta}\sum_{{\overline{n}\{\beta,\varepsilon_1'\}(2,1)}} \left(\left(\prod_{i=1}^{ \frac{n-2\frac{y+1}{2}}{2}} \frac{2i+2\frac{y+1}{2}}{2i}\right)\beta^{2\frac{y+1}{2}}\left(1-\beta^2\right)^{ \frac{n-2\frac{y+1}{2}}{2}}\right).$$
    
    Ясно, что если $y$ пробегает все значения в множестве ${{\overline{n}\{\beta,\varepsilon_1'\}(2,1)}},$ при $n\in\mathbb{N}_0:$ $n\;mod\;{2}=0$, то $\frac{y+1}{2}$ пробегает все значения в множестве $\overline{n}_{01}\{\beta,\varepsilon_1'\}\textbackslash \{0\}$ (просто по определению этого множества), то есть наше выражение равняется следующему: 
    $$\frac{1}{\beta}\sum_{y'\in\left( \overline{n}_{01}\{\beta,\varepsilon_1'\}\textbackslash\{0\}\right)} \left(\left(\prod_{i=1}^{ \frac{n-2y'}{2}} \frac{2i+2y'}{2i}\right)\beta^{2y'}\left(1-\beta^2\right)^{ \frac{n-2y'}{2}}\right)=$$
    $$=\frac{1}{\beta}\sum_{y'\in\left( \overline{n}_{01}\{\beta,\varepsilon_1'\}\textbackslash\{0\}\right)} \left(\left(\prod_{i=1}^{\frac{n}{2}-y' } \frac{i+y'}{i}\right)\beta^{2y'}\left(1-\beta^2\right)^{ \frac{n}{2}-y' }\right).$$
    
    Тут есть два случая:
    \begin{enumerate}
        \item $0\notin\overline{n}_{01}\{\beta,\varepsilon_1'\}$.
        
        В данном случае $\overline{n}_{01}\{\beta,\varepsilon_1'\}\textbackslash \{0\}=\overline{n}_{01}\{\beta,\varepsilon_1'\},$
        а значит
        $$\frac{1}{\beta}\sum_{y'\in\left( \overline{n}_{01}\{\beta,\varepsilon_1'\}\textbackslash\{0\}\right)} \left(\left(\prod_{i=1}^{\frac{n}{2}-y' } \frac{i+y'}{i}\right)\beta^{2y'}\left(1-\beta^2\right)^{ \frac{n}{2}-y' }\right)=$$
        $$=\frac{1}{\beta}\sum_{y'\in \overline{n}_{01}\{\beta,\varepsilon_1'\}} \left(\left(\prod_{i=1}^{\frac{n}{2}-y' } \frac{i+y'}{i}\right)\beta^{2y'}\left(1-\beta^2\right)^{ \frac{n}{2}-y' }\right).$$
    
        \item $0\in\overline{n}_{01}\{\beta,\varepsilon_1'\}$.
        
        В данном случае
        $$\frac{1}{\beta}\sum_{y'\in\left( \overline{n}_{01}\{\beta,\varepsilon_1'\}\textbackslash\{0\}\right)} \left(\left(\prod_{i=1}^{\frac{n}{2}-y' } \frac{i+y'}{i}\right)\beta^{2y'}\left(1-\beta^2\right)^{ \frac{n}{2}-y' }\right)\le \text{(Так как $\beta\in(0,1)$)}\le$$
        $$\le\frac{1}{\beta}\sum_{y'\in\left( \overline{n}_{01}\{\beta,\varepsilon_1'\}\textbackslash\{0\}\right)} \left(\left(\prod_{i=1}^{\frac{n}{2}-y' } \frac{i+y'}{i}\right)\beta^{2y'}\left(1-\beta^2\right)^{ \frac{n}{2}-y' }\right)+\frac{1}{\beta} \left(\prod_{i=1}^{\frac{n}{2} } \frac{i}{i}\right)\beta^{0}\left(1-\beta^2\right)^{ \frac{n}{2} }=$$
            $$=\frac{1}{\beta}\sum_{y'\in\left( \overline{n}_{01}\{\beta,\varepsilon_1'\}\textbackslash\{0\}\right)} \left(\left(\prod_{i=1}^{\frac{n}{2}-y' } \frac{i+y'}{i}\right)\beta^{2y'}\left(1-\beta^2\right)^{ \frac{n}{2}-y' }\right)+$$
            $$+\frac{1}{\beta}\sum_{y'=0}^{0} \left(\left(\prod_{i=1}^{\frac{n}{2}-y' } \frac{i+y'}{i}\right)\beta^{2y'}\left(1-\beta^2\right)^{ \frac{n}{2}-y' }\right)=$$
    $$=\frac{1}{\beta}\sum_{y'\in \overline{n}_{01}\{\beta,\varepsilon_1'\}} \left(\left(\prod_{i=1}^{\frac{n}{2}-y' } \frac{i+y'}{i}\right)\beta^{2y'}\left(1-\beta^2\right)^{ \frac{n}{2}-y' }\right).$$
    
    \end{enumerate}
    
    В обоих случаях (ясно, что других нет) наше выражение не превосходит следующее:
    $$\frac{1}{\beta}\sum_{y'\in \overline{n}_{01}\{\beta,\varepsilon_1'\}} \left(\left(\prod_{i=1}^{\frac{n}{2}-y' } \frac{i+y'}{i}\right)\beta^{2y'}\left(1-\beta^2\right)^{ \frac{n}{2}-y' }\right)=$$
    $$=\frac{1}{\beta}\sum_{y'\in \overline{n}_{01}\{\beta,\varepsilon_1'\}} \left(\left(\frac{\displaystyle\prod_{i=y'+1}^{\frac{n}{2}}i}{\displaystyle\prod_{i=1}^{\frac{n}{2}-y'}i}\right)\left(\beta^2\right)^{y'}\left(1-\beta^2\right)^{ \frac{n}{2}-y'}\right)=$$
    $$=\frac{1}{\beta}\sum_{y'\in \overline{n}_{01}\{\beta,\varepsilon_1'\}  } \left(\binom{\frac{n}{2}}{y'}\left(\beta^2\right)^{y'}\left(1-\beta^2\right)^{ \frac{n}{2} -y'}\right).$$
    
    Ясно, что
    $$\lim_{n\to\infty}^{(2,0)}\left(\frac{1}{\beta}\sum_{y'\in \overline{n}_{01}\{\beta,\varepsilon_1'\}  } \left(\binom{\frac{n}{2}}{y'}\left(\beta^2\right)^{y'}\left(1-\beta^2\right)^{ \frac{n}{2} -y'}\right)\right)=$$
    $$=\lim_{n\to\infty}\left(\frac{1}{\beta}\sum_{y'\in \overline{2n}_{01}\{\beta,\varepsilon_1'\}} \left(\binom{n}{y'}\left(\beta^2\right)^{y'}\left(1-\beta^2\right)^{n -y'}\right)\right)=$$
    $$=\left(\text{просто по определению множества $\overline{n}_{01,2}\{\beta,\varepsilon_1'\}$}\right)=$$
    $$=\lim_{n\to\infty}\left(\frac{1}{\beta}\sum_{y'\in \overline{n}_{01,2}\{\beta,\varepsilon_1'\}} \left(\binom{n}{y'}\left(\beta^2\right)^{y'}\left(1-\beta^2\right)^{n -y'}\right)\right).$$
    
        Ясно, что
    \begin{itemize}
        \item     $$\overline{n}_{01,2}\{\beta,\varepsilon_1'\}= \overline{{n}}\textbackslash\left(n\left(\beta^2-\varepsilon_1'\right)+\frac{1}{2},n\left(\beta^2+\varepsilon_1'\right)+\frac{1}{2}\right);$$
        \item Если $n\in\mathbb{N}_0:$ $n>\left\lceil\frac{1}{\varepsilon_1'}\right\rceil$, то
        $$ \frac{n\left(\beta^2-\varepsilon_1'\right)+\frac{1}{2}}{n}=\beta^2-\varepsilon_1'+\frac{1}{2n}<\beta^2-\varepsilon_1'+\frac{1}{2\left\lceil\frac{1}{\varepsilon_1'}\right\rceil}\le \beta^2-\varepsilon_1'+\frac{1}{2\frac{1}{\varepsilon_1'}}=\beta^2-\varepsilon_1'+\frac{\varepsilon_1'}{2}=\beta^2-\frac{\varepsilon_1'}{2}\Longrightarrow$$
        $$\Longrightarrow {n\left(\beta^2-\varepsilon_1'\right)+\frac{1}{2}}<n\left(\beta^2-\frac{\varepsilon_1'}{2}\right);$$
        
        \item Если $n\in\mathbb{N}_0:$ $n>\left\lceil\frac{1}{\varepsilon_1'}\right\rceil$, то
        $$ \frac{n\left(\beta^2+\varepsilon_1'\right)+\frac{1}{2}}{n}=\beta^2+\varepsilon_1'+\frac{1}{2n}>\beta^2+\frac{\varepsilon_1'}{2}\Longrightarrow$$
        $$\Longrightarrow {n\left(\beta^2+\varepsilon_1'\right)+\frac{1}{2}}>n\left(\beta^2+\frac{\varepsilon_1'}{2}\right).$$
        \end{itemize}
    
    А значит если $n\in\mathbb{N}_0:$ $n>\left\lceil\frac{1}{\varepsilon_1'}\right\rceil$, то 
         $$\overline{n}_{01,2}\{\beta,\varepsilon_1'\}\subset \overline{{n}}\textbackslash\left(n\left(\beta^2-\frac{\varepsilon_1'}{2}\right),n\left(\beta^2+\frac{\varepsilon_1'}{2}\right)\right)=\overline{n}\left\{\beta,\frac{\varepsilon_1'}{2}\right\}.$$

     А значит (так как $\beta\in(0,1)$), если $n\in\mathbb{N}_0:$ $n>\left\lceil\frac{1}{\varepsilon_1'}\right\rceil$, то
    $$\frac{1}{\beta}\sum_{y'\in \overline{n}\left\{\beta,\frac{\varepsilon_1'}{2}\right\}} \left(\binom{n}{y'}\left(\beta^2\right)^{y'}\left(1-\beta^2\right)^{n -y'}\right)\ge \frac{1}{\beta}\sum_{y'\in \overline{n}_{01,2}\{\beta,\varepsilon_1'\}} \left(\binom{n}{y'}\left(\beta^2\right)^{y'}\left(1-\beta^2\right)^{n -y'}\right)\ge 0.$$
     
    Ясно, что при $\beta\in(0,1)$ по закону распределения биномиальных коэффициентов
    $$\lim_{n\to\infty}\left(\frac{1}{\beta}\sum_{y'\in \overline{n}\left\{\beta,\frac{\varepsilon_1'}{2}\right\}} \left(\binom{n}{y'}\left(\beta^2\right)^{y'}\left(1-\beta^2\right)^{n -y'}\right)\right)=0.$$
    
     А значит, по Лемме о двух полицейских,
    $$\lim_{n\to\infty}\left(\frac{1}{\beta}\sum_{y'\in \overline{n}_{01,2}\{\beta,\varepsilon_1'\}} \left(\binom{n}{y'}\left(\beta^2\right)^{y'}\left(1-\beta^2\right)^{n -y'}\right)\right)=0.$$
    
    А значит
    $$\lim_{n\to\infty}^{(2,0)}\left(\frac{1}{\beta}\sum_{y'\in\overline{n}_{01}\{\beta,\varepsilon_1'\}} \left(\binom{\frac{n}{2}}{y'}\left(\beta^2\right)^{y'}\left(1-\beta^2\right)^{ \frac{n}{2} -y'}\right)\right)=0.$$
    
    В силу доказанного выше, а также неотрицительности функции $T$, ясно, что при $n\in\mathbb{N}_0: n \; mod \; 2=0$
    $$0\le \sum_{{y\in\overline{n}\{\beta,\varepsilon_1'\}(2,1)}}\left(\sum_{x\in\mathbb{YF}_n} T_{w,\beta,n}(x,y)\right)\le\frac{1}{\beta} \sum_{y'\in \overline{n}_{01}\{\beta,\varepsilon_1'\}  } \left(\binom{\frac{n}{2}}{y'}\left(\beta^2\right)^{y'}\left(1-\beta^2\right)^{ \frac{n}{2} -y'}\right),$$
    а значит, по Лемме о двух полицейских,
    $$\lim_{n\to\infty}^{(2,0)}\left(\sum_{{y\in\overline{n}\{\beta,\varepsilon_1'\}(2,1)}}\left(\sum_{x\in\mathbb{YF}_n} T_{w,\beta,n}(x,y)\right)\right)= 0.$$

   \item Подпоследовательность $n\in\mathbb{N}_0:$ $n\;mod\;{2}=1$.
    
    Для начала рассмотрим чётные игреки.
        
    Зафиксируем какое-то $n\in\mathbb{N}_0:$ $n\;mod\;{2}=1$.
    
    Если $y\in\mathbb{N}_0:$ $y \le n,$ то по Утверждению \ref{stolb} при $w\in\mathbb{YF}_\infty$, $\beta\in(0,1]$, $n,y\in\mathbb{N}_0$
    $$\sum_{x\in\mathbb{YF}_n} T_{w,\beta,n}(x,y)\le\left( \prod_{i=1}^{\left\lfloor \frac{n-y}{2} \right\rfloor} \frac{2i+y}{2i}\right)\beta^{y}\left(1-\beta^2\right)^{\left\lfloor \frac{n-y}{2} \right\rfloor}.$$
    
    Ясно, что мы можем просуммировать это выражение по $y\in{\overline{n}\{\beta,\varepsilon_1'\}(2,0)}$. Просуммируем:
    $$\sum_{y\in{\overline{n}\{\beta,\varepsilon_1'\}(2,0)}}\left(\sum_{x\in\mathbb{YF}_n} T_{w,\beta,n}(x,y)\right)\le$$
    $$\le\sum_{y\in{\overline{n}\{\beta,\varepsilon_1'\}(2,0)}}\left(\left( \prod_{i=1}^{\left\lfloor \frac{n-y}{2} \right\rfloor} \frac{2i+y}{2i}\right)\beta^{y}\left(1-\beta^2\right)^{\left\lfloor \frac{n-y}{2} \right\rfloor}\right)=$$
    $$=\left(\text{Так как ясно, что если $n \;mod\;{2}=1$ и $y\;mod\;{2}=0$, то $\left\lfloor \frac{n-y}{2} \right\rfloor=\frac{n-y-1}{2}$}\right)=$$
    $$=\sum_{y\in{\overline{n}\{\beta,\varepsilon_1'\}(2,0)}} \left(\left(\prod_{i=1}^{ \frac{n-y-1}{2}} \frac{2i+y}{2i}\right)\beta^{y}\left(1-\beta^2\right)^{ \frac{n-y-1}{2}}\right)=$$
    $$=\sum_{y\in{\overline{n}\{\beta,\varepsilon_1'\}(2,0)}} \left(\left(\prod_{i=1}^{ \frac{n-1-2\frac{y}{2}}{2}} \frac{2i+2\frac{y}{2}}{2i}\right)\beta^{2\frac{y}{2}}\left(1-\beta^2\right)^{ \frac{n-1-2\frac{y}{2}}{2}}\right).$$
    
    Ясно, что если $y$ пробегает все значения в множестве ${{\overline{n}\{\beta,\varepsilon_1'\}(2,0)}},$ при $n\in\mathbb{N}_0:$ $n\;mod\;{2}=1$, то $\frac{y}{2}$ пробегает все значения в множестве $\overline{n}_{10}\{\beta,\varepsilon_1'\}$ (просто по определению этого множества), то есть наше выражение равняется следующему: 
    $$\sum_{y'\in\overline{n}_{10}\{\beta,\varepsilon_1'\}} \left(\left(\prod_{i=1}^{ \frac{n-1-2y'}{2}} \frac{2i+2y'}{2i}\right)\beta^{2y'}\left(1-\beta^2\right)^{ \frac{n-1-2y'}{2}}\right)=$$
    $$=\sum_{y'\in\overline{n}_{10}\{\beta,\varepsilon_1'\}} \left(\left(\prod_{i=1}^{\frac{n-1}{2}-y' } \frac{i+y'}{i}\right)\beta^{2y'}\left(1-\beta^2\right)^{ \frac{n-1}{2}-y' }\right)=$$
    $$=\sum_{y'\in\overline{n}_{10}\{\beta,\varepsilon_1'\}} \left(\left(\frac{\displaystyle\prod_{i=y'+1}^{\frac{n-1}{2}}i}{\displaystyle\prod_{i=1}^{\frac{n-1}{2}-y'}i}\right)\left(\beta^2\right)^{y'}\left(1-\beta^2\right)^{ \frac{n-1}{2}-y'}\right)=$$
    $$=\sum_{y'\in\overline{n}_{10}\{\beta,\varepsilon_1'\}} \left(\binom{\frac{n-1}{2}}{y'}\left(\beta^2\right)^{y'}\left(1-\beta^2\right)^{ \frac{n-1}{2} -y'}\right).$$

    Ясно, что
    $$\lim_{n\to\infty}^{(2,1)}\sum_{y'\in\overline{n}_{10}\{\beta,\varepsilon_1'\}} \left(\binom{\frac{n-1}{2}}{y'}\left(\beta^2\right)^{y'}\left(1-\beta^2\right)^{ \frac{n-1}{2} -y'}\right)=$$
    $$=\lim_{n\to\infty}\left(\sum_{y'\in \overline{2n+1}_{10}\{\beta,\varepsilon_1'\}} \left(\binom{n}{y'}\left(\beta^2\right)^{y'}\left(1-\beta^2\right)^{n -y'}\right)\right)=$$
    $$=\left(\text{просто по определению множества $\overline{n}_{10,2}\{\beta,\varepsilon_1'\}$}\right)=$$
    $$=\lim_{n\to\infty}\left(\sum_{y'\in \overline{n}_{10,2}\{\beta,\varepsilon_1'\}} \left(\binom{n}{y'}\left(\beta^2\right)^{y'}\left(1-\beta^2\right)^{n -y'}\right)\right).$$
    
    Ясно, что
    \begin{itemize}
        \item     $$\overline{n}_{10,2}\{\beta,\varepsilon_1'\}= \overline{n}\textbackslash\left(\frac{(2n+1)\left(\beta^2-\varepsilon_1'\right)}{2},\frac{(2n+1)\left(\beta^2+\varepsilon_1'\right)}{2}\right);$$
        \item Если $n\in\mathbb{N}_0:$ $n>\left\lceil\frac{1}{\varepsilon_1'}\right\rceil$, то
        $$ \frac{\displaystyle\frac{(2n+1)\left(\beta^2-\varepsilon_1'\right)}{\displaystyle2}}{\displaystyle n}=\frac{{(2n+1)\left(\beta^2-\varepsilon_1'\right)}}{2n}=\beta^2-\varepsilon_1'+\frac{\left(\beta^2-\varepsilon_1'\right)}{2n}<\left(\text{Так как $\beta\in(0,1)$}\right)<$$
        $$< \beta^2-\varepsilon_1'+\frac{1}{2n}<\beta^2-\varepsilon_1'+\frac{1}{2\left\lceil\frac{1}{\varepsilon_1'}\right\rceil}\le \beta^2-\varepsilon_1'+\frac{1}{2\frac{1}{\varepsilon_1'}}=\beta^2-\varepsilon_1'+\frac{\varepsilon_1'}{2}=\beta^2-\frac{\varepsilon_1'}{2}\Longrightarrow$$
        $$\Longrightarrow {\displaystyle\frac{(2n+1)\left(\beta^2-\varepsilon_1'\right)}{\displaystyle2}}<n\left(\beta^2-\frac{\varepsilon_1'}{2}\right);$$
        \item Если $n\in\mathbb{N}_0:$ $n>\left\lceil\frac{1}{\varepsilon_1'}\right\rceil$, то
        $$ \frac{\displaystyle \frac{(2n+1)\left(\beta^2+\varepsilon_1'\right)}{\displaystyle 2}}{\displaystyle n}=\frac{\displaystyle {(2n+1)\left(\beta^2+\varepsilon_1'\right)}}{\displaystyle 2n}>\frac{\displaystyle {2n\left(\beta^2+\varepsilon_1'\right)}}{\displaystyle 2n}=\beta^2+\varepsilon_1'>\beta^2+\frac{\varepsilon_1'}{2}\Longrightarrow$$
        $$\Longrightarrow {\displaystyle\frac{(2n+1)\left(\beta^2+\varepsilon_1'\right)}{\displaystyle2}}>n\left(\beta^2+\frac{\varepsilon_1'}{2}\right);$$
        
        \end{itemize}
    
    А значит если $n\in\mathbb{N}_0:$ $n>\left\lceil\frac{1}{\varepsilon_1'}\right\rceil$, то 
         $$\overline{n}_{10,2}\{\beta,\varepsilon_1'\}\subset \overline{{n}}\textbackslash\left(n\left(\beta^2-\frac{\varepsilon_1'}{2}\right),n\left(\beta^2+\frac{\varepsilon_1'}{2}\right)\right)=\overline{n}\left\{\beta,\frac{\varepsilon_1'}{2}\right\}.$$
    
     А значит (так как $\beta\in(0,1)$), если $n\in\mathbb{N}_0:$ $n>\left\lceil\frac{1}{\varepsilon_1'}\right\rceil$, то
    $$\sum_{y'\in \overline{n}\left\{\beta,\frac{\varepsilon_1'}{2}\right\}} \left(\binom{n}{y'}\left(\beta^2\right)^{y'}\left(1-\beta^2\right)^{n -y'}\right)\ge \sum_{y'\in \overline{n}_{10,2}\{\beta,\varepsilon_1'\}} \left(\binom{n}{y'}\left(\beta^2\right)^{y'}\left(1-\beta^2\right)^{n -y'}\right)\ge 0.$$
     
    Ясно, что при $\beta\in(0,1)$ по закону распределения биномиальных коэффициентов
    $$\lim_{n\to\infty}\left(\sum_{y'\in \overline{n}\left\{\beta,\frac{\varepsilon_1'}{2}\right\}} \left(\binom{n}{y'}\left(\beta^2\right)^{y'}\left(1-\beta^2\right)^{n -y'}\right)\right)=0.$$
    
     А значит, по Лемме о двух полицейских,
    $$\lim_{n\to\infty}\left(\sum_{y'\in \overline{n}_{10,2}\{\beta,\varepsilon_1'\}} \left(\binom{n}{y'}\left(\beta^2\right)^{y'}\left(1-\beta^2\right)^{n -y'}\right)\right)=0.$$
    
    А значит
    $$\lim_{n\to\infty}^{(2,1)}\sum_{y'\in\overline{n}_{10}\{\beta,\varepsilon_1'\}} \left(\binom{\frac{n-1}{2}}{y'}\left(\beta^2\right)^{y'}\left(1-\beta^2\right)^{ \frac{n-1}{2} -y'}\right)=0.$$
    
    В силу доказанного выше, а также неотрицительности функции $T$, ясно, что при $n\in\mathbb{N}_0: n \; mod \; 2=1$
    $$0\le \sum_{{y\in\overline{n}\{\beta,\varepsilon_1'\}(2,0)}}\left(\sum_{x\in\mathbb{YF}_n} T_{w,\beta,n}(x,y)\right)\le \sum_{y'\in \overline{n}_{10}\{\beta,\varepsilon_1'\} } \left(\binom{\frac{n-1}{2}}{y'}\left(\beta^2\right)^{y'}\left(1-\beta^2\right)^{ \frac{n-1}{2} -y'}\right),$$
    а значит, по Лемме о двух полицейских,
    $$\lim_{n\to\infty}^{(2,1)}\left(\sum_{{y\in\overline{n}\{\beta,\varepsilon_1'\}(2,0)}}\left(\sum_{x\in\mathbb{YF}_n} T_{w,\beta,n}(x,y)\right)\right)= 0.$$

    Теперь рассмотрим нечётные игреки.
    
    Зафиксируем какое-то $n\in\mathbb{N}_0:n\;mod\;{2}=1$.
    
    Если $y\in\mathbb{N}_0:$ $y \le n,$ то по Утверждению \ref{stolb} при $w\in\mathbb{YF}_\infty$, $\beta\in(0,1]$, $n,y\in\mathbb{N}_0$
    $$\sum_{x\in\mathbb{YF}_n} T_{w,\beta,n}(x,y)\le \left(\prod_{i=1}^{\left\lfloor \frac{n-y}{2} \right\rfloor} \frac{2i+y}{2i}\right)\beta^{y}\left(1-\beta^2\right)^{\left\lfloor \frac{n-y}{2} \right\rfloor}.$$
    
    Ясно, что мы можем просуммировать это выражение по $y\in{\overline{n}\{\beta,\varepsilon_1'\}(2,1)}$. Просуммируем:
    $$\sum_{y\in{\overline{n}\{\beta,\varepsilon_1'\}(2,1)}}\left(\sum_{x\in\mathbb{YF}_n} T_{w,\beta,n}(x,y)\right)\le\sum_{y\in{\overline{n}\{\beta,\varepsilon_1'\}(2,1)}}\left(\left( \prod_{i=1}^{\left\lfloor \frac{n-y}{2} \right\rfloor} \frac{2i+y}{2i}\right)\beta^{y}\left(1-\beta^2\right)^{\left\lfloor \frac{n-y}{2} \right\rfloor}\right)\le$$
    $$\le \sum_{y\in{\overline{n}\{\beta,\varepsilon_1'\}(2,1)}} \left(\left(\prod_{i=1}^{\left\lfloor \frac{n-y}{2} \right\rfloor} \frac{2i+y+1}{2i}\right)\beta^{y}\left(1-\beta^2\right)^{\left\lfloor \frac{n-y}{2} \right\rfloor}\right)=$$
    $$=\left(\text{Так как ясно, что если $n \;mod\;{2}=1$ и $y\;mod\;{2}=1$, то $\left\lfloor \frac{n-y}{2} \right\rfloor=\frac{n-y}{2}$}\right)=$$
    $$=\frac{1}{\beta}\sum_{y\in{\overline{n}\{\beta,\varepsilon_1'\}(2,1)}} \left(\left(\prod_{i=1}^{ \frac{n-y}{2}} \frac{2i+y+1}{2i}\right)\beta^{y+1}\left(1-\beta^2\right)^{ \frac{n-y}{2}}\right)=$$
    $$=\frac{1}{\beta}\sum_{y\in{\overline{n}\{\beta,\varepsilon_1'\}(2,1)}} \left(\left(\prod_{i=1}^{ \frac{n+1-2\frac{y+1}{2}}{2}} \frac{2i+2\frac{y+1}{2}}{2i}\right)\beta^{2\frac{y+1}{2}}\left(1-\beta^2\right)^{ \frac{n+1-2\frac{y+1}{2}}{2}}\right).$$
    
    Ясно, что если $y$ пробегает все значения в множестве ${{\overline{n}\{\beta,\varepsilon_1'\}(2,1)}}$ при $n\in\mathbb{N}_0:$ $n\;mod\;{2}=1$, то $\frac{y+1}{2}$ пробегает все значения в множестве $\overline{n}_{11}\{\beta,\varepsilon_1'\}\textbackslash \{0\}$ (просто по определению этого множества), то есть наше выражение равняется следующему: 
    $$\frac{1}{\beta}\sum_{y'\in\left(\overline{n}_{11}\{\beta,\varepsilon_1'\}\textbackslash \{0\}\right)} \left(\left(\prod_{i=1}^{ \frac{n+1-2y'}{2}} \frac{2i+2y'}{2i}\right)\beta^{2y'}\left(1-\beta^2\right)^{ \frac{n+1-2y'}{2}}\right)=$$
    $$=\frac{1}{\beta}\sum_{y'\in\left(\overline{n}_{11}\{\beta,\varepsilon_1'\}\textbackslash \{0\}\right)} \left(\left(\prod_{i=1}^{\frac{n+1}{2}-y' } \frac{i+y'}{i}\right)\beta^{2y'}\left(1-\beta^2\right)^{ \frac{n+1}{2}-y' }\right).$$
    
        Тут есть два случая:
    \begin{enumerate}
        \item $0\notin\overline{n}_{11}\{\beta,\varepsilon_1'\}$.
        
        В данном случае $\overline{n}_{11}\{\beta,\varepsilon_1'\}\textbackslash \{0\}=\overline{n}_{11}\{\beta,\varepsilon_1'\},$
        а значит
        $$\frac{1}{\beta}\sum_{y'\in\left(\overline{n}_{11}\{\beta,\varepsilon_1'\}\textbackslash \{0\}\right)} \left(\left(\prod_{i=1}^{\frac{n+1}{2}-y' } \frac{i+y'}{i}\right)\beta^{2y'}\left(1-\beta^2\right)^{ \frac{n+1}{2}-y' }\right)= $$
        $$=\frac{1}{\beta}\sum_{y'\in\overline{n}_{11}\{\beta,\varepsilon_1'\}} \left(\left(\prod_{i=1}^{\frac{n+1}{2}-y' } \frac{i+y'}{i}\right)\beta^{2y'}\left(1-\beta^2\right)^{ \frac{n+1}{2}-y' }\right). $$
    
        \item $0\in\overline{n}_{11}\{\beta,\varepsilon_1'\}$.
        
        В данном случае
        $$\frac{1}{\beta}\sum_{y'\in\left(\overline{n}_{11}\{\beta,\varepsilon_1'\}\textbackslash \{0\}\right)} \left(\left(\prod_{i=1}^{\frac{n+1}{2}-y' } \frac{i+y'}{i}\right)\beta^{2y'}\left(1-\beta^2\right)^{ \frac{n+1}{2}-y' }\right)\le \text{(Так как $\beta\in(0,1)$)}\le$$
        $$\le\frac{1}{\beta}\sum_{y'\in\left(\overline{n}_{11}\{\beta,\varepsilon_1'\}\textbackslash \{0\}\right)} \left(\left(\prod_{i=1}^{\frac{n+1}{2}-y' } \frac{i+y'}{i}\right)\beta^{2y'}\left(1-\beta^2\right)^{ \frac{n+1}{2}-y' }\right)+\frac{1}{\beta} \left(\prod_{i=1}^{\frac{n+1}{2} } \frac{i}{i}\right)\beta^{0}\left(1-\beta^2\right)^{ \frac{n+1}{2} }=$$
    $$=\frac{1}{\beta}\sum_{y'\in\left(\overline{n}_{11}\{\beta,\varepsilon_1'\}\textbackslash \{0\}\right)} \left(\left(\prod_{i=1}^{\frac{n+1}{2}-y' } \frac{i+y'}{i}\right)\beta^{2y'}\left(1-\beta^2\right)^{ \frac{n+1}{2}-y' }\right)+$$
    $$+\frac{1}{\beta}\sum_{y'=0}^{0} \left(\left(\prod_{i=1}^{\frac{n+1}{2}-y' } \frac{i+y'}{i}\right)\beta^{2y'}\left(1-\beta^2\right)^{ \frac{n+1}{2}-y' }\right)=$$
    $$=\frac{1}{\beta}\sum_{y'\in\overline{n}_{11}\{\beta,\varepsilon_1'\}} \left(\left(\prod_{i=1}^{\frac{n+1}{2}-y' } \frac{i+y'}{i}\right)\beta^{2y'}\left(1-\beta^2\right)^{ \frac{n+1}{2}-y' }\right).$$
    \end{enumerate}

    В обоих случаях (ясно, что других нет) наше выражение не превосходит
    $$\frac{1}{\beta}\sum_{y'\in\overline{n}_{11}\{\beta,\varepsilon_1'\}} \left(\left(\prod_{i=1}^{\frac{n+1}{2}-y' } \frac{i+y'}{i}\right)\beta^{2y'}\left(1-\beta^2\right)^{ \frac{n+1}{2}-y' }\right)=$$
    $$=\frac{1}{\beta}\sum_{y'\in\overline{n}_{11}\{\beta,\varepsilon_1'\}} \left(\left(\frac{\displaystyle\prod_{i=y'+1}^{\frac{n+1}{2}}i}{\displaystyle\prod_{i=1}^{\frac{n+1}{2}-y'}i}\right)\left(\beta^2\right)^{y'}\left(1-\beta^2\right)^{ \frac{n+1}{2}-y'}\right)=$$
    $$=\frac{1}{\beta}\sum_{y'\in\overline{n}_{11}\{\beta,\varepsilon_1'\}} \left(\binom{\frac{n+1}{2}}{y'}\left(\beta^2\right)^{y'}\left(1-\beta^2\right)^{ \frac{n+1}{2} -y'}\right).$$

    Ясно, что 
    $$\lim_{n\to\infty}^{(2,1)}\left(\frac{1}{\beta}\sum_{y'\in\overline{n}_{11}\{\beta,\varepsilon_1'\}} \left(\binom{\frac{n+1}{2}}{y'}\left(\beta^2\right)^{y'}\left(1-\beta^2\right)^{ \frac{n+1}{2} -y'}\right)\right)=$$
    $$=\lim_{n\to\infty}\left(\frac{1}{\beta}\sum_{y'\in\overline{2n+1}_{11}\{\beta,\varepsilon_1'\}} \left(\binom{n+1}{y'}\left(\beta^2\right)^{y'}\left(1-\beta^2\right)^{ n+1 -y'}\right)\right)=$$
    $$=\lim_{n\to\infty}\left(\frac{1}{\beta}\sum_{y'\in \overline{2n-1}_{11}\{\beta,\varepsilon_1'\}} \left(\binom{n}{y'}\left(\beta^2\right)^{y'}\left(1-\beta^2\right)^{n -y'}\right)\right)=$$
    $$=\left(\text{просто по определению множества $\overline{n}_{11,2}\{\beta,\varepsilon_1'\}$}\right)=$$
    $$=\lim_{n\to\infty}\left(\frac{1}{\beta}\sum_{y'\in \overline{n}_{11,2}\{\beta,\varepsilon_1'\}} \left(\binom{n}{y'}\left(\beta^2\right)^{y'}\left(1-\beta^2\right)^{n -y'}\right)\right).$$
    
    Ясно, что
    \begin{itemize}
        \item     $$\overline{n}_{11,2}\{\beta,\varepsilon\}= \overline{n}\textbackslash\left(\frac{(2n-1)\left(\beta^2-\varepsilon\right)+1}{2},\frac{(2n-1)\left(\beta^2+\varepsilon\right)+1}{2}\right);$$
        
                \item Если $n\in\mathbb{N}_0:n>2\left\lceil\frac{1}{\varepsilon_1'}\right\rceil$, то
        $$ \frac{\displaystyle\frac{(2n-1)\left(\beta^2-\varepsilon_1'\right)+1}{\displaystyle2}}{\displaystyle n}=\frac{{(2n-1)\left(\beta^2-\varepsilon_1'\right)+1}}{2n}=\beta^2-\varepsilon_1'+\frac{-\beta^2+\varepsilon_1'+1}{2n}< \beta^2-\varepsilon_1'+\frac{\varepsilon_1'}{2n}+\frac{1}{2n}< $$
        $$<\left(\text{так как мы рассматриваем случай $n\in\mathbb{N}_0:$ $n>2\left\lceil\frac{1}{\varepsilon_1'}\right\rceil>2$}\right)<$$
        $$<\beta^2-\varepsilon_1'+\frac{\varepsilon_1'}{4}+\frac{1}{2n}<\beta^2-\varepsilon_1'+\frac{\varepsilon_1'}{4}+\frac{1}{4\left\lceil\frac{1}{\varepsilon_1'}\right\rceil}<\beta^2-\varepsilon_1'+\frac{\varepsilon_1'}{4}+\frac{1}{4\frac{1}{\varepsilon_1'}}=\beta^2+\varepsilon_1'+\frac{\varepsilon_1'}{4}+\frac{\varepsilon_1'}{4}=\beta^2-\frac{\varepsilon_1'}{2}\Longrightarrow$$
        $$\Longrightarrow {\displaystyle\frac{(2n-1)\left(\beta^2-\varepsilon_1'\right)+1}{\displaystyle2}}<n\left(\beta^2-\frac{\varepsilon_1'}{2}\right);$$
        \item Если $n\in\mathbb{N}_0:n>2\left\lceil\frac{1}{\varepsilon_1'}\right\rceil$, то
        $$ \frac{\displaystyle \frac{(2n-1)\left(\beta^2+\varepsilon_1'\right)+1}{\displaystyle 2}}{\displaystyle n}=\frac{\displaystyle {(2n-1)\left(\beta^2+\varepsilon_1'\right)+1}}{\displaystyle 2n}=\beta^2+\varepsilon_1'+\frac{-\beta^2-\varepsilon_1'+1}{2n}>$$
        $$>\left(\text{Так как $\beta\in(0,1)$}\right)>\beta^2+\varepsilon_1'-\frac{\varepsilon_1'}{2n}>$$
        $$>\left(\text{так как мы рассматриваем случай $n\in\mathbb{N}_0:$ $n>2\left\lceil\frac{1}{\varepsilon_1'}\right\rceil>1$}\right)>$$
        $$>\beta^2+\varepsilon_1'-\frac{\varepsilon_1'}{2}=\beta^2+\frac{\varepsilon_1'}{2}\Longrightarrow$$
        $$\Longrightarrow {\displaystyle\frac{(2n-1)\left(\beta^2+\varepsilon_1'\right)}{\displaystyle2}}>n\left(\beta^2+\frac{\varepsilon_1'}{2}\right).$$
        \end{itemize}
    
        А значит если $n\in\mathbb{N}_0:$ $n>2\left\lceil\frac{1}{\varepsilon_1'}\right\rceil$, то 
        $$\overline{n}_{11,2}\{\beta,\varepsilon_1'\}\subset \overline{{n}}\textbackslash\left(n\left(\beta^2-\frac{\varepsilon_1'}{2}\right),n\left(\beta^2+\frac{\varepsilon_1'}{2}\right)\right)=\overline{n}\left\{\beta,\frac{\varepsilon_1'}{2}\right\}.$$

     А значит $\left(\text{так как $\beta\in(0,1)$}\right)$, если $n\in\mathbb{N}_0:$ $n>2\left\lceil\frac{1}{\varepsilon_1'}\right\rceil$, то
    $$\frac{1}{\beta}\sum_{y'\in \overline{n}\left\{\beta,\frac{\varepsilon_1'}{2}\right\}} \left(\binom{n}{y'}\left(\beta^2\right)^{y'}\left(1-\beta^2\right)^{n -y'}\right)\ge \frac{1}{\beta}\sum_{y'\in \overline{n}_{11,2}\{\beta,\varepsilon_1'\}} \left(\binom{n}{y'}\left(\beta^2\right)^{y'}\left(1-\beta^2\right)^{n -y'}\right)\ge 0.$$
     
    Ясно, что при $\beta\in(0,1)$ по закону распределения биномиальных коэффициентов
    $$\lim_{n\to\infty}\left(\frac{1}{\beta}\sum_{y'\in \overline{n}\left\{\beta,\frac{\varepsilon_1'}{2}\right\}} \left(\binom{n}{y'}\left(\beta^2\right)^{y'}\left(1-\beta^2\right)^{n -y'}\right)\right)=0,$$
    а значит, по Лемме о двух полицейских,
    $$\lim_{n\to\infty}\left(\frac{1}{\beta}\sum_{y'\in \overline{n}_{11,2}\{\beta,\varepsilon_1'\}} \left(\binom{n}{y'}\left(\beta^2\right)^{y'}\left(1-\beta^2\right)^{n' -y'}\right)\right)=0.$$

    А значит
    $$\lim_{n\to\infty}^{(2,1)}\left(\frac{1}{\beta}\sum_{y'\in\overline{n}_{11}\{\beta,\varepsilon_1'\}} \left(\binom{\frac{n+1}{2}}{y'}\left(\beta^2\right)^{y'}\left(1-\beta^2\right)^{ \frac{n+1}{2} -y'}\right)\right)=0.$$
    
    В силу доказанного выше, а также неотрицительности функции $T$, ясно, что при $n\in\mathbb{N}_0: n \; mod \; 2=1$
    $$0\le \sum_{{y\in\overline{n}\{\beta,\varepsilon_1'\}(2,1)}}\left(\sum_{x\in\mathbb{YF}_n} T_{w,\beta,n}(x,y)\right)\le\frac{1}{\beta} \sum_{y'\in \overline{n}_{11}\{\beta,\varepsilon_1'\} } \left(\binom{\frac{n+1}{2}}{y'}\left(\beta^2\right)^{y'}\left(1-\beta^2\right)^{ \frac{n+1}{2} -y'}\right).$$
    
    а значит, по Лемме о двух полицейских,
    $$\lim_{n\to\infty}^{(2,1)}\left(\sum_{{y\in\overline{n}\{\beta,\varepsilon_1'\}(2,1)}}\left(\sum_{x\in\mathbb{YF}_n} T_{w,\beta,n}(x,y)\right)\right)= 0.$$
        \end{enumerate}

Итак, подытожив написанное выше, получаем, что мы разбиваем последовательность $n\in\mathbb{N}_0$ на две подпоследовательности, а также, что при наших $w\in\mathbb{YF}_\infty$, $\beta\in(0,1)$, $\varepsilon,\varepsilon_1'\in\mathbb{R}_{>0}$
\begin{itemize}
    \item $\forall n\in\mathbb{N}_0$
    $$0\le \sum_{y\in \overline{n}\{\beta,\varepsilon_1'\}}\left({\sum_{v\in \overline{R}(w,\beta,n,\varepsilon)} T_{w,\beta,n}(v,y)}\right)\le\sum_{y\in \overline{n}\{\beta,\varepsilon_1'\}}\left({\sum_{x\in \mathbb{YF}_n} T_{w,\beta,n}(v,y)}\right)=$$
    $$=\left(\sum_{y\in \overline{n}\{\beta,\varepsilon_1'\}(2,0)}\left(\sum_{x\in\mathbb{YF}_n} T_{w,\beta,n}(x,y)\right)\right)+\left(\sum_{y\in \overline{n}\{\beta,\varepsilon_1'\}(2,1)}\left(\sum_{x\in\mathbb{YF}_n} T_{w,\beta,n}(x,y)\right)\right);$$
\item $$\lim_{n\to\infty}^{(2,0)}\left(\left(\sum_{y\in \overline{n}\{\beta,\varepsilon_1'\}(2,0)}\left(\sum_{x\in\mathbb{YF}_n} T_{w,\beta,n}(x,y)\right)\right)+\left(\sum_{y\in \overline{n}\{\beta,\varepsilon_1'\}(2,1)}\left(\sum_{x\in\mathbb{YF}_n} T_{w,\beta,n}(x,y)\right)\right)\right)=$$
$$=\lim_{n\to\infty}^{(2,0)}\left(\sum_{y\in \overline{n}\{\beta,\varepsilon_1'\}(2,0)}\left(\sum_{x\in\mathbb{YF}_n} T_{w,\beta,n}(x,y)\right)\right)+\lim_{n\to\infty}^{(2,0)}\left(\sum_{y\in \overline{n}\{\beta,\varepsilon_1'\}(2,1)}\left(\sum_{x\in\mathbb{YF}_n} T_{w,\beta,n}(x,y)\right)\right)=0;$$
\item $$\lim_{n\to\infty}^{(2,1)}\left(\left(\sum_{y\in \overline{n}\{\beta,\varepsilon_1'\}(2,0)}\left(\sum_{x\in\mathbb{YF}_n} T_{w,\beta,n}(x,y)\right)\right)+\left(\sum_{y\in \overline{n}\{\beta,\varepsilon_1'\}(2,1)}\left(\sum_{x\in\mathbb{YF}_n} T_{w,\beta,n}(x,y)\right)\right)\right)=$$
$$=\lim_{n\to\infty}^{(2,1)}\left(\sum_{y\in \overline{n}\{\beta,\varepsilon_1'\}(2,0)}\left(\sum_{x\in\mathbb{YF}_n} T_{w,\beta,n}(x,y)\right)\right)+\lim_{n\to\infty}^{(2,1)}\left(\sum_{y\in \overline{n}\{\beta,\varepsilon_1'\}(2,1)}\left(\sum_{x\in\mathbb{YF}_n} T_{w,\beta,n}(x,y)\right)\right)=0.$$
\end{itemize}
А значит, по Лемме о двух полицейских
\begin{itemize}
    \item $$\lim_{n\to\infty}^{(2,0)}\left(\sum_{y\in \overline{n}\{\beta,\varepsilon_1'\}}\left({\sum_{v\in \overline{R}(w,\beta,n,\varepsilon)} T_{w,\beta,n}(v,y)}\right)\right)= 0;$$
    \item $$\lim_{n\to\infty}^{(2,1)}\left(\sum_{y\in \overline{n}\{\beta,\varepsilon_1'\}}\left({\sum_{v\in \overline{R}(w,\beta,n,\varepsilon)} T_{w,\beta,n}(v,y)}\right)\right)= 0.$$
\end{itemize}
а из этого ясно, что если $w\in\mathbb{YF}_\infty$, $\beta\in(0,1)$, $\varepsilon,\varepsilon_1'\in\mathbb{R}_{>0}$, то
$$\sum_{y\in \overline{n}\{\beta,\varepsilon_1'\}}\left({\sum_{v\in \overline{R}(w,\beta,n,\varepsilon)} T_{w,\beta,n}(v,y)}\right)\xrightarrow{n\to\infty} 0,$$
что и требовалось.

Лемма доказана.

\end{proof}

Вернёмся к доказательству Теоремы. Вначале вспомним, что мы вообще доказываем:
$\forall w\in \mathbb{YF}_\infty^+$, $\beta \in(0,1),$ $ \varepsilon\in\mathbb{R}_{>0}$
\begin{enumerate}
    \item $$\lim_{n \to \infty}{\sum_{v\in \overline{R}(w,\beta,n,\varepsilon)}\mu_{w,\beta}(v)}=0;$$
    \item $$\lim_{n \to \infty}{\sum_{v\in R(w,\beta,n,\varepsilon)}\mu_{w,\beta}(v)}=1.$$
\end{enumerate}

Давайте доказывать: 

При наших $w\in \mathbb{YF}_\infty^+$, $\beta \in(0,1)$, $\varepsilon\in\mathbb{R}_{>0}$ и произвольном $n\in\mathbb{N}_0$
$$0\le \left(\text{По Следствию \ref{neotr} при наших $w\in \mathbb{YF}_\infty^+$, $\beta \in(0,1]$ и всех $v\in\overline{R}(w,\beta,n,\varepsilon)$}\right) \le{\sum_{v\in \overline{R}(w,\beta,n,\varepsilon)}\mu_{w,\beta}(v)}\le$$
$$\le \left(\text{По Утверждению \ref{zabe} при наших $w\in \mathbb{YF}_\infty$, $\beta \in(0,1)$, $n\in\mathbb{N}_0$ и всех $v\in\overline{R}(w,\beta,n,\varepsilon)\subseteq \mathbb{YF}$}\right) \le$$
$$\le \sum_{v\in \overline{R}(w,\beta,n,\varepsilon)}\left(\sum_{y=0}^{|v|} T_{w,\beta,n}(v,y)\right)=\sum_{v\in \overline{R}(w,\beta,n,\varepsilon)}\left(\sum_{y=0}^{n} T_{w,\beta,n}(v,y)\right)=\sum_{y=0}^{n}\left(\sum_{v\in \overline{R}(w,\beta,n,\varepsilon)} T_{w,\beta,n}(v,y)\right).$$

Зафиксируем произвольный $\overline{\varepsilon}\in\mathbb{R}>0$.

Пусть $\overline{\varepsilon'}=\frac{\overline{\varepsilon}}{3}$.

Заметим, что
\begin{itemize}
    \item (Лемма \ref{pohoronil}) При наших $w\in\mathbb{YF}_\infty^+$, $\beta\in(0,1)$, $\varepsilon\in\mathbb{R}_{>0}$ $\exists\varepsilon_3\in\mathbb{R}_{>0}:$ $\forall \varepsilon_1'\in\left(0,\beta^2\right),$ $ \overline{\varepsilon'}\in\mathbb{R}_{>0}$ $\exists N'\in\mathbb{N}_0$: $\forall n\in\mathbb{N}_0:$ $n\ge N'$
    $${\sum_{y\in \overline{n}[\beta,\varepsilon_1']}\left({\sum_{v\in \overline{R}(w,\beta,n,\varepsilon,y,\varepsilon_3)} T_{w,\beta,n}(v,y)}\right) <\overline{\varepsilon'}}=\frac{\overline{\varepsilon}}{3}.$$

    Зафиксируем данный $\varepsilon_3\in\mathbb{R}_{>0}$.
    
    \item (Лемма \ref{planedead}) При наших $w\in\mathbb{YF}_\infty$, $\beta\in(0,1)$, $\varepsilon,\varepsilon_3\in\mathbb{R}_{>0}$ $\exists\varepsilon_1'\in\left(0,\beta^2\right)$: $\forall \overline{\varepsilon'}\in\mathbb{R}_{>0}$ $\exists N''\in\mathbb{N}_0:$ $\forall n\in\mathbb{N}_0:$ $ n\ge N''$
    $${\sum_{y\in \overline{n}[\beta,\varepsilon_1']}\left(\left({\sum_{v\in \overline{R'}(w,\beta,n,\varepsilon,y,\varepsilon_3)} T_{w,\beta,n}(v,y)}\right)+\left({\sum_{v\in \overline{R''}(w,\beta,n,\varepsilon,y)} T_{w,\beta,n}(v,y)}\right)\right) <\overline{\varepsilon'}}=\frac{\overline{\varepsilon}}{3}.$$

    Зафиксируем данный $\varepsilon_1'\in\left(0,\beta^2\right)$.

\end{itemize}

Таким образом, сложив эти два факта, получаем, что при наших $w\in\mathbb{YF}_\infty^+$, $\beta\in(0,1)$, $\varepsilon,\varepsilon_3,\varepsilon_1'\in\mathbb{R}_{>0}$ $\exists N'''=\max(N',N'')\in\mathbb{N}_0:$ $\forall n\in\mathbb{N}_0:$ $n\ge N'''$ 
    $${\sum_{y\in \overline{n}[\beta,\varepsilon_1']}\left(\sum_{v\in \overline{R}(w,\beta,n,\varepsilon,y,\varepsilon_3)} T_{w,\beta,n}(v,y)+{\sum_{v\in \overline{R'}(w,\beta,n,\varepsilon,y,\varepsilon_3)} T_{w,\beta,n}(v,y)}+{\sum_{v\in \overline{R''}(w,\beta,n,\varepsilon,y,\varepsilon_3)} T_{w,\beta,n}(v,y)}\right) }< \frac{2\overline{\varepsilon}}{3}\Longleftrightarrow$$
$$\Longleftrightarrow(\text{По Замечанию $\ref{zhopa}$ при $w\in\mathbb{YF}_\infty$, $\beta\in(0,1]$, $n\in\mathbb{N}_0$, $\varepsilon\in\mathbb{R}_{>0}$, $y\in\mathbb{N}_0$, $\varepsilon_3\in\mathbb{R}_{>0}:$ $y\le n$})\Longleftrightarrow$$    
    $$\Longleftrightarrow{\sum_{y\in \overline{n}[\beta,\varepsilon_1']}\left(\sum_{v\in \overline{R}(w,\beta,n,\varepsilon)} T_{w,\beta,n}(v,y)\right) }< \frac{2\overline{\varepsilon}}{3}.$$

Таким образом, мы поняли, что при наших $w\in\mathbb{YF}_\infty^+$, $\beta\in(0,1)$, $\varepsilon,\varepsilon_1'\in\mathbb{R}_{>0}$ $\exists N'''\in\mathbb{N}_0:$ $\forall n\in\mathbb{N}_0:$ $n\ge N'''$ 
$${\sum_{y\in \overline{n}[\beta,\varepsilon_1']}\left(\sum_{v\in \overline{R}(w,\beta,n,\varepsilon)} T_{w,\beta,n}(v,y)\right) }< \frac{2\overline{\varepsilon}}{3}.$$

По Лемме \ref{hyde} при наших $w\in\mathbb{YF}_\infty$, $\beta\in(0,1)$, $\varepsilon,\varepsilon_1'\in\mathbb{R}_{>0}$
$${\sum_{y\in \overline{n}\{\beta,\varepsilon_1'\}}\left({\sum_{v\in \overline{R}(w,\beta,n,\varepsilon)} T_{w,\beta,n}(v,y)}\right) \xrightarrow{n\to \infty}0},$$
то есть, по определению предела, при наших $w\in\mathbb{YF}_\infty$, $\beta\in(0,1)$, $\varepsilon,\varepsilon_1'\in\mathbb{R}_{>0}$ $\exists N''''\in\mathbb{N}_0:$ $\forall n\in\mathbb{N}_0:$ $n\ge N''''$
$${\sum_{y\in \overline{n}\{\beta,\varepsilon_1'\}}\left({\sum_{v\in \overline{R}(w,\beta,n,\varepsilon)} T_{w,\beta,n}(v,y)}\right)}<\frac{\overline{\varepsilon}}{3}.$$

Итак, мы поняли, что
\begin{itemize}
    \item (Из Лемм \ref{pohoronil} и \ref{planedead}) При наших $w\in\mathbb{YF}_\infty^+$, $\beta\in(0,1)$, $\varepsilon,\varepsilon_1'\in\mathbb{R}_{>0}$ $\exists N'''\in\mathbb{N}_0:$ $\forall n\in\mathbb{N}_0:$ $n\ge N'''$
    $${\sum_{y\in \overline{n}[\beta,\varepsilon_1']}\left(\sum_{v\in \overline{R}(w,\beta,n,\varepsilon)} T_{w,\beta,n}(v,y)\right) }< \frac{2\overline{\varepsilon}}{3};$$

    \item (Из Леммы \ref{hyde}) При наших $w\in\mathbb{YF}_\infty$, $\beta\in(0,1)$, $\varepsilon,\varepsilon_1'\in\mathbb{R}_{>0}$ $\exists N''''\in\mathbb{N}_0:$ $\forall n\in\mathbb{N}_0:$ $n\ge N''''$
    $${\sum_{y\in \overline{n}\{\beta,\varepsilon_1'\}}\left({\sum_{v\in \overline{R}(w,\beta,n,\varepsilon)} T_{w,\beta,n}(v,y)}\right)}<\frac{\overline{\varepsilon}}{3}.$$
\end{itemize}

Таким образом, сложив эти два факта, получаем, что при наших $w\in\mathbb{YF}_\infty$, $\beta\in(0,1)$, $\varepsilon,\varepsilon_1'\in\mathbb{R}_{>0}$ $\exists N=\max(N''',N'''')\in\mathbb{N}_0:$ $\forall n\in\mathbb{N}_0:$ $n\ge N$
$$\left({\sum_{y\in \overline{n}[\beta,\varepsilon_1']}\left(\sum_{v\in \overline{R}(w,\beta,n,\varepsilon)} T_{w,\beta,n}(v,y)\right) }\right)+\left(
{\sum_{y\in \overline{n}\{\beta,\varepsilon_1'\}}\left({\sum_{v\in \overline{R}(w,\beta,n,\varepsilon)} T_{w,\beta,n}(v,y)}\right)}\right)<{\overline{\varepsilon}}\Longleftrightarrow$$
$$\Longleftrightarrow(\text{По Замечанию \ref{popa} при $n\in\mathbb{N}_0$, $\beta\in(0,1]$, $\varepsilon_1'\in\mathbb{R}_{>0}$})\Longleftrightarrow$$
$$\Longleftrightarrow {\sum_{y=0}^n\left(\sum_{v\in \overline{R}(w,\beta,n,\varepsilon)} T_{w,\beta,n}(v,y)\right) }<{\overline{\varepsilon}}.$$

Таким образом, мы поняли, что при наших $w\in\mathbb{YF}_\infty^+$, $\beta\in(0,1)$, $\varepsilon\in\mathbb{R}_{>0}$ $\exists N:$ $\forall n\in\mathbb{N}_0:$ $n\ge N$ 
$${\sum_{y=0}^n\left(\sum_{v\in \overline{R}(w,\beta,n,\varepsilon)} T_{w,\beta,n}(v,y)\right) }<{\overline{\varepsilon}}.$$

То есть в силу неотрицательности функции $T$
$$\lim_{n\to\infty}\left({\sum_{y=0}^n\left(\sum_{v\in \overline{R}(w,\beta,n,\varepsilon)} T_{w,\beta,n}(v,y)\right) }\right)= 0.$$

Мы уже поняли, что при наших $w\in\mathbb{YF}_\infty^+$, $\beta\in(0,1)$, $\varepsilon\in\mathbb{R}_{>0}$ и произвольном $ n\in\mathbb{N}_0$
$$0\le {\sum_{v\in \overline{R}(w,\beta,n,\varepsilon)}\mu_{w,\beta}(v)}\le\sum_{y=0}^{n}\left(\sum_{v\in \overline{R}(w,\beta,n,\varepsilon)} T_{w,\beta,n}(v,y)\right),$$
а значит, по Лемме о двух полицейских ясно, что при наших $w\in\mathbb{YF}_\infty^+$, $\beta\in(0,1)$, $\varepsilon\in\mathbb{R}_{>0}$
$$\lim_{n\to\infty} {\sum_{v\in \overline{R}(w,\beta,n,\varepsilon)}\mu_{w,\beta}(v)}=0,$$
что доказывает первый пункт.

Кроме того, ясно, что
\begin{itemize}
    \item $$\overline{R}\left(w,\beta,n,\varepsilon\right)\cup R\left(w,\beta,n,\varepsilon\right)=$$
    $$=\left\{v\in\mathbb{YF}_n:\; \pi(v)\notin(\pi(w)(\beta-\varepsilon),\pi(w)(\beta+\varepsilon)) \right\}
    \cup$$
    $$\cup \left\{v\in\mathbb{YF}_n:\; \pi(v)\in(\pi(w)(\beta-\varepsilon),\pi(w)(\beta+\varepsilon)) \right\}
    =\mathbb{YF}_n;$$
    \item $$\overline{R}\left(w,\beta,n,\varepsilon\right)\cap R\left(w,\beta,n,\varepsilon\right)=$$
    $$=\left\{v\in\mathbb{YF}_n: \pi(v)\notin(\pi(w)(\beta-\varepsilon),\pi(w)(\beta+\varepsilon)) \right\}
    \cap$$
    $$\cap \left\{v\in\mathbb{YF}_n: \pi(v)\in(\pi(w)(\beta-\varepsilon),\pi(w)(\beta+\varepsilon)) \right\}
    =\varnothing;$$
    \item (Следствие \ref{mera1}) $\forall w\in\mathbb{YF}_\infty^+,$ $\beta\in(0,1]$, $n\in\mathbb{N}_0$
    $${\sum_{v\in \mathbb{YF}_n}\mu_{w,\beta}(v)}=1.$$
\end{itemize}

А из этого очевидно, что
$$\lim_{n \to \infty}\sum_{v\in R(w,\beta,n,\varepsilon)}\mu_{w,\beta}(v)=1,$$
что доказывает второй пункт.

Таким образом, оба пункта доказаны.

Теорема доказана.
\end{proof}

\newpage

\section{Завершение доказательства гипотезы}

\begin{Col}\label{supermain1}
Пусть $\{w_i'\}_{i=1}^\infty\in\left(\mathbb{YF}\right)^\infty$, $w\in\mathbb{YF}_\infty^+,$ $\beta\in(0,1],$ $l\in\mathbb{N}_0:$  $\{w_i'\}\xrightarrow{i\to\infty}w$ и при этом существует предел
$$\lim_{m\to\infty}\frac{\pi(w'_m)}{\pi(w)}=\beta.$$
Тогда
\begin{enumerate}
    \item $$\lim_{n \to \infty}\sum_{v\in \overline{Q}(w,n,l)}\mu_{\{w_i'\}}(v)=0;$$
    \item $$\lim_{n \to \infty}\sum_{v\in {Q}(w,n,l)}\mu_{\{w_i'\}}(v)=1.$$
\end{enumerate}
\end{Col}
\begin{proof}

Начнём с первого пункта.

По обозначению
$$\lim_{n \to \infty}\sum_{v\in \overline{Q}(w,n,l)}\mu_{\{w_i'\}}(v)=\lim_{n \to \infty}\left(\sum_{v\in \overline{Q}(w,n,l)}\left(\lim_{m\to\infty}\mu_{\{w_i'\}}(v,m)\right)\right)=$$
$$=\left(\text{По Лемме \ref{yavsyo} при наших $\{w_i'\}_{i=1}^\infty\in\left(\mathbb{YF}\right)^\infty$, $w\in\mathbb{YF}_\infty^+,$ $\beta\in(0,1]$ и всех  $v\in\overline{Q}(w,n,l)\subseteq\mathbb{YF}$}\right)=$$
$$=\lim_{n \to \infty}\sum_{v\in \overline{Q}(w,n,l)}\mu_{w,\beta}(v).$$
Далее рассмотрим два случая:

\renewcommand{\labelenumi}{\arabic{enumi}$^\circ$}

\begin{enumerate}
    \item $\beta=1$.
    
    В данном случае 
    $$\lim_{n \to \infty}\sum_{v\in \overline{Q}(w,n,l)}\mu_{\{w_i'\}}(v)=\lim_{n \to \infty}\sum_{v\in \overline{Q}(w,n,l)}\mu_{w,\beta}(v)=\lim_{n \to \infty}\sum_{v\in \overline{Q}(w,n,l)}\mu_{w,1}(v)=$$
    $$=\left(\text{По Следствию \ref{granatakerambita} при нашем $w\in\mathbb{YF}_\infty$ и всех  $v\in\overline{Q}(w,n,l)\subseteq\mathbb{YF}$}\right)=\lim_{n \to \infty}\sum_{v\in \overline{Q}(w,n,l)}\mu_{w}(v)=$$
    $$=\text{(По Теореме \ref{main1} при наших $w\in\mathbb{YF}_\infty^+$, $l\in\mathbb{N}_0$)}=0,$$
    что доказывает первый пункт в данном случае.
    \item $\beta\in(0,1)$.
    
    В данном случае 
    $$\lim_{n \to \infty}\sum_{v\in \overline{Q}(w,n,l)}\mu_{\{w_i'\}}(v)=\lim_{n \to \infty}\sum_{v\in \overline{Q}(w,n,l)}\mu_{w,\beta}(v)=$$
    $$=\text{(По Теореме \ref{t2} при наших $w\in\mathbb{YF}_\infty^+$, $\beta\in(0,1)$, $l\in\mathbb{N}_0$)}=0,$$
    что доказывает первый пункт в данном случае.
    \end{enumerate}
    
    Ясно, что все случаи разобраны, первый пункт доказан.
    
    Перейдём ко второму пункту:
    
    По обозначению
$$\lim_{n \to \infty}\sum_{v\in {Q}(w,n,l)}\mu_{\{w_i'\}}(v)=\lim_{n \to \infty}\left(\sum_{v\in {Q}(w,n,l)}\left(\lim_{m\to\infty}\mu_{\{w_i'\}}(v,m)\right)\right)=$$
$$=\left(\text{По Лемме \ref{yavsyo} при наших $\{w_i'\}_{i=1}^\infty\in\left(\mathbb{YF}\right)^\infty$, $w\in\mathbb{YF}_\infty^+,$ $\beta\in(0,1]$ и всех  $v\in{Q}(w,n,l)\subseteq\mathbb{YF}$}\right)=$$
$$=\lim_{n \to \infty}\sum_{v\in {Q}(w,n,l)}\mu_{w,\beta}(v).$$
Далее рассмотрим два случая:
\begin{enumerate}
    \item $\beta=1$.
    
    В данном случае 
    $$\lim_{n \to \infty}\sum_{v\in {Q}(w,n,l)}\mu_{\{w_i'\}}(v)=\lim_{n \to \infty}\sum_{v\in {Q}(w,n,l)}\mu_{w,\beta}(v)=\lim_{n \to \infty}\sum_{v\in {Q}(w,n,l)}\mu_{w,1}(v)=$$
    $$=\left(\text{По Следствию \ref{granatakerambita} при нашем $w\in\mathbb{YF}_\infty$ и всех  $v\in {Q}(w,n,l)\subseteq\mathbb{YF}$}\right)=\lim_{n \to \infty}\sum_{v\in {Q}(w,n,l)}\mu_{w}(v)=$$
    $$=\text{(По Теореме \ref{main1} при наших $w\in\mathbb{YF}_\infty^+$, $l\in\mathbb{N}_0$)}=1,$$
    что доказывает второй пункт в данном случае.
    \item $\beta\in(0,1)$.
    
    В данном случае 
    $$\lim_{n \to \infty}\sum_{v\in {Q}(w,n,l)}\mu_{\{w_i'\}}(v)=\lim_{n \to \infty}\sum_{v\in {Q}(w,n,l)}\mu_{w,\beta}(v)=$$
    $$=\text{(По Теореме \ref{t2} при $w\in\mathbb{YF}_\infty^+$, $\beta\in(0,1)$, $l\in\mathbb{N}_0$)}=1,$$
    что доказывает второй пункт в данном случае.
    \end{enumerate}
    
    Ясно, что все случаи разобраны, второй пункт доказан.

    Следствие доказано.
    
    \end{proof}
\begin{Col}\label{supermain2}
Пусть $\{w_i'\}_{i=1}^\infty\in\left(\mathbb{YF}\right)^\infty$, $w\in\mathbb{YF}_\infty^+,$ $\beta\in(0,1]$, $\varepsilon\in\mathbb{R}_{>0}:$  $\{w_i'\}\xrightarrow{i\to\infty}w$ и при этом существует предел
$$\lim_{m\to\infty}\frac{\pi(w'_m)}{\pi(w)}=\beta.$$
Тогда
\begin{enumerate}
    \item $$\lim_{n \to \infty}\sum_{v\in \overline{R}(w,\beta,n,\varepsilon)}\mu_{\{w_i'\}}(v)=0;$$
    \item $$\lim_{n \to \infty}\sum_{v\in {R}(w,\beta,n,\varepsilon)}\mu_{\{w_i'\}}(v)=1.$$
\end{enumerate}
\end{Col}
\begin{proof}
Начнём с первого пункта.

По обозначению
$$\lim_{n \to \infty}\sum_{v\in \overline{R}(w,\beta,n,\varepsilon)}\mu_{\{w_i'\}}(v)=\lim_{n \to \infty}\left(\sum_{v\in \overline{R} (w,\beta,n,\varepsilon)}\left(\lim_{m\to\infty}\mu_{\{w_i'\}}(v,m)\right)\right)=$$
$$=\left(\text{По Лемме \ref{yavsyo} при наших $\{w_i'\}_{i=1}^\infty\in\left(\mathbb{YF}\right)^\infty$, $w\in\mathbb{YF}_\infty^+,$ $\beta\in(0,1]$, и всех  $v\in\overline{R} (w,\beta,n,\varepsilon)\subseteq\mathbb{YF}$}\right)=$$
$$=\lim_{n \to \infty}\sum_{v\in \overline{R} (w,\beta,n,\varepsilon)}\mu_{w,\beta}(v).$$
Далее рассмотрим два случая:

\renewcommand{\labelenumi}{\arabic{enumi}$^\circ$}

\begin{enumerate}
    \item $\beta=1$.
    
    В данном случае 
    $$\lim_{n \to \infty}\sum_{v\in \overline{R} (w,\beta,n,\varepsilon)}\mu_{\{w_i'\}}(v)=\lim_{n \to \infty}\sum_{v\in \overline{R} (w,\beta,n,\varepsilon)}\mu_{w,\beta}(v)=\lim_{n \to \infty}\sum_{v\in \overline{R} (w,\beta,n,\varepsilon)}\mu_{w,1}(v)=$$
    $$=\left(\text{По Следствию \ref{granatakerambita} при нашем $w\in\mathbb{YF}_\infty$ и всех  $v\in\overline{R}(w,\beta,n,l)\subseteq\mathbb{YF}$}\right)=\lim_{n \to \infty}\sum_{v\in \overline{R} (w,\beta,n,\varepsilon)}\mu_{w}(v)=$$
    $$=\left(\text{По Замечанию \ref{final} при наших $w\in\mathbb{YF}_\infty,$ $\varepsilon\in\mathbb{R}_{>0}$ и всех $n\in\mathbb{N}_0$}\right)=\lim_{n \to \infty}\sum_{v\in \overline{R} (w,n,\varepsilon)}\mu_{w}(v)=$$
    $$=\text{(По Теореме \ref{main2} при наших $w\in\mathbb{YF}_\infty^+$, $\varepsilon\in\mathbb{R}_{>0}$)}=0,$$
    что доказывает первый пункт в данном случае.
    \item $\beta\in(0,1)$.
    
    В данном случае 
    $$\lim_{n \to \infty}\sum_{v\in \overline{R} (w,\beta,n,\varepsilon)}\mu_{\{w_i'\}}(v)=\lim_{n \to \infty}\sum_{v\in \overline{R} (w,\beta,n,\varepsilon)}\mu_{w,\beta}(v)=$$
    $$=\text{(По Теореме \ref{t3} при наших $w\in\mathbb{YF}_\infty^+$, $\beta\in(0,1)$, $\varepsilon\in\mathbb{R}_{>0}$)}=0,$$
    что доказывает первый пункт в данном случае.
    \end{enumerate}
    
    Ясно, что все случаи разобраны, первый пункт доказан.
    
    Перейдём ко второму пункту:
    
    По обозначению
$$\lim_{n \to \infty}\sum_{v\in {R} (w,\beta,n,\varepsilon)}\mu_{\{w_i'\}}(v)=\lim_{n \to \infty}\left(\sum_{v\in {R} (w,\beta,n,\varepsilon)}\left(\lim_{m\to\infty}\mu_{\{w_i'\}}(v,m)\right)\right)=$$
$$=\left(\text{По Лемме \ref{yavsyo} при наших $\{w_i'\}_{i=1}^\infty\in\left(\mathbb{YF}\right)^\infty$, $w\in\mathbb{YF}_\infty^+,$ $\beta\in(0,1]$, и всех  $v\in {R} (w,\beta,n,\varepsilon)\subseteq\mathbb{YF}$}\right)=$$
$$=\lim_{n \to \infty}\sum_{v\in {R} (w,\beta,n,\varepsilon)}\mu_{w,\beta}(v).$$
Далее рассмотрим два случая:
\begin{enumerate}
    \item $\beta=1$.
    
    В данном случае 
    $$\lim_{n \to \infty}\sum_{v\in {R} (w,\beta,n,\varepsilon)}\mu_{\{w_i'\}}(v)=\lim_{n \to \infty}\sum_{v\in {R} (w,\beta,n,\varepsilon)}\mu_{w,\beta}(v)=\lim_{n \to \infty}\sum_{v\in {R} (w,\beta,n,\varepsilon)}\mu_{w,1}(v)=$$
    $$=\left(\text{По Следствию \ref{granatakerambita} при нашем $w\in\mathbb{YF}_\infty$ и всех  $v\in {R}(w,\beta,n,l)\subseteq\mathbb{YF}$}\right)=\lim_{n \to \infty}\sum_{v\in {R} (w,\beta,n,\varepsilon)}\mu_{w}(v)=$$
    $$=\left(\text{По Замечанию \ref{final} при наших $w\in\mathbb{YF}_\infty,$ $\varepsilon\in\mathbb{R}_{>0}$ и всех $n\in\mathbb{N}_0$}\right)=\lim_{n \to \infty}\sum_{v\in {R} (w,n,\varepsilon)}\mu_{w}(v)=$$
    $$=\text{(По Теореме \ref{main2} при наших $w\in\mathbb{YF}_\infty^+$, $\varepsilon\in\mathbb{R}_{>0}$)}=1,$$
    что доказывает второй пункт в данном случае.
    
    \item $\beta\in(0,1)$.
    
    В данном случае 
    $$\lim_{n \to \infty}\sum_{v\in {R} (w,\beta,n,\varepsilon)}\mu_{\{w_i'\}}(v)=\lim_{n \to \infty}\sum_{v\in {R}(w,\beta,n,\varepsilon)}\mu_{w,\beta}(v)=$$
    $$=\text{(По Теореме \ref{t3} при наших $w\in\mathbb{YF}_\infty^+$, $\beta\in(0,1)$, $\varepsilon\in\mathbb{R}_{>0}$)}=1,$$
    что доказывает второй пункт в данном случае.
    \end{enumerate}
    
    Ясно, что все случаи разобраны, второй пункт доказан.

    Следствие доказано.

\end{proof}
\begin{Col}[Из Следствий \ref{supermain1} и \ref{supermain2}]
Любая мера с границы Мартина графа Юнга--Фибоначчи эргодична.
\end{Col}
\begin{proof}
Рассмотрим центральную меру $\mu_{w_i'}=\mu_{w,\beta}$ при некоторых $w\in\mathbb{YF}_\infty^+,$ $\beta\in(0,1]$ на пространстве путей в графе Юнга--Фибоначчи: мера цилиндрического множества, соответствующего данному начальному отрезку пути от вершины $\varepsilon$ до $v$ равна $\frac{\mu_{w,\beta}(v)}{d(\varepsilon,v)}$.  

Мы доказали такое свойство этой меры: для любого $k$ существует такое $n_k$, что мера тех путей $u_0u_1\ldots$, у которых вершина $u_{n_k}$ имеет последние $k$ цифр не такие как у слова $w$ либо $|\pi(u_{n_k})-\beta|>\frac{1}{k},$ меньше чем $1/2^k$.

Пусть $A_m$ -- объединение множеств путей из предыдущего абзаца по $k=m,m+1,\ldots$; $B_m$ -- дополнение $A_m$. Тогда мера $A_m$ не больше чем $2/2^m$. Значит, пересечение $A_m$ имеет меру $0$ и почти все пути по нашей мере сосредоточены на множестве $B=\cup_m B_m$.

С другой стороны, по каждой из остальных мер множество $A_m$ имеет меру $1$: для мер вида $\mu_{w,\beta}$ это следует из того же утверждения, а для меры Планшереля из работы Керова — Гнедина\cite{KerGned}. Стало быть, множество $B$ имеет меру $0$.

Таким образом, если мера $\mu_{w,\beta}$ является смесью других мер, сужая на множество $B$ получаем противоречие.
\end{proof}

\newpage

\section{Благодарности}

\begin{itemize}
    \item Работа поддержана грантом в форме субсидий из федерального бюджета на создание и развитие международных математических центров мирового уровня, соглашение между МОН и ПОМИ РАН № 075-15-2019-1620 от 8 ноября 2019 г., а также грантом фонда поддержки теоретической физики и математики "БАЗИС", договор No 19-7-2-39-1 от 1 сентября 2019 г.
    \item Я признателен Фёдору Владимировичу Петрову за постановку задачи, помощь в публикации статьи и моральную поддержку на протяжении всего периода работы, Ивану Алексеевичу Бочкову за помощь в работе над первой частью цикла, а также Павлу Андреевичу Ходунову за проявленное при проверке доказательства терпение.
\end{itemize}

\newpage


\end{document}